\documentclass[12pt]{amsart}

\usepackage[paperheight=11in,paperwidth=8.5in,left=1in,top=1.25in,right=1in,bottom=1.25in]{geometry}
\usepackage[pdfstartpage={1},pdfstartview={FitH},pdftitle={},pdfauthor={},pdfsubject={},pdfcreator={},pdfproducer={},pdfkeywords={},pagebackref=false,colorlinks=false,linktocpage=true]{hyperref}

\usepackage{accents}
\usepackage{afterpage}
\usepackage{amssymb}
\usepackage{appendix}
\usepackage{array}
\usepackage{bm}
\usepackage{bbm}
\usepackage{cite}
\usepackage{enumitem}
\usepackage{graphicx}
\usepackage{indentfirst}
\usepackage{mathrsfs}
\usepackage{mathtools}
\usepackage{pgf}
\usepackage{tikz-cd}
\usetikzlibrary{matrix,arrows,backgrounds,positioning}
\allowdisplaybreaks[4]

\MHInternalSyntaxOn
\renewcommand{\dcases}
 {
  \MT_start_cases:nnnn
    {\quad}
    {$\m@th\displaystyle##$\hfil}
    {$\m@th\displaystyle##$\hfil}
    {\lbrace}
 }
\MHInternalSyntaxOff

\numberwithin{equation}{section}

\newtheorem{Thm}{Theorem}[section]
\newtheorem{thm}{Theorem}[subsection]
\newtheorem{cor}[thm]{Corollary}
\newtheorem{Cor}[Thm]{Corollary}
\newtheorem{lemma}[thm]{Lemma}
\newtheorem{prop}[thm]{Proposition}

\theoremstyle{definition}
\newtheorem{defn}[thm]{Definition}
\newtheorem{Defn}[Thm]{Definition}
\newtheorem{eg}[thm]{Example}

\theoremstyle{remark}
\newtheorem{rem}[thm]{Remark}
\newtheorem{Rem}[Thm]{Remark}

\newcommand{\N}{\mathbb{N}}

\newcommand{\R}{\mathbb{R}}
\newcommand{\C}{\mathbb{C}}

\newcommand{\wsc}{\mathcal{SC}}
\newcommand{\gn}{\mathcal{N}}

\newcommand{\prob}{\mathbb{P}}
\newcommand{\E}{\mathbb{E}}
\newcommand{\var}{\operatorname{Var}}

\newcommand{\salg}[1]{\mathcal{#1}}

\newcommand{\asto}{\stackrel{\operatorname{a.s.}}{\to}}
\newcommand{\wto}{\stackrel{w}{\to}}
\newcommand{\dto}{\stackrel{d}{\to}}
\newcommand{\deq}{\stackrel{d}{=}}
\newcommand{\teq}{\stackrel{\tau}{=}}

\newcommand{\indc}[1]{\mathbbm{1}\{#1\}}

\newcommand{\trace}{\text{tr}}

\newcommand{\source}{\operatorname{src}}
\newcommand{\target}{\operatorname{tar}}
\newcommand{\inn}[2]{\langle #1 , #2 \rangle}

\newcommand{\interior}[1]{\accentset{\circ}{#1}}

\newcommand{\wtilde}[1]{\widetilde{#1}}

\newcommand{\ssim}{\mathord{\sim}}

\newcommand{\vin}{v_{\operatorname{in}}}
\newcommand{\vout}{v_{\operatorname{out}}}
\newcommand{\mbf}[1]{\mathbf{#1}}
\newcommand{\mfk}[1]{\mathfrak{#1}}
\newcommand{\op}[1]{\operatorname{#1}}
\newcommand{\matn}{\op{Mat}_n}
\newcommand{\ssum}{\displaystyle\sum}

\newcommand{\sgps}{*$-$\text{graph polynomials}}
\newcommand{\ssgm}[1]{\salg{G}\langle\mbf{#1}, \mbf{#1}^*\rangle}
\newcommand{\ssgp}[1]{\C\salg{G}\langle\mbf{#1}, \mbf{#1}^*\rangle}

\newcommand*{\Scale}[2][4]{\scalebox{#1}{$#2$}}

\DeclareRobustCommand{\SkipTocEntry}[5]{} 

\begin{document}
\author{Benson Au}
\date{\today}
\thanks{Research partially supported by a Julia B. Robinson Graduate Fellowship in Mathematics and by NSF grants DMS-0907630 and DMS-1512933}
\title{Traffic Distributions of Random Band Matrices}

\address{University of California, Berkeley\\
         Department of Mathematics\\
         970 Evans Hall \#3840\\
         Berkeley, CA 94720-3840}
\email{\href{mailto:bensonau@math.berkeley.edu}{bensonau@math.berkeley.edu}}


\begin{abstract}\label{abstract}
We study random band matrices within the framework of traffic probability, an operadic non-commutative probability theory introduced by Male based on graph operations. As a starting point, we revisit the familiar case of the permutation invariant Wigner matrices and compare the situation to the general case in the absence of this invariance. Here, we find a departure from the usual free probabilistic universality of the joint distribution of independent Wigner matrices. We then show how the traffic space of Wigner matrices completely realizes the traffic central limit theorem. We further prove general Markov-type concentration inequalities for the joint traffic distribution of independent Wigner matrices. We then extend our analysis to random band matrices, as studied by Bogachev, Molchanov, and Pastur, and investigate the extent to which the joint traffic distribution of independent copies of these matrices deviates from the Wigner case.
\end{abstract}

\maketitle
\tableofcontents

\section{Introduction and main results}\label{intro}

For a real symmetric (or complex Hermitian) $n \times n$ matrix $\mbf{A}_n$, let $(\lambda_k(\mbf{A}_n) : 1 \leq k \leq n)$ denote the eigenvalues of $\mbf{A}_n$, counting multiplicity, arranged in a non-increasing order. We write $\mu(\mbf{A}_n)$ for the empirical spectral distribution (or ESD for short) of $\mbf{A}_n$, i.e.,
\[
\mu(\mbf{A}_n) = \frac{1}{n}\sum_{k=1}^n \delta_{\lambda_k(\mbf{A}_n)}, \qquad \lambda_1(\mbf{A}_n) \geq \cdots \geq \lambda_n(\mbf{A}_n).
\]
For a random matrix $\mbf{A}_n$, the ESD $\mu(\mbf{A}_n)$ then becomes a random probability measure on the real line $(\R,\salg{B}(\R))$. Wigner initiated the modern study of random matrices by proving the weak convergence of the ESD in expectation as the dimension $n \to \infty$ for a general class of random real symmetric matrices \cite{Wig55,Wig58}. We recall the so-called Wigner matrices, formulated deliberately in such a way below in order to suit our purposes later.

\begin{Defn}[Wigner matrix]\label{defn1.1}
Let $(X_{i, j})_{1 \leq i < j < \infty}$ and $(X_{i, i})_{1 \leq i < \infty}$ be independent families of i.i.d. random variables: the former, real-valued (resp., complex-valued), centered, and of unit variance; the latter, real-valued and of finite variance, i.e.,
\begin{equation}\label{eq:1.1}
\E X_{i, j} = 0, \quad \var(X_{i, j}) = \E |X_{i, j}|^2 = 1, \quad \text{and} \quad \var(X_{i, i})<\infty.
\end{equation}
Taken together, the two families $(X_{i, j})$ and $(X_{i, i})$ define a random real symmetric (resp., complex Hermitian) $n \times n$ matrix $\mbf{X}_n$ in a natural way, viz.
\[
\mbf{X}_n(i,j) = 
\begin{dcases}
X_{i,j} &\quad \text{if } i < j, \\
X_{i,i} &\quad \text{if } i=j.
\end{dcases}
\]
We call $\mbf{X}_n$ an \emph{unnormalized real (resp., complex) Wigner matrix}.

We introduce the standard normalization via a Hadamard-Schur product: let $\mbf{J}_n$ denote the $n \times n$ all-ones matrix, and define $\mbf{N}_n = n^{-1/2}\mbf{J}_n$. We call the random real symmetric (resp., complex Hermitian) $n \times n$ matrix $\mbf{W}_n$ defined by
\[
\mbf{W}_n = \mbf{N}_n \circ \mbf{X}_n = \frac{1}{\sqrt{n}}\mbf{X}_n
\]
a \emph{normalized real (resp., complex) Wigner matrix}. We simply refer to \emph{Wigner matrices} when the context is clear, or when considering the definition altogether.

We define the \emph{parameter} $\beta$ of a Wigner matrix as the pseudo-variance of its unnormalized strictly upper triangular entries so that
\[
\beta = \E \mbf{X}(i,j)^2 = \E X_{i, j}^2, \qquad \forall i < j.
\]
We note that a Wigner matrix is a real Wigner matrix iff its parameter $\beta=1$, and so we can specify a Wigner matrix by its parameter. We further note that the distribution of a Wigner matrix is invariant under conjugation by the permutation matrices iff its parameter $\beta \in \R$. This in turn is equivalent to the real and imaginary parts of $X_{i, j}$ being uncorrelated.

We often restrict to a special class of random variables within our Wigner matrices; thus, if $(X_{i, j})$ and $(X_{i, i})$ have finite moments of all orders, then we call both $\mbf{X}_n$ and  $\mbf{W}_n$ \emph{finite-moment Wigner matrices}. We distinguish the important case in which $(X_{i, j})$ and $(X_{i, i})$ are Gaussian by the term \emph{Gaussian Wigner matrix}.
\end{Defn}

In particular, Wigner identified the standard semicircle distribution $\mu_{SC}$ as the universal limiting spectral distribution (or LSD for short) of the Wigner matrices, where
\[
\mu_{SC}(dx) = \frac{1}{2\pi}(4-x^2)_+^{1/2}\, dx.
\]
Considerable work has since been done on the Wigner matrices and other classical random matrix ensembles, e.g., on questions related to maximal eigenvalues, central limit theorems, concentration inequalities, joint eigenvalue distribution, large deviations, eigenvalue spacing, and free probability. The recent monograph \cite{AGZ10} by Anderson, Guionnet, and Zeitouni provides an excellent introduction to this end.

Free probability, introduced by Voiculescu \cite{Voi85}, explains the distinguished role of the semicircle distribution. Motivated by the study of free group factors, Voiculescu defined a suitable notion of independence in the non-commutative probabilistic setting known as \emph{free independence}. Free analogues of classical constructions and operations from (commutative) probability theory abound: for example, the free central limit theorem (CLT), free convolution, free cumulants, and free entropy. In particular, the semicircle distribution, being the attractor in the free CLT, serves as the free analogue of the normal distribution. As with a normal distribution $\gn(m, \sigma^2)$, we specify a semicircular distribution by its mean $m$ and variance $\sigma^2$, writing $\wsc(m, \sigma^2)$ for the distribution
\[
\wsc(m, \sigma^2)(dx) = \frac{1}{2\pi\sigma^2}(4\sigma^2 - (x-m)^2)_+^{1/2}\, dx.
\]

Voiculescu showed that free independence describes the asymptotic behavior of the ESD for a large class of random matrices, such as those invariant in distribution under conjugation by the orthogonal matrices (in the real symmetric case) or the unitary matrices (in the complex Hermitian case) \cite{Voi91}. Wigner's semicircle law can thus be seen as a consequence of the free CLT. We refer the reader to the standard introductions to free probability \cite{VDN92,NS06}. 

On the other hand, many random matrix models of interest do not possess the aforementioned invariance, e.g., the adjacency matrices of random graphs. This consideration led Male to introduce a non-commutative probability theory that describes the asymptotic behavior of matrices invariant in distribution under conjugation by the permutation matrices \cite{Mal11}. Male termed the corresponding notion of independence \emph{traffic independence} and proved a traffic CLT that interpolates between the free CLT and the classical CLT. The family of free convolutions $(\wsc(0,\sigma_1^2) \boxplus \gn(0,\sigma_2^2): \sigma_1^2+\sigma_2^2=1)$ form the attractors in the traffic CLT, where the parameters $\sigma_i^2$ depend on the particular random variables in consideration. We identify a random matrix ensemble that plays the role of the Wigner matrices in the traffic setting, giving a complete realization to this interpolation.

In his review article \cite{Bai99}, Bai proposed the study of random matrix ensembles with additional linear algebraic structure: the random Hankel, Markov, and Toeplitz matrices. Bryc, Dembo, and Jiang proved the almost sure convergence of the ESD to certain universal distributions for these matrices \cite{BDJ06} (see also \cite{HM05}). We focus on the (random) Markov matrices, using a modified definition in order to absorb the dimensional normalization.

\begin{Defn}[Markov matrix]\label{defn1.2}
Let $(X_{i,j})_{1 \leq i < j < \infty}$ and $(X_{i,i})_{1 \leq i < \infty}$ be real-valued random variables as in Definition \hyperref[defn1.1]{1.1}. For $1 \leq i < j$, define $X_{j,i} = X_{i,j}$. We write $\mbf{W}_n$ for the corresponding Wigner matrix and $\mbf{D}_n=\op{deg}(\mbf{W}_n)$ for the diagonal matrix of row sums of $\mbf{W}_n$ so that
\[
\mbf{D}_n(i,i) = \sum_{j=1}^n \mbf{W}_n(i,j) = \frac{1}{\sqrt{n}} \sum_{j=1}^n X(i,j).
\]
We call the random real symmetric $n\times n$ matrix $\mbf{M}_n$ defined by
\[
\mbf{M}_n = \mbf{W}_n - \mbf{D}_n = \frac{1}{\sqrt{n}}
\begin{pmatrix}
-\ssum_{j\neq 1}^n X_{1,j} & X_{1,2}                   & X_{1,3} & \cdots                  &        & X_{1,n}                 \\
X_{2,1}                  & -\ssum_{j\neq 2}^n X_{2,j}  & X_{2,3} & \cdots                  &        & X_{2,n}                 \\
X_{3,1}                  & X_{3,2}                   & \ddots &                         &        & \vdots                 \\
\vdots                  & \vdots                   &        &                         &        &                        \\
X_{k,1}                  & X_{k,2}                   & \cdots & -\ssum_{j\neq k}^n X_{k,j}  & \cdots & X_{k,n}                  \\
\vdots                  & \vdots                   &        &                         & \ddots & \vdots                  \\
X_{n,1}                  & X_{n,2}                   & \cdots &                         &        & -\ssum_{j\neq n}^n X_{n,j}  
\end{pmatrix}
\]
a \emph{Markov matrix}.
\end{Defn}

\begin{Rem}\label{rem1.3}
The ``Markov'' in Definition \hyperref[defn1.2]{1.2} comes from the zero row-sums property of the matrix, a property shared by the infinitesimal generator of a continuous-time Markov process on a finite state space. This class of matrices also includes the graph Laplacian; however, the Laplacian of a random graph does not fit into our model as equation \eqref{eq:1.1} precludes the a.s.\@ non-negativity of $X_{i,j}$ (see, e.g., the works of Ding and Jiang \cite{DJ10} and Jiang \cite{Jia12.1,Jia12.2} on this related problem).
\end{Rem}

Theorem 1.3 in \cite{BDJ06} shows that the sequence of ESDs $\{\mu(\mbf{M}_n)\}_{n = 1}^\infty$ converges weakly almost surely to the free convolution $\wsc(0,1) \boxplus \gn(0,1)$, for which the authors give two proofs. The first proof relies on the method of moments, using a combinatorial characterization of the moments of $\wsc(0,1) \boxplus \gn(0,1)$ effected by the machinery of Bo{\.z}ejko and Speicher \cite{BS96}. The same machinery applies more generally to the free convolutions $\wsc(0,\sigma_1^2) \boxplus \gn(0,\sigma_2^2)$, but this characterization becomes unwieldy when $\sigma_1^2\neq\sigma_2^2$. The second proof relies on a comparison method, showing that the expected moments of $\mu(\mbf{M}_n)$ are asymptotically equivalent to the expected moments of $\mu(\mbf{M}_n')$, where $\mbf{M}_n' = \mbf{W}_n' - \mbf{D}_n'$ can be written as the difference of two \emph{independent} matrices. One can then appeal to a result of Pastur and Vasilchuk \cite{PV00} to prove the aforementioned convergence for the surrogate $\{\mu(\mbf{M}_n')\}_{n = 1}^\infty$.
 
The techniques for dealing with dependent random matrices can be often quite ad hoc; however, traffic probability provides a unifying framework for a large class of such matrices. In particular, the Markov matrices fit quite naturally into this framework, wherein they can be realized as graph polynomials of Wigner matrices. More generally, for $p,q\in\R$, let $\mbf{M}_{n,p,q} = p\mbf{W}_n + q\mbf{D}_n$, where $\mbf{W}_n$ and $\mbf{D}_n$ are as in the definition of a Markov matrix. Accordingly, we call the random real symmetric $n\times n$ matrix $\mbf{M}_{n,p,q}$ a \emph{$(p,q)$-Markov matrix} after Definition \hyperref[defn1.2]{1.2}. We show that independent finite-moment $(p,q)$-Markov matrices are asymptotically traffic independent with a stable universal limiting traffic distribution (or LTD for short). This allows us to pair a convenient Gaussian realization of our ensemble with the traffic CLT to show that the ESDs $\mu(\mbf{M}_{n,p,q})$ converge weakly almost surely to the free convolution $\wsc(0,p^2)\boxplus\gn(0,q^2)$, extending the result of Bryc, Dembo, and Jiang.

Free convolutions with semicircular distributions enjoy nice regularity properties. In particular, the work \cite{Bia97} of Biane implies that $\wsc(0,p^2)\boxplus\gn(0,q^2)$ is absolutely continuous with a bounded, continuous density (Corollary 2 and Proposition 5) that is analytic off of its zero set (Corollary 4). Moreover, Proposition A.3 in \cite{BDJ06} can be easily adapted to show that $\wsc(0,p^2)\boxplus\gn(0,q^2)$ has bounded support iff $q = 0$.

The weak convergence $\mu(\mbf{M}_{n,p,q}) \wto \wsc(0,p^2) \boxplus \gn(0,q^2)$ for all $p,q\in\R$ suggests that the matrices $\mbf{W}_n$ and $\mbf{D}_n$ are asymptotically free despite the fact that the latter matrix is completely determined by the former. One can prove that freeness does in fact govern the asymptotic behavior of $\mbf{W}_n$ and $\mbf{D}_n$ by working with the LTD of the Wigner matrices and exploiting the relationship between traffic independence and free independence; however, this behavior can be better seen as part of a much more general phenomenon for graph polynomials of random matrices.

For a tracial $*$-probability space $(\salg{A},\varphi)$, C\'{e}bron, Dahlqvist, and Male constructed a universal enveloping traffic space $(\salg{G}(\salg{A}),\tau)$ that extends the trace \cite{CDM16}. The authors further proved a coherent convergence property for $(\salg{G}(\salg{A}),\tau)$: if a family of unitarily invariant random matrices $\mathfrak{A}_n$ converges in $*$-distribution to a family of random variables $\mathfrak{a}$ in $(\salg{A},\varphi)$ and further satisfies a certain factorization property, then $\mathfrak{A}_n$ converges in traffic distribution to $\mathfrak{a}$ in $(\salg{G}(\salg{A}),\tau)$. This construction comes equipped with a canonical (free) independence structure: in the forthcoming work \cite{AM17}, we show that the traffic space $(\salg{G}(\salg{A}),\tau)$, regarded simply as a $*$-probability space, can be realized as the free product (in the sense of Voiculescu) of three natural unital $*$-subalgebras. Taken together with the results of \cite{Mal11,CDM16}, this gives another proof of the asymptotic freeness of $\mbf{W}_n$ and $\mbf{D}_n$.

Yet, in both cases, we rely crucially on the strong invariance property of our ensemble. The universality of \emph{non-invariant} ensembles constitutes a major ongoing program of research. We recall one prominent model of interest: the random band matrices.

\begin{Defn}[Band matrix]\label{defn1.4}
Let $(b_n)$ be a sequence of nonnegative integers. We write $\mbf{B}_n$ for the corresponding $n \times n$ band matrix of ones with band width $b_n$, i.e.,
\[
\mbf{B}_n(i,j) = \indc{|i-j| \leq b_n}.
\]
Let $\mbf{X}_n$ be an unnormalized Wigner matrix. We call the random matrix $\mbf{\Xi}_n$ defined by
\[
\mbf{\Xi}_n = \mbf{B}_n \circ \mbf{X}_n
\]
an \emph{unnormalized random band matrix}.  We introduce a normalization based on the growth rate of the band width $b_n$. We say that $(b_n)$ is of \emph{slow growth} (resp., \emph{proportional growth}) if
\[
\lim_{n \to \infty} b_n = \infty \quad \text{and} \quad b_n = o(n) \qquad (\text{resp., } \lim_{n \to \infty} \frac{b_n}{n} = c \in (0,1]),
\]
in which case we use the normalization
\[
\mbf{\Upsilon}_n = (2b_n)^{-1/2}\mbf{J}_n \qquad (\text{resp., } \mbf{\Upsilon}_n = (2c-c^2)^{-1/2}n^{-1/2}\mbf{J}_n).
\]
We call $c$ the \emph{proportionality constant}: we say that $(b_n)$ is of \emph{full proportion} if $c=1$ and \emph{proper} otherwise. For a fixed band width $b_n \equiv b$, we use the normalization $\mbf{\Upsilon}_n = (2b+1)^{-1/2}\mbf{J}_n$. In any case, we call the random matrix $\mbf{\Theta}_n$ defined by
\[
\mbf{\Theta}_n = \mbf{\Upsilon}_n \circ \mbf{\Xi}_n
\]
a \emph{normalized random band matrix}. We simply refer to random band matrices (or RBMs for short) when the context is clear, or when considering the definition altogether.
\end{Defn}

Following Wigner, one expects universality to hold for any large quantum system of sufficient complexity (see \cite{Meh04} for more on this perspective; see \cite{BEYY16,EY16} and the references therein for progress in this direction). In particular, a fundamental conjecture of Fyodorov and Mirlin proposes a dichotomy for the local spectral statistics of RBMs \cite{FM91}: random matrix theory statistics (weak disorder) for large band widths; Poisson statistics (strong disorder) for small band widths; and a sharp transition around the critical value $b_n = \sqrt{n}$ (again, we refer the reader to \cite{BEYY16,EY16} for progress in this direction).

At the macroscopic level, Bogachev, Molchanov, and Pastur proved that the class of band widths in Definition \hyperref[defn1.4]{1.4} determine the global universality classes of the RBMs \cite{BMP91}: for slow growth RBMs, $\mu(\mbf{\Theta}_n)$ converges to the semicircle distribution $\mu_{SC}$; for proportional growth RBMs of proper proportion, $\mu(\mbf{\Theta}_n)$ converges to a non-semicircular distribution $\mu_c$ of bounded support; and for fixed band width RBMs having a symmetric distribution for the entries, $\mu(\mbf{\Theta}_n)$ converges to a non-universal symmetric distribution $\mu_b$. The authors further proved a continuity result for these distributions, namely,
\begin{equation}\label{eq:1.2}
\lim_{c \to 0^+} \mu_c = \lim_{c \to 1^-} \mu_c = \mu_{SC} \quad \text{and} \quad \lim_{b \to \infty} \mu_b = \mu_{SC}.
\end{equation}

The work \cite{BMP91} considered the distribution of a single RBM: naturally, this invites the question of the joint distribution of such matrices. Shlyakhtenko showed that freeness \emph{with amalgamation} in the context of \emph{operator-valued} free probability governs what he called Gaussian RBMs \cite{Shl96}; otherwise, to our knowledge, RBMs have not received much attention from the non-commutative probabilistic perspective. Nevertheless, we show that the framework of traffic probability allows for effective, tractable computations in multiple RBMs. Our main result identifies the joint LTD of independent RBMs of possibly mixed band width types.
\begin{Thm}\label{thm1.5}
Let $\salg{X}_n = (\mbf{X}_n^{(i)})_{i \in I}$ be a family of independent unnormalized finite-moment Wigner matrices. We assume that the parameters $\beta_i \in \R$ and write $\salg{W}_n = (\mbf{W}_n^{(i)})_{i \in I}$ for the corresponding family of normalized Wigner matrices. Consider a family of band widths
\[
(b_n^{(i)})_{i \in I} = (b_n^{(i)})_{i \in I_1} \cup (b_n^{(i)})_{i \in I_2} \cup (b_n^{(i)})_{i \in I_3} \cup (b_n^{(i)})_{i \in I_4}
\]
of slow growth, proper proportion, full proportion, and fixed band width respectively, and form the corresponding family of normalized RBMs $\salg{O}_n = (\mbf{\Theta}_n^{(i)})_{i \in I}$. Then the family $\salg{O}_n$ converges in traffic distribution. In fact, the LTDs of the families $(\mbf{\Theta}_n^{(i)})_{i \in I_1 \cup I_3}$ and $(\mbf{W}_n^{(i)})_{i \in I_1 \cup I_3}$ are identical, the latter already being known from \cite{Mal11}. 
\end{Thm}

The precise form of this LTD requires a good deal of preparation, and we do not state it here in the introduction (see Theorems \hyperref[thm4.3.3]{4.3.3} and \hyperref[thm4.4.1]{4.4.1}). Instead, we opt for a more familiar free probabilistic statement.

Knowledge of the traffic distribution, which is defined in terms of graph observables, can be difficult to interpret; however, the traffic distribution does encode the information of the usual $*$-distribution. For example, as a consequence of our earlier discussion on the Wigner matrices, we immediately obtain the following corollary:  

\begin{Cor}\label{cor1.6}
The family $(\mbf{\Theta}_n^{(i)})_{i \in I_1 \cup I_3}$ converges in distribution to a semicircular system. If we further assume that $\beta_i = 1$, then the augmented family $(\mbf{\Theta}_n^{(i)}, \deg(\mbf{\Theta}_n^{(i)}))_{i \in I_1 \cup I_3}$ also converges in distribution, where $(\mbf{\Theta}_n^{(i)})_{i \in I_1 \cup I_3}$ and $(\deg(\mbf{\Theta}_n^{(i)}))_{i \in I_1 \cup I_3}$ are asymptotically free and $(\deg(\mbf{\Theta}_n^{(i)}))_{i \in I_1 \cup I_3}$ converges in distribution to a Gaussian system.
\end{Cor}

\begin{Rem}\label{rem1.7}
We do not make any assumptions on the relative rates of growth for the band widths $(b_n^{(i)})_{i \in I_1}$; thus, for example, it could be that $(b_n^{(i_1)}, b_n^{(i_2)}, b_n^{(i_3)}, b_n^{(i_4)})$ are each of slow growth with $b_n^{(i_1)}, b_n^{(i_2)} \ll \sqrt{n} \ll b_n^{(i_3)}, b_n^{(i_4)}$. In particular, perhaps not surprisingly, we fail to observe any sort of transition around the conjectured critical value for the local spectral statistics at the level of first order freeness.
\end{Rem}

In fact, Theorem \hyperref[thm1.5]{1.5} allows us to translate any statement about the limiting distribution (or, more generally, the limiting traffic distribution) of the Wigner matrices $(\mbf{W}_n^{(i)})_{i \in I_1 \cup I_3}$ to the RBMs $(\mbf{\Theta}_n^{(i)})_{i \in I_1 \cup I_3}$, culminating in Theorem \hyperref[thm4.3.6]{4.3.6}, which holds for general parameters $\beta_i \in \C$. This line of investigation, particularly along the traffic distribution, is further pursued in \cite{AM17}.

At the same time, our result shows that the traffic distribution, despite all of its additional structure, falls short of capturing even other macroscopic features. In particular, Theorem 3 in \cite{BMP91} implies that $\lambda_1(\mbf{\Theta}_n) \asto \infty$ for slow growth finite-moment RBMs, whereas Bai and Yin showed that $\lambda_1(\mbf{W}_n) \asto 2$ iff the entries of $\mbf{X}_n$ have finite fourth moments \cite{BY88}.  

Unfortunately, traffic probability has less to say about proportional growth RBMs and less still about fixed band width RBMs. We show that independent proportional growth (resp., fixed band width) RBMs are not asymptotically traffic independent unless $c = 1$ (resp., $b=0$). Nonetheless, we prove the traffic analogue of equation \eqref{eq:1.2}, showing that the continuity of the LSD in the band width extends to the LTD as well. Here, we find a subtle difference in how these limits are attained, leading into our analysis of mixed band width types.

We organize the paper as follows. Section \hyperref[sec2]{2} provides the necessary background in traffic probability following \cite{Mal11,CDM16}. Section \hyperref[sec3]{3} contains the results in our motivating discussion on the Wigner matrices. We further prove general Markov-type concentration inequalities for the traffic distribution of independent Wigner matrices, which allows us to upgrade our convergence to the almost sure sense. Section \hyperref[sec4]{4} treats the case of the RBMs, beginning with a preliminary version of our main result for \emph{periodic} RBMs. We work throughout in the context of finite-moment Wigner matrices $\mbf{X}_n$ with a slightly more general model that replaces the identically distributed assumption with a strong uniform control on the moments. Finally, we gather some miscellaneous results in the appendix.

\begin{center}
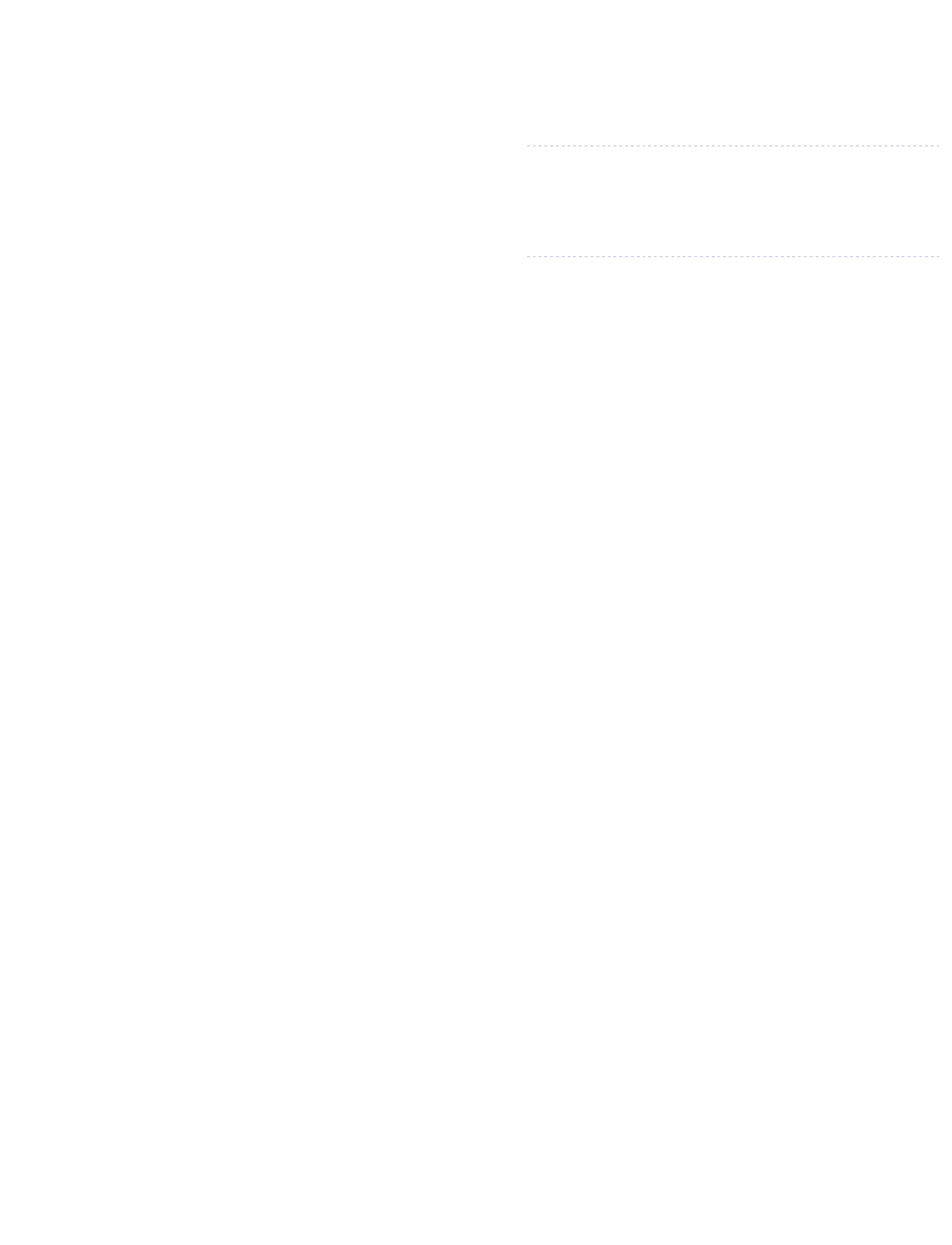
\end{center}

\begin{center}
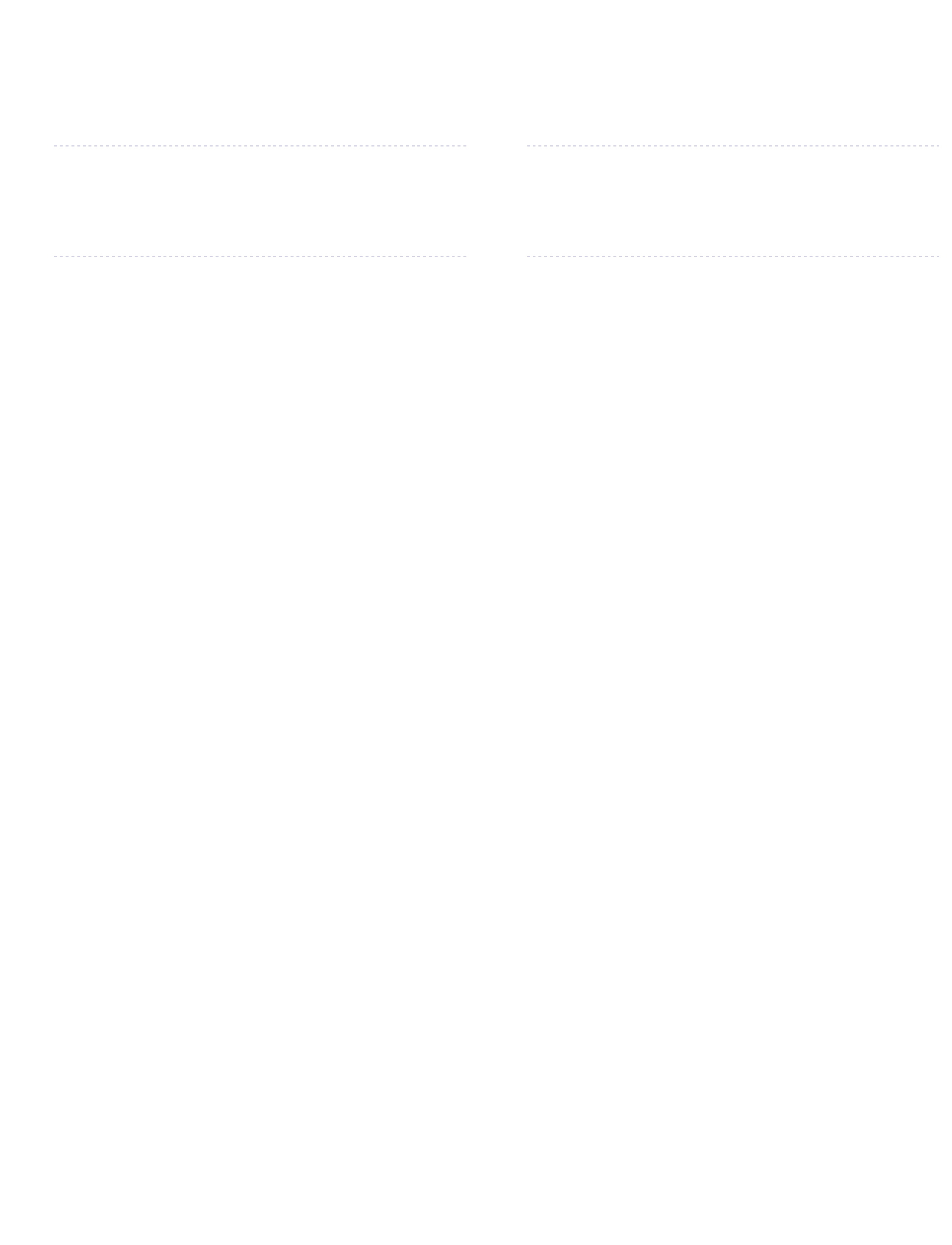
\end{center}

\begin{center}
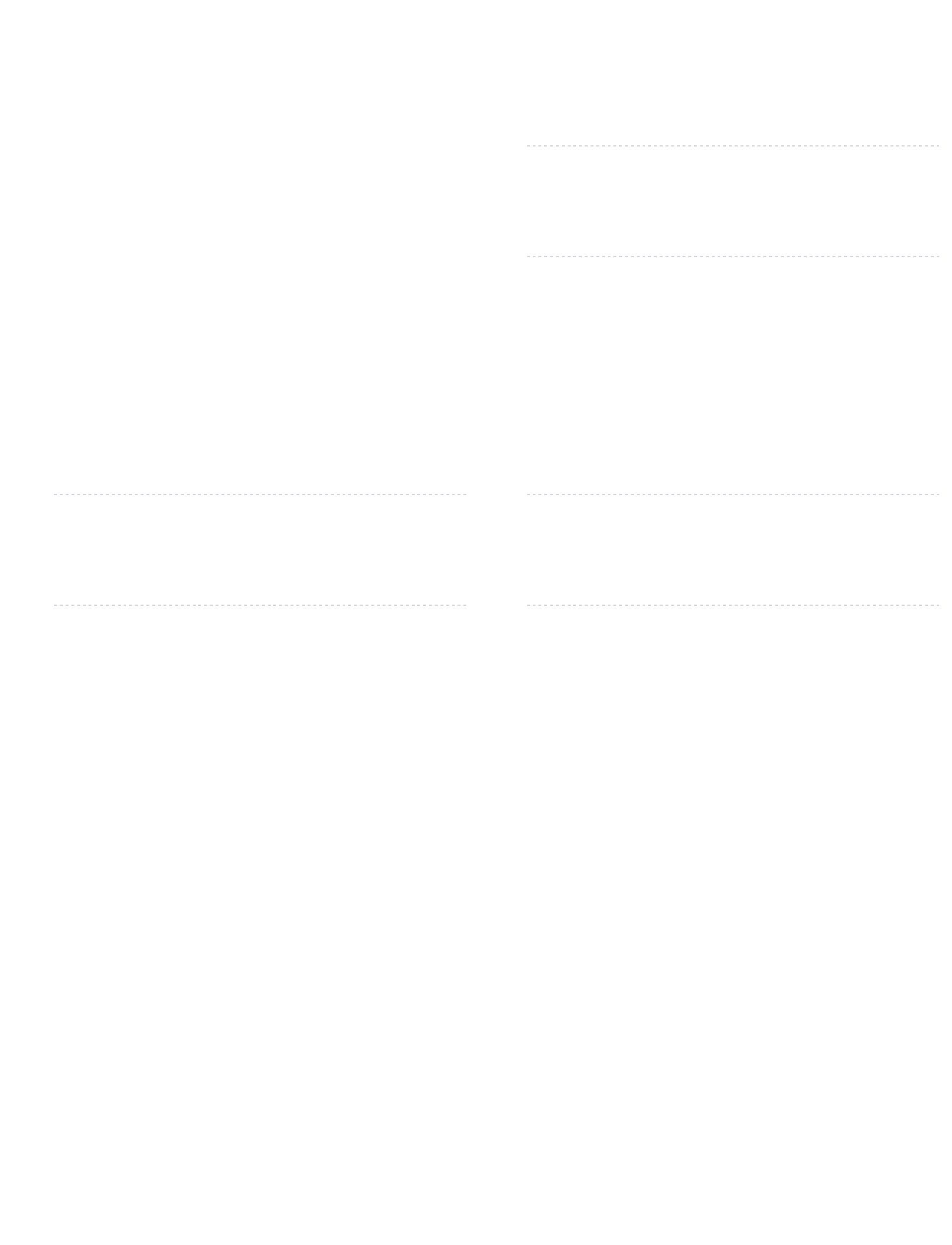
\end{center}

\noindent exceptionally small proportion $c = \frac{1}{1000}$ is actually further from a semicircular shape than the larger proportion $c = \frac{1}{10}$. This can be explained by the simple fact that the RBM in the case $c = \frac{1}{1000}$ has far fewer nontrivial entries than in the case $c = \frac{1}{10}$; or, put another way, there is simply not enough randomness for the convergence $\mu(\mbf{\Theta}_n^{(c)}) \wto \mu_c$ to take yet.

\addtocontents{toc}{\SkipTocEntry}
\subsection*{Acknowledgements}\label{acknowledgements}

The author thanks his advisor, Steve Evans, for his guidance, patience, and support; Alice Guionnet, for a helpful conversation in Montr\'{e}al on the occasion of the 2015 CRM-PIMS Summer School in Probability, during which the author was supported by MSRI; Camille Male, for many helpful comments and suggestions; Raj Rao Nadakuditi, for his insights on an earlier version of our paper and suggesting an investigation into RBMs on the occasion of the 2016 University of Michigan Summer School on Random Matrices; and the organizers of the summer schools, for their hospitality.

\section{Traffic probability}\label{sec2}

We begin with an exposition of \emph{traffic probability}; we refer the reader to \cite{Mal11,CDM16} for the definitive references. For the convenience of the reader, we recall in Section \hyperref[sec2.1]{2.1} the basic framework of \emph{non-commutative probability} following \cite{NS06}. Section \hyperref[sec2.2]{2.2} introduces the \emph{graph polynomials}, a combinatorial generalization of the non-commutative polynomials giving additional structure to the non-commutative probability spaces. As an example, Section \hyperref[sec2.3]{2.3} defines \emph{graph operations on matrices}, forming the prototype of a \emph{traffic space}. We define traffic spaces in full generality in Section \hyperref[sec2.4]{2.4} and devote Section \hyperref[sec2.5]{2.5} to the associated notion of \emph{traffic independence}.

\subsection{Non-commutative probability}\label{sec2.1}

Consider first the usual case of a measurable space $(\Omega, \salg{F})$. For a given probability measure $\prob$ on $(\Omega, \salg{F})$, we can form the (commutative) unital $*$-algebra $L^{\infty-}(\Omega, \salg{F}, \prob)$ of measurable complex-valued functions with finite moments of all orders, i.e.,
\[
L^{\infty-}(\Omega, \salg{F}, \prob) = \bigcap_{p=1}^\infty L^p(\Omega,\salg{F}, \prob).
\]
The expectation $\E: L^{\infty-}(\Omega, \salg{F}, \prob) \to \C$ recovers the probability measure $\prob$; thus, the passage from the probability space $(\Omega, \salg{F}, \prob)=((\Omega, \salg{F}), \prob)$ to the pair $(L^{\infty-}(\Omega, \salg{F}, \prob), \E)$ incurs no loss of information. This correspondence motivates the definition of a non-commutative probability space.

\begin{defn}[Non-commutative probability space]\label{defn2.1.1}
A \emph{non-commutative probability space} is a pair $(\salg{A}, \varphi)$ consisting of a unital algebra $\salg{A}$ over $\C$ equipped with a unital linear functional $\varphi: \salg{A} \to \C$. We call the elements $a \in \salg{A}$ \emph{non-commutative random variables} (or simply \emph{random variables}) and refer to $\varphi$ as the \emph{expectation functional}. If $\salg{A}$ has the additional structure of a $*$-algebra, then we say that $\varphi$ is \emph{positive} if $\varphi(a^*a) \geq 0$ for all $a \in \salg{A}$, in which case we term $\varphi$ a \emph{state} and $(\salg{A}, \varphi)$ a \emph{$*$-probability space}.
\end{defn}

\begin{eg}\label{eg2.1.2}
In keeping with the introduction, a classical probability space $(\Omega, \salg{F}, \prob)$ gives rise to a $*$-probability space $(L^{\infty-}(\Omega, \salg{F}, \prob), \E)$. We abstract another feature from the classical case: in the setting of a $*$-probability space $(\salg{A}, \varphi)$, we say that the expectation functional $\varphi$ is \emph{faithful} if $\varphi(a^*a)=0$ implies $a=0$.
\end{eg}

\begin{eg}\label{eg2.1.3}
Let $\matn(\C)$ denote the usual $*$-algebra of $n\times n$ complex matrices. The normalized trace $\frac{1}{n}\trace$ is clearly positive (indeed, faithful), giving rise to the $*$-probability space $(\matn(\C), \frac{1}{n}\trace)$. The trace of course vanishes on the commutators; in general, we say that the expectation functional $\varphi$ of a non-commutative probability space $(\salg{A},\varphi)$ is a \emph{trace} if $\varphi$ vanishes on the commutators of $\salg{A}$.
\end{eg}

\begin{eg}\label{eg2.1.4}
Combining the two previous examples, we obtain the $*$-probability space $(\matn(L^{\infty-}(\Omega, \salg{F}, \prob)), \E\frac{1}{n}\trace)$ of random matrices whose entries have finite moments of all orders equipped with the expected normalized trace. We leave it to the reader to verify that $\E\frac{1}{n}\trace$ is indeed a faithful trace.
\end{eg}

The distribution of a non-commutative random variable $a \in \salg{A}$ is defined as the pushforward of the expectation functional $\varphi$ by the element $a$. To make this precise, we introduce the \emph{non-commutative polynomials}. For an index set $I$, we write $\C\langle\mbf{x}\rangle$ (resp., $\C\langle\mbf{x}, \mbf{x}^*\rangle$) for the free unital algebra (resp., free unital $*$-algebra) on the indeterminates $\mbf{x} = (x_i)_{i\in I}$. Given a family of random variables $\mbf{a} = (a_i)_{i\in I}$ in a non-commutative probability space (resp., $*$-probability space) $(\salg{A}, \varphi)$, we have the usual evaluation map 
\[
\C\langle\mbf{x}\rangle \ni P \mapsto P(\mbf{a}) \in \salg{A} \quad (\text{resp., } \C\langle\mbf{x}, \mbf{x}^*\rangle \ni Q \mapsto Q(\mbf{a}) \in \salg{A}).
\]
This allows us to formalize

\begin{defn}[Joint distribution]\label{defn2.1.5}
Let $(\salg{A}, \varphi)$ be a non-commutative probability space. The \emph{joint distribution} of a family of random variables $\mbf{a} = (a_i)_{i\in I}$ in $\salg{A}$ is the linear functional
\[
\mu_{\mbf{a}}: \C\langle\mbf{x}\rangle \to \C, \qquad P \mapsto \varphi(P(\mbf{a})).
\]
If $(\salg{A},\varphi)$ has the additional structure of a $*$-probability space, we further define the \emph{joint $*$-distribution} of $\mbf{a}$ as the linear functional
\[
\nu_{\mbf{a}}: \C\langle\mbf{x}, \mbf{x}^*\rangle \to \C, \qquad Q \mapsto \varphi(Q(\mbf{a})).
\]
\end{defn}

\begin{defn}[Convergence in distribution]\label{defn2.1.6}
Let $(\salg{A}_n, \varphi_n)_{1 \leq n < \infty}$ and $(\salg{A},\varphi)$ be non-commutative probability spaces. Suppose that for each $n \in \N$ we have a family of random variables $\mbf{a}_n = (a_n^{(i)})_{i\in I}$ in $\salg{A}_n$. We say that the $\mbf{a}_n$ \emph{converge in distribution} to $\mbf{a} = (a_i)_{i\in I} \subset \salg{A}$ if the corresponding joint distributions $\mu_{\mbf{a}_n}$ converge pointwise to $\mu_{\mbf{a}}$, i.e.,
\[
\lim_{n\to\infty} \mu_{\mbf{a}_n}(P) = \mu_{\mbf{a}}(P), \qquad \forall P \in \C\langle\mbf{x}\rangle.
\]
If $(\salg{A}_n, \varphi_n)_{1 \leq n < \infty}$ and $(\salg{A}, \varphi)$ have the additional structure of a $*$-probability space, we may further say that the $\mbf{a}_n$ \emph{converge in $*$-distribution} to $\mbf{a}$ if the corresponding joint $*$-distributions $\nu_{\mbf{a}_n}$ converge pointwise to $\nu_{\mbf{a}}$, i.e.,
\[
\lim_{n\to\infty} \nu_{\mbf{a}_n}(Q) = \nu_{\mbf{a}}(Q), \qquad \forall Q \in \C\langle\mbf{x}, \mbf{x}^*\rangle.
\]
\end{defn}

We conclude with two notions of independence in the non-commutative probabilistic setting. To facilitate the definitions, we introduce some notation. For a collection $\salg{S} \subset \salg{A}$ of random variables in a non-commutative probability space $(\salg{A}, \varphi)$, we write $\interior{\salg{S}}=(a\in\salg{S}: \varphi(a)=0)$ for the subcollection (possibly empty) of \emph{centered} random variables.

\begin{defn}[Tensor independence]\label{defn2.1.7}
Let $(\salg{A}, \varphi)$ be a non-commutative probability space. We say that unital subalgebras $(\salg{A}_i: i \in I)$ of $\salg{A}$ are \emph{tensor independent} (or \emph{classically independent}) if the $\salg{A}_i$ commute and $\varphi$ is multiplicative across the $\salg{A}_i$ in the following sense: for any $k \geq 1$ and distinct indices $i(1), \ldots, i(k) \in I$,
\begin{equation}\label{eq:2.1}
\varphi\bigg(\prod_{j=1}^k a_{i(j)}\bigg) = \prod_{j=1}^k \varphi(a_{i(j)}), \qquad \forall a_{i(j)} \in \salg{A}_{i(j)}.
\end{equation}
\end{defn}

We note that the multiplicative property \eqref{eq:2.1} is equivalent to
\begin{equation*}\tag{$2.1^\prime$}\label{eq:2.1'}
\varphi\bigg(\prod_{j=1}^k a_{i(j)}\bigg) = 0, \qquad \forall a_{i(j)} \in \interior{\salg{A}}_{i(j)}.
\end{equation*}
We contrast this with

\begin{defn}[Free independence]\label{defn2.1.8}
Let $(\salg{A}, \varphi)$ be a non-commutative probability space. We say that unital subalgebras $(\salg{A}_i: i \in I)$ of $\salg{A}$ are \emph{freely independent} (or simply \emph{free}) if for any $k \geq 1$ and consecutively distinct indices $i(1) \neq i(2) \neq \cdots \neq i(k) \in I$,
\begin{equation}\label{eq:2.2}
\varphi\bigg(\prod_{j=1}^k a_{i(j)}\bigg) = 0, \qquad \forall a_{i(j)}\in\interior{\salg{A}}_{i(j)}.
\end{equation}
\end{defn}

We define the tensor independence (resp., free independence) of subsets $(\salg{S}_i: i\in I)$ of $\salg{A}$ as the tensor independence (resp., free independence) of the generated unital subalgebras $(\text{alg}(1_{\salg{A}},\salg{S}_i): i\in I)$. If $(\salg{A},\varphi)$ has the additional structure of a $*$-probability space, we may further define the \emph{$*$-tensor independence} (resp., \emph{$*$-free independence}) of $(\salg{S}_i: i\in I)$ as the tensor independence (resp., free independence) of the generated unital $*$-subalgebras $(*\text{-alg}(1_{\salg{A}},\salg{S}_i): i\in I)$.

The reader will no doubt notice that equations \eqref{eq:2.1'} and \eqref{eq:2.2} are identical; however, the admissible indices $i(j)$ to which they apply crucially differ. The corresponding CLTs, recorded below (see, e.g., Theorems 8.5 and 8.10 in \cite{NS06}), illustrate the considerable extent to which these two notions diverge.

\begin{thm}[CLTs, classical and free]\label{thm2.1.9}
Let $(a_n)$ be a sequence of identically distributed self-adjoint random variables in a $*$-probability space $(\salg{A},\varphi)$. Assume that the $a_n$ are centered with unit variance, i.e., $\varphi(a_n)=0$ and $\varphi(a_n^2)=1$, and write $s_n=\frac{1}{\sqrt{n}}\sum_{j=1}^n a_j$ for the normalized sum. We consider two cases:
\begin{enumerate}[label=(\roman*)]
\item If the $a_n$ are classically independent, then $(s_n)$ converges in distribution to a standard normal random variable, i.e.,
\[
\lim_{n\to\infty} \varphi(s_n^m) = \int_\R t^m \cdot \frac{1}{\sqrt{2\pi}} e^{-t^2/2}\, dt, \qquad \forall m \in \N.
\]
\item If the $a_n$ are freely independent, then $(s_n)$ converges in distribution to a standard semicircular random variable, i.e.,
\[
\lim_{n\to\infty} \varphi(s_n^m) = \int_{-2}^2 t^m \cdot \frac{1}{2\pi} \sqrt{4-t^2}\, dt, \qquad \forall m \in \N.
\]
\end{enumerate}
\end{thm}

Finally, we recall one of the most basic (and frequently appearing) families of random variables in this framework.

\begin{defn}[Semicircular system]\label{defn2.1.10}
Let $(\salg{A}, \varphi)$ be a $*$-probability space. We say that a family of random variables $\mbf{a} = (a_i)_{i \in I}$ in $\salg{A}$ is a \emph{semicircular system} if $\mbf{a}$ is a family of freely independent self-adjoint standard semicircular random variables. Similarly, we say that $\mbf{a}$ is a \emph{Gaussian system} if $\mbf{a}$ is a family of classically independent self-adjoint standard normal random variables.
\end{defn}

\begin{rem}\label{rem2.1.11}
On a purely combinatorial level, tensor independence and free independence simply amount to rules for obtaining the expectation of non-commutative polynomials in independent random variables from the expectation of non-commutative polynomials in the individual random variables themselves. Naturally, one may then ask if there exist other such rules and hence other notions of independence in the non-commutative probabilistic setting. Speicher showed that if we require the rules to be suitably universal in an algebraic sense, then tensor independence and free independence remain the only candidates \cite{Spe97} (but do see \cite{BGS02,Mur03} for further reading). The distinct notion of traffic independence is consistent with this dichotomy precisely because it is defined in terms of the more general \emph{graph polynomials}, which we introduce in the next section.
\end{rem}

\subsection{Graph polynomials}\label{sec2.2}

To begin, we fix some notation. As before, we write $\mbf{x} = (x_i)_{i \in I}$ for a set of indeterminates. We implicitly assume a corresponding set of indeterminates $\mbf{x}^* = (x_i^*)_{i \in I}$ such that $(\mbf{x}, \mbf{x}^*) = (x_i, x_i^*)_{i \in I}$ forms a set of pairwise distinct indeterminates satisfying the natural involutive $*$-relation.

\begin{defn}[Graph monomial]\label{defn2.2.1}
A \emph{directed multigraph} (or \emph{multidigraph} for short) is a quadruple $G = (V, E, \source, \target)$ consisting of a (non-empty) set of vertices $V$, a set of edges $E$, and maps $\source, \target: E \to V$ specifying the source $\source(e)$ and target $\target(e)$ of each edge $e \in E$. 

A \emph{graph $T$ in $\mbf{x}$} is a multidigraph with edge labels in the indeterminates $\mbf{x} = (x_i)_{i \in I}$: formally, $T = (G,\gamma)$, where $\gamma: E \to I$ specifies the label $x_{\gamma(e)} \in \mbf{x}$ of each edge $e$. We may further define a \emph{$*$-graph in $\mbf{x}$} by including the additional information of a map $\varepsilon: E \to \{1, *\}$ to indicate the label $x_{\gamma(e)}^{\varepsilon(e)} \in (\mbf{x}, \mbf{x}^*)$. We obtain the \emph{conjugate} $\overline{T}$ of a $*$-graph $T$ by reversing the edges of $T$ and replacing the labels $x_{\gamma(e)}^{\varepsilon(e)}$ by $(x_{\gamma(e)}^{\varepsilon(e)})^*$ so that $\overline{T} = (V, E, \target, \source, \gamma, \varepsilon^*)$. For notational convenience, we often omit the source/target and simply write $T = (V, E, \gamma)$ (resp., $T = (V, E, \gamma, \varepsilon)$) when the context is clear.

A \emph{test graph in $\mbf{x}$} is a finite, connected graph in $\mbf{x}$. We define a \emph{$*$-test graph} analogously.

A \emph{bi-rooted graph $t$ in $\mbf{x}$} is a graph $T = (V, E, \gamma)$ in $\mbf{x}$ together with an ordered pair of distinguished (not necessarily distinct) vertices $(\vin, \vout) \in V^2$ whose coordinates we term the \emph{input} and the \emph{output} respectively. Formally, this amounts to a triple $t = (T, \vin, \vout)$. We define a \emph{bi-rooted $*$-graph} analogously. We obtain the \emph{transpose} $t^\intercal$ of $t$ by interchanging the input and the output so that $t^\intercal = (T, \vout, \vin)$. We further obtain the \emph{adjoint} $t^*$ of $t$ by taking the conjugate transpose (in either order) of $t$, i.e., $t^* = (\overline{T}, \vout, \vin)$. 

Finally, a \emph{graph monomial in $\mbf{x}$} is a bi-rooted test graph in $\mbf{x}$. We define a \emph{$*$-graph monomial} analogously.
\end{defn}

\begin{center}
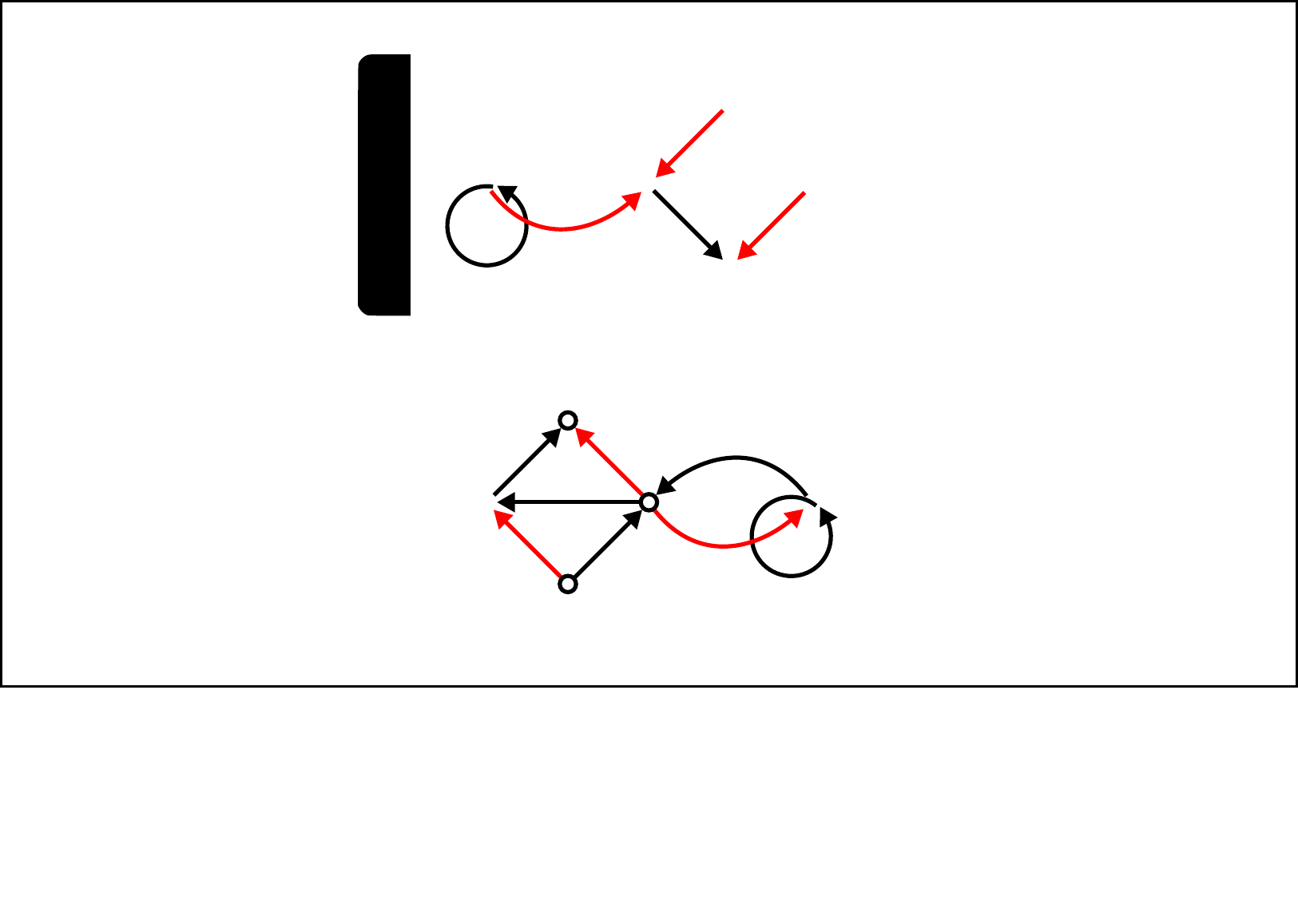
\end{center}

\begin{rem}\label{rem2.2.2}
We consider graphs as defined only up to the appropriate notion of isomorphism; thus, for example, we write $\ssgm{x}$ for the set of $*$-graph monomials in $\mbf{x}$ up to isomorphism of bi-rooted, edge-labeled multidigraphs. 
\end{rem}

We work exclusively in the context of a $*$-probability space in the sequel. To suit our needs, we develop the framework for $*$-graph polynomials. One may of course consider the simpler case of graph polynomials in parallel.

\begin{defn}[$*$-algebra of $*$-graph polynomials]\label{defn2.2.3}
Let $\ssgp{x}$ denote the complex vector space of finite linear combinations in $\ssgm{x}$, the elements of which we call the \emph{$*$-graph polynomials}. We give $\ssgp{x}$ the additional structure of a unital $*$-algebra over $\C$ as follows. For $*$-graph monomials $t_1 = (T_1, \vin^{(1)}, \vout^{(1)})$ and $t_2 = (T_2, \vin^{(2)}, \vout^{(2)})$, we define the product $t_1t_2$ as the concatenation of $t_1$ and $t_2$ by merging the output $\vout^{(2)}$ with the input $\vin^{(1)}$. Formally, $t_1t_2 = (T_3, \vin^{(2)}, \vout^{(1)})$, where $T_3$ corresponds to the $*$-graph obtained from the disjoint union of $T_1$ and $T_2$ by identifying the vertices $\vout^{(2)}$ and $\vin^{(1)}$. We extend this to a bilinear operation on the $*$-graph polynomials to obtain the multiplicative structure on $\ssgp{x}$. The $*$-graph monomial consisting of a single vertex (necessarily both the input and the output) serves as the identity element for this multiplication. As suggested by Definition \hyperref[defn2.2.1]{2.2.1}, we define the $*$-operation as giving the adjoint on the $*$-graph monomials so that $(t_1)^* = t_1^* = (\overline{T}_1, \vout^{(1)}, \vin^{(1)})$.
\end{defn}

\begin{rem}\label{rem2.2.4}
The right-to-left convention for $*$-graph monomials comes from the usual convention for function composition: if we imagine the input as the domain and the output as the codomain, then we see how the two notions align.
\end{rem}

\begin{center}
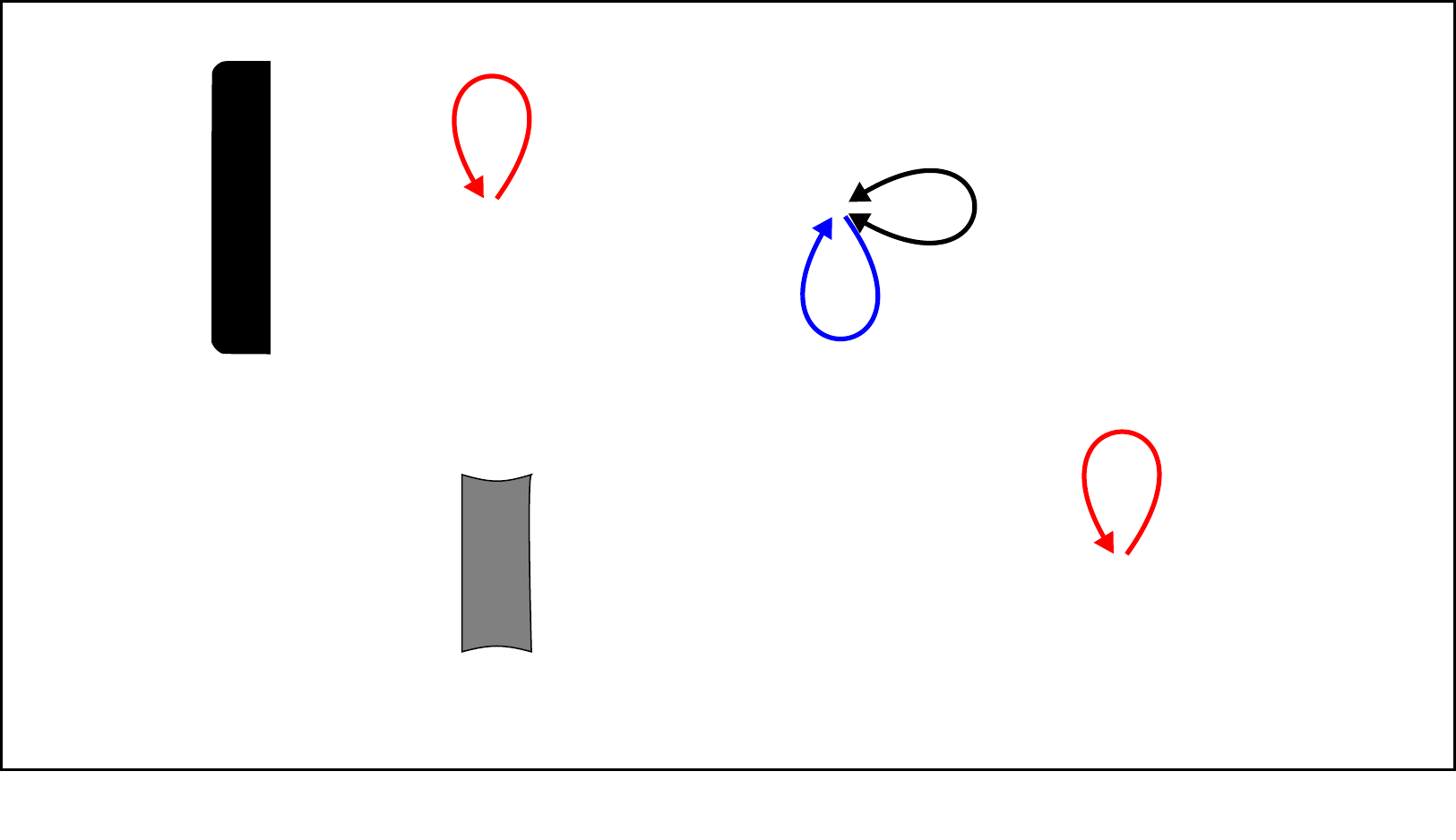
\end{center}

We can encode a $*$-monomial $P = x_{i(n)}^{\varepsilon(n)} \cdots x_{i(1)}^{\varepsilon(1)} \in \C\langle\mbf{x}, \mbf{x}^*\rangle$ as a $*$-graph monomial $t_P \in \ssgp{x}$ by considering each indeterminate in the product as labeling an edge directly below it in a directed path on $n+1$ vertices, right to left, starting at the input and ending at the output. Formally, $t_P = (T_P, v_1, v_{n+1})$, where
\[
T_P = ((v_j)_{j=1}^{n+1}, (e_k)_{k=1}^n, \source, \target, \gamma, \varepsilon), \quad \source(e_k) = v_k, \quad \target(e_k) = v_{k+1}, \quad \text{and} \quad \gamma(e_k) = i(k).
\]
The correspondence $P \mapsto t_P$ defines an embedding of unital $*$-algebras
\begin{equation}\label{eq:2.3}
\eta_{\mbf{x}}: \C\langle\mbf{x}, \mbf{x}^*\rangle \hookrightarrow \ssgp{x},
\end{equation}
hence the term $*$-graph polynomial. We extend our notation $t_P = \eta_{\mbf{x}}(P)$ to $P \in \C\langle\mbf{x}, \mbf{x}^*\rangle$.

\begin{center}
\begingroup%
  \makeatletter%
  \providecommand\color[2][]{%
    \errmessage{(Inkscape) Color is used for the text in Inkscape, but the package 'color.sty' is not loaded}%
    \renewcommand\color[2][]{}%
  }%
  \providecommand\transparent[1]{%
    \errmessage{(Inkscape) Transparency is used (non-zero) for the text in Inkscape, but the package 'transparent.sty' is not loaded}%
    \renewcommand\transparent[1]{}%
  }%
  \providecommand\rotatebox[2]{#2}%
  \ifx\svgwidth\undefined%
    \setlength{\unitlength}{468bp}%
    \ifx\svgscale\undefined%
      \relax%
    \else%
      \setlength{\unitlength}{\unitlength * \real{\svgscale}}%
    \fi%
  \else%
    \setlength{\unitlength}{\svgwidth}%
  \fi%
  \global\let\svgwidth\undefined%
  \global\let\svgscale\undefined%
  \makeatother%
  \begin{picture}(1,0.27307692)%
    \put(0,0){\includegraphics[width=\unitlength,page=1]{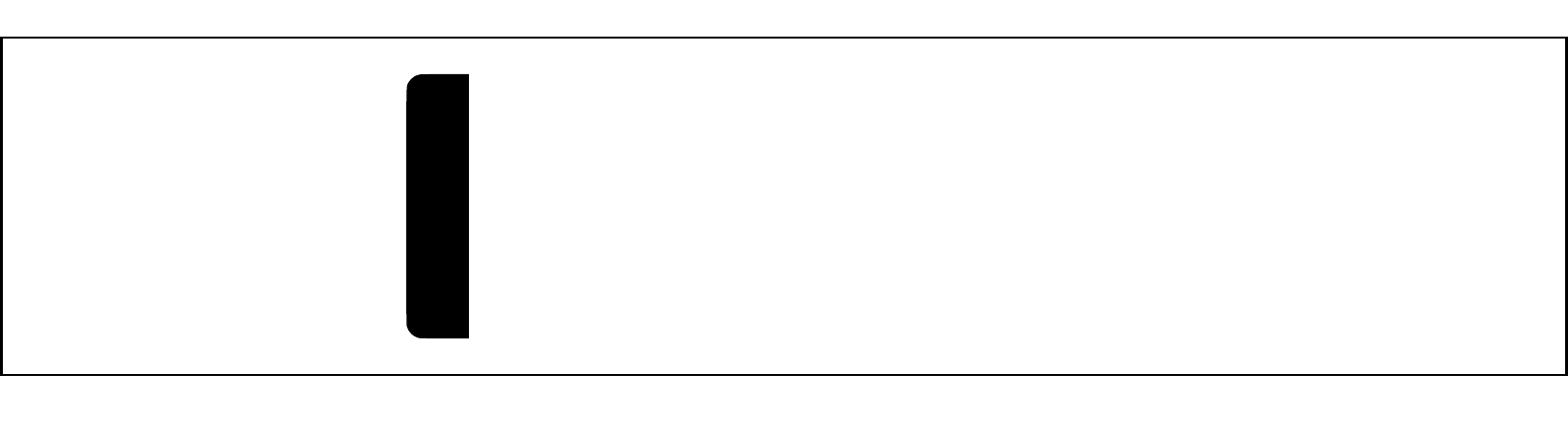}}%
    \put(-0.00141487,0.01911686){\color[rgb]{0,0,0}\makebox(0,0)[lt]{\begin{minipage}{0.99822193\unitlength}\raggedright Figure 6: An example of the embedding $\eta_{\mbf{x}}: \C\langle\mbf{x}, \mbf{x}^*\rangle \ni P \mapsto t_P \in \ssgp{x}$.\end{minipage}}}%
    \put(0.12170716,1.58554485){\color[rgb]{0,0,0}\makebox(0,0)[lt]{\begin{minipage}{0.11106842\unitlength}\raggedright \end{minipage}}}%
    \put(0.03997094,0.15173738){\color[rgb]{0,0,0}\makebox(0,0)[lt]{\begin{minipage}{0.50837182\unitlength}\raggedright $y^*xz^*x \mapsto t_{y^*xz^*x} = $\end{minipage}}}%
    \put(0.92962456,0.15195258){\color[rgb]{0,0,0}\rotatebox{-90}{\makebox(0,0)[lb]{\smash{in}}}}%
    \put(0,0){\includegraphics[width=\unitlength,page=2]{fig6_embed.pdf}}%
    \put(0.79835088,0.15577662){\color[rgb]{0,0,0}\makebox(0,0)[lb]{\smash{$x$}}}%
    \put(0.54862336,0.15577662){\color[rgb]{0,0,0}\makebox(0,0)[lb]{\smash{$x$}}}%
    \put(0.67348714,0.15577662){\color[rgb]{0,0,0}\makebox(0,0)[lb]{\smash{$z^*$}}}%
    \put(0.42376116,0.15724376){\color[rgb]{0,0,0}\makebox(0,0)[lb]{\smash{$y^*$}}}%
    \put(0.28612738,0.12082528){\color[rgb]{0,0,0}\rotatebox{90}{\makebox(0,0)[lb]{\smash{\textcolor{white}{out}}}}}%
  \end{picture}%
\endgroup%

\end{center}

We can also define a notion of substitution for $*$-graph polynomials that generalizes the corresponding notion for $*$-polynomials.

\begin{defn}[Substitution in $\sgps$]\label{defn2.2.5}
Let $t$ be a $*$-graph monomial in $\mbf{x}=(x_i)_{i\in I}$. Suppose that for each $i\in I$ we have a $*$-graph monomial $t_i$ in the indeterminates $\mbf{y}_i= (y_{i,j})_{j\in J_i}$. Then we may substitute the $(t_i)_{i\in I}$ for the indeterminates $\mbf{x}$ in $t$ by replacing the edges $e$ labeled by $x_{\gamma(e)}^{\varepsilon(e)}$ with the $*$-graph monomial $t_{\gamma(e)}^{\varepsilon(e)}$: simply identify the source (resp., target) of $e$ with the input (resp., output) of $t_{\gamma(e)}^{\varepsilon(e)}$. We denote the resulting $*$-graph monomial in the indeterminates $\mbf{y}=\bigcup_{i\in I} \mbf{y}_i =(y_{i,j})_{i\in I,j\in J_i}$ by $\op{Subs}_{\mbf{x},(\mbf{y}_i)_{i\in I}}(t, \bigtimes_{i\in I} t_i)$. We extend this operation to the $*$-graph polynomials in the obvious way to obtain the \emph{substitution map}
\[
\op{Subs}_{\mbf{x},(\mbf{y}_i)_{i\in I}}: \ssgp{x} \times \bigtimes_{i \in I}\C\salg{G}\langle\mbf{y}_i,\mbf{y}_i^*\rangle \to \ssgp{y}.
\]
\end{defn}

\begin{center}
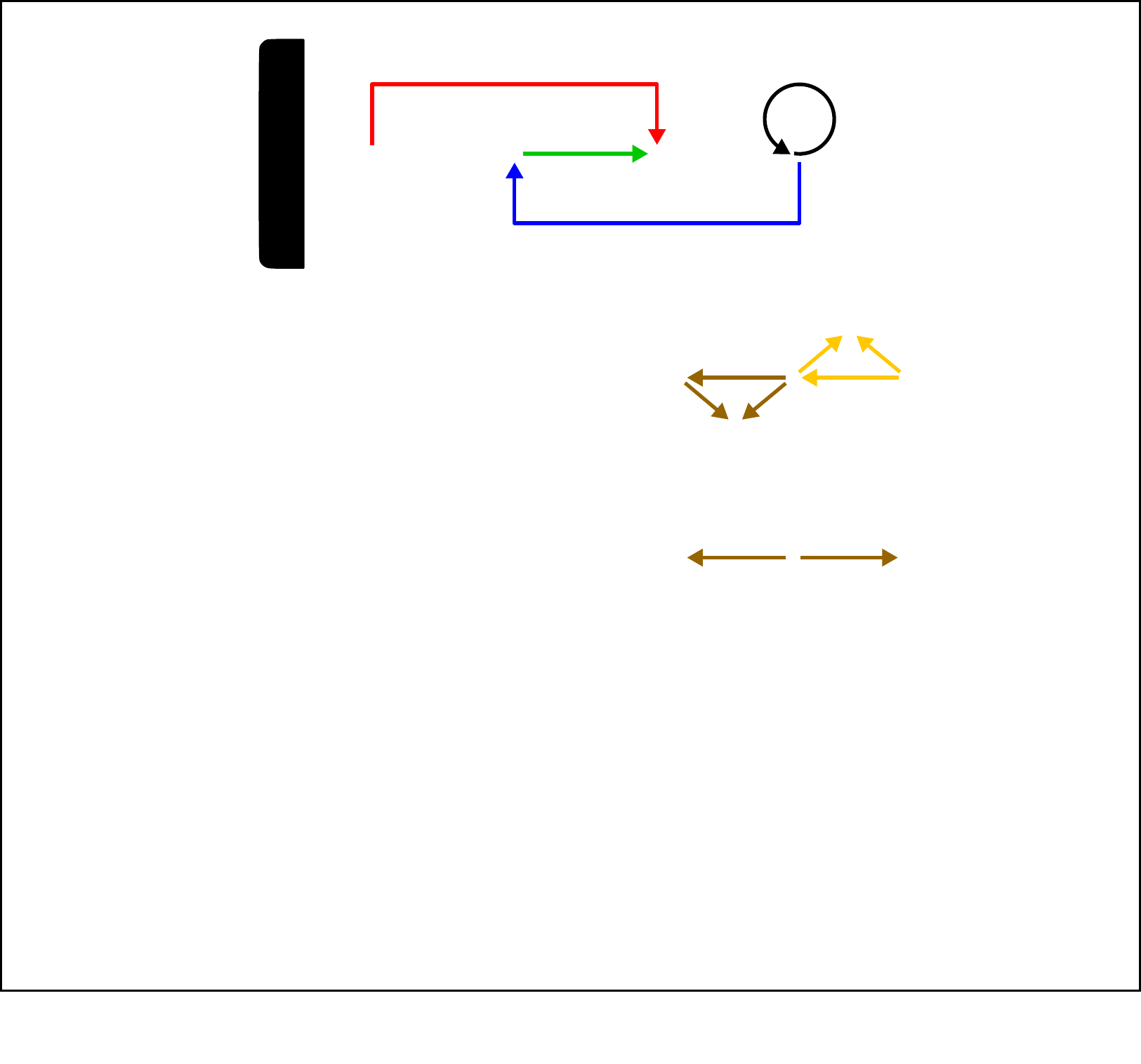
\end{center}

The substitution operation is of course associative, i.e., for indeterminates
\[
\mbf{x}=(x_i)_{i\in I}, \quad \mbf{y}=\bigcup_{i\in I} \mbf{y}_i =(y_{i,j})_{i\in I,j\in J}, \quad \text{and} \quad \mbf{z} = \bigcup_{i\in I} \mbf{z}_i = \bigcup_{i\in I} \bigcup_{j\in J_i} \mbf{z}_{i,j} = (z_{i,j,k})_{i\in I,j\in J_i,k\in K_{i,j}},
\]
the following diagram commutes:
\[
\begin{tikzcd}[row sep=3cm, column sep=4cm]
\Scale[.9]{\displaystyle\C\salg{G}\langle\mbf{x},\mbf{x}^*\rangle\times\bigtimes_{i\in I}\C\salg{G}\langle\mbf{y}_i,\mbf{y}_i^*\rangle\times\bigtimes_{i\in I,j\in J_i}\C\salg{G}\langle\mbf{z}_{i,j},\mbf{z}_{i,j}^*\rangle}\arrow{r}{\Scale[1]{\operatorname{id}\times\bigtimes_{i\in I} \operatorname{Subs}_{\mbf{y}_i,(\mbf{z}_{i,j})_{j\in J_i}}}} \arrow{d}[swap]{\Scale[1]{\operatorname{Subs}_{\mbf{x},(\mbf{y}_i)_{i\in I}}\times \operatorname{id}}}& \Scale[.9]{\displaystyle\C\salg{G}\langle\mbf{x},\mbf{x}^*\rangle\times\bigtimes_{i\in I}\C\salg{G}\langle\mbf{z}_i,\mbf{z}_i^*\rangle} \arrow{d}{\Scale[1]{\operatorname{Subs}_{\mbf{x},(\mbf{z}_i)_{i\in I}}}} \\
\Scale[.9]{\displaystyle\C\salg{G}\langle\mbf{y},\mbf{y}^*\rangle\times\bigtimes_{i\in I, j\in J_i}\C\salg{G}\langle\mbf{z}_{i,j},\mbf{z}_{i,j}^*\rangle} \arrow{r}[swap]{\Scale[1]{\operatorname{Subs}_{\mbf{y},(\mbf{z}_{i,j})_{i\in I, j\in J_i}}}}& \Scale[.9]{\C\salg{G}\langle\mbf{z},\mbf{z}^*\rangle}
\end{tikzcd}
\]

If, by a slight abuse of notation, we use the same notation
\[
\op{Subs}_{\mbf{x},(\mbf{y}_i)_{i\in I}}: \C\langle\mbf{x}, \mbf{x}^*\rangle \times \bigtimes_{i \in I} \C\langle \mbf{y}_i, \mbf{y}_i^*\rangle \to \C\langle\mbf{y}, \mbf{y}^*\rangle
\]
for the usual substitution of $*$-polynomials, then we also have the commutative diagram
\[
\begin{tikzcd}[row sep=3cm, column sep=3cm]
\Scale[1]{\displaystyle\C\langle\mbf{x},\mbf{x}^*\rangle\times\bigtimes_{i\in I}\C\langle\mbf{y}_i,\mbf{y}_i^*\rangle}\arrow{r}{\Scale[1]{\op{Subs}_{\mbf{x},(\mbf{y}_i)_{i\in I}}}} \arrow{d}[swap]{\Scale[1]{\eta_{\mbf{x}} \times \bigtimes_{i \in I} \eta_{\mbf{y}_i}}}& \Scale[1]{\displaystyle\C\langle\mbf{y},\mbf{y}^*\rangle} \arrow{d}{\Scale[1]{\eta_{\mbf{y}}}} \\
\Scale[1]{\displaystyle\C\salg{G}\langle\mbf{x},\mbf{x}^*\rangle\times\bigtimes_{i\in I} \C\salg{G}\langle\mbf{y}_i,\mbf{y}_i^*\rangle} \arrow{r}[swap]{\Scale[1]{\operatorname{Subs}_{\mbf{x},(\mbf{y}_{i})_{i \in I}}}}& \Scale[1]{\C\salg{G}\langle\mbf{y},\mbf{y}^*\rangle}
\end{tikzcd}
\]
which says that the substitution operation commutes with the embedding \eqref{eq:2.3}.

For natural reasons, one often prefers to consider the substitution operation as a function of the edges of a graph as opposed to the indeterminates. We can of course accomplish this by using the edge set as the indexing set for our indeterminates $\mbf{x} = (x_e)_{e \in E}$, but we opt for the more intrinsic notion of a \emph{$K$-graph operation} following \cite{CDM16}.

\begin{defn}[Graph operation]\label{defn2.2.6}
Let $g = (V, E, \source, \target, \vin, \vout, o)$ be a finite, connected, bi-rooted multidigraph together with an ordering of the edges $o: E \stackrel{\sim}{\to} [\#(E)]$. We interpret $g = g(\cdot_1, \ldots, \cdot_K)$ as a function of $K = \#(E)$ arguments, one for each edge $e \in E$, specified by the ordering $o$. We call such a graph $g$ a \emph{$K$-graph operation}. We denote the set of all $K$-graph operations by $\salg{G}_K$ and write $\salg{G} = \bigcup_{K \geq 0} \salg{G}_K$ for the set of all graph operations.

We define an action of the symmetric group $\mfk{S}_K$ on $\salg{G}_K$ by permuting the ordering of the edges. Formally, for $\sigma \in \mfk{S}_K$ and $g \in \salg{G}_K$, we define $g_\sigma$ as the $K$-graph operation $g_\sigma = (V, E, \vin, \vout, \sigma \circ o)$. We further define an involution $*$ on $\salg{G}$ by analogy with the $*$-graph polynomials. Formally, $g^* = (V, E, \target, \source, \vout, \vin, o)$, the only difference being that we do not have edge labels to modify.
\end{defn}

In fact, one may carry forward the entirety of this section to the graph operations: for example, the $*$-algebra of graph operations, substitution in graph operations $g(g_1, \ldots, g_K)$, etc. We leave the relatively straightforward details to the interested reader.

While the two notions largely coincide, the flexibility to work interchangeably between $*$-graph polynomials and graph operations facilitates many of the definitions. For example, the $*$-graph polynomials allow us to define the traffic analogue of the $*$-distribution, whereas the graph operations allow us to formulate the axioms of a traffic space in a more natural setting. Roughly speaking, a traffic space is a $*$-probability space $(\salg{A}, \varphi)$ with the additional structure to evaluate graph operations in the random variables $a \in \salg{A}$. We often refer to the random variables in a traffic space as \emph{traffic random variables} (or simply \emph{traffics}) to emphasize this distinction. The $*$-probability space of random $n \times n$ matrices (see Example \hyperref[eg2.1.4]{2.1.4}) is the prototype of a traffic space; we construct this example in the next section.

\subsection{Graph operations on matrices}\label{sec2.3}

\begin{defn}[Graph of matrices]\label{defn2.3.1}
Let $\mfk{A}_n = (\mbf{A}_n^{(k)})_{k=1}^K$ be a $K$-tuple of random $n \times n$ matrices. For a $K$-graph operation $g = (V, E, \vin, \vout, o)$, we define the \emph{graph of matrices} $Z_g(\mfk{A}_n) = Z_g(\mbf{A}_n^{(1)}, \ldots, \mbf{A}_n^{(K)})$ as the random $n \times n$ matrix with entries 
\begin{equation}\label{eq:2.4}
Z_g(\mfk{A}_n)(i, j) = \sum_{\substack{\phi: V \to [n] \text{ s.t.}\\ \phi(\vout) = i,\ \phi(\vin) = j}} \prod_{e \in E} \mbf{A}_n^{(o(e))}(\phi(\target(e)), \phi(\source(e))).
\end{equation}
For simplicity, we often write $\phi(e) = (\phi(\target(e)), \phi(\source(e)))$. We extend the operation \eqref{eq:2.4} to a multilinear function 
\[
Z_g: \matn(\salg{M}_{\C}(\Omega, \salg{F}, \prob))^{\otimes k} \to \matn(\salg{M}_{\C}(\Omega, \salg{F}, \prob)), \qquad \mbf{A}_n^{(1)} \otimes \cdots \otimes \mbf{A}_n^{(K)} \mapsto Z_g(\mfk{A}_n).
\]

We visualize a graph of matrices $Z_g(\mfk{A}_n)$ in the natural way: as a bi-rooted test graph with edge labels in $\matn(\salg{M}_{\C}(\Omega, \salg{F}, \prob))$. In particular, this suggests that the ordering of the edges should only play a formal role in the definition of a graph of matrices. We capture this intuition with the following \emph{equivariance} property: writing $\mfk{A}_n^\sigma = (\mbf{A}_n^{(\sigma(k))})_{k=1}^K$ for a permutation $\sigma \in \mfk{S}_K$, we have the equality
\[
Z_{g_\sigma}(\mfk{A}_n^{\sigma^{-1}}) = Z_g(\mfk{A}_n), \qquad \forall \sigma \in \mfk{S}_K.
\]

If instead $\mfk{A}_n = (\mbf{A}_n^{(i)})_{i \in I}$ is a family of random $n \times n$ matrices and $t = (V, E, \vin, \vout, \gamma, \varepsilon)$ is a $*$-graph monomial in $\mbf{x} = (x_i)_{i \in I}$, then we define the \emph{graph of matrices} $t(\mfk{A}_n)$ as the random $n \times n$ matrix with entries
\begin{equation}\label{eq:2.5}
t(\mfk{A}_n)(i, j) = \sum_{\substack{\phi: V \to [n] \text{ s.t.}\\ \phi(\vout) = i,\ \phi(\vin) = j}} \prod_{e \in E} (\mbf{A}_n^{(\gamma(e))})^{\varepsilon(e)}(\phi(e)).
\end{equation}
We extend the operation \eqref{eq:2.5} to the $*$-graph polynomials by linearity to obtain a $*$-homomorphic evaluation map
\begin{equation}\label{eq:2.6}
\op{eval}_{\ssgp{x}, \mfk{A}_n}: \ssgp{x} \to \matn(\salg{M}_{\C}(\Omega, \salg{F}, \prob)), \qquad t \mapsto t(\mfk{A}_n).
\end{equation}
We leave it to the reader to formulate the equivariance property in this context.
\end{defn}

\begin{rem}\label{rem2.3.2}
We think of the maps $Z_g$ as defining an action of the graph operations $\salg{G}$ on $\matn(\salg{M}(\Omega, \salg{F}, \prob))$, the data of which we record abstractly in the $*$-graph polynomials with the maps $\op{eval}_{\ssgp{x}, \mfk{A}_n}$. For consistency, we find it convenient to work exclusively with the $*$-graph polynomials for the remainder of this section. 
\end{rem}

We note that the evaluation \eqref{eq:2.6} extends the $*$-algebra structure of $\matn(\salg{M}_{\C}(\Omega, \salg{F}, \prob))$. To see this, recall the embedding \eqref{eq:2.3},
\[
\eta_{\mbf{x}}: \C\langle\mbf{x}, \mbf{x}^*\rangle \hookrightarrow \ssgp{x}, \qquad P \mapsto t_P. 
\]
The usual $*$-polynomial evaluation
\[
\op{eval}_{\C\langle\mbf{x}, \mbf{x}^*\rangle, \mfk{A}_n}: \C\langle\mbf{x}, \mbf{x}^*\rangle \to \matn(\salg{M}_{\C}(\Omega, \salg{F}, \prob)), \qquad P \mapsto P(\mfk{A}_n)
\]
then factors through the $*$-graph polynomials via $\eta_{\mbf{x}}$, i.e., the diagram
\[
\begin{tikzcd}[row sep=2cm, column sep=.75cm]
\Scale[1]{\C\langle\mbf{x}, \mbf{x}^*\rangle} \arrow{dr}[swap]{\eta_{\mbf{x}}} \arrow{rr}{\Scale[1]{\op{eval}_{\C\langle\mbf{x}, \mbf{x}^*\rangle, \mfk{A}_n}}}& &\Scale[1]{\matn(\salg{M}_{\C}(\Omega, \salg{F}, \prob))} \\
&\Scale[1]{\ssgp{x}} \arrow{ur}[swap]{\op{eval}_{\ssgp{x}, \mfk{A}_n}}&
\end{tikzcd}
\]
commutes. Even more, the evaluation in $*$-graph polynomials also produces matrices with additional linear algebraic structure.

\begin{eg}\label{eg2.3.3}
For $*$-graph monomials $t_1 = (T_1, \vin^{(1)}, \vout^{(1)})$ and $t_2 = (T_2, \vin^{(2)}, \vout^{(2)})$ in $\mbf{x}$, we define the \emph{Hadamard-Schur product} $t_1 \circ t_2$ as the superimposition of $t_1$ and $t_2$ according to their distinguished vertices. More precisely, $t_1 \circ t_2 = (T_3, \vin^{(3)}, \vout^{(3)})$, where $T_3$ is the $*$-graph obtained from the disjoint union of $T_1$ and $T_2$ by identifying the vertices $\vin^{(1)}$ and $\vin^{(2)}$ (which we then call $\vin^{(3)}$) and the vertices $\vout^{(1)}$ and $\vout^{(2)}$ (which we then call $\vout^{(3)}$). We extend this to a bilinear operation on the $*$-graph polynomials to obtain a commutative, associative product $\circ$ on $\C\salg{G}\langle\mbf{x}, \mbf{x}^*\rangle$. The evaluation \eqref{eq:2.6} then defines a morphism (of semigroups) for this product and the usual Hadamard-Schur product of matrices:
\[
(t_1 \circ t_2)(\mfk{A}_n) = t_1(\mfk{A}_n) \circ t_2(\mfk{A}_n), \qquad \forall t_1, t_2 \in \C\salg{G}\langle\mbf{x}, \mbf{x}^*\rangle.
\]
\end{eg}

\begin{eg}\label{eg.2.3.4}
The transpose operation on the $*$-graph monomials defines a linear involution on $\C\salg{G}\langle\mbf{x}, \mbf{x}^*\rangle$ that we again call the transpose and denote by $\cdot^\intercal$. Using the same notation for the matrices, we have the equality
\[
t^\intercal(\mfk{A}_n) = t(\mfk{A}_n)^\intercal, \qquad \forall t \in \C\salg{G}\langle\mbf{x}, \mbf{x}^*\rangle.
\]
\end{eg}

\begin{eg}\label{eg2.3.5}
For an indeterminate $x$, we write $\op{row}(t_x)$ for the $*$-graph monomial with two vertices ($v_1$ and $v_2$), a single edge (from $v_2$ to $v_1$ with label $x$), and input and output both equal (to $v_1$). Evaluating $\op{row}(t_x)$ in an $n \times n$ matrix $\mbf{A}_n$ outputs the diagonal matrix of row sums of $\mbf{A}_n$, i.e.,
\[
(\text{row}(t_x)(\mbf{A}_n))(i, j) = \indc{i = j} \sum_{k = 1}^n \mbf{A}_n(i, k).
\]
Reversing the direction of the lone edge in $\op{row}(t_x)$, we obtain a $*$-graph monomial $\op{col}(t_x)$ that evaluates to the diagonal matrix of column sums:
\[
(\text{col}(t_x)(\mbf{A}_n))(i, j) = \indc{i = j}\sum_{k = 1}^n \mbf{A}_n(k, j).
\] 
\end{eg}

\begin{center}
\begingroup%
  \makeatletter%
  \providecommand\color[2][]{%
    \errmessage{(Inkscape) Color is used for the text in Inkscape, but the package 'color.sty' is not loaded}%
    \renewcommand\color[2][]{}%
  }%
  \providecommand\transparent[1]{%
    \errmessage{(Inkscape) Transparency is used (non-zero) for the text in Inkscape, but the package 'transparent.sty' is not loaded}%
    \renewcommand\transparent[1]{}%
  }%
  \providecommand\rotatebox[2]{#2}%
  \ifx\svgwidth\undefined%
    \setlength{\unitlength}{468bp}%
    \ifx\svgscale\undefined%
      \relax%
    \else%
      \setlength{\unitlength}{\unitlength * \real{\svgscale}}%
    \fi%
  \else%
    \setlength{\unitlength}{\svgwidth}%
  \fi%
  \global\let\svgwidth\undefined%
  \global\let\svgscale\undefined%
  \makeatother%
  \begin{picture}(1,0.32307691)%
    \put(0,0){\includegraphics[width=\unitlength,page=1]{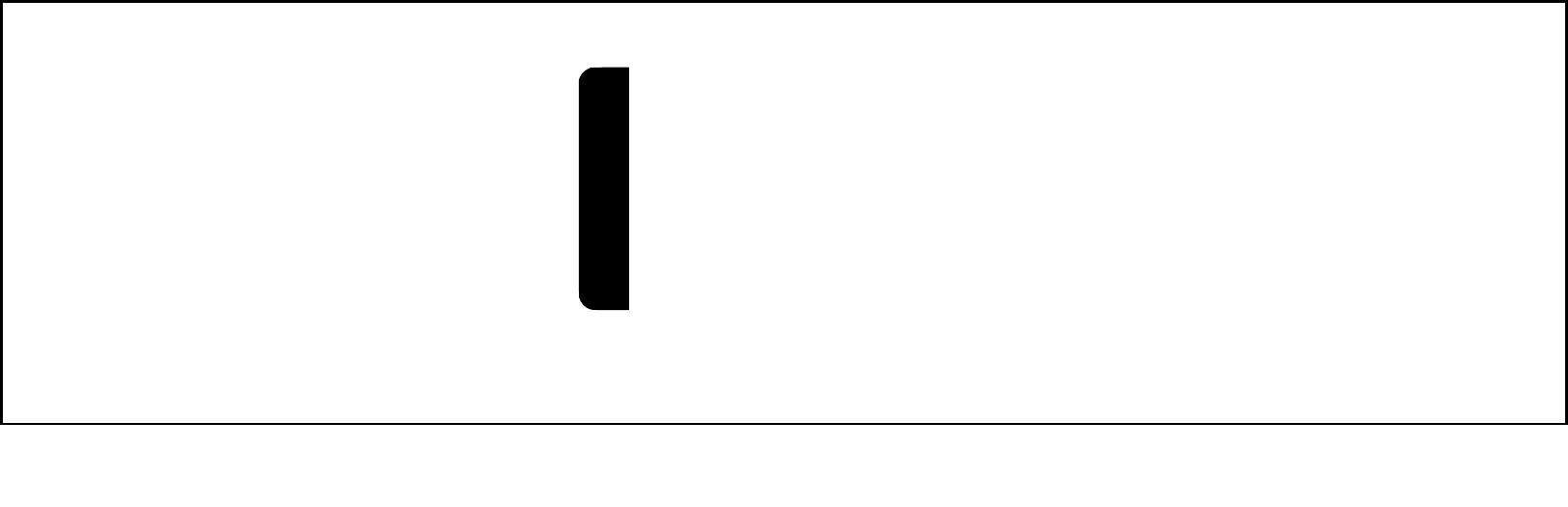}}%
    \put(-0.00141487,0.03834763){\color[rgb]{0,0,0}\makebox(0,0)[lt]{\begin{minipage}{0.99822195\unitlength}\raggedright Figure 8: Examples of $*$-graph monomials with linear algebraic structure.\end{minipage}}}%
    \put(0.12170717,1.59355766){\color[rgb]{0,0,0}\makebox(0,0)[lt]{\begin{minipage}{0.11106843\unitlength}\raggedright \end{minipage}}}%
    \put(0,0){\includegraphics[width=\unitlength,page=2]{fig8_la.pdf}}%
    \put(0.60579484,0.21391918){\color[rgb]{0,0,0}\rotatebox{-90}{\makebox(0,0)[lb]{\smash{in}}}}%
    \put(0.4899918,0.21587272){\color[rgb]{0,0,0}\makebox(0,0)[lb]{\smash{$x$}}}%
    \put(0.39216741,0.18225164){\color[rgb]{0,0,0}\rotatebox{90}{\makebox(0,0)[lb]{\smash{\textcolor{white}{out}}}}}%
    \put(0.1462418,0.081791){\color[rgb]{0,0,0}\makebox(0,0)[lb]{\smash{$t_x \circ t_{y^*}$}}}%
    \put(0.4899918,0.081791){\color[rgb]{0,0,0}\makebox(0,0)[lb]{\smash{$t_x^\intercal$}}}%
    \put(0,0){\includegraphics[width=\unitlength,page=3]{fig8_la.pdf}}%
    \put(0.28202207,0.21391908){\color[rgb]{0,0,0}\rotatebox{-90}{\makebox(0,0)[lb]{\smash{in}}}}%
    \put(0.06839459,0.18225164){\color[rgb]{0,0,0}\rotatebox{90}{\makebox(0,0)[lb]{\smash{\textcolor{white}{out}}}}}%
    \put(0.78726744,0.081791){\color[rgb]{0,0,0}\makebox(0,0)[lb]{\smash{$\op{row}(t_x)$}}}%
    \put(0,0){\includegraphics[width=\unitlength,page=4]{fig8_la.pdf}}%
    \put(0.92956522,0.21391918){\color[rgb]{0,0,0}\rotatebox{-90}{\makebox(0,0)[lb]{\smash{in}}}}%
    \put(0.71593784,0.18225174){\color[rgb]{0,0,0}\rotatebox{90}{\makebox(0,0)[lb]{\smash{\textcolor{white}{out}}}}}%
    \put(0,0){\includegraphics[width=\unitlength,page=5]{fig8_la.pdf}}%
    \put(0.83854954,0.2058567){\color[rgb]{0,0,0}\makebox(0,0)[lb]{\smash{$x$}}}%
    \put(0,0){\includegraphics[width=\unitlength,page=6]{fig8_la.pdf}}%
    \put(0.17228347,0.25593682){\color[rgb]{0,0,0}\makebox(0,0)[lb]{\smash{$x$}}}%
    \put(0.17228347,0.13721176){\color[rgb]{0,0,0}\makebox(0,0)[lb]{\smash{$y^*$}}}%
  \end{picture}%
\endgroup%

\end{center}

For a permutation $\sigma \in \mfk{S}_n$, we write
\[
\mbf{P}_\sigma(i, j) = \indc{\sigma(i) = j}
\]
for the corresponding $n \times n$ permutation matrix. The following result states that the graph operations commute with conjugation by the permutation matrices; the proof follows directly from the definitions.

\begin{prop}\label{prop2.3.6}
Let $\mfk{A}_n = (\mbf{A}_n^{(i)})_{i \in I}$ be a family of random $n \times n$ matrices. For any $*$-graph polynomial $t \in \C\salg{G}\langle\mbf{x}, \mbf{x}^*\rangle$ and permutation $\sigma \in \mfk{S}_n$, we have the equality
\[
t(\mbf{P}_\sigma \mfk{A}_n \mbf{P}_\sigma^*) = \mbf{P}_\sigma t(\mfk{A}_n) \mbf{P}_\sigma^*,
\]
where $\mbf{P}_\sigma \mfk{A}_n \mbf{P}_\sigma^*= (\mbf{P}_\sigma \mbf{A}_n^{(i)} \mbf{P}_\sigma^*)_{i \in I}$.
\end{prop}

Note that the trace of a graph of matrices $t(\mfk{A}_n)$ depends on $t=(T, \vin, \vout)$ only up to the $*$-graph $\Delta(t)=(\wtilde{V}, E, \gamma, \varepsilon)$ obtained from $T=(V, E, \gamma, \varepsilon)$ by identifying the input $\vin$ and the output $\vout$ and forgetting their distinguished roles. Indeed,
\begin{equation}\label{eq:2.7}
\begin{aligned}
\trace(t(\mfk{A}_n)) = \sum_{i = 1}^n t(\mfk{A}_n)(i, i) &= \sum_{i=1}^n \sum_{\substack{\phi: V \to [n] \text{ s.t.}\\ \phi(\vout)=\phi(\vin)=i}} \prod_{e\in E} (\mbf{A}_n^{(\gamma(e))})^{\varepsilon(e)}(\phi(e)) \\
&= \sum_{\phi: \wtilde{V} \to [n]} \prod_{e\in E} (\mbf{A}_n^{(\gamma(e))})^{\varepsilon(e)}(\phi(e)).
\end{aligned}
\end{equation}
We define the traffic distribution of the matrices accordingly.

Recall that a $*$-test graph in $\mbf{x}$ is a finite, connected $*$-graph in $\mbf{x}$. We write $\salg{T}\langle\mbf{x}, \mbf{x}^*\rangle$ for the set of $*$-test graphs in $\mbf{x}$. We further write $\C\salg{T}\langle\mbf{x}, \mbf{x}^*\rangle$ for the complex vector space of finite linear combinations in $\salg{T}\langle\mbf{x}, \mbf{x}^*\rangle$. The gluing operation $\Delta$ consisting of identifying the input and the output then extends to a linear map $\Delta: \C\salg{G}\langle\mbf{x}, \mbf{x}^*\rangle \to \C\salg{T}\langle\mbf{x}, \mbf{x}^*\rangle$.

For a $*$-test graph $T=(V, E, \gamma, \varepsilon) \in \salg{T}\langle\mbf{x}, \mbf{x}^*\rangle$, we define the random variable
\begin{equation}\label{eq:2.8}
\trace\big[T(\mfk{A}_n)\big] = \sum_{\phi: V \to [n]} \prod_{e \in E} (\mbf{A}_n^{(\gamma(e))})^{\varepsilon(e)}(\phi(e)).
\end{equation}
We emphasize that we do not define $T(\mfk{A}_n)$ itself: the identity \eqref{eq:2.7} explains the notation.

\begin{center}
\begingroup%
  \makeatletter%
  \providecommand\color[2][]{%
    \errmessage{(Inkscape) Color is used for the text in Inkscape, but the package 'color.sty' is not loaded}%
    \renewcommand\color[2][]{}%
  }%
  \providecommand\transparent[1]{%
    \errmessage{(Inkscape) Transparency is used (non-zero) for the text in Inkscape, but the package 'transparent.sty' is not loaded}%
    \renewcommand\transparent[1]{}%
  }%
  \providecommand\rotatebox[2]{#2}%
  \ifx\svgwidth\undefined%
    \setlength{\unitlength}{468bp}%
    \ifx\svgscale\undefined%
      \relax%
    \else%
      \setlength{\unitlength}{\unitlength * \real{\svgscale}}%
    \fi%
  \else%
    \setlength{\unitlength}{\svgwidth}%
  \fi%
  \global\let\svgwidth\undefined%
  \global\let\svgscale\undefined%
  \makeatother%
  \begin{picture}(1,0.35384615)%
    \put(0,0){\includegraphics[width=\unitlength,page=1]{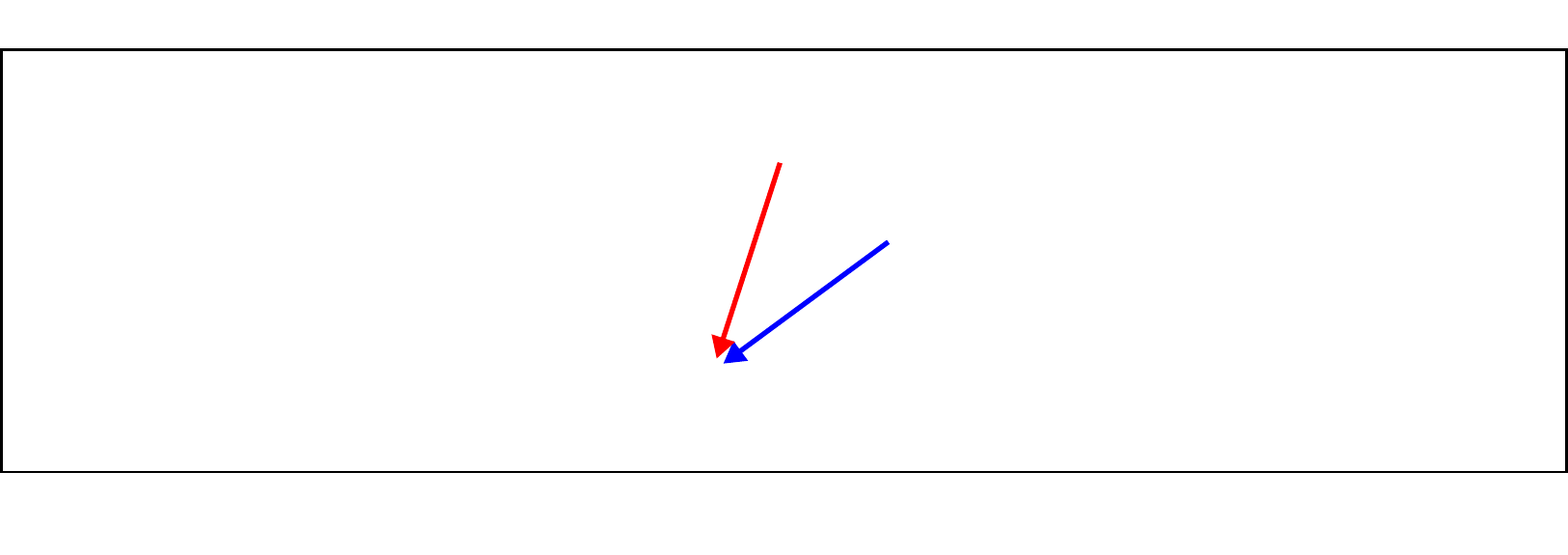}}%
    \put(-0.00141487,0.03834761){\color[rgb]{0,0,0}\makebox(0,0)[lt]{\begin{minipage}{0.99822194\unitlength}\raggedright Figure 9: Examples of $*$-test graphs.\end{minipage}}}%
    \put(0.12170716,1.59355762){\color[rgb]{0,0,0}\makebox(0,0)[lt]{\begin{minipage}{0.11106842\unitlength}\raggedright \end{minipage}}}%
    \put(0,0){\includegraphics[width=\unitlength,page=2]{fig9_stg.pdf}}%
    \put(0.71579922,0.26819654){\color[rgb]{0,0,0}\rotatebox{-5.5830001}{\makebox(0,0)[lt]{\begin{minipage}{0.06799102\unitlength}\raggedright  \end{minipage}}}}%
    \put(0,0){\includegraphics[width=\unitlength,page=3]{fig9_stg.pdf}}%
    \put(0.78598361,0.24217968){\color[rgb]{0,0,0}\makebox(0,0)[lb]{\smash{$y^*$}}}%
    \put(0.73149643,0.19730789){\color[rgb]{0,0,0}\makebox(0,0)[lb]{\smash{$y^*$}}}%
    \put(0.76755412,0.1231893){\color[rgb]{0,0,0}\makebox(0,0)[lb]{\smash{$y^*$}}}%
    \put(0.83886822,0.13480789){\color[rgb]{0,0,0}\makebox(0,0)[lb]{\smash{$y^*$}}}%
    \put(0.85569514,0.20532071){\color[rgb]{0,0,0}\makebox(0,0)[lb]{\smash{$y^*$}}}%
    \put(0,0){\includegraphics[width=\unitlength,page=4]{fig9_stg.pdf}}%
    \put(0.44992769,0.24391757){\color[rgb]{0,0,0}\makebox(0,0)[lb]{\smash{$x^*$}}}%
    \put(0.54928666,0.23590475){\color[rgb]{0,0,0}\makebox(0,0)[lb]{\smash{$x$}}}%
    \put(0.48678667,0.09447847){\color[rgb]{0,0,0}\makebox(0,0)[lb]{\smash{$x$}}}%
    \put(0.41467128,0.15257142){\color[rgb]{0,0,0}\makebox(0,0)[lb]{\smash{$x$}}}%
    \put(0.56780108,0.14602584){\color[rgb]{0,0,0}\makebox(0,0)[lb]{\smash{$y^*$}}}%
    \put(0.51883795,0.14295603){\color[rgb]{0,0,0}\makebox(0,0)[lb]{\smash{$z$}}}%
    \put(0.44992769,0.18929507){\color[rgb]{0,0,0}\makebox(0,0)[lb]{\smash{$y^*$}}}%
    \put(0,0){\includegraphics[width=\unitlength,page=5]{fig9_stg.pdf}}%
    \put(0.18888121,0.22705684){\color[rgb]{0,0,0}\makebox(0,0)[lb]{\smash{$y$}}}%
    \put(0.21034436,0.11741642){\color[rgb]{0,0,0}\makebox(0,0)[lb]{\smash{$x^*$}}}%
  \end{picture}%
\endgroup%

\end{center}

\begin{defn}[Traffic distribution of matrices]\label{defn2.3.7}
Let $\mfk{A}_n$ be a family of random $n \times n$ matrices in $(\matn(L^{\infty-}(\Omega, \salg{F}, \prob)), \E\frac{1}{n}\trace)$. We define the \emph{traffic distribution} of $\mfk{A}_n$ as the linear functional $\tau_{\mfk{A}_n} : \C\salg{T}\langle\mbf{x}, \mbf{x}^*\rangle \to \C$ determined by
\[
T \mapsto \E\bigg[\frac{1}{n}\trace\big[T(\mfk{A}_n)\big]\bigg], \qquad \forall T \in \salg{T}\langle\mbf{x}, \mbf{x}^*\rangle.
\]
We say that a sequence of families $(\mfk{A}_n)$ \emph{converges in traffic distribution} if the corresponding sequence of traffic distributions $(\tau_{\mfk{A}_n})$ converges pointwise, i.e.,
\[
\lim_{n \to \infty} \tau_{\mfk{A}_n}(T) \in \C, \qquad \forall T \in \C\salg{T}\langle\mbf{x}, \mbf{x}^*\rangle.
\]
\end{defn}

Note that the usual $*$-distribution of $\mfk{A}_n$ factors through the traffic distribution via the embedding $\eta_{\mbf{x}}$ and the gluing operation $\Delta$. We can formalize this with the commutative diagram
\begin{equation}\label{eq:2.9}
\begin{tikzcd}[row sep=2.5cm, column sep=1.5cm]
\Scale[1]{\displaystyle\C\langle\mbf{x},\mbf{x}^*\rangle} \arrow{r}{\Scale[1]{\nu_{\mfk{A}_n}}} \arrow{d}[swap]{\Scale[1]{\eta_{\mbf{x}}}}& \Scale[1]{\C} \\
\Scale[1]{\displaystyle\C\salg{G}\langle\mbf{x},\mbf{x}^*\rangle} \arrow{r}[swap]{\Scale[1]{\Delta}}& \Scale[1]{\C\salg{T}\langle\mbf{x},\mbf{x}^*\rangle} \arrow{u}[swap]{\Scale[1]{\tau_{\mfk{A}_n}}}
\end{tikzcd}
\end{equation}

\begin{center}
\begingroup%
  \makeatletter%
  \providecommand\color[2][]{%
    \errmessage{(Inkscape) Color is used for the text in Inkscape, but the package 'color.sty' is not loaded}%
    \renewcommand\color[2][]{}%
  }%
  \providecommand\transparent[1]{%
    \errmessage{(Inkscape) Transparency is used (non-zero) for the text in Inkscape, but the package 'transparent.sty' is not loaded}%
    \renewcommand\transparent[1]{}%
  }%
  \providecommand\rotatebox[2]{#2}%
  \ifx\svgwidth\undefined%
    \setlength{\unitlength}{468bp}%
    \ifx\svgscale\undefined%
      \relax%
    \else%
      \setlength{\unitlength}{\unitlength * \real{\svgscale}}%
    \fi%
  \else%
    \setlength{\unitlength}{\svgwidth}%
  \fi%
  \global\let\svgwidth\undefined%
  \global\let\svgscale\undefined%
  \makeatother%
  \begin{picture}(1,0.41080002)%
    \put(0,0){\includegraphics[width=\unitlength,page=1]{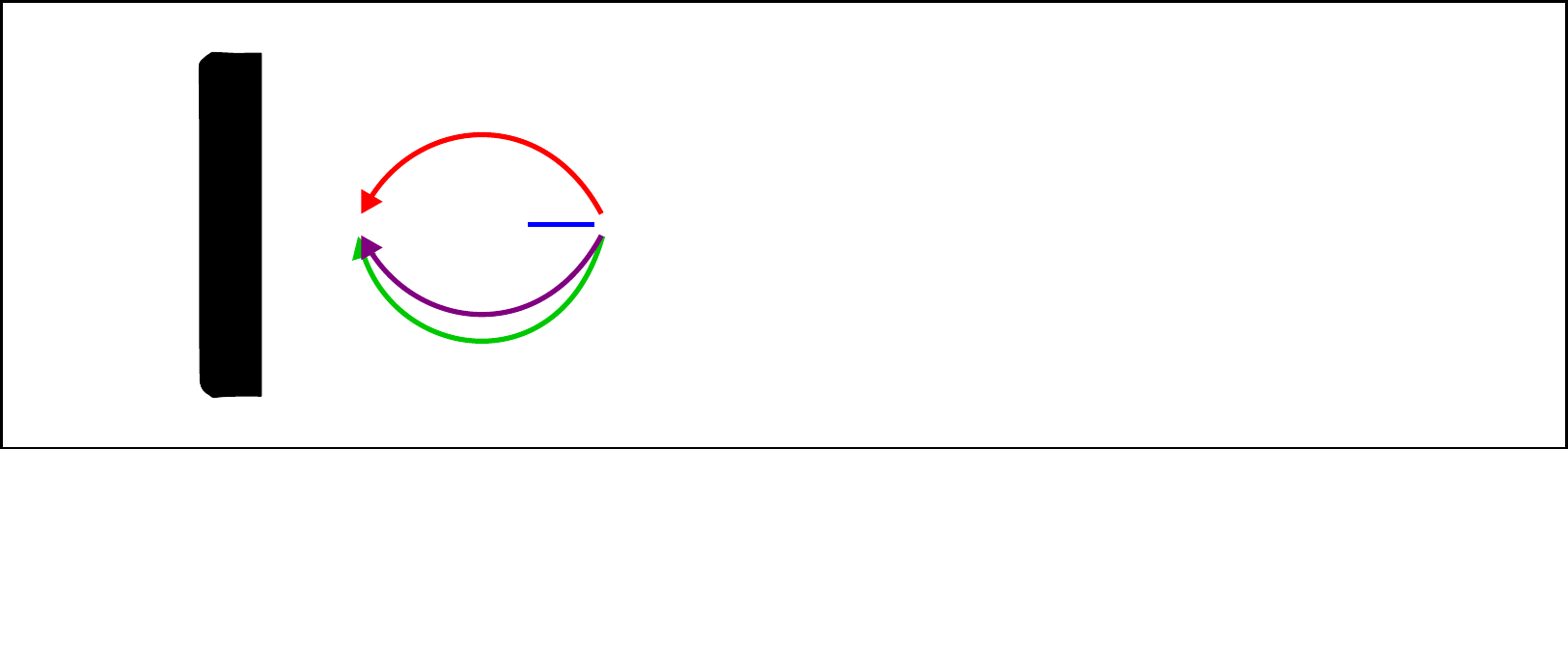}}%
    \put(-0.00141487,0.1105186){\color[rgb]{0,0,0}\makebox(0,0)[lt]{\begin{minipage}{0.99822195\unitlength}\raggedright Figure 10: An example of the trace of a graph of matrices. Here, we consider the trace of a Hadamard product $\circ_{j = 1}^k\, \mbf{A}_n^{(i(j))} = (\circ_{j = 1}^k\, t_{x_{i(j)}})(\mfk{A}_n)$ in the form \eqref{eq:2.8}, as recorded by the traffic distribution via \eqref{eq:2.9}. \end{minipage}}}%
    \put(0.12170717,1.66592649){\color[rgb]{0,0,0}\makebox(0,0)[lt]{\begin{minipage}{0.11106843\unitlength}\raggedright \end{minipage}}}%
    \put(0,0){\includegraphics[width=\unitlength,page=2]{fig10_trace.pdf}}%
    \put(0.45843015,0.27868509){\color[rgb]{0,0,0}\rotatebox{-90}{\makebox(0,0)[lb]{\smash{in}}}}%
    \put(0,0){\includegraphics[width=\unitlength,page=3]{fig10_trace.pdf}}%
    \put(0.08774821,0.26065635){\color[rgb]{0,0,0}\makebox(0,0)[lb]{\smash{$\trace($}}}%
    \put(0.64150279,0.23684118){\color[rgb]{0,0,0}\makebox(0,0)[lt]{\begin{minipage}{0.18632479\unitlength}\raggedright  \end{minipage}}}%
    \put(0.49479949,0.26132471){\color[rgb]{0,0,0}\makebox(0,0)[lb]{\smash{$(\mfk{A}_n)) = \trace\big[$}}}%
    \put(0.8581809,0.26132534){\color[rgb]{0,0,0}\makebox(0,0)[lb]{\smash{$(\mfk{A}_n)\big]$}}}%
    \put(0.154131,0.24755612){\color[rgb]{0,0,0}\rotatebox{90}{\makebox(0,0)[lb]{\smash{\textcolor{white}{out}}}}}%
    \put(0,0){\includegraphics[width=\unitlength,page=4]{fig10_trace.pdf}}%
    \put(0.72076103,0.17180221){\color[rgb]{0,0,0}\makebox(0,0)[lb]{\smash{$\cdots$}}}%
    \put(0.71395013,0.32217562){\color[rgb]{0,0,0}\makebox(0,0)[lb]{\smash{$x_{i(1)}$}}}%
    \put(0.65906231,0.26488395){\color[rgb]{0,0,0}\makebox(0,0)[lb]{\smash{$x_{i(2)}$}}}%
    \put(0.76883795,0.26488395){\color[rgb]{0,0,0}\makebox(0,0)[lb]{\smash{$x_{i(k)}$}}}%
    \put(0,0){\includegraphics[width=\unitlength,page=5]{fig10_trace.pdf}}%
    \put(0.28646616,0.35582946){\color[rgb]{0,0,0}\makebox(0,0)[lb]{\smash{$x_{i(1)}$}}}%
    \put(0.28646616,0.300541){\color[rgb]{0,0,0}\makebox(0,0)[lb]{\smash{$x_{i(2)}$}}}%
    \put(0.3024918,0.2578359){\color[rgb]{0,0,0}\makebox(0,0)[lb]{\smash{$\vdots$}}}%
    \put(0.28646616,0.17233587){\color[rgb]{0,0,0}\makebox(0,0)[lb]{\smash{$x_{i(k)}$}}}%
    \put(0.27604949,0.23082946){\color[rgb]{0,0,0}\makebox(0,0)[lb]{\smash{$x_{i(k-1)}$}}}%
  \end{picture}%
\endgroup%

\end{center}

Finally, Proposition \hyperref[prop2.3.6]{2.3.6} further implies that the traffic distribution is invariant under conjugation by the permutation matrices, i.e.,
\[
\mfk{A}_n \teq \mbf{P}_\sigma \mfk{A}_n \mbf{P}_\sigma^*, \qquad \forall\sigma\in\mfk{S}_n.
\]

\begin{rem}\label{rem2.3.8}
The construction in this section originates in the work \cite{MS12} of Mingo and Speicher, who were interested in bounding partition restricted sums of products of matrix entries. An earlier notion also appears as an example in the work \cite[Example 2.6]{Jon99} of Jones on planar algebras. We abstract the features in the matricial setting to give the formal definition of a traffic space in the next section.
\end{rem}

\subsection{Traffic spaces}\label{sec2.4}

We use the language of commutative diagrams to define a traffic space: the content of a diagram is its commutativity. We encourage the reader to follow through the axioms of a traffic space with the matrices in mind.

\begin{defn}[Traffic space]\label{defn2.4.1}
A \emph{traffic space} is a tracial $*$-probability space $(\salg{A}, \varphi)$ together with a compatible action of the operad of graph operations $\salg{G} = \bigcup_{K \geq 0} \salg{G}_K$ (see, e.g., \cite{May97}). By this, we mean that for any $g \in \salg{G}_K \subset \salg{G}$, there exists a multilinear map
\[
Z_g: \salg{A}^{\otimes K}  \to \salg{A}, \qquad a_1 \otimes \cdots \otimes a_K \mapsto Z_g(a_1 \otimes \cdots \otimes a_K),
\]
satisfying certain consistency properties. We note that the action of $\salg{G}$ defines a linear evaluation map $\op{eval}_{\ssgp{x}, \salg{A}}: \ssgp{x} \to \salg{A}$ for any set of indeterminates $\mbf{x}$. We make use of both the maps $Z_g$ and $\op{eval}_{\ssgp{x}, \salg{A}}$ in formalizing the following properties:
\begin{enumerate}[label=(\roman*)]
\item (Associativity) For any graph operations $g_1, \ldots, g_K$ ($g_i \in \salg{G}_{L_i}$), the action of the substituted graph operation $g(g_1, \ldots, g_K)$ factors through the action of the $g_i$, i.e.,
\[
\begin{tikzcd}[row sep=2cm, column sep=1cm]
\Scale[1]{\salg{A}^{\otimes \sum_{i=1}^K L_i}} \arrow{dr}[swap]{\Scale[1]{\otimes_{i=1}^K Z_{g_i}}} \arrow{rr}{\Scale[1]{Z_{g(g_1, \ldots, g_K)}}}& &\Scale[1]{\salg{A}} \\
&\Scale[1]{\salg{A}^{\otimes K}} \arrow{ur}[swap]{\Scale[1]{Z_g}}&
\end{tikzcd}
\]

\item (Compatibility) The usual evaluation map $\op{eval}_{\C\langle\mbf{x}, \mbf{x}^*\rangle, \salg{A}}: \C\langle\mbf{x}, \mbf{x}^*\rangle \times \salg{A}^I \to \salg{A}$ factors through the $*$-graph polynomials via the embedding $\eta_{\mbf{x}}: \C\langle\mbf{x}, \mbf{x}^*\rangle \to \ssgp{x}$, i.e.,
\[
\begin{tikzcd}[row sep=2cm, column sep=0cm]
\Scale[1]{\C\langle\mbf{x}, \mbf{x}^*\rangle \times \salg{A}^I} \arrow{dr}[swap]{\Scale[1]{\eta_{\mbf{x}} \times \op{id}}} \arrow{rr}{\Scale[1]{\op{eval}_{\C\langle\mbf{x}, \mbf{x}^*\rangle, \salg{A}}}}& &\Scale[1]{\salg{A}} \\
&\Scale[1]{\ssgp{x} \times \salg{A}^I} \arrow{ur}[swap]{\Scale[1]{\op{eval}_{\ssgp{x},\salg{A}}}}&
\end{tikzcd}
\]

\item (Equivariance) For any $\sigma \in \mfk{S}_K$,
\[
\begin{tikzcd}[row sep=1.5cm, column sep=.75cm]
\Scale[1]{\salg{A}^{\otimes K}} \arrow{dr}[swap]{\sigma^{-1}} \arrow{rr}{\Scale[1]{Z_g}}& &\Scale[1]{\salg{A}} \\
&\Scale[1]{\salg{A}^{\otimes K}} \arrow{ur}[swap]{Z_{g_\sigma}}&
\end{tikzcd}
\]

\item (Involutivity) The actions of the graph operations $g$ and $g^*$ are adjoint to each other with respect to the $*$-operation on $\salg{A}$, i.e.,
\[
\begin{tikzcd}[row sep=2cm, column sep=1.75cm]
\Scale[1]{\salg{A}^{\otimes K}} \arrow{r}{\Scale[1]{*^{\otimes K}}} \arrow{d}[swap]{\Scale[1]{Z_g}}& \Scale[1]{\salg{A}^{\otimes K}} \arrow{d}{Z_{g^*}} \\
\Scale[1]{\salg{A}} \arrow{r}[swap]{\Scale[1]{*}}& \Scale[1]{\salg{A}}
\end{tikzcd}
\]
In particular, in view of properties (i) and (ii), $\op{eval}_{\ssgp{x}, \salg{A}}: \ssgp{x} \to \salg{A}$ defines a morphism of $*$-algebras.

\item (Unity) Evaluating an edge $e$ in a graph operation $g$ on the identity $1_{\salg{A}}$ corresponds to the graph operation $\wtilde{g}$ obtained from $g$ by identifying the vertices $\source(e)$ and $\target(e)$ and deleting the edge $e$ from $g$, i.e.,
\[
Z_g(\cdot_1 \otimes \cdots \otimes \cdot_{K-1} \otimes 1_{\salg{A}}) = Z_{\wtilde{g}}(\cdot_1 \otimes \cdots \otimes \cdot_{K-1}).
\]
Note that this property follows from (i) and (ii). We emphasize it here for the convenience of the reader.
\end{enumerate}

We further require that the trace $\varphi$ factor through the $*$-test graphs via the gluing map $\Delta: \ssgp{x} \to \C\salg{T}\langle\mbf{x}, \mbf{x}^*\rangle$ as in the matricial setting: for any set of indeterminates $\mbf{x}$, there exists a map
\[
\tau_{\mbf{x}, \salg{A}}: \C\salg{T}\langle\mbf{x}, \mbf{x}^*\rangle \times \salg{A}^I \to \C, \qquad (T,\mbf{a}) \mapsto \tau\big[T(\mbf{a})\big]
\]
such that
\[
\begin{tikzcd}[row sep=3.5cm, column sep=2.625cm]
\Scale[1]{\ssgp{x} \times \salg{A}^I} \arrow{r}{\Scale[1]{\op{eval}_{\ssgp{x}, \salg{A}}}} \arrow{d}[swap]{\Scale[1]{\Delta \times \op{id}}}& \Scale[1]{\salg{A}} \arrow{d}{\varphi} \\
\Scale[1]{\C\salg{T}\langle\mbf{x}, \mbf{x}^*\rangle \times \salg{A}^I} \arrow{r}[swap]{\Scale[1]{\tau_{\mbf{x}, \salg{A}}}}& \Scale[1]{\C}
\end{tikzcd}
\]
The maps $\tau_{\mbf{x}, \salg{A}}$ implicitly define a function $\tau: (T, \mbf{a}) \mapsto \tau\big[T(\mbf{a})\big]$ we call the \emph{traffic state}. One visualizes a pair $(T, \mbf{a})$ as a test graph with edge labels in $\salg{A}$. Writing $\C\salg{T}\langle\salg{A}\rangle$ for the vector space of finite linear combinations of test graphs in $\salg{A}$, we can formally define the traffic state as a linear functional $\tau: \C\salg{T}\langle\salg{A}\rangle \to \C$.

We specify a traffic space by a triple $(\salg{A}, \varphi, \tau)$, though we often omit the trace $\varphi$. We require that the traffic state satisfy a technical positivity condition analogous to the positivity condition in a $*$-probability space. We state this condition separately in Definition \hyperref[defn2.4.3]{2.4.3}.
\end{defn}

To define the positivity of the traffic state, we need the notion of a test graph with an arbitrary number of distinguished vertices. As suggested by the above, we extend the notion of a test graph to a general labeling set $S$ with an involution $*: S \to S$. 

\begin{defn}[$n$-graph polynomial]\label{defn2.4.2}
An \emph{$n$-graph monomial} $t = (T, \mbf{v})$ in $S$ consists of a test graph $T = (V, E, \gamma)$ in $S$ (i.e., $\gamma: E \to S$) and an $n$-tuple $\mbf{v} = (v_1, \ldots, v_n) \in V^n$ of distinguished (not necessarily distinct) vertices. We write $\salg{G}^{(n)}\langle S\rangle$ for the set of $n$-graph monomials in $S$. We further write $\C\salg{G}^{(n)}\langle S\rangle$ for the complex vector space of finite linear combinations in $\salg{G}^{(n)}\langle S\rangle$, the elements of which we call the \emph{$n$-graph polynomials}. We define the \emph{adjoint} $t^* = (\overline{T}, \mbf{v})$ of $t$ as the $n$-graph monomial obtained from $t$ by conjugating the underlying test graph $T$ (as in Definition \hyperref[defn2.2.1]{2.2.1}). In particular, in contrast to the $*$-graph monomials, we do not permute the distinguished vertices when taking the adjoint of an $n$-graph monomial. We extend the adjoint operation to a conjugate linear involution on $\C\salg{G}^{(n)}\langle S\rangle$.
\end{defn}

\phantomsection\label{fig11_ngp}
\begin{center}
\begingroup%
  \makeatletter%
  \providecommand\color[2][]{%
    \errmessage{(Inkscape) Color is used for the text in Inkscape, but the package 'color.sty' is not loaded}%
    \renewcommand\color[2][]{}%
  }%
  \providecommand\transparent[1]{%
    \errmessage{(Inkscape) Transparency is used (non-zero) for the text in Inkscape, but the package 'transparent.sty' is not loaded}%
    \renewcommand\transparent[1]{}%
  }%
  \providecommand\rotatebox[2]{#2}%
  \ifx\svgwidth\undefined%
    \setlength{\unitlength}{468bp}%
    \ifx\svgscale\undefined%
      \relax%
    \else%
      \setlength{\unitlength}{\unitlength * \real{\svgscale}}%
    \fi%
  \else%
    \setlength{\unitlength}{\svgwidth}%
  \fi%
  \global\let\svgwidth\undefined%
  \global\let\svgscale\undefined%
  \makeatother%
  \begin{picture}(1,0.32307691)%
    \put(0,0){\includegraphics[width=\unitlength,page=1]{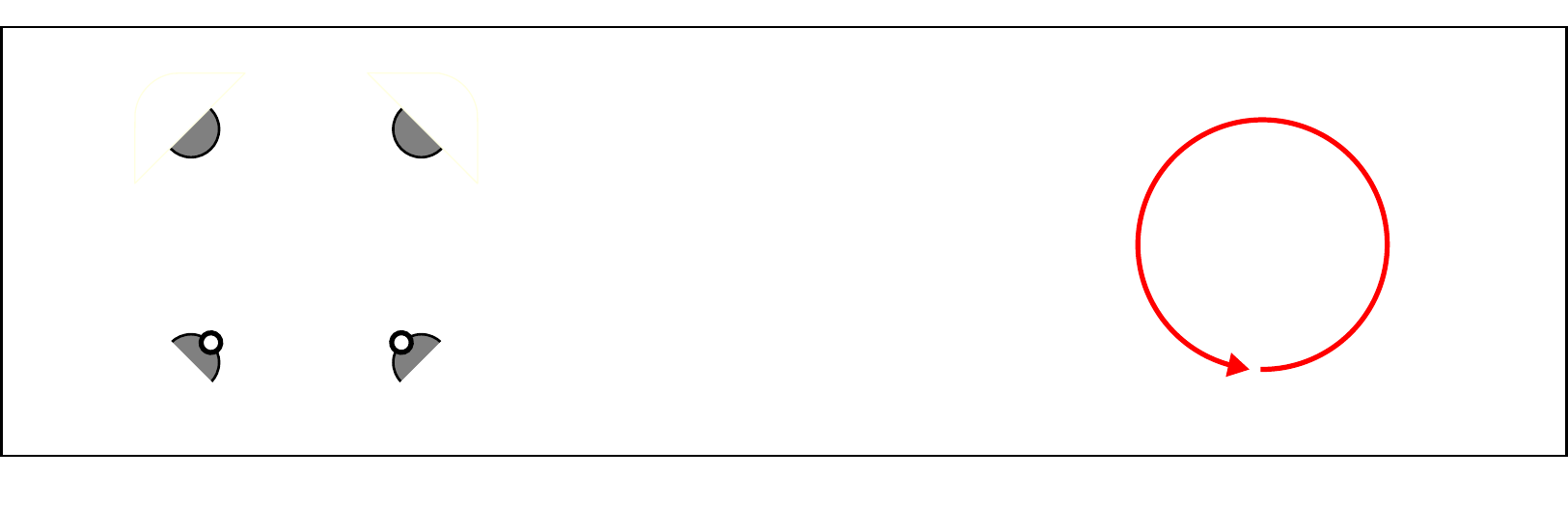}}%
    \put(-0.00141487,0.0176395){\color[rgb]{0,0,0}\makebox(0,0)[lt]{\begin{minipage}{0.99822195\unitlength}\raggedright Figure 11: Examples of $n$-graph monomials. In the last example, we have $n = 0$.\end{minipage}}}%
    \put(0.12170717,1.57325017){\color[rgb]{0,0,0}\makebox(0,0)[lt]{\begin{minipage}{0.11106843\unitlength}\raggedright \end{minipage}}}%
    \put(0,0){\includegraphics[width=\unitlength,page=2]{fig11_ngp.pdf}}%
    \put(0.10306872,0.24796153){\color[rgb]{0,0,0}\makebox(0,0)[lb]{\smash{$1$}}}%
    \put(0,0){\includegraphics[width=\unitlength,page=3]{fig11_ngp.pdf}}%
    \put(0.27444693,0.24796153){\color[rgb]{0,0,0}\makebox(0,0)[lb]{\smash{$2$}}}%
    \put(0,0){\includegraphics[width=\unitlength,page=4]{fig11_ngp.pdf}}%
    \put(0.10306872,0.06967607){\color[rgb]{0,0,0}\makebox(0,0)[lb]{\smash{$3$}}}%
    \put(0.27444693,0.06967607){\color[rgb]{0,0,0}\makebox(0,0)[lb]{\smash{$4$}}}%
    \put(0,0){\includegraphics[width=\unitlength,page=5]{fig11_ngp.pdf}}%
    \put(0.18710718,0.24195191){\color[rgb]{0,0,0}\makebox(0,0)[lb]{\smash{$x^*$}}}%
    \put(0.27164244,0.16182371){\color[rgb]{0,0,0}\makebox(0,0)[lb]{\smash{$z$}}}%
    \put(0.18710718,0.07368279){\color[rgb]{0,0,0}\makebox(0,0)[lb]{\smash{$x^*$}}}%
    \put(0.1061777,0.16182371){\color[rgb]{0,0,0}\makebox(0,0)[lb]{\smash{$z$}}}%
    \put(0.19832513,0.17531018){\color[rgb]{0,0,0}\makebox(0,0)[lb]{\smash{$y^*$}}}%
    \put(0,0){\includegraphics[width=\unitlength,page=6]{fig11_ngp.pdf}}%
    \put(0.7976841,0.2193807){\color[rgb]{0,0,0}\makebox(0,0)[lb]{\smash{$y$}}}%
    \put(0,0){\includegraphics[width=\unitlength,page=7]{fig11_ngp.pdf}}%
    \put(0.40705911,0.21628209){\color[rgb]{0,0,0}\makebox(0,0)[lb]{\smash{$1$}}}%
    \put(0.40705911,0.16897119){\color[rgb]{0,0,0}\makebox(0,0)[lb]{\smash{$2$}}}%
    \put(0.40705911,0.10326606){\color[rgb]{0,0,0}\makebox(0,0)[lb]{\smash{$3$}}}%
    \put(0.4851841,0.08570171){\color[rgb]{0,0,0}\makebox(0,0)[lb]{\smash{$x$}}}%
    \put(0.45914244,0.14920759){\color[rgb]{0,0,0}\makebox(0,0)[lb]{\smash{$y^*$}}}%
    \put(0.49640205,0.21330608){\color[rgb]{0,0,0}\makebox(0,0)[lb]{\smash{$x$}}}%
    \put(0.55309276,0.14774045){\color[rgb]{0,0,0}\makebox(0,0)[lb]{\smash{$x$}}}%
  \end{picture}%
\endgroup%

\end{center}

For $n \geq 1$ and $n$-graph monomials $t_1 = (T_1, \mbf{v})$ and $t_2 = (T_2, \mbf{v}_2)$ in $S$, we define
\[
\Delta_n(t_1, t_2) \in \salg{T}\langle S\rangle = \salg{G}^{(0)}\langle S\rangle
\]
as the test graph obtained from disjoint copies of $T_1$ and $T_2$ by identifying the distinguished vertices $\mbf{v}_1 = (v_1^{(1)}, \ldots, v_n^{(1)})$ and $\mbf{v}_2 = (v_1^{(2)}, \ldots, v_n^{(2)})$ coordinatewise. We extend this operation to a bilinear map
\[
\Delta_n: \C\salg{G}^{(n)}\langle S\rangle^{\otimes 2} \to \C\salg{T}\langle S\rangle,
\]
which allows us to formalize

\begin{defn}[Positivity]\label{defn2.4.3}
We say that a function $\tau: \C\salg{T}\langle S\rangle \to \C$ is \emph{positive} if
\[
\tau\big[\Delta_n(t^*,t)\big] \geq 0, \qquad \forall t \in \C\salg{G}^{(n)}\langle S\rangle.
\]
\end{defn}

\begin{center}
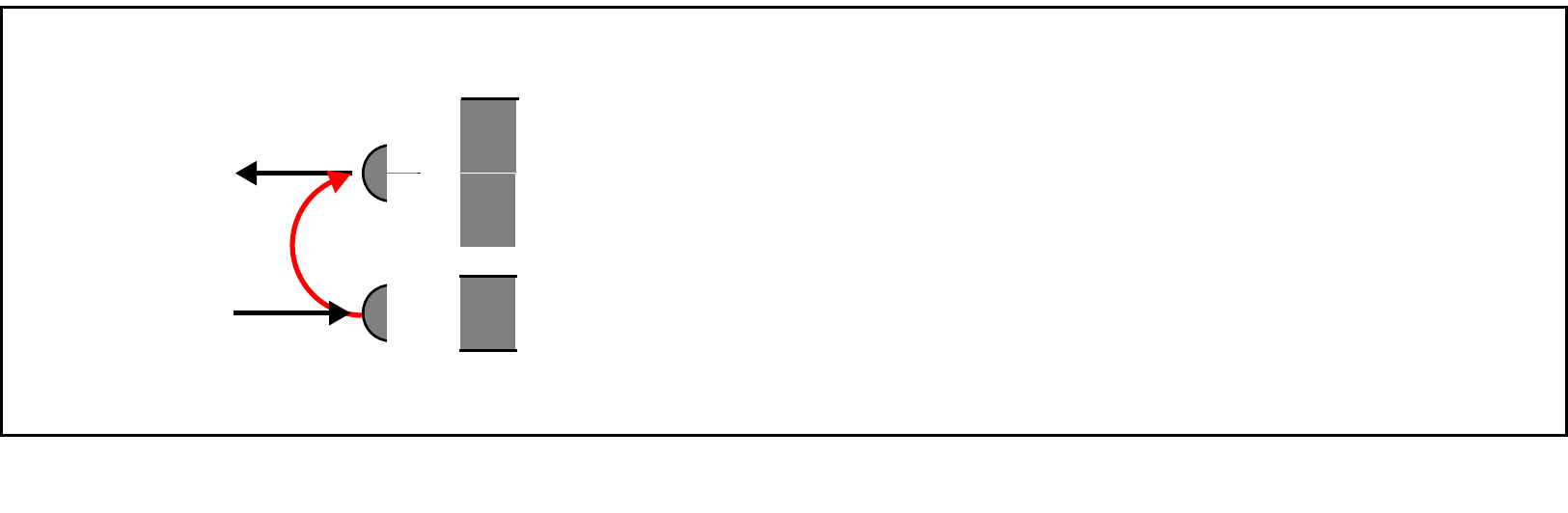
\end{center}

For $S = (\salg{A}, \varphi, \tau)$ a traffic space and $n=2$, the positivity condition is equivalent to the positivity of the trace $\varphi$. We call an element $a \in \salg{A}$ a \emph{traffic random variable} (or simply a \emph{traffic}). We define the traffic distribution of a family of traffics as the pushforward of the traffic state by said family (cf. Definitions \hyperref[defn2.1.5]{2.1.5} and \hyperref[defn2.1.6]{2.1.6}).

\begin{defn}[Traffic distribution]\label{defn2.4.4}
The \emph{traffic distribution} of a family $\mbf{a} = (a_i)_{i \in I}$ in $(\salg{A}, \tau)$ is the linear functional
\[
\tau_{\mbf{a}} : \C\salg{T}\langle\mbf{x}, \mbf{x}^*\rangle \to \C, \qquad T \mapsto \tau\big[T(\mbf{a})\big].
\]
Suppose that for each $n \in \N$ we have a family of traffics $\mbf{a}_n = (a_n^{(i)})_{i \in I}$ in a traffic space $(\salg{A}_n, \tau_n)$. We say that the $\mbf{a}_n$ \emph{converge in traffic distribution} to $\mbf{a}$ if the corresponding traffic distributions $\tau_{\mbf{a}_n}$ converge pointwise to $\tau_{\mbf{a}}$, i.e.,
\[
\lim_{n \to \infty} \tau_{\mbf{a}_n}(T) = \tau_{\mbf{a}}(T), \qquad \forall T \in \C\salg{T}\langle\mbf{x}, \mbf{x}^*\rangle.
\]
\end{defn}

Lemma 2.9 in \cite[v5]{Mal11} establishes the positivity of the traffic state for the random matrices $(\matn(L^{\infty-}(\Omega, \salg{F}, \prob)), \E\frac{1}{n}\trace)$ of Section \hyperref[sec2.3]{2.3}. The positivity condition is of course closed under convergence in traffic distribution; thus, for any traffic convergent sequence of random matrices $\mfk{A}_n \subset \matn(L^{\infty-}(\Omega, \salg{F}, \prob))$, the tracial $*$-probability space $(\ssgp{x}, \lim_{n\to \infty} \E\frac{1}{n}\trace)$ equipped with the traffic state $\tau = \lim_{n\to \infty} \tau_{\mfk{A}_n}$ is again a traffic space. We study the asymptotics of large random matrices within this framework.

\subsection{Traffic independence}\label{sec2.5}

We formulate traffic independence in terms of a combinatorial transform of the traffic state. The construction resembles that of the cumulants (see, e.g., Lecture 11 in \cite{NS06}).

\begin{defn}[Injective version of the traffic state]\label{defn2.5.1}
Let $T = (V, E, \gamma) \in \salg{T}\langle\salg{A}\rangle$ be a test graph in a traffic space $(\salg{A}, \tau)$. We write $\salg{P}(V)$ for the usual poset of partitions of $V$ with its M\"{o}bius function $\mu$. For a partition $\pi \in \salg{P}(V)$, we construct a new test graph $T^\pi$ from $T$ by identifying the vertices $V$ according to the block structure of $\pi$ so that $T^\pi = (V/\mathord{\sim_\pi}, E, \gamma)$.

We define the \emph{injective version of the traffic state} (or \emph{injective traffic state} for short) as the (dual) M\"{o}bius transform of $\tau$, i.e.,
\begin{equation}\label{eq:2.10}
\tau^0: \C\salg{T}\langle\salg{A}\rangle \to \C, \qquad T \mapsto \sum_{\pi \in \salg{P}(V)} \tau\big[T^\pi\big]\mu(0_V,\pi).
\end{equation}
We recover the traffic state $\tau$ via the M\"{o}bius inversion formula (see, e.g., Proposition 3.7.2 in \cite{Sta12}):
\begin{equation}\label{eq:2.11}
\tau\big[T\big] = \sum_{\pi \in \salg{P}(V)} \tau^0\big[T^\pi\big].
\end{equation}
The relations \eqref{eq:2.10} and \eqref{eq:2.11} allow us to work interchangeably between the traffic state and the injective traffic state as convenient.
\end{defn}

\begin{center}
\begingroup%
  \makeatletter%
  \providecommand\color[2][]{%
    \errmessage{(Inkscape) Color is used for the text in Inkscape, but the package 'color.sty' is not loaded}%
    \renewcommand\color[2][]{}%
  }%
  \providecommand\transparent[1]{%
    \errmessage{(Inkscape) Transparency is used (non-zero) for the text in Inkscape, but the package 'transparent.sty' is not loaded}%
    \renewcommand\transparent[1]{}%
  }%
  \providecommand\rotatebox[2]{#2}%
  \ifx\svgwidth\undefined%
    \setlength{\unitlength}{468bp}%
    \ifx\svgscale\undefined%
      \relax%
    \else%
      \setlength{\unitlength}{\unitlength * \real{\svgscale}}%
    \fi%
  \else%
    \setlength{\unitlength}{\svgwidth}%
  \fi%
  \global\let\svgwidth\undefined%
  \global\let\svgscale\undefined%
  \makeatother%
  \begin{picture}(1,0.34615385)%
    \put(-0.00141487,0.09283481){\color[rgb]{0,0,0}\makebox(0,0)[lt]{\begin{minipage}{0.99822194\unitlength}\raggedright Figure 13: An example of the relationship between the trace $\varphi$, the traffic state $\tau$, and the injective traffic state $\tau^0$. Here, the mixed moment $\varphi(a_1 a_2)$ splits into a sum of injective moments.\end{minipage}}}%
    \put(0.12170717,1.64804481){\color[rgb]{0,0,0}\makebox(0,0)[lt]{\begin{minipage}{0.11106842\unitlength}\raggedright \end{minipage}}}%
    \put(0,0){\includegraphics[width=\unitlength,page=1]{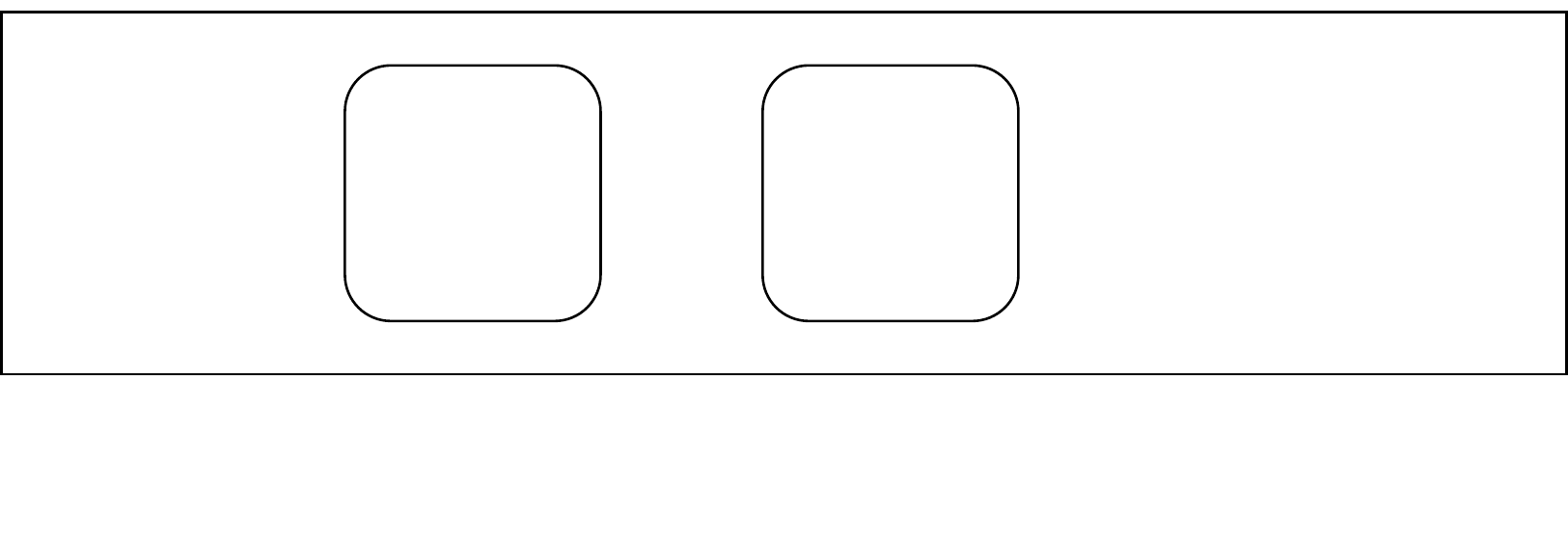}}%
    \put(0.0681168,0.21656828){\color[rgb]{0,0,0}\makebox(0,0)[lb]{\smash{$\varphi(a_1a_2) = \tau\big[$}}}%
    \put(0.39544051,0.21589919){\color[rgb]{0,0,0}\makebox(0,0)[lb]{\smash{$\big] = \tau^0\big[$}}}%
    \put(0.66226743,0.21589919){\color[rgb]{0,0,0}\makebox(0,0)[lb]{\smash{$\big] + \tau^0\big[$}}}%
    \put(0.92548859,0.21589919){\color[rgb]{0,0,0}\makebox(0,0)[lb]{\smash{$\big]$}}}%
    \put(0,0){\includegraphics[width=\unitlength,page=2]{fig13_partition.pdf}}%
    \put(0.82070531,0.18571662){\color[rgb]{0,0,0}\makebox(0,0)[lb]{\smash{$a_1$}}}%
    \put(0.82070494,0.24741972){\color[rgb]{0,0,0}\makebox(0,0)[lb]{\smash{$a_2$}}}%
    \put(0.55689885,0.25606992){\color[rgb]{0,0,0}\makebox(0,0)[lb]{\smash{$a_2$}}}%
    \put(0.55689885,0.17794492){\color[rgb]{0,0,0}\makebox(0,0)[lb]{\smash{$a_1$}}}%
    \put(0,0){\includegraphics[width=\unitlength,page=3]{fig13_partition.pdf}}%
    \put(0.29047257,0.25606992){\color[rgb]{0,0,0}\makebox(0,0)[lb]{\smash{$a_2$}}}%
    \put(0.29047257,0.17794492){\color[rgb]{0,0,0}\makebox(0,0)[lb]{\smash{$a_1$}}}%
  \end{picture}%
\endgroup%

\end{center}

The injective traffic state admits an explicit form without reference to the M\"{o}bius function in the matricial setting. Indeed, the (regular) traffic state $\tau_n$ of $(\matn(L^{\infty-}(\Omega, \salg{F}, \prob)), \E\frac{1}{n}\trace)$ follows from equations \eqref{eq:2.7} and \eqref{eq:2.8}:
\[
\tau_n: \C\salg{T}\langle\matn(L^{\infty-}(\Omega, \salg{F}, \prob))\rangle \to \C, \qquad \tau_n\big[T\big] = \E\bigg[\frac{1}{n}\trace\big[T\big]\bigg] = \E\bigg[\frac{1}{n}\sum_{\phi: V \to [n]} \prod_{e \in E} \gamma(e)(\phi(e))\bigg].
\]
If we define the random variable
\[
\trace^0\big[T\big] = \sum_{\substack{\phi: V \to [n] \text{ s.t.} \\ \phi \text{ is injective}}} \prod_{e \in E} \gamma(e)(\phi(e)) = \sum_{\phi: V \hookrightarrow [n]} \prod_{e \in E} \gamma(e)(\phi(e)),
\]
then we have the equality
\[
\trace\big[T\big] = \sum_{\pi \in \salg{P}(V)} \trace^0\big[T^\pi\big].
\]
The injective version of $\tau_n$ then follows from the relation \eqref{eq:2.11}:
\begin{equation}\label{eq:2.12}
\tau_n^0\big[T\big] = \E\bigg[\frac{1}{n}\trace^0\big[T\big]\bigg] = \E\bigg[\frac{1}{n}\sum_{\phi: V \hookrightarrow [n]} \prod_{e \in E} \gamma(e)(\phi(e))\bigg],
\end{equation}
hence the term \emph{injective} traffic state.

In particular, traffic independence specifies the behavior of the injective traffic state on test graphs of a particular form (cf. Definitions \hyperref[defn2.1.7]{2.1.7} and \hyperref[defn2.1.8]{2.1.8}).

\begin{defn}[Free product of test graphs]\label{defn2.5.2}
Let $\salg{S} = \bigcup_{i \in I} S_i$ be a union of pairwise disjoint labeling sets $S_i$. For a test graph $T \in \salg{T}\langle\salg{S}\rangle$, we define $\chi(T)$ as the simple bipartite graph obtained from $T$ as follows. For each $i \in I$, let $(T_{i, \ell})_{\ell = 1}^{k(i)}$ denote the connected components of the subgraph of $T$ spanned by the labels $S_i$ so that
\[
T_{i, \ell} \in \salg{T}\langle S_i\rangle, \qquad \forall \ell \in [k(i)].
\]
Note that $\sum_{i \in I} k(i) < \infty$ since $T$ is a finite graph. We write $(v_m)_{m=1}^n$ for the vertices of $T$ that belong to more than one of the components $(T_{i, \ell})_{i \in I, \ell \in [k(i)]}$. Together, the components $(T_{i, \ell})_{i \in I, \ell \in [k(i)]}$ and the vertices $(v_m)_{m=1}^n$ form the vertices of $\chi(T)$ with edges determined by inclusion, i.e.,
\[
v_m \sim_{\chi(T)} T_{i, \ell} \quad \Longleftrightarrow \quad v_m \in T_{i, \ell}.
\]
We say that $T$ is a \emph{free product} in $(S_i)_{i \in I}$ if $\chi(T)$ is a tree.
\end{defn}

\begin{center}
\begingroup%
  \makeatletter%
  \providecommand\color[2][]{%
    \errmessage{(Inkscape) Color is used for the text in Inkscape, but the package 'color.sty' is not loaded}%
    \renewcommand\color[2][]{}%
  }%
  \providecommand\transparent[1]{%
    \errmessage{(Inkscape) Transparency is used (non-zero) for the text in Inkscape, but the package 'transparent.sty' is not loaded}%
    \renewcommand\transparent[1]{}%
  }%
  \providecommand\rotatebox[2]{#2}%
  \ifx\svgwidth\undefined%
    \setlength{\unitlength}{468bp}%
    \ifx\svgscale\undefined%
      \relax%
    \else%
      \setlength{\unitlength}{\unitlength * \real{\svgscale}}%
    \fi%
  \else%
    \setlength{\unitlength}{\svgwidth}%
  \fi%
  \global\let\svgwidth\undefined%
  \global\let\svgscale\undefined%
  \makeatother%
  \begin{picture}(1,0.38461538)%
    \put(0,0){\includegraphics[width=\unitlength,page=1]{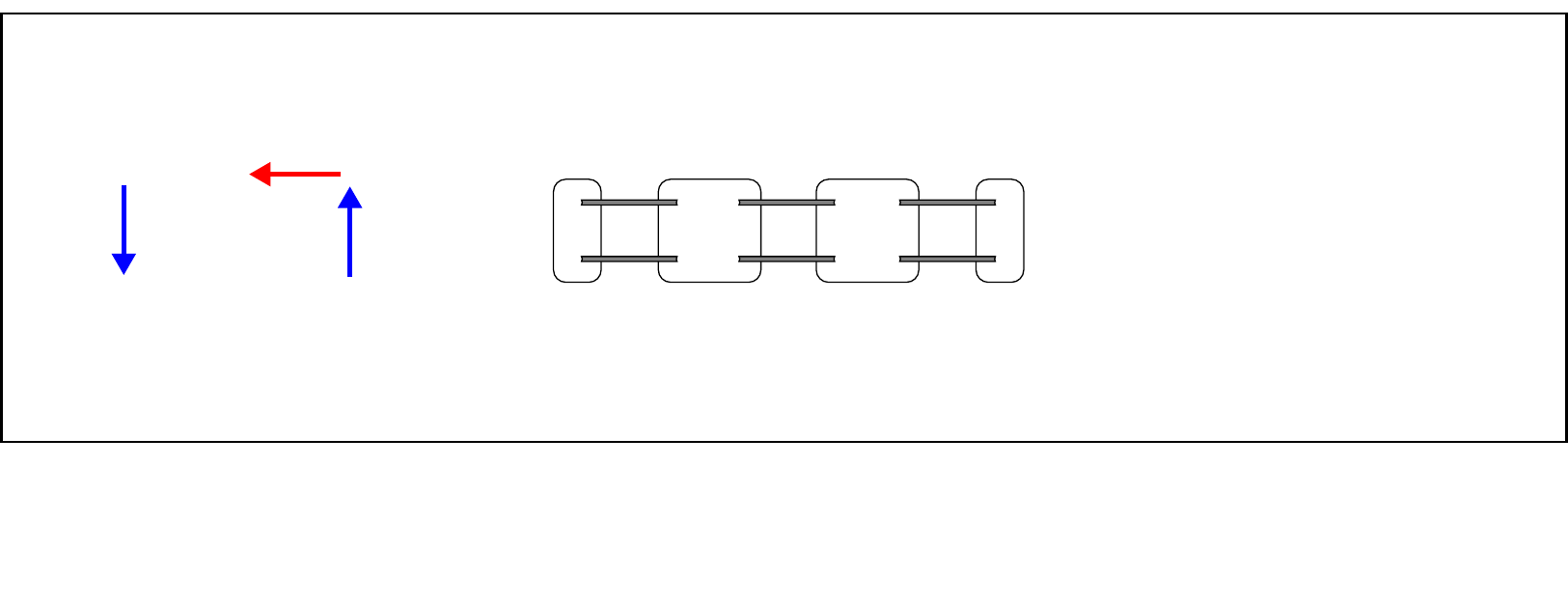}}%
    \put(-0.00141487,0.08819488){\color[rgb]{0,0,0}\makebox(0,0)[lt]{\begin{minipage}{0.99822194\unitlength}\raggedright Figure 14: An example of the construction of the simple graph $\chi(T)$ from a test graph $T$. Here, we color the vertices of $\chi(T)$ to clarify the construction. Note that $T$ is not a free product in this example.\end{minipage}}}%
    \put(0.12170717,1.64340494){\color[rgb]{0,0,0}\makebox(0,0)[lt]{\begin{minipage}{0.11106842\unitlength}\raggedright \end{minipage}}}%
    \put(0.28646616,0.23271809){\color[rgb]{0,0,0}\makebox(0,0)[lb]{\smash{$\stackrel{\chi}{\mapsto}$}}}%
    \put(0.69391808,0.23125095){\color[rgb]{0,0,0}\makebox(0,0)[lb]{\smash{$=$}}}%
    \put(0,0){\includegraphics[width=\unitlength,page=2]{fig14_free.pdf}}%
    \put(0.05609757,0.23590475){\color[rgb]{0,0,0}\makebox(0,0)[lb]{\smash{$z$}}}%
    \put(0,0){\includegraphics[width=\unitlength,page=3]{fig14_free.pdf}}%
    \put(0.23438282,0.23590475){\color[rgb]{0,0,0}\makebox(0,0)[lb]{\smash{$z^*$}}}%
    \put(0.11382513,0.28798809){\color[rgb]{0,0,0}\makebox(0,0)[lb]{\smash{$x$}}}%
    \put(0.18594051,0.28945533){\color[rgb]{0,0,0}\makebox(0,0)[lb]{\smash{$y$}}}%
    \put(0.11783154,0.23590475){\color[rgb]{0,0,0}\makebox(0,0)[lb]{\smash{$x$}}}%
    \put(0.17191806,0.237372){\color[rgb]{0,0,0}\makebox(0,0)[lb]{\smash{$y$}}}%
    \put(0.1018059,0.17180219){\color[rgb]{0,0,0}\makebox(0,0)[lb]{\smash{$x^*$}}}%
    \put(0.17392128,0.17326943){\color[rgb]{0,0,0}\makebox(0,0)[lb]{\smash{$y^*$}}}%
    \put(0,0){\includegraphics[width=\unitlength,page=4]{fig14_free.pdf}}%
  \end{picture}%
\endgroup%

\end{center}

\begin{defn}[Traffic independence]\label{defn2.5.3}
Let $(\salg{A}, \tau)$ be a traffic space. We say that subsets $(\mbf{a}_i : i \in I)$ of $\salg{A}$ (with union $\mbf{a} = \bigcup_{i \in I} \mbf{a}_i$) are \emph{traffic independent} if for any $T \in \salg{T}\langle\mbf{x}, \mbf{x}^*\rangle$,
\begin{equation}\label{eq:2.13}
\tau^0\big[T(\mbf{a})\big] =
\begin{cases}
\prod_{i \in I} \prod_{\ell=1}^{k(i)} \tau^0\big[T_{i, \ell}(\mbf{a}_i)\big] & \quad \text{if $T$ is a free product in $((\mbf{x}_i, \mbf{x}_i^*) : i \in I)$,}  \\
\hfil 0                                                    & \quad \text{otherwise.}
\end{cases}
\end{equation}

Suppose instead that for each $n \in \N$ we have subsets $(\mbf{a}_n^{(i)} : i \in I)$ of a traffic space $(\salg{A}_n, \tau_n)$ with union $\mbf{a}_n = \bigcup_{i \in I} \mbf{a}_n^{(i)} = (a_n^{(i,j)})_{i \in I, j \in J_i}$. We say that the $(\mbf{a}_n^{(i)} : i \in I)$ are \emph{asymptotically traffic independent} if the joint traffic distributions $\tau_{\mbf{a}_n}: \C\salg{T}\langle\mbf{x}, \mbf{x}^*\rangle \to \C$ converge pointwise to a limit $\tau$ such that for any $T \in \salg{T}\langle\mbf{x}, \mbf{x}^*\rangle$,
\begin{equation}\label{eq:2.14}
\tau^0\big[T\big] =
\begin{cases}
\prod_{i \in I} \prod_{\ell=1}^{k(i)} \tau^0\big[T_{i, \ell}\big] & \quad \text{if $T$ is a free product in $((\mbf{x}_i, \mbf{x}_i^*) : i \in I)$,}  \\
\hfil 0                                                    & \quad \text{otherwise.}
\end{cases}
\end{equation}
\end{defn}

\begin{rem}\label{rem2.5.4}
The relations \eqref{eq:2.10} and \eqref{eq:2.11} characterize the joint traffic distribution of traffic independent random variables in terms of their corresponding marginal traffic distributions. In the asymptotic case, we can realize the limit $\tau$ defined by \eqref{eq:2.14} as the traffic state of the traffic space $(\ssgp{x}, \lim_{n \to \infty} \tau_n)$, in which case we have that the $\mbf{a}_n$ converge in traffic distribution to
\[
t_{\mbf{x}} = \bigcup_{i \in I} t_{\mbf{x}_i} = (t_{x_{i ,j}})_{i \in I, j \in J_i} \subset (\ssgp{x}, \lim_{n \to \infty} \tau_n)
\]
with $(t_{\mbf{x}_i} : i \in I)$ traffic independent (recall the notation $t_{x_{i,j}}$ from \eqref{eq:2.3}).

We note that the traffic distribution of a subset $\mbf{a}_i$ specifies the traffic distribution of the generated traffic space $\mathcal{A}_i$ (i.e., $\mathcal{A}_i$ is the smallest unital $*$-subalgebra containing $\mbf{a}_i$ that is closed under the action of the graph operations). Naturally, the traffic independence of any family of subsets $(\mbf{a}_i : i \in I)$ then extends to the generated traffic spaces $(\salg{A}_i : i \in I)$ \cite[Proposition 3.4]{Mal11}. In the context of the previous paragraph, this implies that we actually have the traffic independence of the traffic spaces $(\C\salg{G}\langle\mbf{x}_i, \mbf{x}_i^*\rangle : i \in I)$ in $(\ssgp{x}, \lim_{n \to \infty} \tau_n)$.  
\end{rem}

For a family of traffic spaces $(\salg{A}_j, \tau_j)_{j \in J}$, we can find a traffic independent realization of the $(\salg{A}_j, \tau_j)_{j \in J}$ inside of a larger traffic space $(\salg{A}, \tau)$. Intuitively, we imagine $\salg{A}$ as a suitable set of graphs in $\bigcup_{j \in J} \salg{A}_j$, while equation \eqref{eq:2.13} completely determines our choice of $\tau = *_{j \in J}\, \tau_j$. The formal construction involves a number of technical details: most notably, in establishing the positivity of $\tau$. We refer the reader to \cite{CDM16} for the existence of such a (traffic) free product; however, one need not appeal to the free product construction in order to find instances of traffic independence. More concretely, Theorem 2.8 in \cite{Mal11}, recorded below, shows that traffic independence describes the asymptotic behavior of permutation invariant random matrices.

\begin{thm}[Criteria for asymptotic traffic independence]\label{thm2.5.5}
Let $I$ be an index set, and suppose that for each $n \in \N$ and $i \in I$ we have a family $\mfk{A}_n^{(i)} = (\mbf{A}_n^{(i,j)})_{j \in J_i}$ of random $n \times n$ matrices satisfying the following properties: 
\begin{enumerate}[label=(\roman*)]
\item (Independence) The families $(\mfk{A}_n^{(i)} : i \in I)$ are independent.
\item (Permutation invariance) The distribution of all but at most one of the families $\mfk{A}_n^{(i)}$ is invariant under conjugation by the permutation matrices, i.e.,
\[
\mbf{P}_\sigma \mfk{A}_n^{(i)} \mbf{P}_\sigma^* \deq \mfk{A}_n^{(i)}, \qquad \forall \sigma \in \mfk{S}_n.
\]
\item (Convergence in traffic distribution) For each $i \in I$, the sequence $(\mfk{A}_n^{(i)})$ converges in traffic distribution.
\item (Factorization) For each $i \in I$ and any finite collection of $*$-test graphs $T_1, \ldots, T_\ell \in \salg{T}\langle\mbf{x}_i, \mbf{x}_i^*\rangle$,
\begingroup
\setlength\abovedisplayskip{8.5pt}
\setlength\belowdisplayskip{8.5pt}
\begin{equation}\label{eq:2.15}
\lim_{n \to \infty} \E\bigg[\prod_{m=1}^\ell \frac{1}{n}\emph{tr}\big[T_m(\mfk{A}_n^{(i)})\big]\bigg] = \prod_{m=1}^\ell \bigg(\lim_{n \to \infty} \E\bigg[\frac{1}{n}\emph{tr}\big[T_m(\mfk{A}_n^{(i)})\big]\bigg]\bigg),
\end{equation}
\endgroup
where the limits on the right exist by (iii).
\end{enumerate}
Then the families $(\mfk{A}_n^{(i)} : i \in I)$ are asymptotically traffic independent and satisfy the joint factorization property
\[
\lim_{n \to \infty} \E\bigg[\prod_{m=1}^\ell \frac{1}{n}\emph{tr}\big[T_m(\mfk{A}_n)\big]\bigg] = \prod_{m=1}^\ell \bigg(\lim_{n \to \infty} \E\bigg[\frac{1}{n}\emph{tr}\big[T_m(\mfk{A}_n)\big]\bigg]\bigg)
\]
for any finite collection of $*$-test graphs $T_1, \ldots, T_\ell \in \salg{T}\langle\mbf{x}, \mbf{x}^*\rangle$, where $\mfk{A}_n = \bigcup_{i \in I} \mfk{A}_n^{(i)}$.
\end{thm}

The assumptions of Theorem \hyperref[thm2.5.5]{2.5.5} turn out to be surprisingly mild in practice and hold for many classical random matrix ensembles: for example, the Wigner matrices, Haar distributed unitary matrices, and uniformly distributed permutation matrices \cite{Mal11}. Of course, these ensembles are already well-studied within the context of free probability, but Theorem \hyperref[thm2.5.5]{2.5.5} also applies to random matrix ensembles traditionally outside of the domain of free probability: for example, the heavy Wigner matrices \cite{Mal17}. In Section \hyperref[sec4]{4}, we further show how the random band matrices fit neatly into the traffic probability framework.

We conclude with a central limit theorem for traffic independence. The version stated below is contained in the more general Theorem 8.18 of \cite{Mal11} and interpolates between the classical CLT and the free CLT (cf. Theorem \hyperref[thm2.1.9]{2.1.9}).

\begin{thm}[Traffic CLT]\label{thm2.5.6}
Let $(a_n)$ be a sequence of identically distributed self-adjoint traffics in a traffic space $(\salg{A}, \varphi, \tau)$. Assume that the $a_n$ are centered with unit variance, i.e. $\varphi(a_n) = 0$ and $\varphi(a_n^2) = 1$, and write $s_n = \frac{1}{\sqrt{n}} \sum_{j=1}^n a_j$ for the normalized sum. We split the variance of $a_n$ as
\[
1 = \varphi(a_n^2) = \tau\big[T_1\big] = \tau^0\big[T_1\big] + \tau^0\big[T_2\big] = \alpha + (1-\alpha),
\]
where
\begin{center}
\begingroup%
  \makeatletter%
  \providecommand\color[2][]{%
    \errmessage{(Inkscape) Color is used for the text in Inkscape, but the package 'color.sty' is not loaded}%
    \renewcommand\color[2][]{}%
  }%
  \providecommand\transparent[1]{%
    \errmessage{(Inkscape) Transparency is used (non-zero) for the text in Inkscape, but the package 'transparent.sty' is not loaded}%
    \renewcommand\transparent[1]{}%
  }%
  \providecommand\rotatebox[2]{#2}%
  \ifx\svgwidth\undefined%
    \setlength{\unitlength}{468bp}%
    \ifx\svgscale\undefined%
      \relax%
    \else%
      \setlength{\unitlength}{\unitlength * \real{\svgscale}}%
    \fi%
  \else%
    \setlength{\unitlength}{\svgwidth}%
  \fi%
  \global\let\svgwidth\undefined%
  \global\let\svgscale\undefined%
  \makeatother%
  \begin{picture}(1,0.16470781)%
    \put(0,0){\includegraphics[width=\unitlength,page=1]{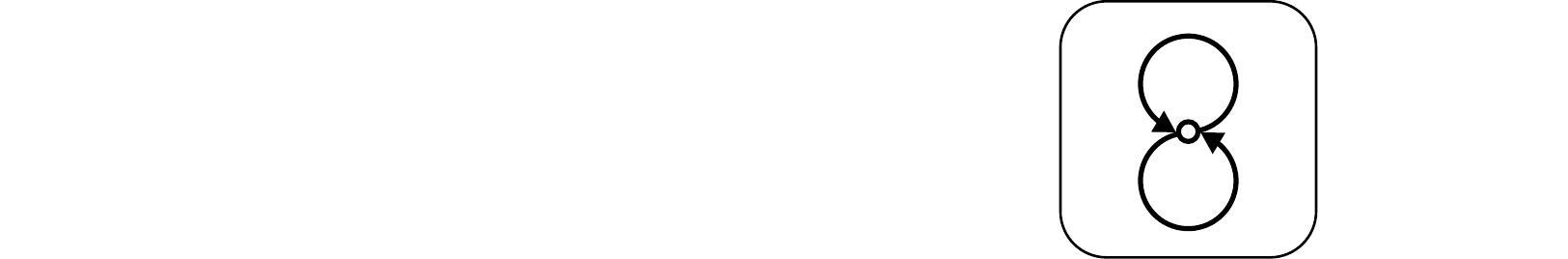}}%
    \put(-0.36386975,1.50609133){\color[rgb]{0,0,0}\makebox(0,0)[lt]{\begin{minipage}{0.11106842\unitlength}\raggedright \end{minipage}}}%
    \put(0,0){\includegraphics[width=\unitlength,page=2]{fig15_clt.pdf}}%
    \put(0.74698736,0.04328332){\color[rgb]{0,0,0}\makebox(0,0)[lb]{\smash{$a_n$}}}%
    \put(0.74698705,0.10498642){\color[rgb]{0,0,0}\makebox(0,0)[lb]{\smash{$a_n$}}}%
    \put(0.29247577,0.11555844){\color[rgb]{0,0,0}\makebox(0,0)[lb]{\smash{$a_n$}}}%
    \put(0.29247577,0.03743344){\color[rgb]{0,0,0}\makebox(0,0)[lb]{\smash{$a_n$}}}%
    \put(0.15826103,0.07605669){\color[rgb]{0,0,0}\makebox(0,0)[lb]{\smash{$T_1 = $}}}%
    \put(0.61258795,0.07605701){\color[rgb]{0,0,0}\makebox(0,0)[lb]{\smash{$T_2 = $}}}%
    \put(0.47773979,0.07538802){\color[rgb]{0,0,0}\makebox(0,0)[lb]{\smash{and}}}%
    \put(0.85239747,0.07038783){\color[rgb]{0,0,0}\makebox(0,0)[lb]{\smash{.}}}%
  \end{picture}%
\endgroup%

\end{center}
If the $a_n$ are traffic independent, then $(s_n)$ converges in distribution to the free convolution $\mu_\alpha = \wsc(0, \alpha) \boxplus \gn(0, 1-\alpha)$, i.e.,
\[
\lim_{n \to \infty} \varphi(s_n^m) = \int_{\R} t^m\, \mu_{\alpha}(dt), \qquad \forall m \in \N.
\]
\end{thm}

We note that \eqref{eq:2.10} and the positivity of the traffic state imply that both
\begin{align*}
\alpha = \tau^0\big[T_1\big] &= \tau\big[T_1\big] - \tau\big[T_2\big] \\
&= \varphi\big((a_n - \Delta(a_n))^2\big) \\
&= \varphi\big((a_n - \Delta(a_n))^*(a_n - \Delta(a_n))\big) \geq 0
\end{align*}
and
\[
1 - \alpha = \tau^0\big[T_2\big] = \tau\big[T_2\big] \geq 0.
\]

\begin{rem}\label{rem2.5.7}
Section 7 in \cite{Mal11} shows how one can realize the traffic CLT for the values $\alpha \in \{0, 1\}$, recovering the usual CLTs. We show how one can realize the traffic CLT for the remaining values $\alpha \in (0, 1)$ in the next section.
\end{rem}

\section{Wigner matrices}\label{sec3}

We now proceed to the details of our motivating discussion on the Wigner matrices. We restrict ourselves to finite-moment Wigner matrices $\salg{X}_n = (\mbf{X}_n^{(i)})_{i \in I}$ with a strong uniform control on the moments in a slight generalization of Definition \hyperref[defn1.1]{1.1}, namely,
\begin{equation}\label{eq:3.1}
\sup_{n \in \N} \sup_{i \in I_0} \sup_{1 \leq j \leq k \leq n} \E|\mbf{X}_n^{(i)}(j, k)|^\ell \leq m_\ell^{(I_0)} < \infty, \qquad \forall I_0 \subset I: \#(I_0) < \infty,
\end{equation}
where the entries $(\mbf{X}_n^{(i)}(j, k) : 1 \leq j \leq k \leq n,\, i \in I)$ are independent with parameter
\begin{equation*}
\E \mbf{X}_n^{(i)}(j, k)^2 = \beta_i, \qquad \forall j < k.
\end{equation*}
In particular, compared to our original definition, we now allow the random variables within our matrices to vary with the dimension $n$; moreover, we no longer assume that they are identically distributed. For technical reasons, we assume that the real and imaginary parts of an off-diagonal entry $\mbf{X}_n^{(i)}(j, k)$ are uncorrelated so that
\begin{equation}\label{eq:3.2}
\E \mbf{X}_n^{(i)}(j, k)^2 = \beta_i = \overline{\beta_i} = \E \mbf{X}_n^{(i)}(k, j)^2.
\end{equation}
For example, this includes the class of all real Wigner matrices ($\beta_i = 1$), but also circularly-symmetric ensembles such as the GUE ($\beta_i = 0$). We comment on the general case of $\beta_i \in \C$ when possible, though the situation becomes much different and often intractable (especially for RBMs). Thus, unless stated otherwise, we assume that $\beta_i = \overline{\beta_i} \in \R$.

\subsection{Traffic distribution}\label{sec3.1}

For such a family $\salg{X}_n$, Male proved the traffic convergence of the corresponding family of normalized Wigner matrices
\[
\salg{W}_n = (\mbf{W}_n^{(i)})_{i \in I} = (\mbf{N}_n \circ \mbf{X}_n^{(i)})_{i \in I}
\]
to the so-called colored double trees \cite[Proposition 4.4]{Mal11}. We review the proof shortly in Proposition \hyperref[prop3.1.2]{3.1.2} for the convenience of the reader, presenting a slightly modified argument that carries forward to the rest of the article. For simplicity, we restrict our attention to test graphs. The general case of a $*$-test graph follows from the self-adjointness of our ensembles.

\begin{defn}[Colored double tree]\label{defn3.1.1}
Let $T = (V, E, \gamma)$ be a test graph in $\mbf{x} = (x_i)_{i \in I}$. We say that $T$ is a \emph{fat tree} if when disregarding the orientation and multiplicity of the edges, $T$ becomes a tree. We further specify that $T$ is a \emph{double tree} if there are exactly two edges between adjacent vertices. We call the pair of edges connecting adjacent vertices in a double tree \emph{twin edges}: \emph{congruent} if they have the same orientation, \emph{opposing} otherwise. Finally, we say that $T$ is a \emph{colored double tree} if $T$ is a double tree such that each pair of twin edges $\{e,e'\}$ shares a common label $\gamma(e) = \gamma(e') \in I$. We record the number $c_i(T)$ of pairs of congruent twin edges with the common label $i$ in a colored double tree $T$.
\end{defn}

\begin{center}
\begingroup%
  \makeatletter%
  \providecommand\color[2][]{%
    \errmessage{(Inkscape) Color is used for the text in Inkscape, but the package 'color.sty' is not loaded}%
    \renewcommand\color[2][]{}%
  }%
  \providecommand\transparent[1]{%
    \errmessage{(Inkscape) Transparency is used (non-zero) for the text in Inkscape, but the package 'transparent.sty' is not loaded}%
    \renewcommand\transparent[1]{}%
  }%
  \providecommand\rotatebox[2]{#2}%
  \ifx\svgwidth\undefined%
    \setlength{\unitlength}{468bp}%
    \ifx\svgscale\undefined%
      \relax%
    \else%
      \setlength{\unitlength}{\unitlength * \real{\svgscale}}%
    \fi%
  \else%
    \setlength{\unitlength}{\svgwidth}%
  \fi%
  \global\let\svgwidth\undefined%
  \global\let\svgscale\undefined%
  \makeatother%
  \begin{picture}(1,0.32692308)%
    \put(0,0){\includegraphics[width=\unitlength,page=1]{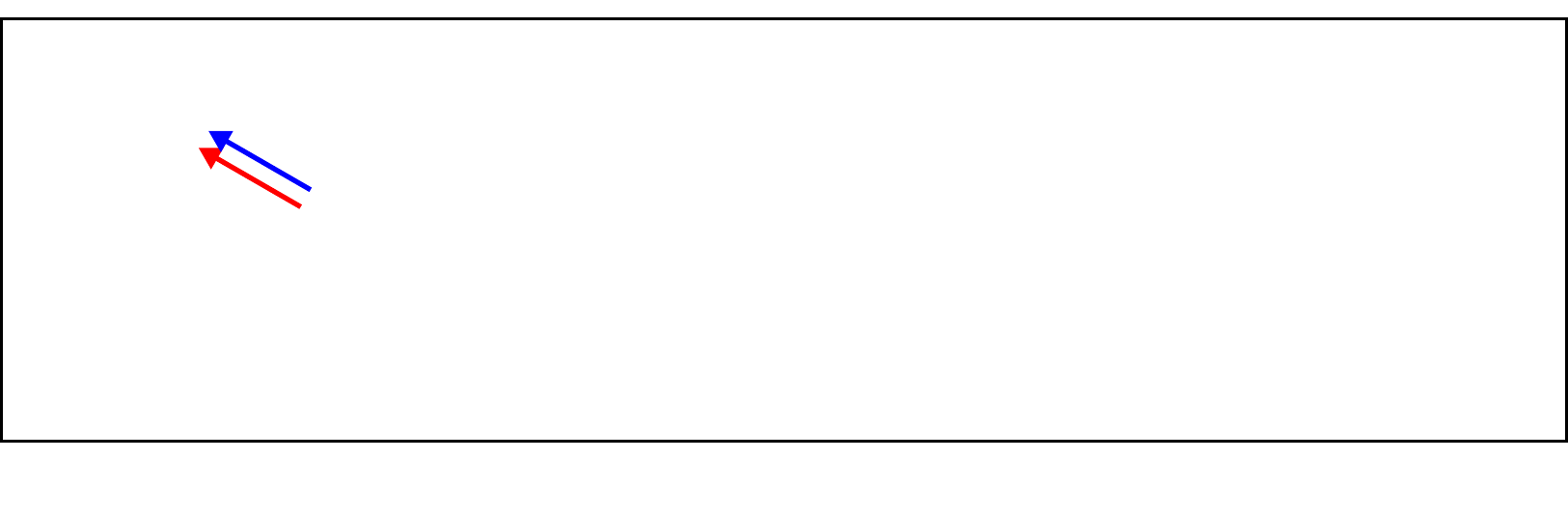}}%
    \put(-0.00141487,0.03099137){\color[rgb]{0,0,0}\makebox(0,0)[lt]{\begin{minipage}{0.99822197\unitlength}\raggedright Figure 15: Examples of a fat tree, a double tree, and a colored double tree respectively.\end{minipage}}}%
    \put(0.12170717,1.58620142){\color[rgb]{0,0,0}\makebox(0,0)[lt]{\begin{minipage}{0.11106843\unitlength}\raggedright \end{minipage}}}%
    \put(0,0){\includegraphics[width=\unitlength,page=2]{fig15_trees.pdf}}%
    \put(0.18830911,0.09168201){\color[rgb]{0,0,0}\makebox(0,0)[lb]{\smash{$x$}}}%
    \put(0.46675463,0.23991116){\color[rgb]{0,0,0}\makebox(0,0)[lb]{\smash{$z$}}}%
    \put(0.52264405,0.23991116){\color[rgb]{0,0,0}\makebox(0,0)[lb]{\smash{$r$}}}%
    \put(0.46393893,0.14776372){\color[rgb]{0,0,0}\makebox(0,0)[lb]{\smash{$x$}}}%
    \put(0,0){\includegraphics[width=\unitlength,page=3]{fig15_trees.pdf}}%
    \put(0.52243252,0.14776372){\color[rgb]{0,0,0}\makebox(0,0)[lb]{\smash{$x$}}}%
    \put(0.44071296,0.19330148){\color[rgb]{0,0,0}\makebox(0,0)[lb]{\smash{$y$}}}%
    \put(0.54828508,0.19183424){\color[rgb]{0,0,0}\makebox(0,0)[lb]{\smash{$r$}}}%
    \put(0.16187103,0.23990949){\color[rgb]{0,0,0}\makebox(0,0)[lb]{\smash{$z$}}}%
    \put(0,0){\includegraphics[width=\unitlength,page=4]{fig15_trees.pdf}}%
    \put(0.13582936,0.19329981){\color[rgb]{0,0,0}\makebox(0,0)[lb]{\smash{$y$}}}%
    \put(0.2434087,0.19183267){\color[rgb]{0,0,0}\makebox(0,0)[lb]{\smash{$r$}}}%
    \put(0.15906231,0.14776372){\color[rgb]{0,0,0}\makebox(0,0)[lb]{\smash{$x$}}}%
    \put(0.21755592,0.14776372){\color[rgb]{0,0,0}\makebox(0,0)[lb]{\smash{$x$}}}%
    \put(0,0){\includegraphics[width=\unitlength,page=5]{fig15_trees.pdf}}%
    \put(0.77145439,0.23991064){\color[rgb]{0,0,0}\makebox(0,0)[lb]{\smash{$y$}}}%
    \put(0.82734382,0.23991064){\color[rgb]{0,0,0}\makebox(0,0)[lb]{\smash{$r$}}}%
    \put(0.76863869,0.1477632){\color[rgb]{0,0,0}\makebox(0,0)[lb]{\smash{$x$}}}%
    \put(0,0){\includegraphics[width=\unitlength,page=6]{fig15_trees.pdf}}%
    \put(0.82713228,0.1477632){\color[rgb]{0,0,0}\makebox(0,0)[lb]{\smash{$x$}}}%
    \put(0.74541273,0.19330096){\color[rgb]{0,0,0}\makebox(0,0)[lb]{\smash{$y$}}}%
    \put(0.85298484,0.19183372){\color[rgb]{0,0,0}\makebox(0,0)[lb]{\smash{$r$}}}%
  \end{picture}%
\endgroup%

\end{center}

We introduce some notation to emphasize the relevant features of our test graphs. This notation will greatly simplify our analysis and features prominently in the remainder of the article. We start with a finite (not necessarily connected) multidigraph $G = (V, E)$. We partition the set of edges $E = L \cup N$ to distinguish between the loops $L$ and the non-loop edges $N = L^c$. As suggested by Definition \hyperref[defn3.1.1]{3.1.1}, we define $\wtilde{G} = (V, \wtilde{E})$ as the undirected graph obtained from $G$ by disregarding the orientation and multiplicity of the edges. Formally, $\wtilde{E} = E/\ssim$ consists of equivalence classes in $E$, where
\[
e \sim e' \quad \Longleftrightarrow \quad \{\source(e), \target(e)\} = \{\source(e'), \target(e')\}.
\]
In this case, our partition $E = L \cup N$ projects down to a partition $\wtilde{E} = \wtilde{L} \cup \wtilde{N}$ between equivalence classes of loops and equivalence classes of non-loops respectively. We may then write the underlying simple graph $\underline{G}$ of $G = (V, E)$ as $\underline{G} = (V, \wtilde{N})$.

\begin{center}
\begingroup%
  \makeatletter%
  \providecommand\color[2][]{%
    \errmessage{(Inkscape) Color is used for the text in Inkscape, but the package 'color.sty' is not loaded}%
    \renewcommand\color[2][]{}%
  }%
  \providecommand\transparent[1]{%
    \errmessage{(Inkscape) Transparency is used (non-zero) for the text in Inkscape, but the package 'transparent.sty' is not loaded}%
    \renewcommand\transparent[1]{}%
  }%
  \providecommand\rotatebox[2]{#2}%
  \ifx\svgwidth\undefined%
    \setlength{\unitlength}{468bp}%
    \ifx\svgscale\undefined%
      \relax%
    \else%
      \setlength{\unitlength}{\unitlength * \real{\svgscale}}%
    \fi%
  \else%
    \setlength{\unitlength}{\svgwidth}%
  \fi%
  \global\let\svgwidth\undefined%
  \global\let\svgscale\undefined%
  \makeatother%
  \begin{picture}(1,0.35384615)%
    \put(0,0){\includegraphics[width=\unitlength,page=1]{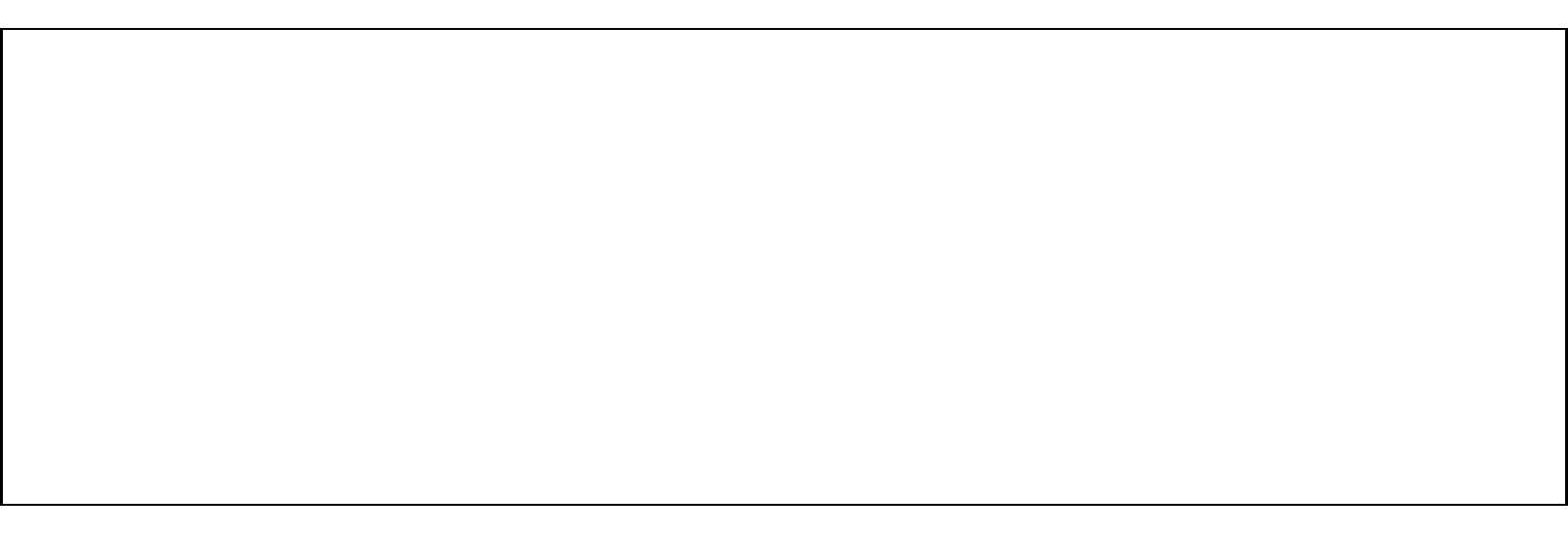}}%
    \put(-0.00141487,0.02501381){\color[rgb]{0,0,0}\makebox(0,0)[lt]{\begin{minipage}{0.99822197\unitlength}\raggedright Figure 16: Examples of the projections $\wtilde{G}$ and $\underline{G}$ starting from a multidigraph $G$.\end{minipage}}}%
    \put(0.12170717,1.59361325){\color[rgb]{0,0,0}\makebox(0,0)[lt]{\begin{minipage}{0.11106843\unitlength}\raggedright \end{minipage}}}%
    \put(0,0){\includegraphics[width=\unitlength,page=2]{fig16_simple.pdf}}%
    \put(0.18541806,0.05161008){\color[rgb]{0,0,0}\makebox(0,0)[lb]{\smash{$G$}}}%
    \put(0.49011998,0.05199726){\color[rgb]{0,0,0}\makebox(0,0)[lb]{\smash{$\wtilde{G}$}}}%
    \put(0.7948187,0.05199726){\color[rgb]{0,0,0}\makebox(0,0)[lb]{\smash{$\underline{G}$}}}%
    \put(0,0){\includegraphics[width=\unitlength,page=3]{fig16_simple.pdf}}%
  \end{picture}%
\endgroup%

\end{center}

Now suppose that our graph $G$ comes with edge labels $\gamma: E \to I$. We count the (undirected) multiplicity of a label $i$ in a class of edges $[e] = \{e' \in E : e \sim e'\} \in \wtilde{E}$ with
\[
m_{i, [e]} = \#(\gamma^{-1}(\{i\}) \cap [e]) \geq 0.
\]
Summing this over the labels in $I$, we of course obtain the multiplicity of the class $[e]$, 
\[
m_{[e]} = \sum_{i \in I} m_{i, [e]} = \#([e]).
\]

If $T = (G, \gamma)$ is a colored double tree, then
\begin{equation}\label{eq:3.3}
m_{i, [e]} \in \{0, 2\} \text{ and } m_{[e]} = 2, \qquad \forall (i, [e]) \in I \times \wtilde{E}. 
\end{equation}
In this case, we write
\begin{equation}\label{eq:3.4}
\gamma([e]) = \gamma(e)
\end{equation}
for the common label $\gamma(e) = \gamma(e')$ of twin edges $[e] = \{e, e'\}$. Conversely, if \eqref{eq:3.3} and \eqref{eq:3.4} hold for a test graph $T$ whose projection $\wtilde{T}$ is a tree, then $T$ is a colored double tree.

\begin{prop}[Semicircular traffics]\label{prop3.1.2}
For any test graph $T$ in $\mbf{x} = (x_i)_{i \in I}$,
\begin{equation}\label{eq:3.5}
\lim_{n \to \infty} \tau^0\big[T(\salg{W}_n)\big] = 
\begin{cases}
\prod_{i \in I} \beta_i^{c_i(T)} & \text{if $T$ is a colored double tree,} \\
\hfil 0 & \text{otherwise}.
\end{cases}
\end{equation}
\end{prop}
\begin{proof}
Suppose that $T = (V, E, \gamma)$. By definition, we have that
\begin{align}
\notag \tau^0\big[T(\salg{W}_n)\big] &= \E \bigg[\frac{1}{n} \sum_{\phi: V \hookrightarrow [n]} \prod_{e \in E} \mbf{W}_n^{(\gamma(e))}(\phi(e))\bigg] \\
&= \frac{1}{n^{1+\frac{\#(E)}{2}}} \sum_{\phi: V \hookrightarrow [n]} \E\bigg[\prod_{e \in E} \mbf{X}_n^{(\gamma(e))}(\phi(e))\bigg]. \label{eq:3.6}
\end{align}
We analyze the asymptotics of \eqref{eq:3.6} by working piecemeal in order to count the number of contributing maps $\phi$ (i.e., maps such that the summand is nonzero). First, we note that the independence of the random variables $\mbf{X}_n^{(i)}(j,k)$ and the injectivity of the maps $\phi$ allow us to factor the product over the expectation, provided that we take into account multi-edges. The relevant information is precisely contained in the projected graph $\wtilde{T} = (V, \wtilde{E})$, which allows us to recast \eqref{eq:3.6} as
\begin{equation}\label{eq:3.7}
\frac{1}{n^{1+\frac{\#(E)}{2}}} \sum_{\phi: V \hookrightarrow [n]} \bigg(\prod_{[\ell] \in \wtilde{L}} \E \bigg[\prod_{\ell' \in [\ell]} \mbf{X}_n^{(\gamma(\ell'))}(\phi(\ell'))\bigg]\bigg) \bigg(\prod_{[e] \in \wtilde{N}} \E \bigg[\prod_{e' \in [e]} \mbf{X}_n^{(\gamma(e'))}(\phi(e'))\bigg]\bigg).
\end{equation}

For non-loop edges $e' \in N$, the independence of the centered random variables $\mbf{X}_n^{(i)}(\phi(e'))$ implies that the second expectation in \eqref{eq:3.7} vanishes if there exists a lone edge $e_0 \in [e]$ with the label $\gamma(e_0) = i_0$. Thus, in order for a summand to be non-zero, each label $i$ present in a class $[e]$ must occur with multiplicity
\begin{equation}\label{eq:3.8}
m_{i,[e]} \geq 2.
\end{equation}
This in turn implies that
\begin{equation}\label{eq:3.9}
\#(N) \geq 2\#(\wtilde{N}).
\end{equation}
The underlying simple graph $\underline{T} = (V, \wtilde{N})$ is of course still connected, whence the inequality
\begin{equation}\label{eq:3.10}
\#(\wtilde{N}) \geq \#(V) - 1.
\end{equation}
Finally, we make use of our strong moment assumption \eqref{eq:3.1} to bound the summands in \eqref{eq:3.7} uniformly in $\phi$ and $n$. In particular, our bound only depends on $T$, i.e.,
\begin{equation}\label{eq:3.11}
\bigg(\prod_{[\ell] \in \wtilde{L}} \E \bigg[\prod_{\ell' \in [\ell]} \mbf{X}_n^{(\gamma(\ell'))}(\phi(\ell'))\bigg]\bigg) \bigg(\prod_{[e] \in \wtilde{N}} \E \bigg[\prod_{e' \in [e]} \mbf{X}_n^{(\gamma(e'))}(\phi(e'))\bigg]\bigg) \leq C_T < \infty.
\end{equation}

Putting everything together, we arrive at the asymptotic
\begin{equation}\label{eq:3.12}
\tau^0\big[T(\salg{W}_n)\big] = O_T(n^{-1-\frac{\#(E)}{2}}n^{\#(V)}) = O_T(n^{-(\frac{\#(N)}{2} - (\#(V)-1))}n^{-\frac{\#(L)}{2}}).
\end{equation}
The inequalities \eqref{eq:3.8}-\eqref{eq:3.10} then imply that $\tau^0\big[T(\salg{W}_n)\big]$ vanishes in the limit unless $T$ is a colored double tree. For such a test graph $T$, \eqref{eq:3.7} becomes
\begin{equation}\label{eq:3.13}
\frac{n^{\underline{\#(V)}}}{n^{\#(V)}} \prod_{[e] \in \wtilde{E}} \bigg(\indc{[e] \text{ are opposing}} + \beta_{\gamma([e])}\indc{[e] \text{ are congruent}}\bigg),
\end{equation}
where $n^{\underline{\#(V)}}$ denotes the falling factorial $n(n-1)\cdots (n-(\#(V)-1))$. The limit \eqref{eq:3.5} now follows.
\end{proof}

Equation \eqref{eq:3.13} explains the apparent asymmetry in the LTD of the Wigner matrices. In particular, if we record the number $o_i(T)$ of pairs of opposing twin edges with the common label $i$ in a colored double tree $T$, then we can rewrite the nontrivial part of \eqref{eq:3.5} as
\[
\prod_{i \in I} \beta_i^{c_i(T)} = \prod_{i \in I} 1^{o_i(T)}\beta_i^{c_i(T)}.
\]
Working directly with this LTD, one can prove the asymptotic traffic independence of the Wigner matrices $\mathcal{W}_n$. To the same end, we can instead appeal to Theorem \hyperref[thm2.5.5]{2.5.5} by choosing a permutation invariant realization of our ensemble and concluding the general result by universality. We employ this technique of instantiation in the next section to realize the traffic CLT.

The careful reader will notice that we have made use of \eqref{eq:3.2} in formulating \eqref{eq:3.13}: by assuming that $\beta_i = \overline{\beta_i}$, we were able to disregard the ordering on the vertices induced by the maps $\phi$ and conclude that congruent twin edges $[e]$ always give a contribution of $\beta_{\gamma([e])}$. In general, for a colored double tree $T$, a summand $S_\phi(T)$ of \eqref{eq:3.7} will depend on $\phi$, namely,
\[
\begin{aligned}
S_\phi(T) = \prod_{[e] \in \wtilde{E}} \bigg(\indc{[e] \text{ are opposing}} &+ \beta_{\gamma([e])}\indc{[e] \text{ are congruent and }  \phi(\target([e])) < \phi(\source([e]))} \\
&+ \overline{\beta_{\gamma([e])}}\indc{[e] \text{ are congruent and } \phi(\target([e])) > \phi(\source([e]))} \bigg).
\end{aligned}
\]
To compute the limit, we must then keep track of the ordering $\psi_\phi$ on the vertices, where 
\[
\psi_\phi: [\#(V)] \stackrel{\sim}{\to} V, \qquad \phi(\psi_\phi(1)) > \cdots > \phi(\psi_\phi(\#(V))).
\]
Note that if $\phi_1: V \hookrightarrow [n_1]$ and $\phi_2: V \hookrightarrow [n_2]$ induce the same ordering $\psi_{\phi_1} = \psi_{\phi_2}$, then the corresponding summands are equal, i.e.,
\[
S_{\phi_1}(T) = \E\bigg[\prod_{e \in E} \mbf{X}_{n_1}^{(\gamma(e))}(\phi_1(e))\bigg] = \E\bigg[\prod_{e \in E} \mbf{X}_{n_2}^{(\gamma(e))}(\phi_2(e))\bigg] = S_{\phi_2}(T).
\]
Thus, for an ordering $\psi: [\#(V)] \stackrel{\sim}{\to} V$, we write $S_\psi(T)$ for the common value of
\[
\{S_\phi(T) : \psi_\phi = \psi\}.
\]
In this case, \eqref{eq:3.7} becomes
\begin{equation}\label{eq:3.14}
\sum_{\psi: [\#(V)] \stackrel{\sim}{\to} V} \frac{\sum_{\phi: V \hookrightarrow [n]} \indc{\psi_\phi = \psi}}{n^{\#(V)}}S_\psi(T).
\end{equation}
One can intuitively verify that
\[
\lim_{n \to \infty} \frac{\sum_{\phi: V \hookrightarrow [n]} \indc{\psi_\phi = \psi}}{n^{\#(V)}} = \frac{1}{\#(V)!}, \qquad \forall \psi: [\#(V)] \stackrel{\sim}{\to} V;
\]
however, in anticipation of Section \hyperref[sec4]{4}, we give a natural integral representation of this limit instead. To this end, we introduce a set of indeterminates $\mbf{x}_V = (x_v)_{v \in V}$ indexed by the vertices of our graph. A straightforward weak convergence argument then shows that
\begin{equation}\label{eq:3.15}
\lim_{n \to \infty} \frac{\sum_{\phi: V \hookrightarrow [n]} \indc{\psi_\phi = \psi}}{n^{\#(V)}} = \int_{[0, 1]^V} \indc{x_{\psi(1)} \geq \cdots \geq x_{\psi(\#(V))}} \, d\mbf{x}_V = \frac{1}{\#(V)!}.
\end{equation}
Indeed, for each $n \in \N$, we can scale a labeling $\phi: V \hookrightarrow [n]$ by $n$ to associate the image $\phi(V) = (\phi(v))_{v \in V}$ with a point $p_{\phi}$ of the latticed hypercube $[0, 1]^V$, namely, 
\[
p_\phi = \bigg(\frac{\phi(v)}{n}\bigg)_{v \in V}.
\]
We imagine integrating the indicator $\indc{x_{\psi(1)} \geq \cdots \geq x_{\psi(\#(V))}}$ against the atomic measure 
\[
\mu_n = \frac{1}{n^{\underline{\#(V)}}} \sum_{\phi: V \hookrightarrow [n]} \delta_{p_\phi}
\]
to obtain the left-hand side of \eqref{eq:3.15} (up to an asymptotically negligible corrective factor). The limit $n \to \infty$ then converts this discretization into the uniform measure on $[0, 1]^V$.

Finally, we arrive at the analogue of \eqref{eq:3.5} for general $\beta_i \in \C$,
\begin{equation}\label{eq:3.16}
\lim_{n \to \infty} \tau^0\big[T(\salg{W}_n)\big] = 
\begin{dcases}
\sum_{\psi: [\#(V)] \stackrel{\sim}{\to} V} \frac{1}{\#(V)!} S_\psi(T) & \text{if $T$ is a colored double tree,} \\
\hfil 0 & \text{otherwise}.
\end{dcases}
\end{equation}
In contrast to Proposition \hyperref[prop3.1.2]{3.1.2}, the LTD \eqref{eq:3.16} does not necessarily describe asymptotically traffic independent matrices $\salg{W}_n$. In fact, if we divide our index set $I$ into two camps $I = I_\R \cup I_\C = \{i \in I: \beta_i \in \R\} \cup \{i \in I: \beta_i \in \C\setminus\R\}$, then the two families $\salg{W}_n^{\R} = (\mbf{W}_n^{(i)})_{i \in I_\R}$ and $\salg{W}_n^{\C} = (\mbf{W}_n^{(i)})_{i \in I_\C}$ are asymptotically traffic independent, but the matrices $\salg{W}_n^{\C}$ are not.

For the first statement, we need only to note that the representative value $S_\psi(T)$ does not depend on the ordering of the vertices that are only adjacent to edges with labels $i \in I_\R$, for which $\beta_i = \overline{\beta_i}$. We can formalize this by considering the subgraphs $T_\R = (V_\R, E_\R)$ and $T_\C = (V_\C, E_\C)$ of $T$ with edge labels in $I_\R$ and $I_\C$ respectively. We write $T_\C = C_1^\C \cup \cdots \cup C_{k_1}^\C$ for the connected components of $T_\C$, each of which is a colored double tree $C_\ell = (V_\ell^\C, E_\ell^\C)$, and similarly for $T_\R = C_1^\R \cup \cdots \cup C_{k_2}^\R$. We call such a graph a \emph{forest of colored double trees}. It follows that a summand $S_\phi(T)$ only depends on the orderings 
\[
\psi_\phi^{(\ell)}: [\#(V_\ell^\C)] \stackrel{\sim}{\to} V_\ell^\C, \qquad \ell \in [k_1]
\]
on each component $C_\ell^\C$. In particular,
\[
S_\phi(T) = \bigg(\prod_{\ell = 1}^{k_1} S_{\psi_\phi^{(\ell)}}(C_\ell^\C)\bigg) \bigg(\prod_{\ell = 1}^{k_2} \prod_{i \in I_\R} \beta_i^{c_i(C_\ell^\R)}\bigg).
\]
In this case, for a concatenation of orderings
\[
\psi = \times_{\ell = 1}^{k_1} \psi_\ell: \bigtimes_{\ell = 1}^{k_1} [\#(V_\ell^\C)] \stackrel{\sim}{\to} \bigtimes_{\ell = 1}^{k_1} V_\ell^\C
\]
with the restrictions
\[
\psi_\ell: [\#(V_\ell^\C)] \stackrel{\sim}{\to} V_\ell^\C,
\]
we write $S_\psi$ for the common value of
\[
\{S_\phi(T): \psi_\phi^{(\ell)} = \psi_\ell \text{ for all } \ell \in [k_1]\}.
\]
We may then write
\[
\tau^0\big[T(\salg{W}_n)\big] = \sum_{\psi: \bigtimes_{\ell = 1}^{k_1} [\#(V_\ell^\C)] \stackrel{\sim}{\to} \bigtimes_{\ell = 1}^{k_1} V_\ell^\C} \frac{\sum_{\phi: V \hookrightarrow [n]} \prod_{\ell = 1}^{k_1} \indc{\psi_\phi^{(\ell)} = \psi_\ell}}{n^{\#(V)}}S_\psi(T),
\]
where
\begin{align}
\notag \lim_{n \to \infty} \frac{\sum_{\phi: V \hookrightarrow [n]} \prod_{\ell = 1}^{k_1} \indc{\psi_\phi^{(\ell)} = \psi_\ell}}{n^{\#(V)}} &= \int_{[0, 1]^V} \prod_{\ell = 1}^{k_1} \indc{x_{\psi_\ell(1)} \geq \cdots \geq x_{\psi_\ell(\#(V_\ell^{\C}))}} \, d\mbf{x}_V \\
&= \bigg(\int_{[0, 1]^{V \setminus V_\C}} \, d\mbf{x}_{V \setminus V_\C}\bigg) \bigg(\prod_{\ell = 1}^{k_1} \int_{[0, 1]^{V_\ell^\C}} \indc{x_{\psi_\ell(1)} \geq \cdots \geq x_{\psi_\ell(\#(V_\ell^{\C}))}} \, d\mbf{x}_{V_\ell^\C}\bigg) \label{eq:3.17} \\
\notag &= \frac{1}{\prod_{\ell = 1}^{k_1} \#(V_\ell^\C)!}.
\end{align}
We conclude that
\begin{align*}
\lim_{n \to \infty} \tau^0\big[T(\salg{W}_n)\big] &= \sum_{\psi: \bigtimes_{\ell = 1}^{k_1} [\#(V_\ell^\C)] \stackrel{\sim}{\to} \bigtimes_{\ell = 1}^{k_1} V_\ell^\C} \frac{1}{\prod_{\ell = 1}^{k_1} \#(V_\ell^\C)!} \bigg(\prod_{\ell = 1}^{k_1} S_{\psi_\ell}(C_\ell^\C)\bigg) \bigg(\prod_{\ell = 1}^{k_2} \prod_{i \in I_\R} \beta_i^{c_i(C_\ell^\R)}\bigg) \\
&= \bigg(\prod_{\ell = 1}^{k_1} \sum_{\psi_\ell: [\#(V_\ell^\C)] \to V_\ell^\C} \frac{1}{\#(V_\ell^\C)!} S_{\psi_\ell}(C_\ell^\C)\bigg) \bigg(\prod_{\ell = 1}^{k_2} \prod_{i \in I_\R} \beta_i^{c_i(C_\ell^\R)} \bigg) \\
&= \bigg(\prod_{\ell = 1}^{k_1} \lim_{n \to \infty} \tau^0\big[C_\ell^\C(\salg{W}_n^{\C})\big]\bigg) \bigg(\prod_{\ell = 1}^{k_2} \lim_{n \to \infty} \tau^0\big[C_\ell^\R(\salg{W}_n^\R)\big]\bigg),
\end{align*}
as was to be shown.

Intuitively, we imagine each pair of twin edges $[e]$ imposing a constraint coming from the ordering of its adjacent vertices $\{\source([e]), \target([e])\}$. We gather these constraints in the ordering $\psi_\phi$ to carry out the calculation of $S_{\phi} = S_{\psi(\phi)}$; however, if $\gamma([e]) \in I_\R$, the constraint becomes vacuous and we can disregard it, which corresponds to discarding the edge $[e]$ (but keeping the adjacent vertices). In this way, we arrive at the integral \eqref{eq:3.17} (and, after discarding the isolated vertices, the forest of colored double trees $T_\C$). We return to this notion of a ``free'' edge $[e]$ in a slightly different context in Section \hyperref[sec4]{4}.

\begin{center}
\begingroup%
  \makeatletter%
  \providecommand\color[2][]{%
    \errmessage{(Inkscape) Color is used for the text in Inkscape, but the package 'color.sty' is not loaded}%
    \renewcommand\color[2][]{}%
  }%
  \providecommand\transparent[1]{%
    \errmessage{(Inkscape) Transparency is used (non-zero) for the text in Inkscape, but the package 'transparent.sty' is not loaded}%
    \renewcommand\transparent[1]{}%
  }%
  \providecommand\rotatebox[2]{#2}%
  \ifx\svgwidth\undefined%
    \setlength{\unitlength}{468bp}%
    \ifx\svgscale\undefined%
      \relax%
    \else%
      \setlength{\unitlength}{\unitlength * \real{\svgscale}}%
    \fi%
  \else%
    \setlength{\unitlength}{\svgwidth}%
  \fi%
  \global\let\svgwidth\undefined%
  \global\let\svgscale\undefined%
  \makeatother%
  \begin{picture}(1,0.42500001)%
    \put(0,0){\includegraphics[width=\unitlength,page=1]{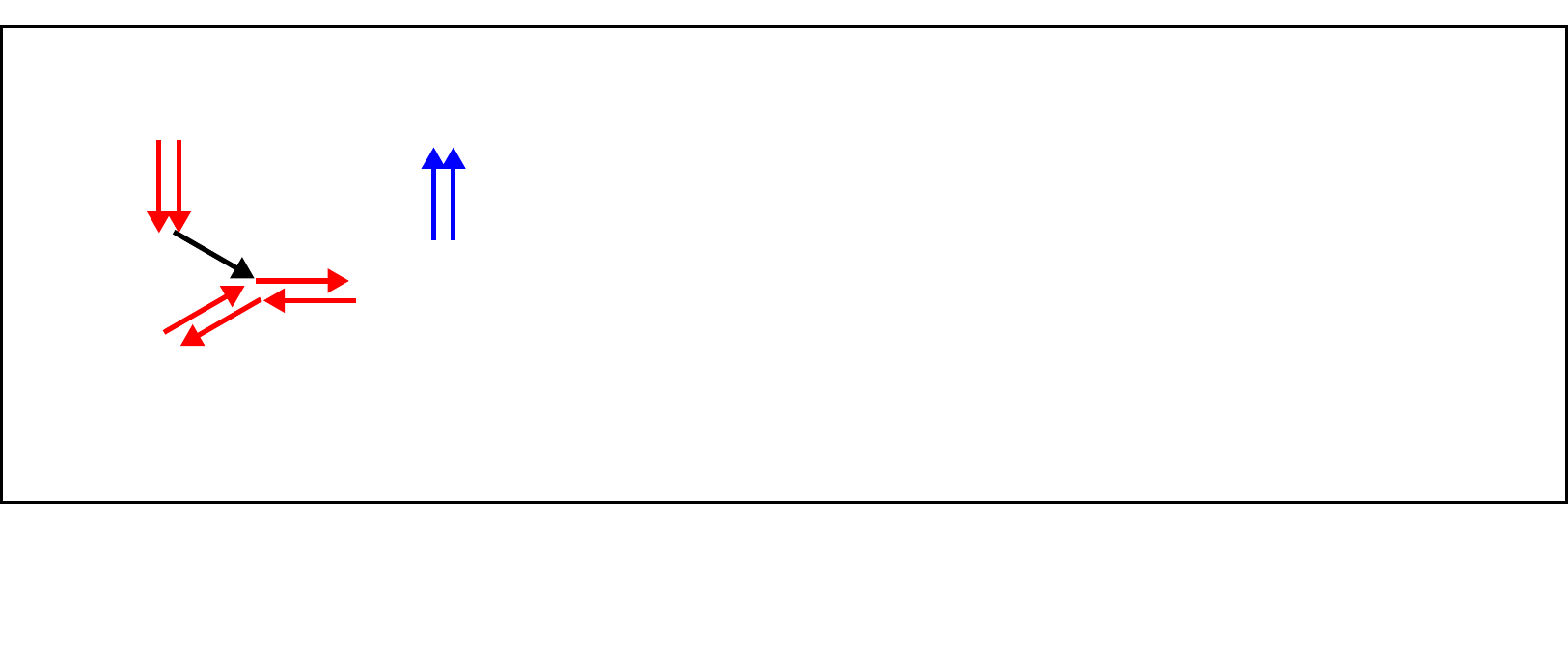}}%
    \put(-0.00141487,0.09011798){\color[rgb]{0,0,0}\makebox(0,0)[lt]{\begin{minipage}{0.99822201\unitlength}\raggedright Figure 17: An example of the forest subgraph construction starting from a colored double tree $T$. We label twin edges $[e]$ with a single common indeterminate $\gamma([e])$ to simplify the presentation. \end{minipage}}}%
    \put(0.12170717,1.66621588){\color[rgb]{0,0,0}\makebox(0,0)[lt]{\begin{minipage}{0.11106843\unitlength}\raggedright \end{minipage}}}%
    \put(0,0){\includegraphics[width=\unitlength,page=2]{fig17_forest.pdf}}%
    \put(0.18541807,0.12246409){\color[rgb]{0,0,0}\makebox(0,0)[lb]{\smash{$T$}}}%
    \put(0.49012,0.1245999){\color[rgb]{0,0,0}\makebox(0,0)[lb]{\smash{$T_\R$}}}%
    \put(0.79481873,0.1245999){\color[rgb]{0,0,0}\makebox(0,0)[lb]{\smash{$T_\C$}}}%
    \put(0,0){\includegraphics[width=\unitlength,page=3]{fig17_forest.pdf}}%
    \put(0.13422258,0.27305713){\color[rgb]{0,0,0}\makebox(0,0)[lb]{\smash{$x_1^\R$}}}%
    \put(0.23318091,0.18731995){\color[rgb]{0,0,0}\makebox(0,0)[lb]{\smash{$x_1^\R$}}}%
    \put(0.22637002,0.27305713){\color[rgb]{0,0,0}\makebox(0,0)[lb]{\smash{$x_1^\R$}}}%
    \put(0.12300463,0.18731995){\color[rgb]{0,0,0}\makebox(0,0)[lb]{\smash{$x_2^\C$}}}%
    \put(0.1822995,0.2025585){\color[rgb]{0,0,0}\makebox(0,0)[lb]{\smash{$x_2^\C$}}}%
    \put(0.23438284,0.31552498){\color[rgb]{0,0,0}\makebox(0,0)[lb]{\smash{$x_3^\C$}}}%
    \put(0.12620975,0.31552498){\color[rgb]{0,0,0}\makebox(0,0)[lb]{\smash{$x_2^\C$}}}%
    \put(0,0){\includegraphics[width=\unitlength,page=4]{fig17_forest.pdf}}%
    \put(0.43892256,0.27305713){\color[rgb]{0,0,0}\makebox(0,0)[lb]{\smash{$x_1^\R$}}}%
    \put(0.5378809,0.18731995){\color[rgb]{0,0,0}\makebox(0,0)[lb]{\smash{$x_1^\R$}}}%
    \put(0.53107,0.27305713){\color[rgb]{0,0,0}\makebox(0,0)[lb]{\smash{$x_1^\R$}}}%
    \put(0,0){\includegraphics[width=\unitlength,page=5]{fig17_forest.pdf}}%
    \put(0.73240496,0.18731995){\color[rgb]{0,0,0}\makebox(0,0)[lb]{\smash{$x_2^\C$}}}%
    \put(0.79169984,0.20255839){\color[rgb]{0,0,0}\makebox(0,0)[lb]{\smash{$x_2^\C$}}}%
    \put(0.84378317,0.31552498){\color[rgb]{0,0,0}\makebox(0,0)[lb]{\smash{$x_3^\C$}}}%
    \put(0.73561009,0.31552498){\color[rgb]{0,0,0}\makebox(0,0)[lb]{\smash{$x_2^\C$}}}%
  \end{picture}%
\endgroup%

\end{center}

For the second statement (about the lack of asymptotic traffic independence for $\salg{W}_n^\C$), we give a simple counterexample, namely, for $\beta_2^\C, \beta_3^\C \in \C \setminus \R$,

\begin{center}
\begingroup%
  \makeatletter%
  \providecommand\color[2][]{%
    \errmessage{(Inkscape) Color is used for the text in Inkscape, but the package 'color.sty' is not loaded}%
    \renewcommand\color[2][]{}%
  }%
  \providecommand\transparent[1]{%
    \errmessage{(Inkscape) Transparency is used (non-zero) for the text in Inkscape, but the package 'transparent.sty' is not loaded}%
    \renewcommand\transparent[1]{}%
  }%
  \providecommand\rotatebox[2]{#2}%
  \ifx\svgwidth\undefined%
    \setlength{\unitlength}{468bp}%
    \ifx\svgscale\undefined%
      \relax%
    \else%
      \setlength{\unitlength}{\unitlength * \real{\svgscale}}%
    \fi%
  \else%
    \setlength{\unitlength}{\svgwidth}%
  \fi%
  \global\let\svgwidth\undefined%
  \global\let\svgscale\undefined%
  \makeatother%
  \begin{picture}(1,0.28076924)%
    \put(0.0956655,1.55415027){\color[rgb]{0,0,0}\makebox(0,0)[lt]{\begin{minipage}{0.11106843\unitlength}\raggedright \end{minipage}}}%
    \put(0.65305273,0.04412947){\color[rgb]{0,0,0}\makebox(0,0)[lb]{\smash{$\displaystyle\big]\bigg) \bigg(\lim_{n \to \infty} \tau^0\big[$}}}%
    \put(0.34055272,0.21400263){\color[rgb]{0,0,0}\makebox(0,0)[lb]{\smash{$\big] = \frac{1}{3}(\beta_2^\C \beta_3^\C + \overline{\beta_2^\C}\overline{\beta_3^\C}) + \frac{1}{6}(\beta_2^\C \overline{\beta_3^\C} + \overline{\beta_2^\C} \beta_3^\C)$}}}%
    \put(0.01202706,0.21400263){\color[rgb]{0,0,0}\makebox(0,0)[lb]{\smash{$\displaystyle \lim_{n \to \infty} \tau^0\big[$}}}%
    \put(0,0){\includegraphics[width=\unitlength,page=1]{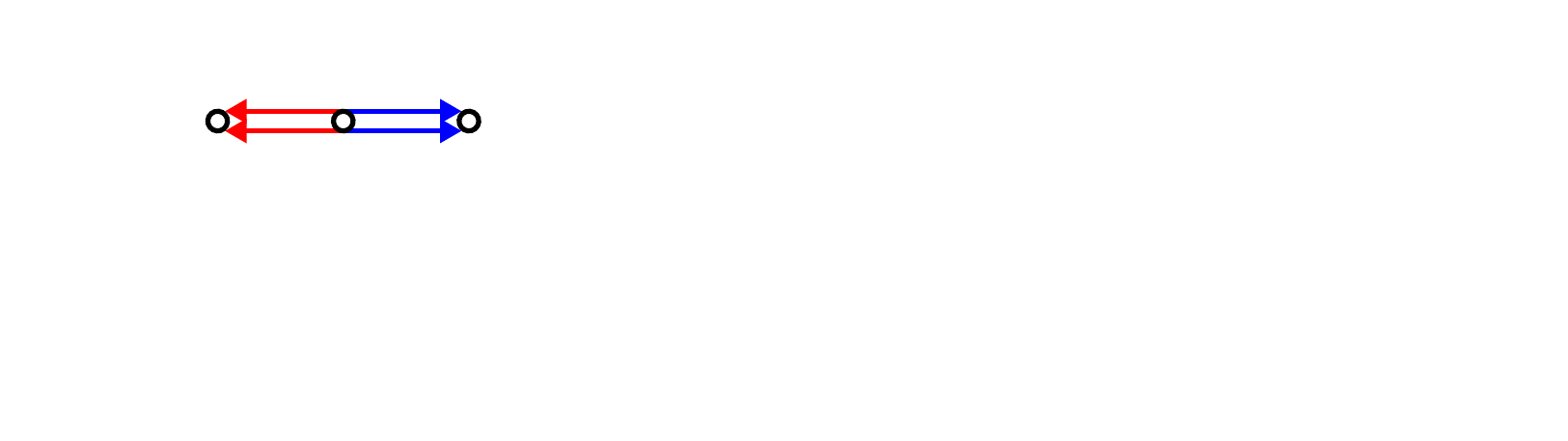}}%
    \put(0.16388047,0.2242244){\color[rgb]{0,0,0}\makebox(0,0)[lb]{\smash{$\mbf{W}_n^{i_2^\C}$}}}%
    \put(0.22958561,0.2242244){\color[rgb]{0,0,0}\makebox(0,0)[lb]{\smash{$\mbf{W}_n^{i_3^\C}$}}}%
    \put(0,0){\includegraphics[width=\unitlength,page=2]{fig18_non.pdf}}%
    \put(0.34055272,0.13387442){\color[rgb]{0,0,0}\makebox(0,0)[lb]{\smash{$\phantom{\big]} \neq  \bigg(\frac{1}{2}(\beta_2^\C + \overline{\beta_2^\C})\bigg)\bigg(\frac{1}{2}(\beta_3^\C + \overline{\beta_3^\C})\bigg)$}}}%
    \put(0.9543348,0.04346007){\color[rgb]{0,0,0}\makebox(0,0)[lb]{\smash{$\big]\bigg)$.}}}%
    \put(0,0){\includegraphics[width=\unitlength,page=3]{fig18_non.pdf}}%
    \put(0.8433638,0.05434957){\color[rgb]{0,0,0}\makebox(0,0)[lb]{\smash{$\mbf{W}_n^{i_3^\C}$}}}%
    \put(0,0){\includegraphics[width=\unitlength,page=4]{fig18_non.pdf}}%
    \put(0.55650241,0.05434957){\color[rgb]{0,0,0}\makebox(0,0)[lb]{\smash{$\mbf{W}_n^{i_2^\C}$}}}%
    \put(0.34055272,0.04413072){\color[rgb]{0,0,0}\makebox(0,0)[lb]{\smash{$\phantom{\big[} = \displaystyle\bigg(\lim_{n \to \infty} \tau^0\big[$}}}%
  \end{picture}%
\endgroup%

\end{center}

Yet, we know that free independence describes the asymptotic behavior of the Wigner matrices $\salg{W}_n$ regardless of the parameters $(\beta_i)_{i \in I}$. Naturally, we would like to know how to extract this information from the LTD (in particular, how this is consistent with the distinct LTDs \eqref{eq:3.5} and \eqref{eq:3.16}). Again, we restrict our attention to the joint distribution, the general case of the joint $*$-distribution following from the self-adjointness of our ensembles. 

From the diagram \eqref{eq:2.9}, we know that the joint distribution $\mu_{\salg{W}_n}$ of $\salg{W}_n$ factors through the traffic distribution $\tau_{\salg{W}_n}$ of $\salg{W}_n$ via 
\[
\mu_{\salg{W}_n} = \tau_{\salg{W}_n} \circ \Delta \circ \eta.
\]
This amounts to computing $\tau\big[C(\salg{W}_n)\big]$ for directed cycles $C = (V, E)$. We use the injective traffic state to rewrite this as
\[
\tau\big[C(\salg{W}_n)\big] = \sum_{\pi \in \salg{P}(V)} \tau^0\big[C^\pi(\salg{W}_n)\big].
\]
In the limit, the only contributions come from (colored) double trees $C^\pi$. We claim that if $C^\pi$ is a double tree, then it can only have opposing twin edges (an \emph{opposing double tree}). Indeed, assume that $\pi \in \salg{P}(V)$ identifies the sources $\source(e_1) \stackrel{\pi}{\sim} \source(e_2)$ and targets $\target(e_1) \stackrel{\pi}{\sim} \target(e_2)$ of two distinct edges $e_1, e_2 \in E$. We write $C^\rho$ for the graph intermediate to $C$ and $C^\pi$ obtained from $C$ by only making these two identifications. If $e_1$ and $e_2$ are consecutive edges in the cycle $C$, then $C^\rho$ consists of a directed cycle with two loops coming out of a particular vertex (``rabbit ears''). Otherwise, $C^\rho$ consists of two almost disjoint directed cycles overlapping in the twin edge $[e] = \{e_1, e_2\}$ (a ``butterfly''). In both cases, we see that no further identifications can possibly result in a double tree $C^\pi$.

\phantomsection\label{fig18_cycle}
\begin{center}
\begingroup%
  \makeatletter%
  \providecommand\color[2][]{%
    \errmessage{(Inkscape) Color is used for the text in Inkscape, but the package 'color.sty' is not loaded}%
    \renewcommand\color[2][]{}%
  }%
  \providecommand\transparent[1]{%
    \errmessage{(Inkscape) Transparency is used (non-zero) for the text in Inkscape, but the package 'transparent.sty' is not loaded}%
    \renewcommand\transparent[1]{}%
  }%
  \providecommand\rotatebox[2]{#2}%
  \ifx\svgwidth\undefined%
    \setlength{\unitlength}{468bp}%
    \ifx\svgscale\undefined%
      \relax%
    \else%
      \setlength{\unitlength}{\unitlength * \real{\svgscale}}%
    \fi%
  \else%
    \setlength{\unitlength}{\svgwidth}%
  \fi%
  \global\let\svgwidth\undefined%
  \global\let\svgscale\undefined%
  \makeatother%
  \begin{picture}(1,0.34230768)%
    \put(0,0){\includegraphics[width=\unitlength,page=1]{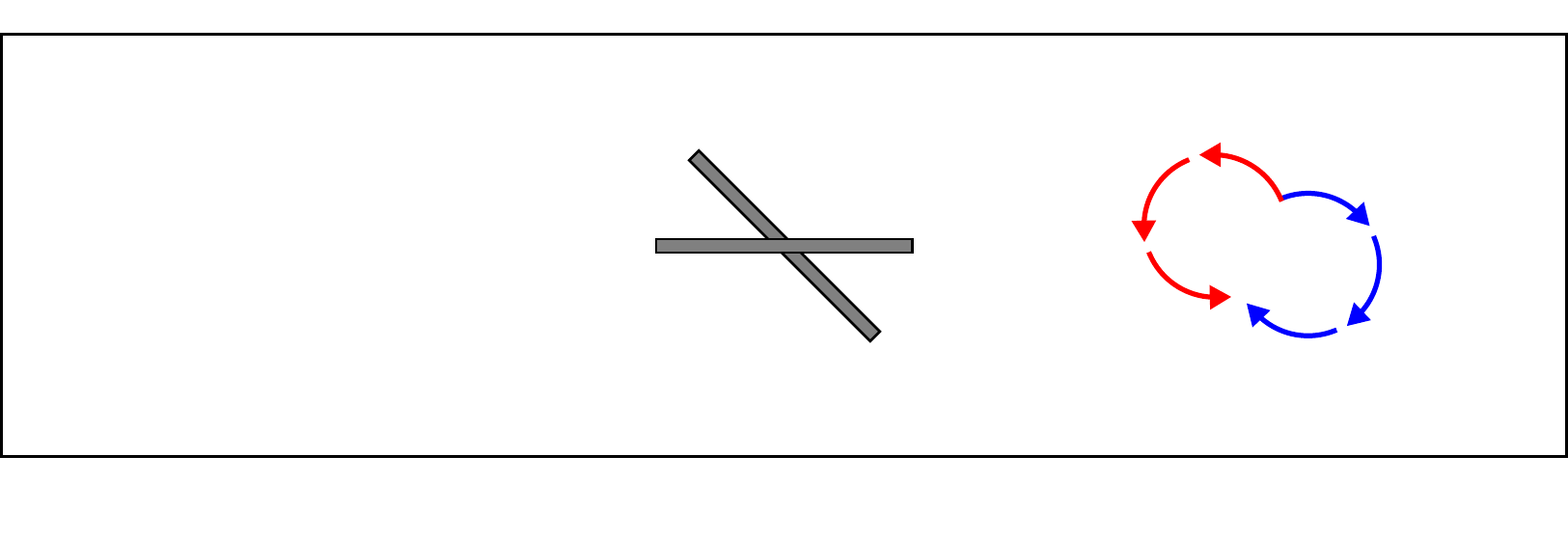}}%
    \put(-0.00141487,0.03646963){\color[rgb]{0,0,0}\makebox(0,0)[lt]{\begin{minipage}{0.99822197\unitlength}\raggedright Figure 18: An example of a butterfly $C^\rho$ starting from a directed cycle $C$.\end{minipage}}}%
    \put(0.12170717,1.59167967){\color[rgb]{0,0,0}\makebox(0,0)[lt]{\begin{minipage}{0.11106843\unitlength}\raggedright \end{minipage}}}%
    \put(0,0){\includegraphics[width=\unitlength,page=2]{fig18_cycle.pdf}}%
    \put(0.3348636,0.17906974){\color[rgb]{0,0,0}\makebox(0,0)[lb]{\smash{$\stackrel{\rho}{\mapsto}$}}}%
    \put(0.64267611,0.17833617){\color[rgb]{0,0,0}\makebox(0,0)[lb]{\smash{$=$}}}%
    \put(0,0){\includegraphics[width=\unitlength,page=3]{fig18_cycle.pdf}}%
  \end{picture}%
\endgroup%

\end{center}

Thus, from the perspective of the joint distribution, we need only to consider the behavior of the LTD on opposing colored double trees $T$. In this case, we see that the LTDs \eqref{eq:3.5} and \eqref{eq:3.16} agree on the value of  
\[
\lim_{n \to \infty} \tau^0\big[T(\salg{W}_n)\big] = 1.
\]

\begin{rem}\label{rem3.1.3}
An important application of traffic probability lies in the relationship between traffic independence and free independence. In certain situations, one can actually deduce free independence from traffic independence \cite{Mal11,CDM16}, the advantage being that the traffic setting might be more tractable. Of course, the two notions do not perfectly align, as seen even in the case of the Wigner matrices (Lemma 4.7 in \cite{Mal11} gives yet another example). In this case, we see that the traffic distribution specifies the behavior of our matrices in situations that might not be relevant to their joint distribution: in a certain sense, traffic independence asks for too much. Nevertheless, we can still use the traffic framework to make free probabilistic statements, even when a LTD might not exist! In particular, from our work above, we see that if a family of self-adjoint traffics $\mbf{a}_n = (a_n^{(i)})_{i \in I}$ in a traffic space $(\salg{A}_n, \tau_n)$ satisfies
\begin{equation}\label{eq:3.18}
\lim_{n \to \infty} \tau_n^0\big[T(\mbf{a}_n)\big] =
\begin{cases}
1 & \text{if $T$ is an opposing colored double tree,}\\
0 & \text{if $T$ is not a colored double tree,}
\end{cases}
\end{equation}
then $\mbf{a}_n$ converges in \emph{joint distribution} to a semicircular system $\mbf{a} = (a_i)_{i \in I}$. Note that we do not specify the behavior of $\tau_n^0\big[T(\mbf{a}_n)\big]$ on general colored double trees $T$ (in particular, we do not assume that the limit $\lim_{n \to \infty} \tau_n^0\big[T(\mbf{a}_n)\big]$ even exists). We will use this criteria in Section \hyperref[sec4]{4} to treat the case of RBMs of a general parameter $\beta_i \in \C$.

Of course, in the other direction, it is possible to have traffic independence but not free independence. We can see this in the context of the traffic CLT (Theorem \hyperref[thm2.5.6]{2.5.6}) by realizing the intermediate values $\alpha \in (0, 1)$.
\end{rem}

\subsection{The traffic CLT}\label{sec3.2}

For simplicity, we restrict our attention to real Wigner matrices $\mathcal{W}_n = (\mbf{W}_n^{(\ell)}: \ell \in \N)$ in this section. A classical result of Dykema shows that the matrices $\salg{W}_n$ are asymptotically free \cite{Dyk93}, thus realizing both the free CLT and the traffic CLT (the latter, for $\alpha = 1$). Yet, Remark \hyperref[rem2.5.4]{2.5.4} extends the asymptotic traffic independence of $\mathcal{W}_n$ to a much larger class of matrices. In particular, we know that the corresponding family of degree matrices $\salg{D}_n = (\mbf{D}_n^{(\ell)} : \ell \in \N)$ are also asymptotically traffic independent, where 
\begin{equation}\label{eq:3.19}
\mbf{D}_n^{(\ell)} = \text{row}(t_{x_\ell})(\mbf{W}_n^{(\ell)}) = \frac{1}{2}\text{row}(t_{x_\ell})(\mbf{W}_n^{(\ell)}) + \frac{1}{2}\text{col}(t_{x_\ell})(\mbf{W}_n^{(\ell)}) \in \C\salg{G}\langle\mbf{W}_n^{(\ell)}\rangle.
\end{equation}
A simple computation shows that the diagonal matrices $\salg{D}_n$ realize the traffic CLT for $\alpha = 0$, in some sense recovering the classical CLT.

Taking linear combinations of the above, we obtain the $(p, q)$-Markov matrices from before:
\[
\mbf{M}_{n, p, q}^{(\ell)} = p\mbf{W}_n^{(\ell)} + q\mbf{D}_n^{(\ell)} \in \C\salg{G}\langle\mbf{W}_n^{(\ell)}\rangle, \qquad \forall p, q \in \R.
\]
Recall that the LSD of the Markov matrices is given by the free convolution $\wsc(0,1) \boxplus \gn(0, 1)$. Naively, one may then hope that the interpolation between $\mbf{W}_n^{(\ell)}$ and $\mbf{D}_n^{(\ell)}$ given by $\mbf{M}_{n, p, q}^{(\ell)}$ passes to the traffic CLT, realizing the intermediate values $\alpha \in (0, 1)$. We show that this is indeed the case. 

\begin{defn}[Stable traffic distribution]\label{defn3.2.1}
Let $\nu: \C\salg{T}\langle\mbf{x}, \mbf{x}^*\rangle \to \C$ denote the traffic distribution of some family of centered random variables. We say that $\nu$ is \emph{stable} if there exists a realization of $\nu$ by traffic independent families $\mbf{a}_1 = (a_1^{(i)})_{i \in I}$ and $\mbf{a}_2 = (a_2^{(i)})_{i \in I}$ in a traffic space $(\salg{A}, \tau)$ such that the sum $\mbf{a} = \mbf{a}_1 + \mbf{a}_2 = (a_1^{(i)} + a_2^{(i)})_{i \in I}$ has the same traffic distribution, up to scale. By this, we mean that
\[
\nu = \tau_{\mbf{a}_1} = \tau_{\mbf{a}_2}
\]
with a scaling parameter $c \in \R_+$ such that
\[
\tau_{\mbf{a}}(T) = c^{\#(E)/2}\nu(T), \qquad \forall T = (V,E,\gamma,\varepsilon) \in \salg{T}\langle\mbf{x}, \mbf{x}^*\rangle.
\]
\end{defn}

\begin{lemma}\label{lemma3.2.2}
The families $(\salg{M}_n^{(\ell)}: \ell \in \N) = ((\mbf{M}_{n, p, q}^{(\ell)})_{p, q \in \R}: \ell \in \N)$ are asymptotically traffic independent with a stable universal limiting traffic distribution.
\end{lemma}
\begin{proof}
We need only to prove the stability of the limiting traffic distribution $\nu = \lim_{n \to \infty} \tau_{\salg{M}_n^{(1)}}$ as the rest follows from Proposition \hyperref[prop3.1.2]{3.1.2} and Remark \hyperref[rem2.5.4]{2.5.4}. To this end, we model the limit of our matrices $(\salg{M}_n^{(\ell)} : \ell \in \N)$ within the traffic space $(\ssgp{x}, \tau)$, where
\[
\tau = *_{\ell \in \N} \bigg(\lim_{n \to \infty} \tau_{\mbf{W}_n^{(\ell)}}\bigg) \quad\text{and}\quad x_\ell = x_\ell^*.
\]
By the universality of \eqref{eq:3.5}, the traffic state $\tau$ does not depend on the particular choice of Wigner matrices $\mbf{W}_n^{(\ell)}$. We single out the Gaussian realization $\mbf{X}_n^{(\ell)}(i,j) \deq \gn(0, \indc{i \neq j})$ for the distributional symmetry 
\[
\salg{S}_n^{(k)} = \frac{1}{\sqrt{k}} \sum_{\ell=1}^k \salg{M}_n^{(\ell)} = \bigg(\frac{1}{\sqrt{k}}\sum_{\ell=1}^k \mbf{M}_{n, p, q}^{(\ell)}\bigg)_{p, q \in \R} \deq (\mbf{M}_{n, p, q}^{(1)})_{p, q \in \R} = \salg{M}_n^{(1)}.
\]
This in turn implies the traffic distributional equality
\begin{equation}\label{eq:3.20}
\salg{S}_n^{(k)} \stackrel{\tau_n}{=} \salg{M}_n^{(1)}.
\end{equation}

By construction, the family $\salg{S}_n^{(k)}$ converges in traffic distribution to
\[
\mbf{s}_k = \frac{1}{\sqrt{k}}\sum_{\ell=1}^k \mbf{m}_\ell = \bigg(\frac{1}{\sqrt{k}}\sum_{\ell=1}^k pt_{x_\ell} + \frac{q}{2}\text{row}(t_{x_\ell}) + \frac{q}{2}\text{col}(t_{x_\ell})\bigg)_{p, q \in \R} \subset (\ssgp{x}, \tau).
\]
Passing to the limit, \eqref{eq:3.20} becomes
\begin{equation}\label{eq:3.21}
\mbf{s}_k \teq \mbf{m}_1 \subset \C\salg{G}\langle x_1, x_1^*\rangle.
\end{equation}
Taking $k=2$ in the above, we have that
\[
\mbf{s}_2 = \frac{1}{\sqrt{2}}(\mbf{m}_1 + \mbf{m}_2) \teq \mbf{m}_1,
\]
where $\mbf{m}_1$ and $\mbf{m}_2$ are traffic independent. We conclude that the limiting traffic distribution $\nu = \lim_{n \to \infty} \tau_{\salg{M}_n^{(1)}}$ is stable with scaling parameter $c = 2$. 
\end{proof}

\begin{cor}\label{cor3.2.3}
The ESDs $\mu(\mbf{M}_{n, p, q}^{(1)})$ converge weakly in expectation to the free convolution $\mu_{p, q} = \wsc(0, p^2) \boxplus \gn(0, q^2)$.
\end{cor}
\begin{proof}
It suffices to prove the result for $p, q \in \R$ of the form $p^2 + q^2 = 1$. Proposition A.3 in \cite{BDJ06} shows that the free convolution $\mu_{1, 1}$ is determined by its moments: the same argument applies wholesale to the family of free convolutions $(\mu_{p, q})_{p, q \in \R}$. We may thus proceed by the method of moments.

Using the same notation as before, we know that $\mbf{M}_{n, p, q}^{(\ell)}$ converges in traffic distribution to the self-adjoint traffic
\[
a_{p,q}^{(\ell)} = pt_{x_\ell} + \frac{q}{2}\text{row}(t_{x_\ell}) + \frac{q}{2}\text{col}(t_{x_\ell}) \in \mbf{m}_\ell \subset \C\salg{G}\langle x_\ell, x_\ell^*\rangle \subset (\ssgp{x}, \tau).
\]
This reduces the problem to showing that the moments of $a_{p, q}^{(1)}$ match those of $\mu_{p, q}$.

Now, note that a special case of \eqref{eq:3.21} implies that
\begin{equation}\label{eq:3.22}
s_{p, q}^{(k)} = \frac{1}{\sqrt{k}}\sum_{\ell=1}^k a_{p, q}^{(\ell)} \teq a_{p, q}^{(1)}.
\end{equation}
We calculate the mean and variance of $a_{p, q}^{(\ell)}$ using the same Gaussian realization as before: 
\[
\varphi(a_{p, q}^{(\ell)}) = \lim_{n \to \infty} \E\bigg[\frac{1}{n}\trace(\mbf{M}_{n, p, q}^{(\ell)})\bigg] = \lim_{n \to \infty} \E\bigg[\frac{1}{n}\sum_{j=1}^n \bigg(q\sum_{k \neq j}^n \frac{\mbf{X}_n^{(\ell)}(j, k)}{\sqrt{n}}\bigg)\bigg] = 0
\]
and
\[
\varphi((a_{p, q}^{(\ell)})^2) = \tau\big[T_1(a_{p, q}^{(\ell)})\big] = \tau^0\big[T_1(a_{p, q}^{(\ell)})\big] + \tau^0\big[T_2(a_{p, q}^{(\ell)})\big],
\]
where $T_1$ and $T_2$ are as in the statement of the traffic CLT (Theorem \hyperref[thm2.5.6]{2.5.6}). A straightforward calculation then shows that
\[
\tau^0\big[T_1(a_{p, q}^{(\ell)})\big] = \lim_{n \to \infty} \tau_n^0\big[T_1(\mbf{M}_{n, p, q}^{(\ell)})\big] = \lim_{n \to \infty} \frac{1}{n}\sum_{j \neq k}^n \E\big[\mbf{M}_{n, p, q}(j, k)^2\big] = \lim_{n \to \infty} \frac{1}{n}n(n-1)\frac{p^2}{n} = p^2
\]
and
\begin{alignat*}{2}
\tau^0\big[T_2(a_{p, q}^{(\ell)})\big] &= \lim_{n \to \infty} \tau_n^0\big[T_2(\mbf{M}_{n, p, q}^{(\ell)})\big] &&= \lim_{n \to \infty} \frac{1}{n}\sum_{j=1}^n \E\big[\mbf{M}^{(\ell)}_{n, p, q}(j,j)^2\big] \\
&= \lim_{n \to \infty} \E\big[\mbf{M}^{(\ell)}_{n, p, q}(1,1)^2\big] &&= \lim_{n \to \infty} \E\bigg[\bigg(q\sum_{j=2}^n\frac{\mbf{X}_n^{(\ell)}(1, j)}{\sqrt{n}}\bigg)^2\bigg] = \lim_{n\to\infty} (n-1)\frac{q^2}{n} = q^2.
\end{alignat*}
Combining \eqref{eq:3.22} with the traffic CLT, we obtain the distributional identity
\[
a_{p, q}^{(1)} \teq s_{p, q}^{(k)} \dto \mu_{p, q} = \wsc(0, p^2) \boxplus \gn(0, q^2) \quad \text{as} \quad k \to \infty,
\]
as was to be shown.
\end{proof}

Taking $(p, q) = (1, -1)$ in the above, we recover the special case of the Markov matrices in \cite{BDJ06}. Corollary \hyperref[cor3.2.3]{3.2.3} explains this convergence in the context of traffic probability, but it also suggests a far more natural free probabilistic interpretation, namely, the asymptotic freeness of $\mbf{W}_n^{(1)}$ and $\mbf{D}_n^{(1)}$. Note that if $\mbf{W}_n^{(1)}$ and $\mbf{W}_n^{(2)}$ are normalized GOE matrices, then the standard techniques apply to show that the \emph{independent} matrices $\mbf{W}_n^{(1)}$ and $\mbf{D}_n^{(2)}$ are asymptotically free \cite{Voi91}; however, in our case, the matrix $\mbf{D}_n^{(1)}$ is completely determined by $\mbf{W}_n^{(1)}$. Nevertheless, one can work directly with the LTD \eqref{eq:3.5} of the Wigner matrices to show that the pairs $(\mbf{W}_n^{(1)}, \mbf{D}_n^{(1)})$ and $(\mbf{W}_n^{(1)}, \mbf{D}_n^{(2)})$ have the same limiting joint distribution, which proves the expected result. Instead, we defer to \cite{AM17}, wherein this convergence follows from the given free product decomposition of the universal enveloping traffic space.

For convenience, we restricted our attention to real Wigner matrices. One can easily adapt the argument to complex Wigner matrices of a real parameter $\beta_\ell \in \R$ by finding an appropriate complex Gaussian realization. In this case, we must take care to choose an analogue of the degree matrix $\mbf{D}_n^{(\ell)}$ to ensure that we have a self-adjoint traffic (in particular, we can use the second equality in \eqref{eq:3.19} so that $\mbf{D}_n^{(\ell)}$ now averages the row sums with the column sums). We leave the relatively straightforward details to the interested reader.

\subsection{Concentration inequalities}\label{sec3.3}

For a test graph $T = (V, E, \gamma) \in \salg{T}\langle\mbf{x}\rangle$, we recall the random variable \eqref{eq:2.8}:
\[
\trace\big[T(\salg{W}_n)\big] = \sum_{\phi: V \to [n]} \prod_{e \in E} (\mbf{W}_n^{(\gamma(e))})(\phi(e)).
\]

For natural reasons, we are interested in bounding the deviation of $\trace\big[T(\salg{W}_n)\big]$ from its mean. In particular, we would like to emulate the usual approach for the Wigner matrices to show that the variance $\var(\frac{1}{n}\trace\big[T(\salg{W}_n)\big]) = O_T(n^{-2})$, which would allow us to upgrade the convergence in Proposition \hyperref[prop3.1.2]{3.1.2} to the almost sure sense. It turns out that this approach will not work in general, but it will be instructive to see just how it falls short.

For notational convenience, we consider instead the deviation of $\trace\big[T(\salg{X}_n)\big]$ (recall that $\salg{X}_n = \sqrt{n}\salg{W}_n$ are the unnormalized Wigner matrices). To begin,
\begin{align}
\notag \var(\trace\big[T(\salg{X}_n)\big]) &= \E\bigg[\bigg|\trace\big[T(\salg{X}_n)]- \E \, \trace\big[T(\salg{X}_n)\big]\bigg|^2\bigg] \\
\notag &= \E\bigg[\bigg(\trace\big[T(\salg{X}_n)]- \E \, \trace\big[T(\salg{X}_n)\big]\bigg)\overline{\bigg(\trace\big[T(\salg{X}_n)]- \E \, \trace\big[T(\salg{X}_n)\big]\bigg)}\bigg] \\
&= \sum_{\phi_1, \phi_2: V \to [n]} \E\bigg[\prod_{\ell=1}^2 \bigg(\prod_{e \in E} \mbf{X}_{n, \ell}^{(\gamma(e))}(\phi_\ell(e)) - \E \prod_{e \in E} \mbf{X}_{n, \ell}^{(\gamma(e))}(\phi_\ell(e)) \bigg)\bigg], \label{eq:3.23}
\end{align}
where
\begin{equation}\label{eq:3.24}
\mbf{X}_{n, \ell}^{(i)}(j, k) =
\begin{cases}
\mbf{X}_n^{(i)}(j, k) & \text{ if } \ell = 1, \\
\mbf{X}_n^{(i)}(k, j) & \text{ if } \ell = 2.
\end{cases}
\end{equation}
We again make use of our strong moment assumption \eqref{eq:3.1}, this time to bound our summands uniformly in $\phi_1$, $\phi_2$, and $n$. In particular, our bound only depends on $T$, i.e.,
\begin{equation}\label{eq:3.25}
\E\bigg[\prod_{\ell=1}^2 \bigg(\prod_{e \in E} \mbf{X}_{n, \ell}^{(\gamma(e))}(\phi_\ell(e)) - \E \prod_{e \in E} \mbf{X}_{n, \ell}^{(\gamma(e))}(\phi_\ell(e)) \bigg)\bigg] \leq C_T < \infty.
\end{equation}

We are then interested in the number of pairs $(\phi_1, \phi_2)$ that actually contribute in \eqref{eq:3.23} (i.e., such that the summand \eqref{eq:3.25} is nonzero). To this end, note that the maps $\phi_\ell$ induce maps $\wtilde{\phi}_\ell : E \to \{\{a, b\} : a, b \in [n]\}$, where
\[
e \mapsto \{\phi_\ell(\source(e)), \phi_\ell(\target(e))\}.
\]
In particular, if $\wtilde{\phi}_1(E) \cap \wtilde{\phi}_2(E) = \emptyset$, then the independence of the $\mbf{X}_n^{(i)}(j, k)$ implies that the outermost product of \eqref{eq:3.25} factors over the expectation, resulting in a zero summand. Thus, we need only to consider so-called \emph{edge-matched} pairs $(\phi_1, \phi_2)$. For our purposes, it will be convenient to incorporate the data of such a pair into the graph $T$ itself.

For a pair $(\phi_1, \phi_2)$, we construct a new graph $T_{\phi_1 \sqcup \phi_2}$ by considering two disjoint copies $T_1$ and $T_2$ of $T$ (associated to $\phi_1$ and $\phi_2$ respectively), reversing the direction of the edges of $T_2$, and then identifying the vertices according to their images under the maps $\phi_1$ and $\phi_2$; formally, the vertices of $T_{\phi_1 \sqcup \phi_2}$ are then given by
\[
V_{\phi_1 \sqcup \phi_2} = (\phi_1^{-1}(m) \cup \phi_2^{-1}(m) : m \in [n]).
\]
An edge match between $\phi_1$ and $\phi_2$ then corresponds to an overlay of edges, though not necessarily in the same direction. Note that
\[
(\phi_1, \phi_2) \text{ is edge-matched} \quad \Longrightarrow \quad T_{\phi_1 \sqcup \phi_2} \text{ is connected}.
\]

\begin{center}
\begingroup%
  \makeatletter%
  \providecommand\color[2][]{%
    \errmessage{(Inkscape) Color is used for the text in Inkscape, but the package 'color.sty' is not loaded}%
    \renewcommand\color[2][]{}%
  }%
  \providecommand\transparent[1]{%
    \errmessage{(Inkscape) Transparency is used (non-zero) for the text in Inkscape, but the package 'transparent.sty' is not loaded}%
    \renewcommand\transparent[1]{}%
  }%
  \providecommand\rotatebox[2]{#2}%
  \ifx\svgwidth\undefined%
    \setlength{\unitlength}{468bp}%
    \ifx\svgscale\undefined%
      \relax%
    \else%
      \setlength{\unitlength}{\unitlength * \real{\svgscale}}%
    \fi%
  \else%
    \setlength{\unitlength}{\svgwidth}%
  \fi%
  \global\let\svgwidth\undefined%
  \global\let\svgscale\undefined%
  \makeatother%
  \begin{picture}(1,0.48076923)%
    \put(0,0){\includegraphics[width=\unitlength,page=1]{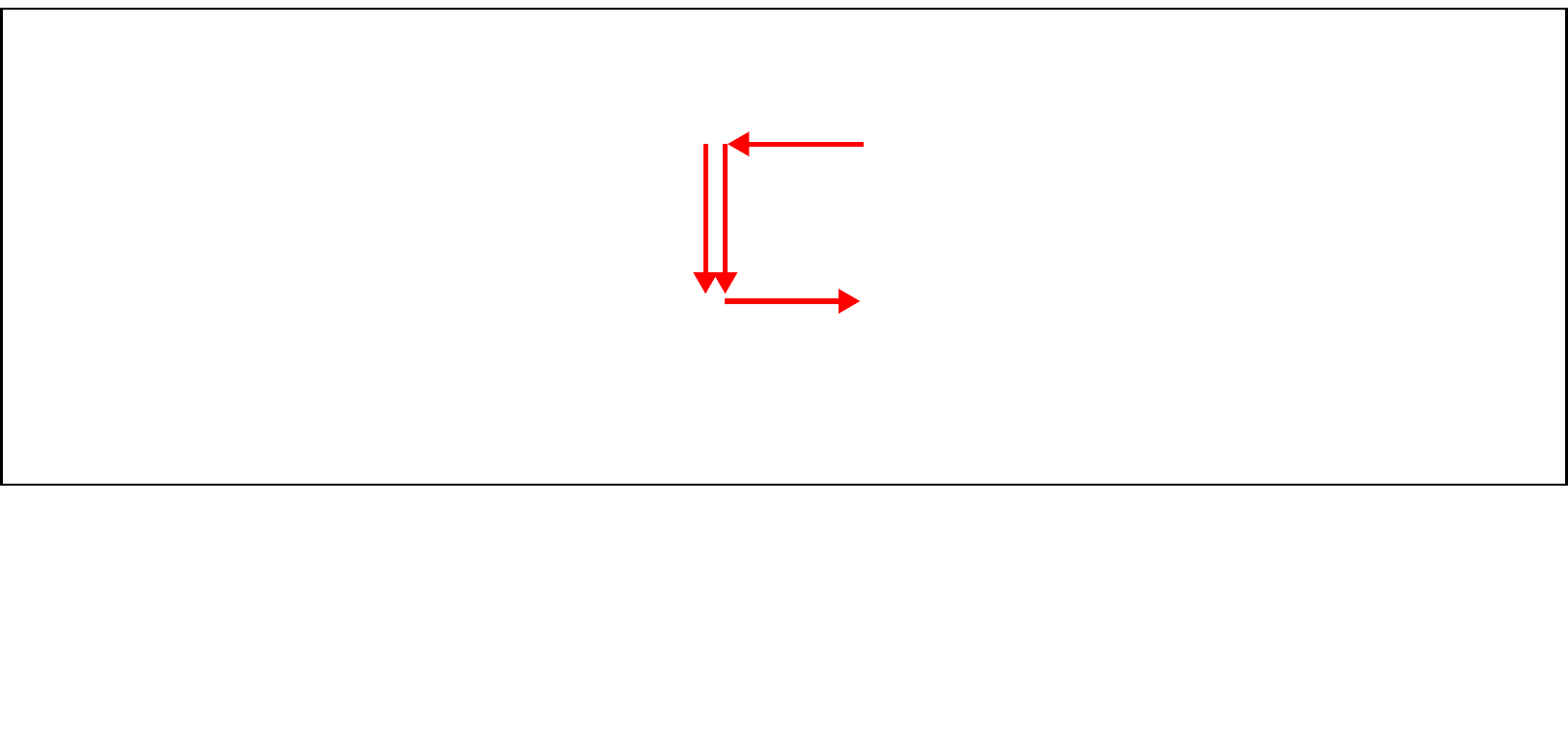}}%
    \put(-0.00141487,0.15726613){\color[rgb]{0,0,0}\makebox(0,0)[lt]{\begin{minipage}{0.99822205\unitlength}\raggedright Figure 19: An example of the construction of the graph $T_{\phi_1 \sqcup \phi_2}$ for an edge-matched pair $(\phi_1, \phi_2)$. Here, we omit the edge labels to emphasize the vertex labels $\phi_\ell(v)$. In this case, we use different colors for the edges of the two copies $T_1$ and $T_2$ of $T$ to keep track of their origins in the new graph $T_{\phi_1 \sqcup \phi_2}$. Recall that we reverse the direction of the edges of the second copy $T_2$ before identifying the vertices. \end{minipage}}}%
    \put(0.12170718,1.73336406){\color[rgb]{0,0,0}\makebox(0,0)[lt]{\begin{minipage}{0.11106844\unitlength}\raggedright \end{minipage}}}%
    \put(0,0){\includegraphics[width=\unitlength,page=2]{fig19_sqcup.pdf}}%
    \put(0.15489565,0.19174848){\color[rgb]{0,0,0}\makebox(0,0)[lb]{\smash{$(T_1, \phi_1)$}}}%
    \put(0.45959597,0.19174807){\color[rgb]{0,0,0}\makebox(0,0)[lb]{\smash{$(T_2, \phi_2)$}}}%
    \put(0.71915848,0.19174807){\color[rgb]{0,0,0}\makebox(0,0)[lb]{\smash{$(T_{\phi_1 \sqcup \phi_2}, \phi_1 \sqcup \phi_2)$}}}%
    \put(0,0){\includegraphics[width=\unitlength,page=3]{fig19_sqcup.pdf}}%
    \put(0.13882994,0.41122602){\color[rgb]{0,0,0}\makebox(0,0)[lb]{\smash{$1$}}}%
    \put(0.23899019,0.41122602){\color[rgb]{0,0,0}\makebox(0,0)[lb]{\smash{$2$}}}%
    \put(0.23899019,0.24932697){\color[rgb]{0,0,0}\makebox(0,0)[lb]{\smash{$4$}}}%
    \put(0.13882994,0.24932697){\color[rgb]{0,0,0}\makebox(0,0)[lb]{\smash{$3$}}}%
    \put(0.5499806,0.41122602){\color[rgb]{0,0,0}\makebox(0,0)[lb]{\smash{$1$}}}%
    \put(0.44982033,0.41122602){\color[rgb]{0,0,0}\makebox(0,0)[lb]{\smash{$2$}}}%
    \put(0.5499806,0.24932697){\color[rgb]{0,0,0}\makebox(0,0)[lb]{\smash{$5$}}}%
    \put(0.44982033,0.24932697){\color[rgb]{0,0,0}\makebox(0,0)[lb]{\smash{$3$}}}%
    \put(0,0){\includegraphics[width=\unitlength,page=4]{fig19_sqcup.pdf}}%
    \put(0.71915848,0.41828138){\color[rgb]{0,0,0}\makebox(0,0)[lb]{\smash{$1$}}}%
    \put(0.82774661,0.33154019){\color[rgb]{0,0,0}\makebox(0,0)[lb]{\smash{$3$}}}%
    \put(0.71915848,0.24477646){\color[rgb]{0,0,0}\makebox(0,0)[lb]{\smash{$4$}}}%
    \put(0.87781229,0.2447944){\color[rgb]{0,0,0}\makebox(0,0)[lb]{\smash{$5$}}}%
    \put(0.87781229,0.41827397){\color[rgb]{0,0,0}\makebox(0,0)[lb]{\smash{$2$}}}%
    \put(0,0){\includegraphics[width=\unitlength,page=5]{fig19_sqcup.pdf}}%
  \end{picture}%
\endgroup%

\end{center}

The sum over the set of edge-matched pairs $(\phi_1,\phi_2)$ can then be decomposed into a double sum: the first, over the set $\salg{S}_T$ of connected graphs $T_\sqcup = (V_\sqcup, E_\sqcup, \gamma_\sqcup)$ obtained by gluing the vertices of two disjoint copies of $T$ with at least one edge overlay (we reverse the direction of the edges of the second copy beforehand, and we keep track of the origin of the edges $E_\sqcup = E_\sqcup^{(1)} \sqcup E_\sqcup^{(2)}$); the second, over the set of injective labelings $\phi_\sqcup: V_\sqcup \hookrightarrow [n]$ of the vertices of $T_\sqcup$. We may then recast \eqref{eq:3.23} as
\begin{equation}\label{eq:3.26}
\sum_{T_\sqcup \in \salg{S}_T} \sum_{\phi_\sqcup: V_\sqcup \hookrightarrow [n]} \E\bigg[\prod_{\ell=1}^2 \bigg(\prod_{e \in E_\sqcup^{(\ell)}} \mbf{X}_n^{(\gamma_\sqcup(e))}(\phi_\sqcup(e)) - \E \prod_{e \in E_\sqcup^{(\ell)}} \mbf{X}_n^{(\gamma_\sqcup(e))}(\phi_\sqcup(e)) \bigg)\bigg]. 
\end{equation}
We defined $\salg{S}_T$ by reversing the direction of the edges of the second copy of $T$ before gluing in order to write \eqref{eq:3.26} without reference to the transposes \eqref{eq:3.24}. Moreover, by keeping track of the origin of the edges, we ensure that $\salg{S}_T$ does not conflate otherwise isomorphic graphs, and so guaranteeing a faithful reconstruction of \eqref{eq:3.23} from \eqref{eq:3.26}. The set $\salg{S}_T$ is of course a finite set whose size only depends on $T$. 

We consider a generic $T_\sqcup \in \salg{S}_T$, iterating the proof of Proposition \hyperref[prop3.1.2]{3.1.2}. We decompose the set of edges $E_\sqcup = L_\sqcup \cup N_\sqcup$ as before, and the same for $\wtilde{E}_\sqcup = \wtilde{L}_\sqcup \cup \wtilde{N}_{\sqcup}$ (recall that $\wtilde{E}_\sqcup$ denotes the set of equivalence classes in $E_\sqcup$). Suppose that there exists a lone edge $e_0 \in [e] \in \wtilde{N}_\sqcup$ with the label $\gamma(e_0) = i_0 \in I$ so that
\[
\gamma(e') \neq \gamma(e_0), \qquad \forall e' \in [e]\setminus\{e_0\}.
\]
Without loss of generality, we may assume that $e_0 \in E_\sqcup^{(1)}$. We write
\[
P_\ell = \prod_{e \in E_\sqcup^{(\ell)}} \mbf{X}_n^{(\gamma_\sqcup(e))}(\phi_\sqcup(e)) \quad \text{and} \quad P_1^{(0)} = \prod_{e \in E_\sqcup^{(1)}\setminus\{e_0\}} \mbf{X}_n^{(\gamma_\sqcup(e))}(\phi_\sqcup(e)).
\]
The independence of the centered random variables $\mbf{X}_n^{(i)}(j, k)$ and the injectivity of the maps $\phi_\sqcup$ imply that
\begin{align*}
\E\big[(P_1 -\E P_1)(P_2 - \E P_2)\big] &= \E\big[(\mbf{X}_n^{(\gamma_\sqcup(e_0))}(\phi_\sqcup(e_0))P_1^{(0)} - \E\mbf{X}_n^{(\gamma_\sqcup(e_0))}(\phi_\sqcup(e_0))\E P_1^{(0)})(P_2 - \E P_2)\big] \\
&= \E\big[\mbf{X}_n^{(\gamma_\sqcup(e_0))}(\phi_\sqcup(e_0))]\E\big[(P_1^{(0)} - \E P_1^{(0)})(P_2 - \E P_2)\big] \\
&= 0.
\end{align*}
Thus, for $T_\sqcup \in \salg{S}_T$ to contribute, each label $i \in I$ present in a class $[e] \in \wtilde{N}_\sqcup$ must occur with multiplicity
\begin{equation}\label{eq:3.27}
m_{i, [e]} \geq 2.
\end{equation}
This in turn implies that
\begin{equation}\label{eq:3.28}
\#(N_\sqcup) \geq 2\#(\wtilde{N}_\sqcup).
\end{equation}

As before, the underlying simple graph $\underline{T_\sqcup}  = (V_\sqcup, \wtilde{N}_\sqcup)$ is still connected, whence
\begin{equation}\label{eq:3.29}
\#(\wtilde{N}_\sqcup) + 1 \geq \#(V_\sqcup).
\end{equation}
Of course, we also have the inherent bound
\begin{equation}\label{eq:3.30}
\#(N_\sqcup) \leq \#(E_\sqcup) = 2\#(E).
\end{equation}
Recalling the uniform bound \eqref{eq:3.25}, we arrive at the asymptotic
\begin{equation}\label{eq:3.31}
\text{Var}(\trace\big[T(\salg{X}_n)\big]) = O_T(n^{\max\{\#(V_\sqcup) : T_\sqcup \in \salg{S}_T\}}) \leq O_T(n^{\#(E) + 1}),
\end{equation}
or, equivalently,
\begin{equation}\label{eq:3.32}
\text{Var}\bigg(\frac{1}{n}\trace\big[T(\salg{W}_n)\big]\bigg) = O_T(n^{-1}),
\end{equation}
falling short of our goal. Of course, one might hope that we were overly generous in our bounds and that equality in
\begin{equation}\label{eq:3.33}
\max\{{\#(V_\sqcup) : T_\sqcup \in \salg{S}_T\}} \leq \#(E) + 1
\end{equation}
is not attainable in practice. In fact, in the usual situation of traces of powers
\begin{equation}\label{eq:3.34}
\trace\big[T(\salg{W}_n)\big] = \trace((\mbf{W}_n^{(i(1))})^{\ell_1} \cdots (\mbf{W}_n^{(i(m))})^{\ell_m} ),
\end{equation}
this is indeed the case; however, in general, \eqref{eq:3.31} is tight. In particular, note that if we start with a tree $T$, we can overlay two disjoint copies $T_1$ and $T_2$ of $T$, the second with reversed edges, to obtain an opposing colored double tree $T_\sqcup$. In this case, we have equality in \eqref{eq:3.27}-\eqref{eq:3.30}. Proposition \hyperref[prop3.1.2]{3.1.2} then shows that the contribution of $T_\sqcup$ in \eqref{eq:3.26} is $\Theta(n^{\#(E) + 1})$.

\begin{center}
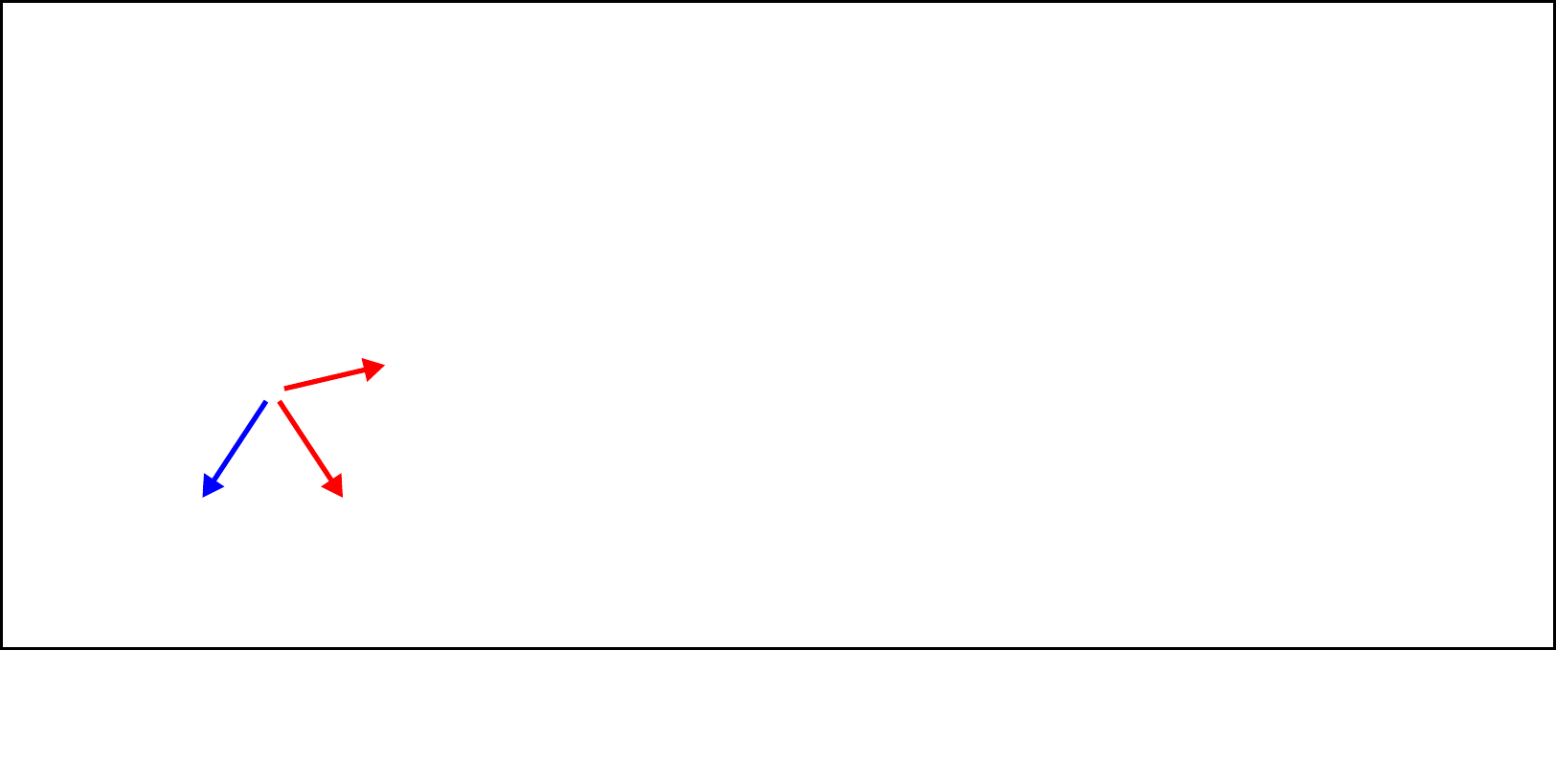
\end{center}

Working backwards, we identify the worst case scenario: for \eqref{eq:3.27}-\eqref{eq:3.30} to hold with equality, we need to glue (not necessarily overlay) disjoint copies $T_1$ and $T_2$ of $T$ with at least one edge overlay to obtain a colored double tree $T_\sqcup$ (though $T$ itself need not be a tree in general). In the classical case \eqref{eq:3.34}, $T$ corresponds to a cycle of length $\ell_1 + \cdots + \ell_m$ and such a gluing does not exist: starting with an edge overlay between two copies of the cycle, we obtain a butterfly as in Figure \hyperref[fig18_cycle]{18}, leading to a strict inequality in \eqref{eq:3.33} (and hence the usual asymptotic $O(n^{-2})$ in place of \eqref{eq:3.32}).

The careful reader will notice that we have actually proven a stronger result in the presence of loops $L \neq \emptyset$: in place of \eqref{eq:3.30}, we can instead use the tighter bound
\[
\#(N_\sqcup) \leq 2\#(N).
\]
We summarize our findings thus far.

\begin{lemma}\label{lemma3.3.1}
For a family of Wigner matrices $\salg{X}_n = (\mbf{X}_n^{(i)})_{i \in I}$, we have the asymptotic
\[
\var(\emph{tr}\big[T(\salg{X}_n)\big]) = O_T(n^{\#(N) + 1}), \qquad \forall T \in \salg{T}\langle\mbf{x}\rangle.
\]
The bound is tight in the sense that there exist test graphs $T$ in $\mbf{x}$ with
\[
\var(\emph{tr}\big[T(\salg{X}_n)\big]) = \Theta_T(n^{\#(N) + 1}).
\]
\end{lemma}

The colored double tree obstruction in Lemma \hyperref[lemma3.3.1]{3.3.1} ramifies into a forest of colored double trees for higher powers, but this construction remains the lone outlier (in particular, things do not get any worse). Drawing inspiration from Proposition 4.15 of \cite{BDJ06}, we prove

\begin{thm}\label{thm3.3.2}
For a family of Wigner matrices $\salg{X}_n = (\mbf{X}_n^{(i)})_{i \in I}$, we have the asymptotic
\[
\E\bigg[\bigg|\emph{tr}\big[T(\salg{X}_n)\big] - \E\,\emph{tr}\big[T(\salg{X}_n)\big]\bigg|^{2m}\bigg] = O_T(n^{m(\#(N)+1)}), \qquad \forall T \in \salg{T}\langle\mbf{x}\rangle.
\]
The bound is tight in the sense that there exist test graphs $T$ in $\mbf{x}$ with
\[
\E\bigg[\bigg|\emph{tr}\big[T(\salg{X}_n)\big] - \E\,\emph{tr}\big[T(\salg{X}_n)\big]\bigg|^{2m}\bigg] = \Theta_T(n^{m(\#(N)+1)}).
\]
\end{thm}
\begin{proof}
The concrete case of $m = 2$ contains all of the essential ideas; we encourage the reader to follow through the proof with this simpler case in mind.

To begin, we expand the absolute value as in \eqref{eq:3.23} to obtain
\begin{equation}\label{eq:3.35}
\sum_{\phi_1, \ldots, \phi_{2m}: V \to [n]} \E\bigg[\prod_{\ell=1}^{2m} \bigg(\prod_{e \in E} \mbf{X}_{n, \ell}^{(\gamma(e))}(\phi_\ell(e)) - \E \prod_{e \in E} \mbf{X}_{n, \ell}^{(\gamma(e))}(\phi_\ell(e)) \bigg)\bigg],
\end{equation}
where
\[
\mbf{X}_{n, \ell}^{(i)}(j, k) =
\begin{cases}
\mbf{X}_n^{(i)}(j, k) & \text{ if $\ell$ is odd}, \\
\mbf{X}_n^{(i)}(k, j) & \text{ if $\ell$ is even.}
\end{cases}
\]
Our strong moment assumption \eqref{eq:3.1} again ensures that we can bound the summands in \eqref{eq:3.35} uniformly in $(\phi_1, \ldots, \phi_{2m})$ and $n$ with a dependence only on $T$, i.e.,
\begin{equation}\label{eq:3.36}
\E\bigg[\prod_{\ell=1}^{2m} \bigg(\prod_{e \in E} \mbf{X}_{n, \ell}^{(\gamma(e))}(\phi_\ell(e)) - \E \prod_{e \in E} \mbf{X}_{n, \ell}^{(\gamma(e))}(\phi_\ell(e)) \bigg)\bigg] \leq C_T < \infty.
\end{equation}

We proceed to an analysis of contributing $2m$-tuples $\Phi = (\phi_1, \ldots, \phi_{2m})$. Using the same notation as before, we say that a coordinate $\phi_\ell$ in a $2m$-tuple $\Phi$ is \emph{unmatched} if
\[
\wtilde{\phi}_\ell(E) \cap \wtilde{\phi}_{\ell'}(E) = \emptyset, \qquad \forall \ell' \neq \ell.
\]
Similarly, we say that the distinct coordinates $\phi_\ell$ and $\phi_{\ell'}$ (i.e., $\ell \neq \ell'$) are \emph{matched} if
\[
\wtilde{\phi}_\ell(E) \cap \wtilde{\phi}_{\ell'}(E) \neq \emptyset.
\]
We further say that a $2m$-tuple $\Phi$ is \emph{unmatched} if it has an unmatched coordinate $\phi_\ell$; otherwise, we say that $\Phi$ is \emph{matched}.

We define an equivalence relation $\sim$ on the coordinates of $\Phi$ by matchings; thus,
\begin{equation*}
\phi_\ell \sim \phi_{\ell'} \quad \Longleftrightarrow \quad \exists \ell_1, \ldots \ell_k \in [2m]: \phi_{\ell_j} \text{ and } \phi_{\ell_{j+1}} \text{ are matched for } j = 0, \ldots, k,
\end{equation*}
where $\ell(0) = \ell$ and $\ell(k+1) = \ell'$. We write $\wtilde{\Phi}$ for the set of equivalence classes in $\Phi$, in which case \eqref{eq:3.36} becomes
\[
\prod_{[\wtilde{\phi}] \in \wtilde{\Phi}} \E\bigg[\prod_{\phi \in [\wtilde{\phi}]} \bigg(\prod_{e \in E} \mbf{X}_{n, \ell(\phi)}^{(\gamma(e))}(\phi(e)) - \E \prod_{e \in E} \mbf{X}_{n, \ell(\phi)}^{(\gamma(e))}(\phi(e)) \bigg)\bigg].
\]
For an unmatched $\Phi$, this product includes a zero term; henceforth, we only consider matched $2m$-tuples. We incorporate the data of such a tuple into the graph $T$ as before.

For a $2m$-tuple $\Phi$, we construct a new graph $T_{\sqcup \Phi}$ by considering $2m$ disjoint copies $(T_1, \ldots, T_{2m})$ of $T$ (associated to $\Phi = (\phi_1, \ldots, \phi_{2m})$ respectively), reversing the direction of the edges of $(T_2, T_4, \ldots, T_{2m})$, and then identifying the vertices according their images under the maps $\Phi$; formally, the vertices of $T_{\sqcup \Phi}$ are then given by
\[
V_{\sqcup \Phi} = (\cup_{\ell=1}^{2m} \phi_\ell^{-1}(m) : m \in [n]).
\]
Note that
\[
\Phi \text{ is matched} \quad \Longrightarrow \quad T_{\sqcup \Phi} \text{ has $\leq m$ connected components.} 
\]
The sum over the set of matched $2m$-tuples $\Phi$ can then be decomposed into a double sum: the first, over the set $\salg{S}_T$ of (not necessarily connected) graphs $T_\sqcup = (V_\sqcup, E_\sqcup, \gamma_\sqcup)$ obtained by gluing the vertices of $2m$ disjoint copies of $T$ such that each copy has at least one edge overlay with at least one other copy (we reverse the direction of the edges of the even copies beforehand, and we again keep track of the origin of the edges $E_\sqcup = E_\sqcup^{(1)} \sqcup \cdots \sqcup E_\sqcup^{(2m)}$); the second, over the set of injective labelings $\phi_\sqcup : V_\sqcup \hookrightarrow [n]$ of the vertices of $T_\sqcup$. We write $C(T_\sqcup) = \{C_1, \ldots, C_{d_{T_\sqcup}}\}$ for the set of connected components of $T_\sqcup$. We emphasize that
\begin{equation}\label{eq:3.37}
d_{T_\sqcup} \leq m.
\end{equation}

Note that the edges $E_p$ of each connected component $C_p$ consists of a union
\[
E_p = E_\sqcup^{(j_p(1))} \sqcup \cdots \sqcup E_\sqcup^{(j_p(k_p))}.
\]
We may then recast \eqref{eq:3.35} as
\begin{equation}\label{eq:3.38}
\sum_{T_\sqcup \in \salg{S}_T} \sum_{\phi_\sqcup: V_\sqcup \hookrightarrow [n]} \prod_{p=1}^{d_{T_\sqcup}} \E\bigg[\prod_{\ell=1}^{k_p} \bigg(\prod_{e \in E_\sqcup^{(j_p(\ell))}} \mbf{X}_n^{(\gamma_\sqcup(e))}(\phi_\sqcup(e)) - \E \prod_{e \in E_\sqcup^{(j_p(\ell))}} \mbf{X}_n^{(\gamma_\sqcup(e))}(\phi_\sqcup(e)) \bigg)\bigg]. 
\end{equation}

We consider a generic $T_\sqcup \in \salg{S}_T$. Note that our analysis from before applies to each of the connected components $C_p = (V_p, E_p, \gamma_p)$. In particular, using the same notation as before, we know that the components of a contributing $T_\sqcup$ must satisfy
\begin{align}
m_{i,[e]} = 0 \text{ or } m_{i,[e]} &\geq 2, \qquad \forall (i, [e]) \in I \times \wtilde{N}_p, \label{eq:3.39} \\
\#(N_p) &\geq 2\#(\wtilde{N}_p), \label{eq:3.40} \\
\#(\wtilde{N}_p) + 1 &\geq \#(V_p). \label{eq:3.41}
\end{align}
Of course, we also have the inherent (in)equalities
\begin{equation}\label{eq:3.42}
\sum_{p=1}^{d_{T_\sqcup}} \#(V_p) = \#(V_\sqcup), \qquad \sum_{p=1}^{d_{T_\sqcup}} \#(N_p) = \#(N_\sqcup) \leq 2m\#(N).
\end{equation}
Putting everything together, we arrive at the asymptotic
\[
\E\bigg[\bigg|\trace\big[T(\salg{X}_n)\big] - \E \, \trace\big[T(\salg{X}_n)\big]\bigg|^{2m}\bigg] = O_T(n^{\max\{\#(V_\sqcup) : T_\sqcup \in \salg{S}_T\}}) \leq O_T(n^{m\#(N)+d_{T_\sqcup}}) \leq O_T(n^{m(\#(N)+1)}).
\]

The tightness of our bound follows much as before. If we start with a tree $T$, we can overlay pairs of the $2m$-disjoint copies $(T_1, \ldots, T_{2m})$ of $T$ to obtain a forest of $d_{T_\sqcup} = m$ opposing colored double trees. In this case, we have equality in \eqref{eq:3.37} and \eqref{eq:3.39}-\eqref{eq:3.42}. Once again, Proposition \hyperref[prop3.1.2]{3.1.2} shows that the contribution of $T_\sqcup$ in \eqref{eq:3.38} is $\Theta(n^{m(\#(N)+1)})$. As was the case for $m=1$, a forest of $m$ colored double trees $T_\sqcup$ corresponds to the worst case scenario.
\end{proof}

Reintroducing the standard normalization $\salg{W}_n = n^{-1/2}\salg{X}_n$, we obtain the asymptotic
\begin{equation}\label{eq:3.43}
\E\bigg[\bigg|\frac{1}{n}\trace\big[T(\salg{W}_n)\big] - \E\frac{1}{n}\trace\big[T(\salg{W}_n)\big]\bigg|^{2m}\bigg] = O_T(n^{-m(\#(L)+1)}), \qquad \forall T \in \salg{T}\langle\mbf{x}\rangle,
\end{equation}
which bounds the deviation
\begin{equation}\label{eq:3.44}
\prob\bigg(\bigg|\frac{1}{n}\trace\big[T(\salg{W}_n)\big] - \E\frac{1}{n}\trace\big[T(\salg{W}_n)\big]\bigg| > \varepsilon\bigg) = O_{T, m}(n^{-m(\#(L)+1)}), \qquad \forall T \in \salg{T}\langle\mbf{x}\rangle.
\end{equation}

We chose to work with the random variable $\trace\big[T(\salg{X}_n)\big]$, but virtually the same proof applies to the injective version
\[
\trace^0\big[T(\salg{X}_n)\big] = \sum_{\phi: V \hookrightarrow [n]} \prod_{e \in E} (\mbf{X}_n^{(\gamma(e))})(\phi(e)).
\]
In particular, Theorem \hyperref[thm3.2.2]{3.2.2} holds with $\trace^0\big[T(\salg{X}_n)\big]$ in place of $\trace\big[T(\salg{X}_n)\big]$, and so too do its implications \eqref{eq:3.43} and \eqref{eq:3.44}. Of course, one could also deduce this from the relations \eqref{eq:2.10} and \eqref{eq:2.11} between $\trace\big[T(\salg{X}_n)\big]$ and $\trace^0\big[T(\salg{X}_n)\big]$, which still hold at the level of random variables (i.e., before taking the expectation). This shows that the two results are in fact equivalent. We may then apply the usual Borel-Cantelli machinery to prove the almost sure version of Proposition \hyperref[prop3.1.2]{3.1.2} (and, as a special case, the a.s.\@ version of Corollary \hyperref[cor3.2.3]{3.2.3}).

The results in this section apply just as well to Wigner matrices of a general parameter $\beta_i \in \C$. In this case, we do not need a separate statement for the general situation.

\section{Random band matrices}\label{sec4}

Our analysis of the Wigner matrices $\salg{W}_n$ in the previous sections crucially relies on two important features of our ensemble, namely, the homogeneity of the vertices in our graphs $T$ and the divergence of our normalization $\sqrt{n}$. By the first property, we mean that the label $\phi(v) \in [n]$ of a vertex $v \in V$ does not constrain our choice of a contributing label $\phi(w)$ for an adjacent vertex $w \sim_e v$ (or, in the case of an injective labeling $\phi$, does so uniformly in the choice of $\phi(v)$). At the level of the matrices $\salg{X}_n$, this corresponds to the fact that any given row (resp., column) of a Wigner matrix looks much the same as any other row (resp., column). For example, if we consider a real Wigner matrix as in Definition \hyperref[defn1.1]{1.1}, then the rows (resp, columns) each have the same distribution up to a cyclic permutation of the entries. More generally, there exists a permutation invariant realization of our ensemble $\salg{X}_n$ iff $\beta_i \in \R$. This property of course does not hold for the random band matrices $\mbf{\Xi}_n = \mbf{B}_n \circ \mbf{X}_n$ (recall Definition \hyperref[defn1.4]{1.4}): rows (resp, columns) near the top or the bottom (resp., the far left or the far right) of our matrix will in general have fewer nonzero entries. This in turn owes to the asymmetry of the band condition $\mbf{B}_n$. We can recover the homogeneity of our ensemble by reflecting the band width across the perimeter of the matrix to obtain the so-called periodic random band matrices, providing an intermediate model between the Wigner matrices and the random band matrices. We start with this technically simpler model and work our way up to the RBMs. We summarize the main results at the end of Section \hyperref[sec4.3]{4.3}.

\begin{Rem}\label{rem4.1}
The so-called homogeneity property mentioned above and the corresponding periodization technique first appeared in the work \cite{BMP91} of Bogachev, Molchanov, and Pastur. The authors used this intermediate model to transfer Wigner's semicircle law to random band matrices of slow growth. We employ the same periodization technique to identify the joint limiting traffic distribution of independent random band matrices.
\end{Rem} 

\subsection{Periodic random band matrices}\label{sec4.1}

To begin, we formalize

\begin{defn}[Periodic RBM]\label{defn4.1.1}
Let $(b_n)$ be a sequence of nonnegative integers. We write $\mbf{P}_n$ for the corresponding $n \times n$ periodic band matrix of ones with band width $b_n$, i.e.,
\[
\mbf{P}_n(i, j) = \indc{|i - j|_n \leq b_n},
\]
where
\[
|i - j|_n = \min\{|i-j|, n-|i-j|\}.
\]
Let $\mbf{X}_n$ be an unnormalized Wigner matrix. We call the random matrix $\mbf{\Gamma}_n$ defined by
\[
\mbf{\Gamma}_n = \mbf{P}_n \circ \mbf{X}_n
\]
an \emph{unnormalized periodic RBM}. Using the normalization $\mbf{\Upsilon}_n = (2b_n)^{-1/2}\mbf{J}_n$, we call the random matrix $\mbf{\Lambda}_n$ defined by
\[
\mbf{\Lambda}_n = \mbf{\Upsilon}_n \circ \mbf{\Gamma}_n
\]
a \emph{normalized periodic RBM}. We simply refer to periodic RBMs when the context is clear, or when considering the definition altogether.
\end{defn}

\begin{center}
\begingroup%
  \makeatletter%
  \providecommand\color[2][]{%
    \errmessage{(Inkscape) Color is used for the text in Inkscape, but the package 'color.sty' is not loaded}%
    \renewcommand\color[2][]{}%
  }%
  \providecommand\transparent[1]{%
    \errmessage{(Inkscape) Transparency is used (non-zero) for the text in Inkscape, but the package 'transparent.sty' is not loaded}%
    \renewcommand\transparent[1]{}%
  }%
  \providecommand\rotatebox[2]{#2}%
  \ifx\svgwidth\undefined%
    \setlength{\unitlength}{468bp}%
    \ifx\svgscale\undefined%
      \relax%
    \else%
      \setlength{\unitlength}{\unitlength * \real{\svgscale}}%
    \fi%
  \else%
    \setlength{\unitlength}{\svgwidth}%
  \fi%
  \global\let\svgwidth\undefined%
  \global\let\svgscale\undefined%
  \makeatother%
  \begin{picture}(1,0.5)%
    \put(0,0){\includegraphics[width=\unitlength,page=1]{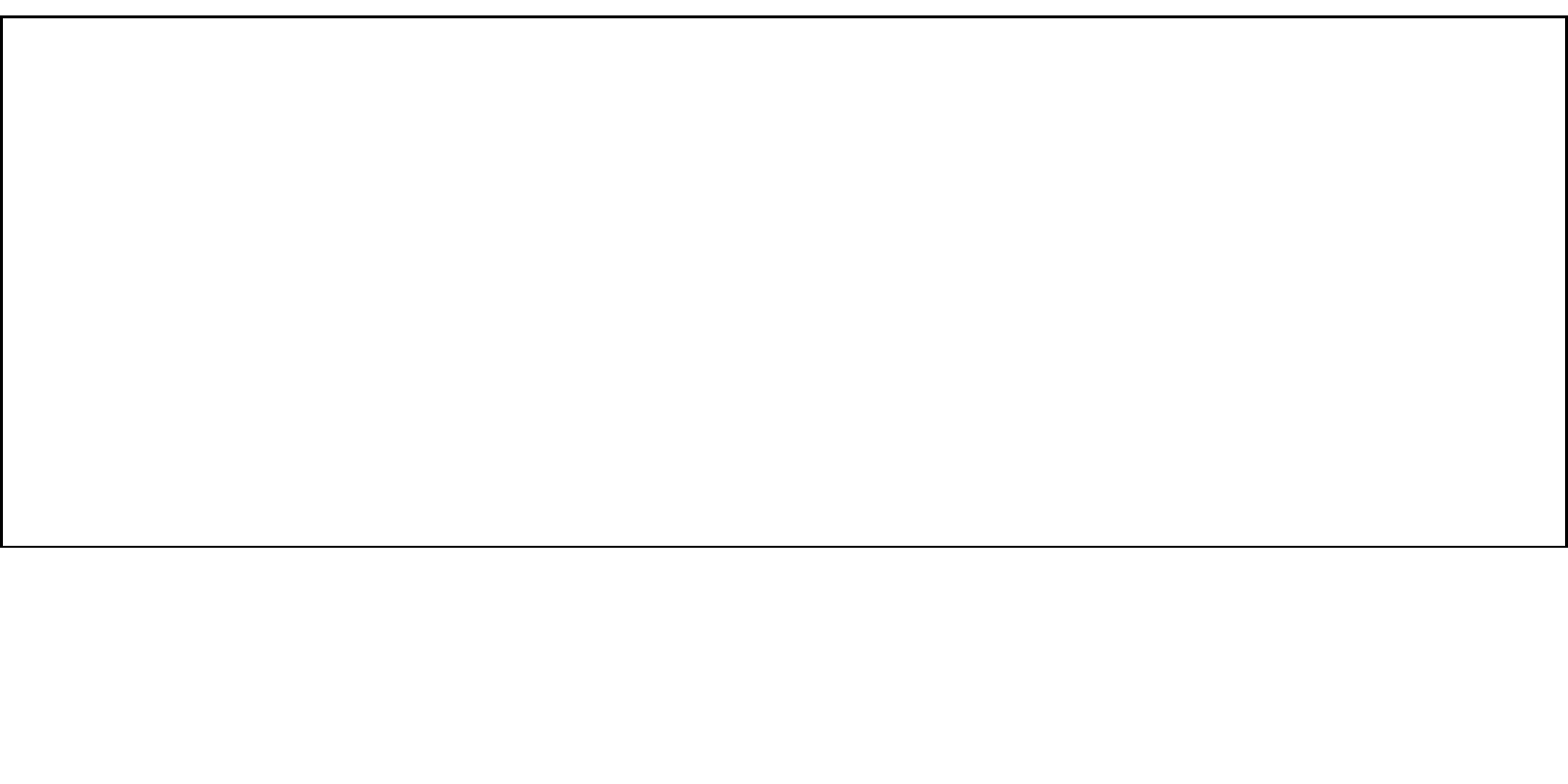}}%
    \put(-0.00141487,0.13570342){\color[rgb]{0,0,0}\makebox(0,0)[lt]{\begin{minipage}{0.99822201\unitlength}\raggedright Figure 21: An example of the periodization of a random band matrix. Here, we scale the matrix to the unit square $[0, 1]^2$. The $(i, j)$-th entry then corresponds to the subsquare $[\frac{j-1}{n}, \frac{j}{n}] \times [\frac{n-i}{n}, \frac{n-i+1}{n}]$, which we then fill in provided the band width condition $|i -j| \leq b_n$ (resp., $|i - j|_n \leq b_n$) is satisfied. \end{minipage}}}%
    \put(0.12170717,1.60355588){\color[rgb]{0,0,0}\makebox(0,0)[lt]{\begin{minipage}{0.11106843\unitlength}\raggedright \end{minipage}}}%
    \put(0,0){\includegraphics[width=\unitlength,page=2]{fig21_periodic.pdf}}%
    \put(0.66064081,0.17393845){\color[rgb]{0,0,0}\makebox(0,0)[lb]{\smash{$(\mbf{\Gamma}_n, \mbf{\Lambda}_n)$}}}%
    \put(0.24540044,0.17393845){\color[rgb]{0,0,0}\makebox(0,0)[lb]{\smash{$(\mbf{\Xi}_n, \mbf{\Theta}_n)$}}}%
  \end{picture}%
\endgroup%

\end{center}

Let $\salg{X}_n = (\mbf{X}_n^{(i)})_{i \in I}$ be a family of unnormalized Wigner matrices as in Section \hyperref[sec3]{3}. We consider a family of divergent band widths $(b_n^{(i)})_{i \in I}$ such that
\begin{equation}\label{eq:4.1}
\lim_{n \to \infty} b_n^{(i)} = \infty,\qquad \forall i \in I,
\end{equation}
for which we form the corresponding family of periodic RBMs, unnormalized $\salg{R}_n = (\mbf{\Gamma}_n^{(i)})_{i \in I}$ and otherwise $\salg{P}_n = (\mbf{\Lambda}_n^{(i)})_{i \in I}$. We identify the LTD of the family $\salg{P}_n$ with that of the familiar Wigner matrices $\salg{W}_n$ from Proposition \hyperref[prop3.1.2]{3.1.2}.

\begin{lemma}\label{lemma4.1.2}
For any test graph $T$ in $\mbf{x} = (x_i)_{i \in I}$,
\begin{equation}\label{eq:4.2}
\lim_{n \to \infty} \tau^0\big[T(\salg{P}_n)\big] = 
\begin{cases}
\prod_{i \in I} \beta_i^{c_i(T)} & \text{if $T$ is a colored double tree,} \\
\hfil 0 & \text{otherwise}.
\end{cases}
\end{equation}
\end{lemma}
\begin{proof}
The proof follows much along the same lines as Proposition \hyperref[prop3.1.2]{3.1.2} except that we must take care to account for the differing rates of growth in the band widths $b_n^{(i)}$. To begin, suppose that $T = (V, E, \gamma)$. By definition, we have that
\begin{align}
\notag \tau^0\big[T(\salg{P}_n)\big] &= \E\bigg[\frac{1}{n}\sum_{\phi: V \hookrightarrow [n]} \prod_{e \in E} \mbf{\Lambda}_n^{(\gamma(e))}(\phi(e))\bigg] \\
&= \frac{1}{n\prod_{e \in E} \sqrt{2b_n^{(\gamma(e))}}}\sum_{\phi: V \hookrightarrow [n]} \E\bigg[\prod_{e \in E} \mbf{\Gamma}_n^{(\gamma(e))}(\phi(e))\bigg]. \label{eq:4.3}
\end{align}
Using the same notation as before, we can recast the sum in \eqref{eq:4.3} as
\begin{equation}\label{eq:4.4}
\sum_{\phi: V \hookrightarrow [n]} \bigg(\prod_{[\ell] \in \wtilde{L}} \E \bigg[\prod_{\ell' \in [\ell]} \mbf{\Gamma}_n^{(\gamma(\ell'))}(\phi(\ell'))\bigg]\bigg) \bigg(\prod_{[e] \in \wtilde{N}} \E \bigg[\prod_{e' \in [e]} \mbf{\Gamma}_n^{(\gamma(e'))}(\phi(e'))\bigg]\bigg).
\end{equation}

Whereas before the label $\phi(v)$ of a vertex $v$ does not constrain our choice of label $\phi(w)$ for an adjacent vertex $w \sim_e v$ (beyond the injectivity requirement), we note that in this case a summand of \eqref{eq:4.4} equals zero if
\[
\exists e_0 \in [e]: |\phi(\source(e_0)) - \phi(\target(e_0))|_n > b_n^{(\gamma(e_0))}.
\]
In fact, we see that such a summand equals zero as soon as
\begin{equation*}
\exists e_0 \in [e]: |\phi(\source(e_0)) - \phi(\target(e_0))|_n > \min_{e' \in [e]} b_n^{(\gamma(e'))}.
\end{equation*}
To keep track of these constraints, we define
\[
|\phi(e)|_n = |\phi(\source(e)) - \phi(\target(e))|_n.
\]
Note that $|\phi(\cdot)|_n$ is constant on equivalence classes $[e] \in \wtilde{N}$, and so we further write $|\phi([e])|_n$ for the common value of
\[
\{|\phi(e')|_n : e' \in [e]\}.
\]
We use the function $|\phi(\cdot)|_n$ to define the band width condition
\[
C_{[e]} = \indc{|\phi([e])|_n \leq \min_{e' \in [e]} b_n^{(\gamma(e'))}},
\]
which allows us to rewrite \eqref{eq:4.4} as
\begin{equation}\label{eq:4.5}
\sum_{\phi: V \hookrightarrow [n]} \bigg(\prod_{[\ell] \in \wtilde{L}} \E \bigg[\prod_{\ell' \in [\ell]} \mbf{X}_n^{(\gamma(\ell'))}(\phi(\ell'))\bigg]\bigg) \bigg(\prod_{[e] \in \wtilde{N}} C_{[e]} \E \bigg[\prod_{e' \in [e]} \mbf{X}_n^{(\gamma(e'))}(\phi(e'))\bigg]\bigg)
\end{equation}
in terms of the usual Wigner matrices $\salg{X}_n = (\mbf{X}_n^{(i)})_{i \in I}$ (cf. \eqref{eq:3.7}). We may then apply our analysis from Proposition \hyperref[prop3.1.2]{3.1.2} to conclude that a contributing graph $T$ satisfies
\begin{equation}\label{eq:4.6}
m_{i, [e]} = 0 \text{ or } m_{i, [e]} \geq 2  \qquad \forall (i,[e]) \in I \times \wtilde{N}.
\end{equation}

The band width condition
\begin{equation}\label{eq:4.7}
|\phi([e])|_n \leq \min_{e' \in [e]} b_n^{(\gamma(e'))}, \qquad \forall [e] \in \wtilde{N}
\end{equation}
bounds the number $A_n(T)$ of contributing maps $\phi: V \hookrightarrow [n]$ by
\begin{equation*}
A_n(T) \leq n\prod_{[e] \in \wtilde{N}} \min_{e' \in [e]} 2b_n^{(\gamma(e'))}.
\end{equation*}
Indeed, fixing an arbitrary vertex $v_0 \in V$, we have $n$ choices for $\phi(v_0) \in [n]$; but, having made this choice, we must take into account the band widths in traversing the remaining edges of the simple graph $\underline{T} = (V, \wtilde{N})$. In fact, we can apply the same reasoning to any spanning tree $\underline{T_0} = (V, \wtilde{N}_0)$ of $\underline{T}$ since any edge $[e_k] \in \wtilde{N}$ in a cycle $([e_1], \ldots, [e_k])$ will have already had the admissible range of labels for its incident vertices determined by the band width conditions coming from the other edges $([e_1], \ldots, [e_{k-1}])$. This leads to the refinement
\begin{equation}\label{eq:4.8}
A_n(T) \leq n\prod_{[e] \in \wtilde{N}_0} \min_{e' \in [e]} 2b_n^{(\gamma(e'))},
\end{equation}
where
\begin{equation}\label{eq:4.9}
\#(\wtilde{N}_0) \leq \#(\wtilde{N}) \leq \#(\wtilde{E}).
\end{equation}

Recycling the bound \eqref{eq:3.11} for the summands of \eqref{eq:4.5}, we arrive at the asymptotic
\begin{equation*}
\begin{aligned}
\tau^0\big[T(\salg{P}_n)\big] &= O_T\bigg(\frac{n\prod_{[e] \in \wtilde{N}_0} \min_{e' \in [e]} 2b_n^{(\gamma(e'))}}{n\prod_{e \in E} \sqrt{2b_n^{(\gamma(e))}}}\bigg) \\
&= O_T\bigg(\frac{\prod_{[e] \in \wtilde{N}_0} \min_{e' \in [e]} 2b_n^{(\gamma(e'))}}{\prod_{e \in N} \sqrt{2b_n^{(\gamma(e))}}\prod_{\ell \in L} \sqrt{2b_n^{(\gamma(\ell))}}}\bigg).
\end{aligned}
\end{equation*}
For the sake of comparison, we draw the reader's attention to \eqref{eq:3.12} for the analogous asymptotic in the case of the Wigner matrices $\salg{W}_n$ (note that $\#(\wtilde{N}_0) = \#(V) - 1$). The divergence \eqref{eq:4.1} of the band widths $b_n^{(i)}$ and the inequalities \eqref{eq:4.6} and \eqref{eq:4.9} then imply that $\tau^0\big[T(\salg{P}_n)\big]$ vanishes in the limit unless $T$ is a colored double tree, in which case one clearly obtains the prescribed limit \eqref{eq:4.2}.
\end{proof}

Here, the situation for general $\beta_i \in \C$ becomes much different. For a single periodic RBM $\mbf{\Lambda}_n$ of divergent band width $b_n \to \infty$, the LTD again follows \eqref{eq:3.16} as in the Wigner case; however, the joint LTD of $\salg{P}_n$ might not exist depending on the fluctuations of the band widths $b_n^{(i)}$. In this case, we need to make additional assumptions on the band widths (e.g., proportional growth) to ensure the existence of an asymptotic proportion for an ordering $\psi$ of the vertices (i.e., the analogue of \eqref{eq:3.15}). We comment more on this situation later.

On the other hand, the orderings $\psi$ play no role in the calculation of $\tau^0\big[T(\salg{P}_n)\big]$ for an \emph{opposing} colored double tree $T$. Consequently, we can apply the criteria \eqref{eq:3.18} in Remark \hyperref[rem3.1.3]{3.1.3} to conclude that $\salg{P}_n = (\mbf{\Lambda}_n^{(i)})_{i \in I}$ converges in joint distribution to a semicircular system $\mbf{a} = (a_i)_{i \in I}$ regardless of $(\beta_i)_{i \in I}$.

Note that a periodic RBM $\mbf{\Lambda}_n$ with band width $b_n = n/2$ corresponds to a standard Wigner matrix $\mbf{W}_n$. As such, we can view Lemma \hyperref[lemma4.1.2]{4.1.2} as a generalization of Proposition \hyperref[prop3.1.2]{3.1.2}. We extend the result to include RBMs of slow growth in the next section. 

\subsection{Slow growth}\label{sec4.2}

To begin, we partition the index set $I$ of our matrices $\salg{X}_n = (\mbf{X}_n^{(i)})_{i \in I}$ into two camps $I = I_1 \cup I_2$. We consider a class of divergent band widths $(b_n^{(i)})_{i \in I}$ as in \eqref{eq:4.1} with the added condition of slow growth for $(b_n^{(i)})_{i \in I_2}$, i.e.,
\begin{equation}\label{eq:4.10}
\lim_{n \to \infty} \frac{b_n}{n} = 0, \qquad \forall i \in I_2.
\end{equation}
We form the corresponding family of periodic RBMs as before,
\[
\salg{R}_n = \salg{R}_n^{(1)} \cup \salg{R}_n^{(2)} = (\mbf{\Gamma}_n^{(i)})_{i \in I_1} \cup (\mbf{\Gamma}_n^{(i)})_{i \in I_2}, \qquad \salg{P}_n = \salg{P}_n^{(1)} \cup \salg{P}_n^{(2)} = (\mbf{\Lambda}_n^{(i)})_{i \in I_1} \cup (\mbf{\Lambda}_n^{(i)})_{i \in I_2}.
\] 
For $i \in I_2$, we also form the corresponding family of slow growth RBMs (see Definition \hyperref[defn1.4]{1.4}),
\[
\salg{S}_n^{(2)} = (\mbf{\Xi}_n^{(i)})_{i \in I_2} = (\mbf{B}_n^{(i)} \circ \mbf{X}_n^{(i)})_{i \in I_2}, \qquad \salg{O}_n^{(2)} = (\mbf{\Theta}_n^{(i)})_{i \in I_2} = (\mbf{\Upsilon}_n^{(i)} \circ \mbf{\Xi}_n^{(i)})_{i \in I_2}.
\]

\begin{lemma}\label{lemma4.2.1}
The family $\salg{M}_n = \salg{P}_n^{(1)} \cup \salg{O}_n^{(2)}$ converges in traffic distribution to the limit
\begin{equation}\label{eq:4.11}
\lim_{n \to \infty} \tau^0\big[T(\salg{M}_n)\big] = 
\begin{cases}
\prod_{i \in I} \beta_i^{c_i(T)} & \text{if $T$ is a colored double tree,} \\
\hfil 0 & \text{otherwise}.
\end{cases}
\end{equation}
\end{lemma}
\begin{proof}
In view of Lemma \hyperref[lemma4.1.2]{4.1.2}, it suffices to show that
\begin{equation}\label{eq:4.12}
\lim_{n \to \infty} \bigg|\tau^0\big[T(\salg{P}_n)\big] - \tau^0\big[T(\salg{M}_n)\big]\bigg| = 0, \qquad \forall T \in \salg{T}\langle\mbf{x}\rangle.
\end{equation}
Of course, the only difference between the families $\salg{P}_n$ and $\salg{M}_n$ comes from the periodization of the slow growth RBMs $\salg{S}_n^{(2)}$. Equation \eqref{eq:4.12} then asserts that the contribution of the additional entries arising from this periodization becomes negligible in the limit.

For convenience, we write $\salg{U}_n = (\mbf{U}_n^{(i)})_{i \in I}$ for the unnormalized version of $\salg{M}_n$ so that
\[
\mbf{U}_n^{(i)} =
\begin{cases}
\mbf{\Gamma}_n^{(i)} & \text{if } i \in I_1, \\
\mbf{\Xi}_n^{(i)} & \text{if } i \in I_2.
\end{cases}
\]
Expanding $\tau^0\big[T(\salg{M}_n)\big]$, we obtain the analogue of \eqref{eq:4.3},
\begin{equation*}
\frac{1}{n\prod_{e \in E} \sqrt{2b_n^{(\gamma(e))}}} \sum_{\phi: V \hookrightarrow [n]} \E\bigg[\prod_{e \in E} \mbf{U}_n^{(\gamma(e))}(\phi(e))\bigg].
\end{equation*}
Our notation works just as well in this case to produce the analogue of \eqref{eq:4.4} for our sum,
\begin{equation*}
\sum_{\phi: V \hookrightarrow [n]} \bigg(\prod_{[\ell] \in \wtilde{L}} \E \bigg[\prod_{\ell' \in [\ell]} \mbf{U}_n^{(\gamma(\ell'))}(\phi(\ell'))\bigg]\bigg) \bigg(\prod_{[e] \in \wtilde{N}} \E \bigg[\prod_{e' \in [e]} \mbf{U}_n^{(\gamma(e'))}(\phi(e'))\bigg]\bigg).
\end{equation*}
Naturally, we then look for the analogue of \eqref{eq:4.5}. Note that the corresponding version of the band width condition \eqref{eq:4.7} must now take into account the index $\gamma(e') \in I_1 \cup I_2$ of $e' \in [e]$. We partition the equivalence classes $[e] = [e]_1 \cup [e]_2$ in $\wtilde{N}$ accordingly, where
\[
[e]_j = [e] \cap \gamma^{-1}(I_j).
\]

For an edge $e \in N$, we define
\[
|\phi(e)| = |\phi(\source(e)) - \phi(\target(e))|.
\]
As before, $|\phi(\cdot)|$ is constant on equivalence classes $[e] \in \wtilde{N}$, and we write $|\phi([e])|$ for the common value of
\[
\{|\phi(e')| : e' \in [e]\}.
\]
More specifically, we write $|\phi([e]_2)|$ for the common value of
\[
\{|\phi(e')| : e' \in [e]_2\}.
\]
Note that $[e]_2$ may be empty, in which case we define $|\phi(\emptyset)| = 0$. We use the same convention for $|\phi([e]_1)|_n$ to define the band width condition
\begin{equation*}
C_{[e]}' = \mathbbm{1}\{|\phi([e]_1)|_n \leq \min_{e' \in [e]_1} b_n^{(\gamma(e'))}\}\mathbbm{1}\{|\phi([e]_2)| \leq \min_{e' \in [e]_2} b_n^{(\gamma(e'))}\}, \qquad \forall [e] \in \wtilde{N}. 
\end{equation*}
We may then write the analogue of \eqref{eq:4.5} for our family $\salg{M}_n$ as
\begin{equation}\label{eq:4.13}
\sum_{\phi: V \hookrightarrow [n]} \bigg(\prod_{[\ell] \in \wtilde{L}} \E \bigg[\prod_{\ell' \in [\ell]} \mbf{X}_n^{(\gamma(\ell'))}(\phi(\ell'))\bigg]\bigg) \bigg(\prod_{[e] \in \wtilde{N}} C_{[e]}'\E \bigg[\prod_{e' \in [e]} \mbf{X}_n^{(\gamma(e'))}(\phi(e'))\bigg]\bigg).
\end{equation}
Of course, the inherent inequality $|\cdot|_ n = \min\{|\cdot|, n - |\cdot|\} \leq |\cdot|$ implies that
\[
C_{[e]}' \leq \indc{|\phi([e])|_n \leq \min_{e' \in [e]} b_n^{(\gamma(e'))}} = C_{[e]}, \qquad \forall [e] \in \wtilde{N},
\]
which bounds the number $B_n(T)$ of maps $\phi: V \hookrightarrow [n]$ satisfying the band width condition
\begin{equation}\label{eq:4.14}
|\phi([e]_1)|_n \leq \min_{e' \in [e]_1} b_n^{(\gamma(e'))} \quad \text{and} \quad |\phi([e]_2)| \leq \min_{e' \in [e]_2} b_n^{(\gamma(e'))}, \qquad \forall [e] \in \wtilde{N}
\end{equation}
by
\begin{equation}\label{eq:4.15}
B_n(T) \leq A_n(T).
\end{equation}
Recall that $A_n(T)$ is the number of maps $\phi: V \hookrightarrow [n]$ satisfying the weaker condition
\begin{equation}\label{eq:4.16}
|\phi([e])|_n \leq \min_{e' \in [e]} b_n^{(\gamma(e'))}, \qquad \forall [e] \in \wtilde{N}
\end{equation}
present in Lemma \hyperref[lemma4.1.2]{4.1.2}. In view of \eqref{eq:4.15}, our work in this previous case implies that
\[
\lim_{n \to \infty} \tau^0\big[T(\salg{M}_n)\big] = 0
\]
unless $T$ is a colored double tree. Thus, it remains to prove \eqref{eq:4.12} for such a test graph $T$.

Comparing the two equations \eqref{eq:4.5} and \eqref{eq:4.13}, we arrive at the asymptotic
\begin{equation}\label{eq:4.17}
\bigg|\tau^0\big[T(\salg{P}_n)\big] - \tau^0\big[T(\salg{M}_n)\big]\bigg| = O_T\bigg(\frac{D_n(T)}{n\prod_{e \in E} \sqrt{2b_n^{(\gamma(e))}}}\bigg),
\end{equation}
where $D_n(T) = A_n(T) - B_n(T)$ is the number of maps $\phi: V \hookrightarrow [n]$ that satisfy the band width condition \eqref{eq:4.16} but not the stronger condition \eqref{eq:4.14}. This formalizes the observation that we made at the beginning of the proof about the only difference between the families $\salg{P}_n$ and $\salg{M}_n$. In particular, for $i \in I_2$, note that the periodic version $\mbf{\Gamma}_n^{(i)}$ of a slow growth RBM $\mbf{\Xi}_n^{(i)}$ only differs in the entries within band width's distance of the perimeter; otherwise, the two matrices are identical. For a map $\phi: V \hookrightarrow [n]$, this means that if $\phi$ stays sufficiently far away from the endpoints of the interval $[n]$, then the two conditions \eqref{eq:4.14} and \eqref{eq:4.16} are actually equivalent. In particular, this holds if
\begin{equation*}
\phi(V) \subset [1 + \max_{e \in E_2} b_n^{(\gamma(e))}, n - \max_{e \in E_2} b_n^{(\gamma(e))}], 
\end{equation*}
where $E_2 = \gamma^{-1}(I_2)$ is of course a finite set. In this case, we have the bound
\[
D_n(T) = A_n(T) - B_n(T) \leq A_n^*(T),
\]
where $A_n^*(T)$ is the number of maps $\phi: V \hookrightarrow [n]$ satisfying \eqref{eq:4.16} with range
\begin{equation}\label{eq:4.18}
\phi(V) \not\subset [1 + \max_{e \in E_2} b_n^{(\gamma(e))}, n - \max_{e \in E_2} b_n^{(\gamma(e))}].
\end{equation}

\begin{center}
\begingroup%
  \makeatletter%
  \providecommand\color[2][]{%
    \errmessage{(Inkscape) Color is used for the text in Inkscape, but the package 'color.sty' is not loaded}%
    \renewcommand\color[2][]{}%
  }%
  \providecommand\transparent[1]{%
    \errmessage{(Inkscape) Transparency is used (non-zero) for the text in Inkscape, but the package 'transparent.sty' is not loaded}%
    \renewcommand\transparent[1]{}%
  }%
  \providecommand\rotatebox[2]{#2}%
  \ifx\svgwidth\undefined%
    \setlength{\unitlength}{468bp}%
    \ifx\svgscale\undefined%
      \relax%
    \else%
      \setlength{\unitlength}{\unitlength * \real{\svgscale}}%
    \fi%
  \else%
    \setlength{\unitlength}{\svgwidth}%
  \fi%
  \global\let\svgwidth\undefined%
  \global\let\svgscale\undefined%
  \makeatother%
  \begin{picture}(1,0.96153846)%
    \put(0,0){\includegraphics[width=\unitlength,page=1]{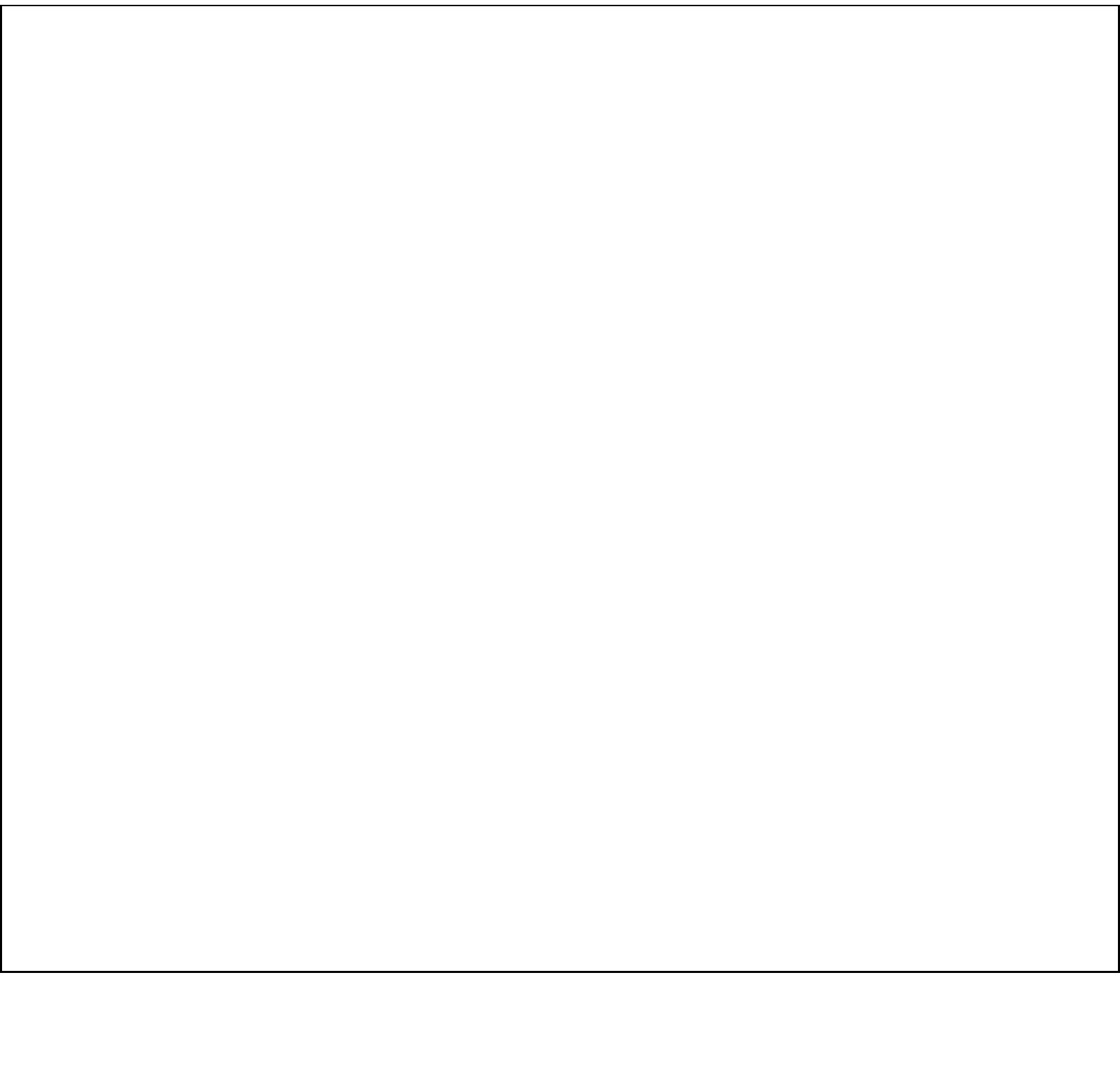}}%
    \put(-0.00141487,0.07930067){\color[rgb]{0,0,0}\makebox(0,0)[lt]{\begin{minipage}{0.99822202\unitlength}\raggedright Figure 22: An illustration of the ``interior" region of a random band matrix (resp., periodic random band matrix) at band width's distance $\frac{b_n}{n} = o(1)$ from the perimeter. Here, we cut off the boundary to see that the two interior regions are indeed identical.\end{minipage}}}%
    \put(0.12170717,1.60355587){\color[rgb]{0,0,0}\makebox(0,0)[lt]{\begin{minipage}{0.11106843\unitlength}\raggedright \end{minipage}}}%
    \put(0,0){\includegraphics[width=\unitlength,page=2]{fig22_interior.pdf}}%
    \put(0.14728346,0.15838553){\color[rgb]{0,0,0}\makebox(0,0)[lb]{\smash{$(\mbf{\Gamma}_n, \mbf{\Lambda}_n)$}}}%
    \put(0,0){\includegraphics[width=\unitlength,page=3]{fig22_interior.pdf}}%
    \put(0.14497738,0.56516443){\color[rgb]{0,0,0}\makebox(0,0)[lb]{\smash{$(\mbf{\Xi}_n, \mbf{\Theta}_n)$}}}%
    \put(0,0){\includegraphics[width=\unitlength,page=4]{fig22_interior.pdf}}%
    \put(0.66447096,0.49443521){\color[rgb]{0,0,0}\makebox(0,0)[lb]{\smash{$\frac{b_n}{n}$}}}%
    \put(0.70954307,0.44369454){\color[rgb]{0,0,0}\makebox(0,0)[lb]{\smash{$\frac{b_n}{n}$}}}%
    \put(0,0){\includegraphics[width=\unitlength,page=5]{fig22_interior.pdf}}%
    \put(0.34175462,0.31515492){\color[rgb]{0,0,0}\makebox(0,0)[lb]{\smash{$\Rightarrow$}}}%
    \put(0.34175462,0.72193433){\color[rgb]{0,0,0}\makebox(0,0)[lb]{\smash{$\Rightarrow$}}}%
    \put(0.72877384,0.72193443){\color[rgb]{0,0,0}\makebox(0,0)[lb]{\smash{$\Rightarrow$}}}%
    \put(0.72877384,0.31515565){\color[rgb]{0,0,0}\makebox(0,0)[lb]{\smash{$\Rightarrow$}}}%
  \end{picture}%
\endgroup%

\end{center}

We give a simple bound on $A_n^*(T)$ as follows: set aside a vertex $v_0 \in V$ (for which there are $\#(V)$ choices) to satisfy \eqref{eq:4.18} (for which there are $\displaystyle  2\max_{e \in E_2} b_n^{(\gamma(e))}$ choices) and pick the labels $\phi(v)$ of the remaining vertices according to \eqref{eq:4.16} (for which there are at most $\prod_{[e] \in \wtilde{E}} \min_{e' \in [e]} 2b_n^{(\gamma(e'))}$ choices) to see that
\begin{equation}\label{eq:4.19}
A_n^*(T) = O_T\bigg(\max_{e \in E_2} b_n^{(\gamma(e))}\prod_{[e] \in \wtilde{E}} \min_{e' \in [e]} 2b_n^{(\gamma(e'))}\bigg). 
\end{equation}
We may then recast \eqref{eq:4.17} as 
\begin{equation}\label{eq:4.20}
\bigg|\tau^0\big[T(\salg{P}_n)\big] - \tau^0\big[T(\salg{M}_n)\big]\bigg| = \frac{\max_{e \in E_2} b_n^{(\gamma(e))}}{n}O_T\bigg(\frac{\prod_{[e] \in \wtilde{E}} \min_{e' \in [e]} 2b_n^{(\gamma(e'))}}{\prod_{e \in E} \sqrt{2b_n^{(\gamma(e))}}}\bigg).
\end{equation}
$T$ being a colored double tree, we know that
\begin{equation*}
\frac{\prod_{[e] \in \wtilde{E}} \min_{e' \in [e]} 2b_n^{(\gamma(e'))}}{\prod_{e \in E} \sqrt{2b_n^{(\gamma(e))}}} = 1.
\end{equation*}
Moreover, since $\#(E_2) < \infty$, the slow growth \eqref{eq:4.10} still holds for the maximum over $E_2$,
\begin{equation}\label{eq:4.21}
\max_{e \in E_2} b_n^{(\gamma(e))} = o(n).
\end{equation}

Equations \eqref{eq:4.19}-\eqref{eq:4.21} formalize our intuition from before: the periodic version of a RBM only differs within band width's distance of the perimeter; for a slow growth RBM, one then needs to be very close to the perimeter to realize this difference; as such, the corresponding interior region accounts for the bulk of the calculations. The result now follows.
\end{proof}

\begin{rem}\label{rem4.2.2}
If we think of choosing a map $\phi: V \hookrightarrow [n]$ satisfying \eqref{eq:4.14} as starting at an arbitrary vertex $v_0$, making a choice $\phi(v_0) \in [n]$, and then choosing the labels of the remaining vertices in a manner compatible with the band width conditions, then each choice of $\phi(v)$ after $\phi(v_0)$ can be thought of as an incremental walk of distance at most $\min_{e' \in [e]} b_n^{(\gamma(e'))}$ for some $[e] \in \wtilde{N}$. If $I = I_2$, then starting from a ``deep'' vertex
\[
\phi(v_0) \in [1+\#(E)\max_{e \in E} b_n^{\gamma(e))}, n - \#(E)\max_{e \in E} b_n^{(\gamma(e))}],
\]
the walk never has a chance to loop across the perimeter of the matrix. This line of reasoning can be used to give a more intuitive geometric proof of Lemma \hyperref[lemma4.2.1]{4.2.1} in the simpler case of $I = I_2$. This notion of a deep vertex originates in the work \cite{BMP91}.

If $I \neq I_2$, then we need to account for the possibility of the band widths of the periodic RBMs being large enough to bring us close to the perimeter so that the walk crosses over with a step from a periodized version of a slow growth RBM. Taking inspiration from the simpler case of $I = I_2$, our analysis shows that a generic walk stays within a region in which the slow growth RBMs and their periodized versions are identical. 
\end{rem}

We encounter the same problem from before when considering general $\beta_i \in \C$: without further assumptions on the band widths $b_n^{(i)}$, their fluctuations could possibly preclude the existence of a joint LTD. In general, we must again settle for the convergence of $\salg{M}_n = (\mbf{\Lambda}_n^{(i)})_{i \in I_1} \cup (\mbf{\Theta}_n^{(i)})_{i \in I_2}$ in joint distribution to a semicircular system $\mbf{a} = (a_i)_{i \in I}$.

Recall that the Wigner matrices $\salg{W}_n$ are asymptotically traffic independent iff $\beta_i \in \R$, and that a permutation invariant realization of our ensemble $\salg{W}_n$ exists iff $\beta_i \in \R$. In view of Theorem \hyperref[thm2.5.5]{2.5.5}, one might then expect that permutation invariance is also a necessary condition for matricial asymptotic traffic independence; however, we see that this is not the case. In particular, one cannot find a permutation invariant realization of the periodic RBMs (except in the trivial case of $b_n \sim n/2$), nor of the slow growth RBMs. Instead, we relied on the aforementioned homogeneity property and the divergence of our normalization. Taken alone, neither of these two properties suffices, as we shall see in the proportional growth regime (which lacks homogeneity) and the fixed band width regime (which has a fixed normalization).

\subsection{Proportional growth}\label{sec4.3}

Not surprisingly, the periodization trick from the previous section fails for proportional growth RBMs unless $c = 1$ (recall that $c = \lim_{n \to \infty} \frac{b_n}{n}\in (0, 1]$). In the case of \emph{proper} proportion $c \in (0, 1)$, the entries in the matrix introduced by reflecting the band width across the perimeter now account for an asymptotically nontrivial region in the unit square and so no longer represent a negligible contribution to the calculations. Nevertheless, we can adapt our work from before to prove the existence of a joint LTD supported on colored double trees $T$, though in general the value of this limit will depend on the degree structure of $T$.

\begin{center}
\begingroup%
  \makeatletter%
  \providecommand\color[2][]{%
    \errmessage{(Inkscape) Color is used for the text in Inkscape, but the package 'color.sty' is not loaded}%
    \renewcommand\color[2][]{}%
  }%
  \providecommand\transparent[1]{%
    \errmessage{(Inkscape) Transparency is used (non-zero) for the text in Inkscape, but the package 'transparent.sty' is not loaded}%
    \renewcommand\transparent[1]{}%
  }%
  \providecommand\rotatebox[2]{#2}%
  \ifx\svgwidth\undefined%
    \setlength{\unitlength}{468bp}%
    \ifx\svgscale\undefined%
      \relax%
    \else%
      \setlength{\unitlength}{\unitlength * \real{\svgscale}}%
    \fi%
  \else%
    \setlength{\unitlength}{\svgwidth}%
  \fi%
  \global\let\svgwidth\undefined%
  \global\let\svgscale\undefined%
  \makeatother%
  \begin{picture}(1,0.58461538)%
    \put(0,0){\includegraphics[width=\unitlength,page=1]{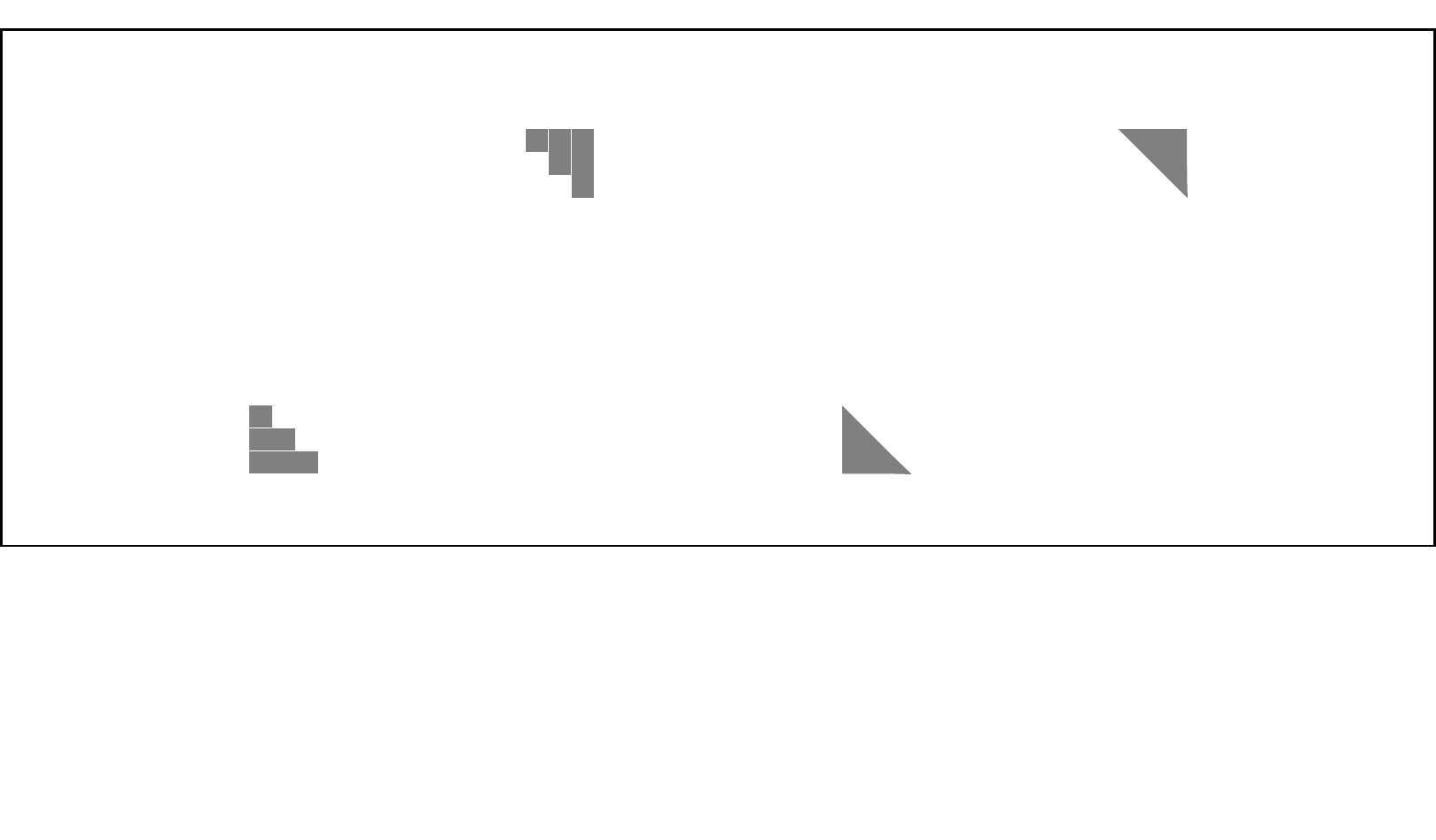}}%
    \put(-0.00141487,0.19059122){\color[rgb]{0,0,0}\makebox(0,0)[lt]{\begin{minipage}{0.99822197\unitlength}\raggedright Figure 23: An illustration of the limit shape of our scaled matrix in the unit square $[0, 1]^2$. Here, we distinguish the periodized version of our matrix with the additional grey area. In the limit, the shape corresponds to the banded region $|x-(1-y)| \leq c$ (resp., the periodic banded region $\min(|x-(1-y)|, 1-|x-(1-y)|) \leq c$). In contrast to slow growth regime, we see a nontrivial contribution from the periodization due to the nonvanishing scale of the band width $\lim_{n \to \infty} \frac{b_n}{n} = c \in (0, 1)$. \end{minipage}}}%
    \put(0.12170717,1.66837212){\color[rgb]{0,0,0}\makebox(0,0)[lt]{\begin{minipage}{0.11106843\unitlength}\raggedright \end{minipage}}}%
    \put(0,0){\includegraphics[width=\unitlength,page=2]{fig23_proportional.pdf}}%
    \put(0.3778684,0.5153973){\color[rgb]{0,0,0}\makebox(0,0)[lb]{\smash{$\frac{b_n}{n}$}}}%
    \put(0.42324499,0.46441541){\color[rgb]{0,0,0}\makebox(0,0)[lb]{\smash{$\frac{b_n}{n}$}}}%
    \put(0,0){\includegraphics[width=\unitlength,page=3]{fig23_proportional.pdf}}%
    \put(0.80008792,0.50745177){\color[rgb]{0,0,0}\makebox(0,0)[lb]{\smash{$c$}}}%
    \put(0.83774817,0.46718735){\color[rgb]{0,0,0}\makebox(0,0)[lb]{\smash{$c$}}}%
    \put(0.48718727,0.36796839){\color[rgb]{0,0,0}\makebox(0,0)[lb]{\smash{$\Rightarrow$}}}%
  \end{picture}%
\endgroup%

\end{center}

To formalize our result, we now split the index set $I = I_1 \cup I_2 \cup I_3 \cup I_4$ into four camps. We consider a class of divergent band widths $(b_n^{(i)})_{i \in I}$ as in \eqref{eq:4.1} with the added conditions of slow growth for $\mbf{b}_n^{(2)} = (b_n^{(i)})_{i \in I_2}$, full proportion for $\mbf{b}_n^{(3)} = (b_n^{(i)})_{i \in I_3}$, and proper proportion for $\mbf{b}_n^{(4)} = (b_n^{(i)})_{i \in I_4}$ so that
\begin{equation*}
\begin{aligned}
\lim_{n \to \infty} \frac{b_n^{(i)}}{n} &= 0, & \qquad \forall i &\in I_2 \\
\lim_{n \to \infty} \frac{b_n^{(i)}}{n} &= c_i = 1, & \qquad \forall i &\in I_3, \\
\lim_{n \to \infty} \frac{b_n^{(i)}}{n} &= c_i \in (0, 1), & \qquad \forall i &\in I_4.
\end{aligned}
\end{equation*}

For $i \in I_1 \cup I_2$, we form the corresponding families of periodic RBMs and slow growth RBMs as before,
\[
\begin{aligned}
\salg{R}_n &= \salg{R}_n^{(1)} \cup \salg{R}_n^{(2)} = (\mbf{\Gamma}_n^{(i)})_{i \in I_1} \cup (\mbf{\Gamma}_n^{(i)})_{i \in I_2}, & \qquad \salg{P}_n &= \salg{P}_n^{(1)} \cup \salg{P}_n^{(2)} = (\mbf{\Lambda}_n^{(i)})_{i \in I_1} \cup (\mbf{\Lambda}_n^{(i)})_{i \in I_2}; \\
\salg{S}_n^{(2)} &= (\mbf{\Xi}_n^{(i)})_{i \in I_2} = (\mbf{B}_n^{(i)} \circ \mbf{X}_n^{(i)})_{i \in I_2}, & \qquad \salg{O}_n^{(2)} &= (\mbf{\Theta}_n^{(i)})_{i \in I_2} = (\mbf{\Upsilon}_n^{(i)} \circ \mbf{\Xi}_n^{(i)})_{i \in I_2}.
\end{aligned}
\]
For $i \in I_3 \cup I_4$, we form the corresponding families of proportional growth RBMs,
\[
\begin{aligned}
\salg{F}_n^{(3)} &= (\mbf{\Xi}_n^{(i)})_{i \in I_3} = (\mbf{B}_n^{(i)} \circ \mbf{X}_n^{(i)})_{i \in I_3}, & \qquad \salg{O}_n^{(3)} &= (\mbf{\Theta}_n^{(i)})_{i \in I_3} = (\mbf{\Upsilon}_n^{(i)} \circ \mbf{\Xi}_n^{(i)})_{i \in I_3}; \\
\salg{C}_n^{(4)} &= (\mbf{\Xi}_n^{(i)})_{i \in I_4} = (\mbf{B}_n^{(i)} \circ \mbf{X}_n^{(i)})_{i \in I_4}, & \qquad \salg{O}_n^{(4)} &= (\mbf{\Theta}_n^{(i)})_{i \in I_4} = (\mbf{\Upsilon}_n^{(i)} \circ \mbf{\Xi}_n^{(i)})_{i \in I_4}.
\end{aligned}
\]

We start with the simpler case of the single family $\salg{O}_n^{(4)}$ of (proper) proportional growth RBMs. In this case, the LTD of $\salg{O}_n^{(4)}$ only depends on the band widths $\mbf{b}_n^{(4)}$ up to the limiting proportions
\[
\mbf{c}_4 = (c_i)_{i \in I_4}.
\]  
\begin{lemma}\label{lemma4.3.1}
For any test graph $T$ in $\mbf{x}_4 = (x_i)_{i \in I_4}$,
\begin{equation}\label{eq:4.22}
\lim_{n \to \infty} \tau^0\big[T(\salg{O}_n^{(4)})\big] = 
\begin{cases}
p_T(\mbf{c}_4)\prod_{i \in I} \beta_i^{c_i(T)} & \text{if $T$ is a colored double tree,} \\
\hfil 0 & \text{otherwise},
\end{cases}
\end{equation}
where $p_T(\mbf{c}_4) > 0$ only depends on the test graph $T$ and the proportions $\mbf{c}_4 = (c_i)_{i \in I_4}$. 
\end{lemma}

\begin{proof}
As usual, we begin by expanding
\[
\tau^0\big[T(\salg{O}_n^{(4)})\big] = \frac{1}{n^{1+\frac{\#(E)}{2}}\prod_{e \in E} \sqrt{2c_{\gamma(e)}-c_{\gamma(e)}^2}} \sum_{\phi: V \hookrightarrow [n]} \E \bigg[\prod_{e \in E} \mbf{\Xi}_n^{(\gamma(e))}(\phi(e))\bigg]
\]
and rewriting the sum as
\[
\sum_{\phi: V \hookrightarrow [n]} \bigg(\prod_{[\ell] \in \wtilde{L}} \E \bigg[\prod_{\ell' \in [\ell]} \mbf{X}_n^{(\gamma(\ell'))}(\phi(\ell'))\bigg]\bigg) \bigg(\prod_{[e] \in \wtilde{N}} \mathbbm{1}\{|\phi([e])| \leq \min_{e' \in [e]} b_n^{(\gamma(e'))}\}\E \bigg[\prod_{e' \in [e]} \mbf{X}_n^{(\gamma(e'))}(\phi(e'))\bigg]\bigg).
\]
At this point, we can already conclude the second half of \eqref{eq:4.22}. Hereafter, $T$ denotes a colored double tree. In this case, we have the equality
\begin{align*}
\tau^0\big[T(\salg{O}_n^{(4)})\big] &= \frac{C_n(T)}{n^{1+\#(\wtilde{E})}\prod_{[e] \in \wtilde{E}} \, (2c_{\gamma([e])} - c_{\gamma([e])}^2)}\prod_{i \in I} \beta_i^{c_i(T)} \\
&=  \frac{C_n(T)}{n^{\#(V)}}\frac{1}{\prod_{[e] \in \wtilde{E}} \, (2c_{\gamma([e])} - c_{\gamma([e])}^2)}\prod_{i \in I} \beta_i^{c_i(T)},
\end{align*}
where $C_n(T)$ is the number of maps $\phi: V \hookrightarrow [n]$ satisfying the band width condition
\begin{equation}\label{eq:4.23}
|\phi([e])| \leq b_n^{(\gamma([e]))}, \qquad \forall [e] \in \wtilde{E}.
\end{equation}

We may think of the ratio
\[
\frac{C_n(T)}{n^{\#(V)}} \sim \frac{C_n(T)}{n^{\underline{\#(V)}}} 
\]
as the proportion of admissible maps $\phi: V \hookrightarrow [n]$. Unfortunately, the vertices of our graph $T$ lack the homogeneity property from before due to the asymmetry of the band condition \eqref{eq:4.23}. This makes the task of computing $C_n(T)$ extremely tedious (and highly dependent on $T$). Nevertheless, we can give an integral representation of the limit of this ratio much as in \cite{BMP91}. In particular, a straightforward weak convergence argument shows that
\begin{equation}\label{eq:4.24}
\lim_{n \to \infty} \frac{C_n(T)}{n^{\#(V)}} = \int_{[0, 1]^V} \prod_{[e] \in \wtilde{E}} \mathbbm{1}\{|x_{\source([e])} - x_{\target([e])}| \leq c_{\gamma([e])}\} \, d\mbf{x}_V.
\end{equation}
The remaining term in \eqref{eq:4.22} follows as
\[
p_T(\mbf{c}_4) = \frac{\int_{[0, 1]^V} \prod_{[e] \in \wtilde{E}} \mathbbm{1}\{|x_{\source([e])} - x_{\target([e])}| \leq c_{\gamma([e])}\} \, d\mbf{x}_V}{\prod_{[e] \in \wtilde{E}} \, (2c_{\gamma([e])} - c_{\gamma([e])}^2)} > 0.
\]
\end{proof}

\begin{rem}\label{rem4.3.2}
For general $\beta_i \in \C$, we must again keep track of the orderings $\psi$ of the vertices. In this case, we combine the integrands of \eqref{eq:3.15} and \eqref{eq:4.24} to define
\[
p_T(\mbf{c}_4, \psi) = \frac{\int_{[0, 1]^V} \indc{x_{\psi(1)} \geq \cdots \geq x_{\psi(\#(V))}}\prod_{[e] \in \wtilde{E}} \mathbbm{1}\{|x_{\source([e])} - x_{\target([e])}| \leq c_{\gamma([e])}\} \, d\mbf{x}_V}{\prod_{[e] \in \wtilde{E}} \, (2c_{\gamma([e])} - c_{\gamma([e])}^2)},
\]
which replaces the $\frac{1}{\#(V)!}$ term in \eqref{eq:3.16}. In particular, we can then write the LTD of $\salg{O}_n^{(4)}$ as 
\begin{equation*}
\lim_{n \to \infty} \tau^0\big[T(\salg{O}_n^{(4)})\big] =
\begin{dcases}
\sum_{\psi: [\#(V)] \stackrel{\sim}{\to} V} p_T(\mbf{c}_4, \psi) S_\psi(T) & \text{if $T$ is a colored double tree,} \\
\hfil 0 & \text{otherwise.}
\end{dcases}
\end{equation*}
\end{rem}

Naturally, we are interested in the behavior of $p_T(\mbf{c}_4)$ as the proportions $\mbf{c}_4$ approach the boundary values $\{0,1\}$. To this end, we fix some notation. Recall that $T = (V, E, \gamma, \source, \target)$ is a colored double tree. We record the labels $L(\wtilde{F})$ appearing in any subset $\wtilde{F} \subset \wtilde{E}$ of twin edges so that
\[
L(\wtilde{F}) = \{\gamma([e]) : [e] \in \wtilde{F}\} \subset I_4.
\]
We write $\{\source([e]), \target([e])\}$ for the pair of vertices adjacent to twin edges $[e] = \{e, e'\}$, which allows us to further record the vertices $V(\wtilde{F})$ appearing in $\wtilde{F}$ as
\[
V(\wtilde{F}) = \{\source([e]), \target([e]) : [e] \in \wtilde{F}\}.
\]
For any collection of real numbers $\mbf{r} = (r_j)_{j \in J}$ in $[0, 1]$ with $L(\wtilde{F}) \subset J$, we define the function
\[
\op{Cut}_{\wtilde{F}, \mbf{r}} : [0, 1]^{V(\wtilde{F})} \to [0, 1]
\]
by the product
\[
\op{Cut}_{\wtilde{F}, \mbf{r}}(\mbf{x}_{V(\wtilde{F})}) = \prod_{[e] \in \wtilde{F}} \mathbbm{1}\{|x_{\source([e])} - x_{\target([e])}| \leq r_{\gamma([e])}\}.
\]
We note that $\op{Cut}_{\wtilde{F}, \mbf{r}}$ is simply the indicator on the banded region cut out of the hypercube $[0, 1]^{V(\wtilde{F})}$ by the constraints $|x_{\source([e])} - x_{\target([e])}| \leq r_{\gamma([e])}$. For example, our notation allows us to succinctly write the integral
\[
\op{Int}_T(\mbf{c}_4) = \lim_{n \to \infty} \frac{C_n(T)}{n^{\#(V)}} = \int_{[0, 1]^{V}} \op{Cut}_{\wtilde{E}, \mbf{c}_4}(\mbf{x}_V) \, d\mbf{x}_V.
\]
Similarly, we group the normalizations coming from the twin edges $\wtilde{F} \subset \wtilde{E}$ with
\begin{equation}\label{eq:4.25}
\op{Norm}_{\wtilde{F}}(\mbf{c}_4) = \prod_{[e] \in \wtilde{F}} (2c_{\gamma([e])} - c_{\gamma([e])}^2).
\end{equation}
If $\wtilde{F} = \wtilde{E}$, we write $\op{Cut}_{T, \mbf{r}} = \op{Cut}_{\wtilde{E}, \mbf{r}}$ (resp., $\op{Norm}_T(\mbf{c}_4) = \op{Norm}_{\wtilde{E}}(\mbf{c}_4)$). In this case,
\[
p_T(\mbf{c}_4) = \frac{\op{Int}_T(\mbf{c}_4)}{\op{Norm}_T(\mbf{c}_4)}.
\]

We will need some simple bounds on the integral $\op{Int}_T(\mbf{c}_4)$. We start with an easy upper bound. Consider a leaf vertex $v_0$ of our colored double tree $T$. Let $v_1 \sim_{[e_0]} v_0$ denote the unique vertex $v_1$ adjacent to $v_0$. We compute the diameter $f(x_{v_1})$ of a cross section in the banded strip of the unit square $[0, 1]^2$ defined by $|x_{v_0} - x_{v_1}| \leq c_{\gamma([e_0])}$, 
\begin{equation}\label{eq:4.26}
\begin{aligned}
f(x_{v_1}) &= \int_0^1 \mathbbm{1}\{|x_{\source([e_0])} - x_{\target([e_0])}| \leq c_{\gamma([e_0])}\}  \, dx_{v_0} \\
&= \int_0^1 \mathbbm{1}\{|x_{v_0} - x_{v_1}| \leq c_{\gamma([e_0])}\}  \, dx_{v_0} \\
&=
\begin{cases}
x_{v_1} + c_{\gamma([e_0])}     & \text{if } x_{v_1} \in [0, c_{\gamma([e_0])} \wedge (1-c_{\gamma([e_0])})] \\
2c_{\gamma([e_0])} \wedge 1 & \text{if } x_{v_1} \in [c_{\gamma([e_0])} \wedge (1-c_{\gamma([e_0])}), c_{\gamma([e_0])} \vee (1-c_{\gamma([e_0])})] \\
1 + c_{\gamma([e_0])} - x_{v_1}   & \text{if } x_{v_1} \in [(c_{\gamma([e_0])} \vee (1-c_{\gamma([e_0])}), 1]
\end{cases}
\end{aligned}
\end{equation}
\begin{center}
\begingroup%
  \makeatletter%
  \providecommand\color[2][]{%
    \errmessage{(Inkscape) Color is used for the text in Inkscape, but the package 'color.sty' is not loaded}%
    \renewcommand\color[2][]{}%
  }%
  \providecommand\transparent[1]{%
    \errmessage{(Inkscape) Transparency is used (non-zero) for the text in Inkscape, but the package 'transparent.sty' is not loaded}%
    \renewcommand\transparent[1]{}%
  }%
  \providecommand\rotatebox[2]{#2}%
  \ifx\svgwidth\undefined%
    \setlength{\unitlength}{468bp}%
    \ifx\svgscale\undefined%
      \relax%
    \else%
      \setlength{\unitlength}{\unitlength * \real{\svgscale}}%
    \fi%
  \else%
    \setlength{\unitlength}{\svgwidth}%
  \fi%
  \global\let\svgwidth\undefined%
  \global\let\svgscale\undefined%
  \makeatother%
  \begin{picture}(1,0.28846154)%
    \put(0.12010459,1.4966977){\color[rgb]{0,0,0}\makebox(0,0)[lt]{\begin{minipage}{0.11106843\unitlength}\raggedright \end{minipage}}}%
    \put(0,0){\includegraphics[width=\unitlength,page=1]{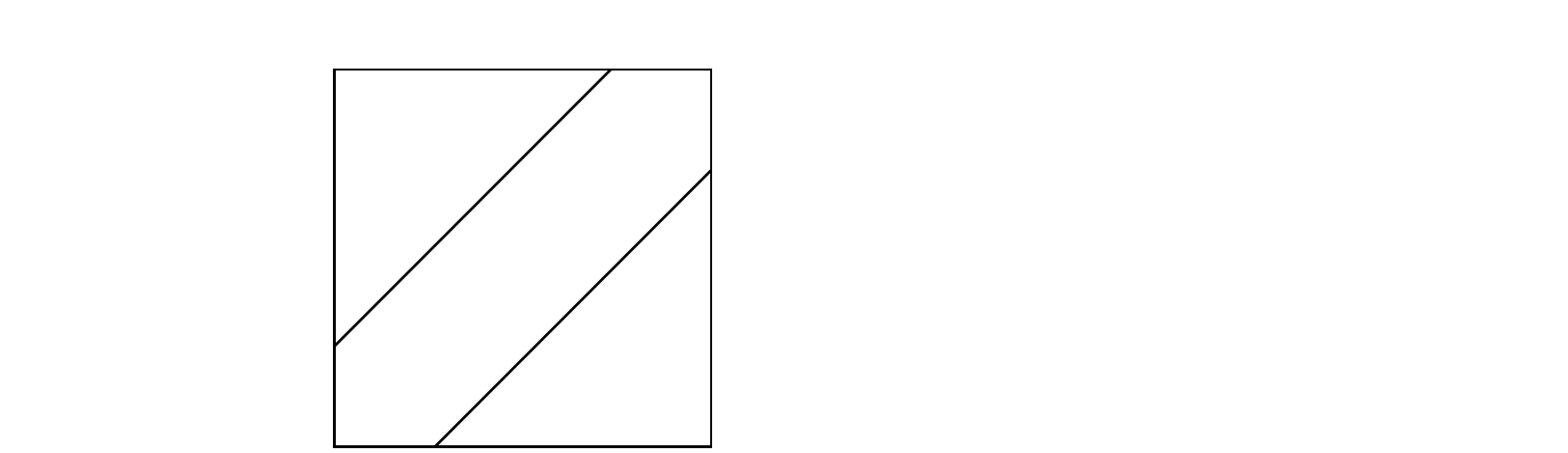}}%
    \put(0.39043251,0.26496409){\color[rgb]{0,0,0}\makebox(0,0)[lb]{\smash{$c_{\gamma([e_0])}$}}}%
    \put(0.46461516,0.20570334){\color[rgb]{0,0,0}\makebox(0,0)[lb]{\smash{$c_{\gamma([e_0])}$}}}%
    \put(0,0){\includegraphics[width=\unitlength,page=2]{fig24_diameter.pdf}}%
    \put(0.34334939,0.14809711){\color[rgb]{0,0,0}\makebox(0,0)[lb]{\smash{$f(x_{v_1})$}}}%
    \put(0,0){\includegraphics[width=\unitlength,page=3]{fig24_diameter.pdf}}%
    \put(0.32171478,-0.02638207){\color[rgb]{0,0,0}\makebox(0,0)[lb]{\smash{$x_{v_1}$}}}%
    \put(0,0){\includegraphics[width=\unitlength,page=4]{fig24_diameter.pdf}}%
    \put(0.17268411,0.11649414){\color[rgb]{0,0,0}\makebox(0,0)[lb]{\smash{$=$}}}%
    \put(0.46675463,0.11649456){\color[rgb]{0,0,0}\makebox(0,0)[lb]{\smash{$.$}}}%
  \end{picture}%
\endgroup%

\end{center}

In particular,
\[
c_{\gamma([e_0])} \leq f(x_{v_1}) \leq 2c_{\gamma([e_0])} \wedge 1.
\]
It follows that
\begin{align*}
\op{Int}_T(\mbf{c}_4) &= \int_{[0, 1]^V} \op{Cut}_{T, \mbf{c}_4}(\mbf{x}_V) d\mbf{x}_V \\
&= \int_{[0, 1]^{V\setminus\{v_0\}}} \op{Cut}_{\wtilde{E}\setminus\{[e_0]\}, \mbf{c}_4}(\mbf{x}_{V \setminus \{v_0\}}) \bigg(\int_0^1 \mathbbm{1}\{|x_{v_0} - x_{v_1}| \leq c_{\gamma([e_0])}\} \, dx_{v_0}\bigg) \, d\mbf{x}_{V\setminus\{v_0\}} \\
&\leq \int_{[0, 1]^{V\setminus\{v_0\}}} \op{Cut}_{\wtilde{E}\setminus\{[e_0]\}, \mbf{c}_4}(\mbf{x}_{V \setminus \{v_0\}})\bigg(2c_{\gamma([e_0])} \wedge 1\bigg) \, d\mbf{x}_{V\setminus\{v_0\}} \\
&= (2c_{\gamma([e_0])} \wedge 1) \op{Int}_{T\setminus[e_0]}(\mbf{c}_4),
\end{align*}
where $T\setminus[e_0]$ is the colored double tree obtained from $T$ by removing the leaf $v_0$ and its adjacent twin edges $[e_0]$. Iterating this construction, we obtain the upper bound
\[
\op{Int}_T(\mbf{c}_4) \leq \prod_{[e] \in \wtilde{E}} (2c_{\gamma([e])} \wedge 1).
\]

The same reasoning of course shows that
\[
\op{Int}_T(\mbf{c}_4) \geq c_{\gamma([e_0])}\op{Int}_{T\setminus[e_0]}(\mbf{c}_4) \geq \cdots \geq \prod_{[e] \in \wtilde{E}} c_{\gamma([e])},
\]
but we can do much better for small proportions $\mbf{c}_4$. In particular, assume that
\[
\hat{c} = \max_{[e] \in \wtilde{E}} c_{\gamma([e])} < \frac{1}{2}.
\]
Then
\begin{align*}
\op{Int}_T(\mbf{c}_4) &= \int_{[0, 1]^V} \op{Cut}_{T, \mbf{c}_4}(\mbf{x}_V) \, d\mbf{x}_V \geq \int_{[\hat{c}, 1-\hat{c}]^V} \op{Cut}_{T, \mbf{c}_4}(\mbf{x}_V) \, d\mbf{x}_V  \\
&= \int_{[\hat{c}, 1-\hat{c}]^{V\setminus\{v_0\}}} \op{Cut}_{\wtilde{E}\setminus\{[e_0]\}, \mbf{c}_4}(\mbf{x}_{V \setminus \{v_0\}})\bigg(\int_{\hat{c}}^{1-\hat{c}} \mathbbm{1}\{|x_{v_0} - x_{v_1}| \leq c_{\gamma([e_0])}\} \, dx_{v_0}\bigg) \, d\mbf{x}_{V\setminus\{v_0\}} \\
&= \int_{[\hat{c}, 1-\hat{c}]^{V\setminus\{v_0\}}} \op{Cut}_{\wtilde{E}\setminus\{[e_0]\}, \mbf{c}_4}(\mbf{x}_{V \setminus \{v_0\}})\bigg((1-2\hat{c})2c_{\gamma([e_0])}\bigg) \, d\mbf{x}_{V\setminus\{v_0\}} \\
&= \cdots = (1-2\hat{c})^{\#(\wtilde{E})}\prod_{[e] \in \wtilde{E}} 2c_{\gamma([e])}.
\end{align*}

Thus, for $\hat{c} < \frac{1}{2}$, we have the bounds
\[
\frac{(1-2\hat{c})^{\#(\wtilde{E})}\prod_{[e] \in \wtilde{E}} 2c_{\gamma([e])}}{\prod_{[e] \in \wtilde{E}} \, (2c_{\gamma([e])} - c_{\gamma([e])}^2)} \leq \frac{\op{Int}_T(\mbf{c}_4)}{\op{Norm}_T(\mbf{c}_4)} \leq \frac{\prod_{[e] \in \wtilde{E}} 2c_{\gamma([e])}}{\prod_{[e] \in \wtilde{E}} \, (2c_{\gamma([e])} - c_{\gamma([e])}^2)},
\]
which imply that
\begin{equation}\label{eq:4.27}
\lim_{\hat{c} \to 0^+} p_T(\mbf{c}_4) = \lim_{\hat{c} \to 0^+} \frac{\op{Int}_T(\mbf{c}_4)}{\op{Norm}_T(\mbf{c}_4)} = 1.
\end{equation}
We view the limit $\hat{c} \to 0^+$ as approaching the slow growth regime. In view of \eqref{eq:4.27}, we see that the LTD \eqref{eq:4.22} of the proportional growth RBMs behaves accordingly (in particular, we have convergence to the LTD \eqref{eq:4.11} of the slow growth RBMs).

In an easier direction, we can also consider the limit
\[
\underline{c} = \min_{[e] \in \wtilde{E}} c_{\gamma([e])} \to 1^-.
\]
One then clearly has
\begin{equation}\label{eq:4.28}
\lim_{\underline{c} \to 1^-} \op{Cut}_{T, \mbf{c}_4}(\mbf{x}_V) = 1, \qquad \forall \mbf{x}_V \in [0, 1]^V.
\end{equation}
We can push this limit through the integral by dominated convergence to obtain
\begin{equation}\label{eq:4.29}
\lim_{\underline{c} \to 1^-} \op{Int}_T(\mbf{c}_4) = \int_{[0, 1]^V} \lim_{\underline{c} \to 1^-} \op{Cut}_{T, \mbf{c}_4}(\mbf{x}_V) \, d\mbf{x}_V = 1.
\end{equation}
Of course, the same convergence also holds for the normalizations \eqref{eq:4.25},
\begin{equation}\label{eq:4.30}
\lim_{\underline{c} \to 1^-} \op{Norm}_{\wtilde{F}}(\mbf{c}_4) = 1, \qquad \forall \wtilde{F} \subset \wtilde{E},
\end{equation}
and so
\begin{equation}\label{eq:4.31}
\lim_{\underline{c} \to 1^-} p_T(\mbf{c}_4) = \lim_{\underline{c} \to 1^-} \frac{\op{Int}_T(\mbf{c}_4)}{\op{Norm}_T(\mbf{c}_4)} = 1.
\end{equation}

We view the limit $\underline{c} \to 1^-$ as approaching the usual Wigner matrices $\salg{W}_n$, or, more generally, the full proportion RBMs. Again, our limit \eqref{eq:4.31} shows that the LTD \eqref{eq:4.22} behaves accordingly (in particular, we have convergence to the LTD \eqref{eq:3.5} of the Wigner matrices). 

Up to now, our analysis of the integral $\op{Int}_T(\mbf{c}_4)$ essentially follows \cite{BMP91}. We take care to account for possibly different band widths by grouping them in the min $\underline{c}$ or the max $\hat{c}$, but in both cases we indiscriminately send the proportions to a single boundary value $\{0, 1\}$. From this point of view, we fail to perceive any differences in the limits
\begin{equation}\label{eq:4.32}
\lim_{\hat{c} \to 0^+} p_T(\mbf{c}_4) = 1 = \lim_{\underline{c} \to 1^-} p_T(\mbf{c}_4);
\end{equation}
yet, the two cases actually differ quite considerably. To see this, we will need to refine our analysis of $p_T(\mbf{c}_4)$ to consider sending only a subset of the proportions $\mbf{c}_4$ to possibly different boundary values. The results will greatly inform our treatment of the joint LTD of the combined families $\salg{P}_n^{(1)} \cup \salg{O}_n^{(2)} \cup \salg{O}_n^{(3)} \cup \salg{O}_n^{(4)}$.

We start with the simpler case of sending the band width $c_{i_0}$ of a single label $i_0 \in I_4$ in our colored double tree $T$ to $1^-$. We write $T_{i_0} = (V_{i_0}, E_{i_0})$ for the subgraph of $T$ with edge labels in $i_0$. In general, $T_{i_0}$ is a forest of colored double trees (in the single ``color'' $i_0$). We define $\wtilde{T}_{i_0} = (V_{i_0}, \wtilde{E}_{i_0})$ as before. We remove the twin edges $\wtilde{E}_{i_0}$ from $T$ to obtain a forest of colored double trees $T\setminus\wtilde{E}_{i_0}$ (say, with connected components $T_1, \ldots, T_k$). We emphasize that we only remove the edges $\wtilde{E}_{i_0}$; in particular, we keep any resulting isolated vertices. We then have the analogues of \eqref{eq:4.28}-\eqref{eq:4.30}:
\begin{equation}\label{eq:4.33}
\lim_{c_{i_0} \to 1^-} \op{Cut}_{T, \mbf{c}_4}(\mbf{x}_V) = \op{Cut}_{\wtilde{E}\setminus\wtilde{E}_{i_0}, \mbf{c}_4}(\mbf{x}_V) = \prod_{\ell = 1}^k \op{Cut}_{T_\ell, \mbf{c}_4}(\mbf{x}_{V_\ell}), \qquad \forall \mbf{x}_V \in [0, 1]^V,
\end{equation}
\begin{equation}\label{eq:4.34}
\begin{aligned}
\lim_{c_{i_0} \to 1^-} \op{Int}_T(\mbf{c}_4) &= \int_{[0, 1]^V} \lim_{c_{i_0} \to 1^-} \op{Cut}_{T, \mbf{c}_4}(\mbf{x}_V) \, d\mbf{x}_V \\
&= \prod_{\ell = 1}^k \int_{[0, 1]^{V_\ell}} \op{Cut}_{T_\ell, \mbf{c}_4}(\mbf{x}_{V_\ell}) \, d\mbf{x}_{V_\ell} = \prod_{\ell = 1}^k \op{Int}_{T_\ell}(\mbf{c}_4),
\end{aligned}
\end{equation}
and
\begin{equation}\label{eq:4.35}
\lim_{c_{i_0} \to 1^-} \op{Norm}_T(\mbf{c}_4) = \op{Norm}_{\wtilde{E}\setminus\wtilde{E}_{i_0}}(\mbf{c}_4) \lim_{c_{i_0} \to 1^-} \op{Norm}_{\wtilde{E}_{i_0}}(\mbf{c}_4) = \prod_{\ell = 1}^k \op{Norm}_{T_\ell}(\mbf{c}_4).
\end{equation}
It follows that
\begin{equation}\label{eq:4.36}
\lim_{c_{i_0} \to 1^-} p_T(\mbf{c}_4) = \lim_{c_{i_0} \to 1^-} \frac{\op{Int}_T(\mbf{c}_4)}{\op{Norm}_T(\mbf{c}_4)} = \frac{\prod_{\ell = 1}^k \op{Int}_{T_\ell}(\mbf{c}_4)}{\prod_{\ell = 1}^k \op{Norm}_{T_\ell}(\mbf{c}_4)} = \prod_{\ell = 1}^k p_{T_\ell}(\mbf{c}_4).
\end{equation}
Of course, if $T_\ell$ consists of an isolated vertex, then $p_{T_\ell}(\mbf{c}_4) = 1$. One can then effectively discard the isolated vertices of $T \setminus \wtilde{E}_{i_0}$ and just consider the resulting forest of nontrivial colored double trees. We choose to keep these vertices in writing a simple, consistent formula for our limit.

The reader will no doubt be easily convinced of \eqref{eq:4.36}, but we give here some intuition for the sake of comparison later. We imagine each vertex $v$ as a country in a league of allied nations $V$. Each value $x_v \in [0, 1]$ represents a proposed amount of aid to be sent by country $v$ to every other country. To avoid showing favoritism, the same amount of aid $x_v$ is sent to each ally $w \neq v$; however, to ensure goodwill, a country can opt to cap the disparity in the amount of aid they exchange with a given ally. We view these restrictions as coming from the edges $\wtilde{E}$, where an edge $v \sim_{[e]} w$ corresponds to a bound $|x_v - x_w| \leq c_{\gamma([e])}$.

We can then interpret the integral $\op{Int}_T(\mbf{c}_4)$ as the percentage of universally acceptable proposals $\mbf{x}_V \in [0, 1]^V$. Each term in our normalization
\[
\op{Norm}_T(\mbf{c}_4) = \prod_{[e] \in \wtilde{E}} (2c_{\gamma([e])} - c_{\gamma([e])}^2) = \prod_{[e] \in \wtilde{E}} \int_0^1 \int_0^1 \mathbbm{1}\{|x_{\source([e])} - x_{\target([e])}| \leq c_{\gamma([e])}\} \, dx_{\source([e])}dx_{\target([e])}
\]
corresponds to the local situation of a single pair of constrained allies $\{\source([e]), \target([e])\}$. Of course, each such pair must agree to a proposal $\mbf{x}_V$ for it to be universally acceptable, though in general this is not sufficient. We can then think of the ratio
\[
p_T(\mbf{c}_4) = \frac{\op{Int}_T(\mbf{c}_4)}{\op{Norm}_T(\mbf{c}_4)}
\]
as conditioning on the proposals that, at the very least, pass at the local level (though it is possible for $p_T(\mbf{c}^{(4)}) > 1$). In the limit $c_{i_0} \to 1^-$, the twin edges $[e] \in \wtilde{E}_{i_0}$ with label $i_0$ represent negotiations between increasingly amicable nations, insomuch that they no longer care to keep track of the disparity in the aid exchanged between them. Here, we again encounter the notion of a free edge. In this case, the proposal $\mbf{x}_V$ need only to satisfy the constraints coming from the remaining edges $\wtilde{E}\setminus\wtilde{E}_{i_0}$, which explains the limit \eqref{eq:4.36}.

Of course, there is nothing special about only sending one of the band widths $c_{i_0} \to 1^-$. In fact, the same argument clearly applies to any collection of labels $i_0, \ldots, i_j$ in a colored double tree $T$. We state the full result later once we have also considered the behavior of $p_T(\mbf{c}_4)$ for band widths $c_{i_0} \to 0^+$, but first we must introduce some more notation.

For any pair of subsets $W \subset V$ and $\wtilde{F} \subset \wtilde{E}$, we define the conditional expectation
\[
\op{Int}_{\wtilde{F}}(\mbf{c}_4 | W): [0, 1]^W \to [0, 1] 
\]
by
\[
\op{Int}_{\wtilde{F}}(\mbf{c}_4 | W)(\mbf{x}_W) = \int_{[0, 1]^{V \setminus W}} \op{Cut}_{\wtilde{F}, \mbf{c}_4}(\mbf{x}_V)\, d\mbf{x}_{V\setminus W}.
\]
For example, the reader can easily verify that
\[
\int_{[0, 1]^W} \op{Int}_T(\mbf{c}_4 | W)(\mbf{x}_W) \, d\mbf{x}_W = \op{Int}_T(\mbf{c}_4).
\]

As before, we start with a single label $i_0 \in I_4$ in $T$, for which we now consider the limit $c_{i_0} \to 0^+$. To simplify the argument, we first assume that there is a unique pair of twin edges $[e_{i_0}]$ with the label $\gamma([e_{i_0}]) = i_0$. For notational convenience, we write
\[
\{a, b\} =\{\source([e_{i_0}]), \target([e_{i_0}])\}.
\]
We condition on the vertices $\{a, b\}$ to obtain
\begin{align}
\notag p_T(\mbf{c}_4) = \frac{\op{Int}_T(\mbf{c}_4)}{\op{Norm}_T(\mbf{c}_4)} &= \int_{[0, 1]^V} \frac{\op{Cut}_{T, \mbf{c}_4}(\mbf{x}_V)}{\op{Norm}_T(\mbf{c}_4)} \, d\mbf{x}_V \\
\notag &= \int_{[0, 1]^2} \frac{\op{Int}_{\wtilde{E}\setminus\{[e_{i_0}]\}}(\mbf{c}_4 | \{x_a, x_b\})(x_a, x_b)}{\op{Norm}_{\wtilde{E}\setminus\{[e_{i_0}]\}}(\mbf{c}_4)} \bigg(\frac{\mathbbm{1}\{|x_a - x_b| \leq c_{i_0}\}}{2c_{i_0} - c_{i_0}^2} \, dx_a dx_b\bigg) \\
&= \int_{[0, 1]^2} f(x_a, x_b)\, \mu_{c_{i_0}}(dx_a, dx_b), \label{eq:4.37}
\end{align}
where
\[
f(x_a, x_b) = \frac{\op{Int}_{\wtilde{E}\setminus\{[e_{i_0}]\}}(\mbf{c}_4 | \{x_a, x_b\})(x_a, x_b)}{\op{Norm}_{\wtilde{E}\setminus\{[e_{i_0}]\}}(\mbf{c}_4)} 
\]
is a bounded continuous function that does not depend on $c_{i_0}$ and
\[
\mu_{c_{i_0}}(dx_a, dx_b)
\]
is the uniform (probability) measure on the banded strip in unit square $[0, 1]^2$ defined by $|x_a-x_b| \leq c_{i_0}$. In the limit, we have the weak convergence
\[
\mu_{c_{i_0}} \wto \mu_{\Delta} \quad \text{as} \quad  c_{i_0} \to 0^+,
\]
where $\mu_{\Delta}$ is the uniform measure on the diagonal $\{(x, x) : x \in [0 ,1]\} \subset [0, 1]^2$. In particular, this implies that
\begin{align*}
\lim_{c_{i_0} \to 0^+} p_T(\mbf{c}_4) &= \lim_{c_{i_0} \to 0^+} \int_{[0, 1]^2} f(x_a, x_b) \, \mu_{c_{i_0}}(dx_a, dx_b) \\
&= \int_{[0, 1]^2} f(x_a, x_b) \, \mu_\Delta(dx_a, dx_b) = \int_0^1 f(x, x)\, dx = p_{T/[e_{i_0}]}(\mbf{c}_4),
\end{align*}
where $T/[e_{i_0}]$ is the colored double tree obtained from $T$ by contracting the twin edges $[e_{i_0}]$ (i.e., we remove the edges $[e_{i_0}]$ and merge the vertices $\{a, b\}$). We note the contrast to the situation \eqref{eq:4.36} in the limit $c_{i_0} \to 1^-$, where we remove the edges but do not otherwise modify the vertices.

We must take care if the label $i_0$ appears in more than one set of twin edges. In any case, we can always identify the subgraph $T_{i_0}$ of $T$ with edge labels in $i_0$. In general, $T_{i_0} = (V_{i_0}, E_{i_0})$ is a forest $T_1 \sqcup \cdots \sqcup T_k$ of colored double trees $T_\ell = (V_\ell, E_\ell)$ (in the single color $i_0$). Conditioning on the vertices $V_{i_0} = V_1 \sqcup \cdots \sqcup V_k$ of $T_{i_0}$, we obtain
\begin{equation}\label{eq:4.38}
p_T(\mbf{c}_4) = \int_{\bigtimes_{\ell=1}^k [0,1]^{V_\ell}} f(\mbf{x}_{V_1}, \ldots, \mbf{x}_{V_k}) \prod_{\ell=1}^k \bigg(\frac{\op{Cut}_{T_\ell, c_{i_0}}(\mbf{x}_{V_\ell})}{\op{Norm}_{T_\ell}(c_{i_0})}\, d\mbf{x}_{V_\ell}\bigg)
\end{equation}
where
\[
f(\mbf{x}_{V_1}, \ldots, \mbf{x}_{V_k}) = \frac{\op{Int}_{\wtilde{E} \setminus \wtilde{E}_{i_0}}(\mbf{c}_4 | V_{i_0})(\mbf{x}_{V_1}, \ldots, \mbf{x}_{V_k})}{\op{Norm}_{\wtilde{E} \setminus \wtilde{E}_{i_0}}(\mbf{c}_4)}
\]
is again a bounded continuous function that does not depend on $c_{i_0}$. In this case, we cannot immediately write \eqref{eq:4.38} in terms of probability measures
\[
\mu_{c_{i_0}}^{(\ell)}(d\mbf{x}_{V_\ell}) = \frac{\op{Cut}_{T_\ell, c_{i_0}}(\mbf{x}_{V_\ell})}{\op{Norm}_{T_\ell}(c_{i_0})}\, d\mbf{x}_{V_\ell}
\]
as we did in \eqref{eq:4.37} since, in general,
\[
\op{Int}_{T_\ell}(c_{i_0}) = \int_{[0,1]^{V_\ell}} \op{Cut}_{T_\ell, c_{i_0}}(\mbf{x}_{V_\ell}) \, d\mbf{x}_{V_\ell} \neq (2c_{i_0} - c_{i_0}^2)^{\#(\wtilde{E}_\ell)} = \op{Norm}_{T_\ell}(c_{i_0});
\]
however, our work \eqref{eq:4.27} from before shows that
\[
\lim_{c_{i_0} \to 0^+} \frac{\op{Int}_{T_\ell}(c_{i_0})}{\op{Norm}_{T_\ell}(c_{i_0})} = 1.
\]
Thus, we can instead write
\begin{equation}\label{eq:4.39}
p_T(\mbf{c}_4) = \delta(c_{i_0}) \int_{\bigtimes_{\ell=1}^k [0,1]^{V_\ell}} f(\mbf{x}_{V_1}, \ldots, \mbf{x}_{V_k}) \, \bigg(\otimes_{\ell=1}^k \mu_{c_{i_0}}^{(\ell)}(d\mbf{x}_{V_\ell})\bigg),
\end{equation}
where $\delta(c_{i_0})$ is a real number depending on $c_{i_0}$ such that
\[
\lim_{c_{i_0} \to 0^+} \delta(c_{i_0}) = 1
\]
and $\mu_{c_{i_0}}^{(\ell)}$ is the uniform measure on the banded region $R_\ell \subset [0, 1]^{V_\ell}$ defined by the constraints 
\[
|x_{\source([e])} - x_{\target([e])}| \leq c_{i_0}, \qquad \forall [e] \in \wtilde{E}_\ell.
\]
As before, we note that
\[
\lim_{c_{i_0} \to 0^+} \mu_{c_{i_0}}^{(\ell)} = \mu_\Delta^{(\ell)},
\]
where $\mu_\Delta^{(\ell)}$ is the uniform measure on the diagonal $\{(x, \ldots, x) : x \in [0, 1]\} \subset [0, 1]^{V_\ell}$. It follows that
\begin{align*}
\lim_{c_{i_0} \to 0^+} p_T(\mbf{c}_4) &= \lim_{c_{i_0} \to 0^+} \int_{\bigtimes_{\ell=1}^k [0,1]^{V_\ell}} f(\mbf{x}_{V_1}, \ldots, \mbf{x}_{V_k}) \prod_{\ell=1}^k \bigg(\frac{\op{Cut}_{T_\ell, c_{i_0}}(\mbf{x}_{V_\ell})}{\op{Norm}_{T_\ell}(c_{i_0})}\, d\mbf{x}_{V_\ell}\bigg) \\
&=  \int_{\bigtimes_{\ell=1}^k [0,1]^{V_\ell}} f(\mbf{x}_{V_1}, \ldots, \mbf{x}_{V_k}) \bigg(\otimes_{\ell=1}^k \mu_\Delta^{(\ell)}(d\mbf{x}_{V_\ell})\bigg) \\
&= \int_{[0, 1]^k} f(x_1, \ldots, x_1, \ldots, x_k, \ldots, x_k)\, dx_1\cdots dx_k = p_{T/T_{i_0}}(\mbf{c}_4),
\end{align*}
where $T/T_{i_0}$ is the colored double tree obtained from $T$ by contracting the edges of $T_{i_0}$ (i.e., for each $\ell \in [k]$, we remove the edges $\wtilde{E}_\ell$ and merge the vertices $V_\ell$ into a single vertex).

We can easily adapt our argument to accommodate multiple band widths $c_{i_0}, \ldots, c_{i_j}$ in the limit $\max(c_{i_0}, \ldots, c_{i_j}) \to 0^+$. In this case, we replace $T_{i_0}$ with $T_{\mbf{i}}$, the subgraph of $T$ with edge labels in $\mbf{i} = \{i_0, \ldots, i_j\}$; otherwise, the same argument goes through just as well.

Returning to our intuition from before, we think of the limit $c_{i_0} \to 0^+$ as representing negotiations between increasing acrimonious nations, insomuch that they become completely intransigent and insist on absolute parity in the aid exchanged between them. Negotiations along such an edge $\gamma([e]) = i_0$ then stall a proposal $\mbf{x}_V$ unless $|x_{\source([e])} - x_{\target([e])}| = 0$. In this case, we can effectively consider the two countries $\source([e])$ and $\target([e])$ as a single entity sending the aid $x_{\source([e])} = x_{\target([e])}$ to the remaining allies. Our normalization then allows us to recast the problem as the proportion of acceptable proposals in this new world order. 

\phantomsection\label{fig24_forest}
\begin{center}
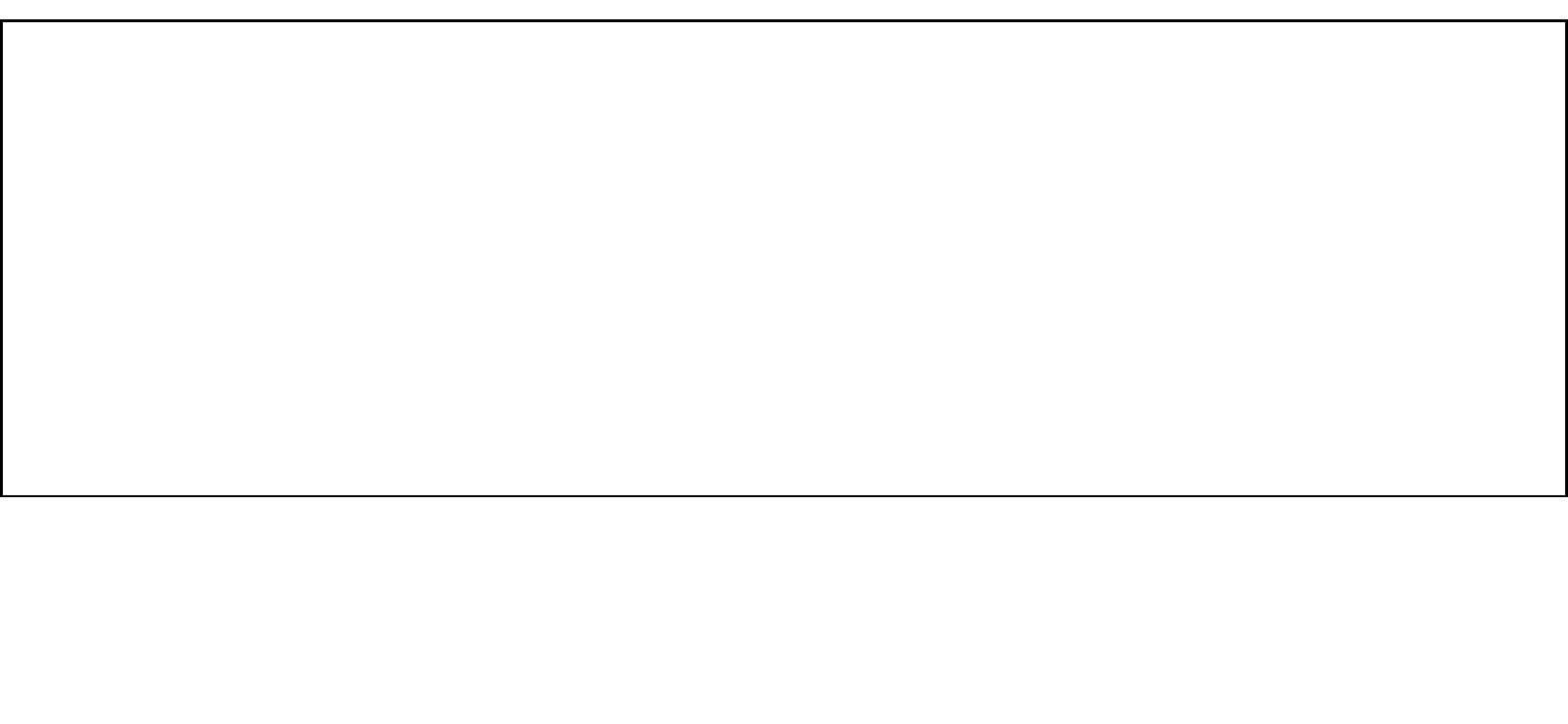
\end{center}

At this point, we see how the limits \eqref{eq:4.32} come about in different ways: in the limit $\underline{c} \to 0^+$, we contract all of the edges, leaving a single isolated vertex; in the limit $\hat{c} \to 1^-$, we remove all of the edges, leaving $\#(V)$ isolated vertices. Finally, the result for a collection of band widths sent to possibly different boundary values should come as no surprise. We combine our work in the two previous cases with care to account for parts moving simultaneously in different directions.

To begin, let $J_0$ (resp., $J_1$) denote the collection of labels in our colored double tree $T$ whose band widths are to be sent to $0^+$ (resp., $1^-$). We define
\begin{alignat*}{2}
\mbf{c}_0 &= (c_i)_{i \in J_0}, \qquad \mbf{c}_1 &&= (c_i)_{i \in J_1}; \\
c_0 &= \max_{i \in J_0} c_i, \qquad \ c_1 &&= \min_{i \in J_1} c_i,
\end{alignat*}
and write $\mbf{c}_2 = \mbf{c}_4 \setminus (\mbf{c}_0 \cup \mbf{c}_1)$ for the remaining band widths. We are then interested in the limit
\[
\lim_{(c_0, c_1) \to (0^+, 1^-)} p_T(\mbf{c}_4).
\]
We decompose our graph as before. We write $T_{0^+}$ for the subgraph of $T$ with edge labels in $J_0$. In general, $T_{0^+} = (V_{0^+}, E_{0^+})$ is a forest $T_{0^+} = T_1^+ \sqcup \cdots \sqcup T_k^+$ of colored double trees $T_\ell^+ = (V_\ell^+, E_\ell^+)$ (except now possibly with multiple colors). Similarly, we write $T_{1^-} = (V_{1^-}, E_{1^-})$ for the subgraph of $T$ with edge labels in $J_1$. Finally, we write $E_2 = E \setminus (E_0 \cup E_1)$ for the remaining edges.

Conditioning on the vertices $V_{0^+} = V_1^+ \sqcup \cdots \sqcup V_k^+$ of $T_{0^+}$, we obtain the analogue of \eqref{eq:4.39},
\begin{align*}
p_T(\mbf{c}_4) = \delta(\mbf{c}_0) \int_{\bigtimes_{\ell=1}^k [0,1]^{V_\ell^+}} f_{\mbf{c}_1}(\mbf{x}_{V_1^+}, \ldots, \mbf{x}_{V_k^+}) \bigg(\otimes_{\ell=1}^k \mu_{\mbf{c}_0}^{(\ell)}(d\mbf{x}_{V_\ell^+})\bigg),
\end{align*}
where $\delta(\mbf{c}_0)$ is a real number depending on $\mbf{c}_0$ such that
\[
\lim_{c_0 \to 0^+} \delta(\mbf{c}_0) = 1
\]
and $\mu_{\mbf{c}_0}^{(\ell)}$ is the uniform measure on the banded region $R_\ell$ in $[0, 1]^{V_\ell^+}$ defined by the constraints
\[
|x_{\source([e])} - x_{\target([e])}| \leq c_{\gamma([e])} \in \mbf{c}_0, \qquad \forall [e] \in \wtilde{E}_\ell^+.
\]
Despite considering multiple band widths $\mbf{c}_0$, we still have the weak convergence
\[
\lim_{c_0 \to 0^+} \mu_{\mbf{c}_0}^{(\ell)} = \mu_\Delta^{(\ell)}.
\]

As before,
\begin{align*}
f_{\mbf{c}_1}(\mbf{x}_{V_1^+}, \ldots, \mbf{x}_{V_\ell^+}) &=  \frac{\op{Int}_{\wtilde{E}\setminus\wtilde{E}_{0^+}}(\mbf{c}_4 | V_{0^+})(\mbf{x}_{V_1^+}, \ldots, \mbf{x}_{V_k^+})}{\op{Norm}_{\wtilde{E} \setminus \wtilde{E}_{0^+}}(\mbf{c}_4)} \\
&= \int_{[0, 1]^{V\setminus V_{0^+}}} \frac{\op{Cut}_{\wtilde{E}_{1^-}, \mbf{c}_1}(\mbf{x}_V)}{\op{Norm}_{\wtilde{E}_{1^-}}(\mbf{c}_1)} \frac{\op{Cut}_{\wtilde{E}_2, \mbf{c}_2}(\mbf{x}_V)}{\op{Norm}_{\wtilde{E}_2}(\mbf{c}_2)} \, d\mbf{x}_{V \setminus V_{0^+}}
\end{align*}
is a bounded continuous function that does not depend on $\mbf{c}_0$; however, $f_{\mbf{c}_1}$ does depend on $\mbf{c}_1$. In particular, the function
\[
\op{Cut}_{\wtilde{E}_{1^-}, \mbf{c}_1} : [0, 1]^V \to [0, 1] 
\]
is monotonic in $\mbf{c}_1$ with
\[
\lim_{c_1 \to 1^-} \op{Cut}_{\wtilde{E}_{1^-}, \mbf{c}_1}(\mbf{x}_V) = 1, \qquad \forall \mbf{x}_V \in [0, 1]^V.
\]
Since
\[
\lim_{c_1 \to 1^-} \op{Norm}_{\wtilde{E}_{1^-}}(\mbf{c}_1) = 1,
\]
it follows that
\[
f(\mbf{x}_{V_1^+}, \ldots, \mbf{x}_{V_k^+}) = \lim_{c_1 \to 1^-} f_{\mbf{c}_1}(\mbf{x}_{V_1^+}, \ldots, \mbf{x}_{V_k^+}) = \int_{[0, 1]^{V \setminus V_0^+}} \frac{\op{Cut}_{\wtilde{E}_2, \mbf{c}_2}(\mbf{x}_V)}{\op{Norm}_{\wtilde{E}_2}(\mbf{c}_2)} \, d\mbf{x}_{V \setminus V_{0^+}}.
\]
The monotonicity of $\op{Cut}_{\wtilde{E}_{1^-}, \mbf{c}_1}$ in the proportions $\mbf{c}_1$ then allows us to conclude that
\begin{align*}
\lim_{(c_0, c_1) \to (0^+, 1^-)} p_T(\mbf{c}_4) &= \lim_{(c_0, c_1) \to (0^+, 1^-)} \int_{\bigtimes_{\ell=1}^k [0,1]^{V_\ell^+}} f_{\mbf{c}_1}(\mbf{x}_{V_1^+}, \ldots, \mbf{x}_{V_k^+}) \prod_{\ell=1}^k \bigg(\frac{\op{Cut}_{\wtilde{E}_\ell^+, \mbf{c}_0}(\mbf{x}_{V_\ell^+})}{\op{Norm}_{T_\ell^+}(\mbf{c}_0)}\, d\mbf{x}_{V_\ell^+}\bigg) \\
&= \int_{\bigtimes_{\ell=1}^k [0,1]^{V_\ell^+}} f(\mbf{x}_{V_1^+}, \ldots, \mbf{x}_{V_k^+}) \bigg(\otimes_{\ell=1}^k \mu_\Delta^{(\ell)}(d\mbf{x}_{V_\ell^+})\bigg) \\
&= \int_{[0, 1]^k} f(x_1, \ldots, x_1, \ldots, x_k, \ldots, x_k) \, dx_1\cdots dx_k = p_F(\mbf{c}_2) = \prod_{r=1}^s p_{T_r}(\mbf{c}_2),
\end{align*}
where $F$ is the forest of colored double trees $F = T_1 \sqcup \cdots \sqcup T_s$ obtained from $T$ by removing the edges $E_{1^-}$ and contracting the edges $E_{0^+}$.

Our treatment of $p_T(\mbf{c}_4)$ suggests the following form for the joint LTD of the matrices $\salg{O}_n^{(2)} \cup \salg{O}_n^{(3)} \cup \salg{O}_n^{(4)}$. We leave the by-now familiar details of the proof to the diligent reader.

\begin{thm}\label{thm4.3.3}
For any test graph $T$ in $\mbf{x}_2 \cup \mbf{x}_3 \cup \mbf{x}_4 = (x_i)_{i \in I_2 \cup I_3 \cup I_4}$,
\begin{equation}\label{eq:4.40}
\lim_{n \to \infty} \tau^0\big[T(\salg{O}_n^{(2)} \cup \salg{O}_n^{(3)} \cup \salg{O}_n^{(4)})\big] = 
\begin{cases}
p_F(\mbf{c}_4)\prod_{i \in I} \beta_i^{c_i(T)} & \text{if $T$ is a colored double tree,} \\
\hfil 0 & \text{otherwise},
\end{cases}
\end{equation}
where $F = T_1 \sqcup \cdots \sqcup T_s$ is the forest of colored double trees obtained from $T$ by contracting the edges with labels in $I_2$ and removing the edges with labels in $I_3$ and
\begin{equation}\label{eq:4.41}
p_F(\mbf{c}_4) = \prod_{r=1}^s p_{T_r}(\mbf{c}_4).
\end{equation}
\end{thm}

\begin{cor}\label{cor4.3.4}
The full proportion RBMs $\salg{O}_n^{(3)}$ and the proper proportion RBMs $\salg{O}_n^{(4)}$ are asymptotically traffic independent, as are the full proportion RBMs $\salg{O}_n^{(3)}$ and the slow growth RBMs $\salg{O}_n^{(2)}$. The slow growth RBMs $\salg{O}_n^{(2)}$ and the proper proportion RBMs $\salg{O}_n^{(4)}$ are not asymptotically traffic independent, nor are independent proper proportion RBMs $\salg{O}_n^{(4)}  = (\mbf{\Theta}_n^{(i)})_{i \in I_4}$. 
\end{cor}
\begin{proof}
The statements about asymptotic traffic independence follow from the calculation of $F$ from our colored double tree $T$ (we simply remove the edges with labels in $I_3$) and the multiplicativity of \eqref{eq:4.41}. For the statements about non-asymptotic traffic independence, we give a simple counterexample, namely, for $i_2 \in I_2$ and $i_4, j_4 \in I_4$ with $0 < c_{i_4} \leq c_{j_4} < 1$,

\begin{center}
\begingroup%
  \makeatletter%
  \providecommand\color[2][]{%
    \errmessage{(Inkscape) Color is used for the text in Inkscape, but the package 'color.sty' is not loaded}%
    \renewcommand\color[2][]{}%
  }%
  \providecommand\transparent[1]{%
    \errmessage{(Inkscape) Transparency is used (non-zero) for the text in Inkscape, but the package 'transparent.sty' is not loaded}%
    \renewcommand\transparent[1]{}%
  }%
  \providecommand\rotatebox[2]{#2}%
  \ifx\svgwidth\undefined%
    \setlength{\unitlength}{468bp}%
    \ifx\svgscale\undefined%
      \relax%
    \else%
      \setlength{\unitlength}{\unitlength * \real{\svgscale}}%
    \fi%
  \else%
    \setlength{\unitlength}{\svgwidth}%
  \fi%
  \global\let\svgwidth\undefined%
  \global\let\svgscale\undefined%
  \makeatother%
  \begin{picture}(1,0.14576923)%
    \put(0.08673121,1.44600133){\color[rgb]{0,0,0}\makebox(0,0)[lt]{\begin{minipage}{0.11106843\unitlength}\raggedright \end{minipage}}}%
    \put(0,0){\includegraphics[width=\unitlength,page=1]{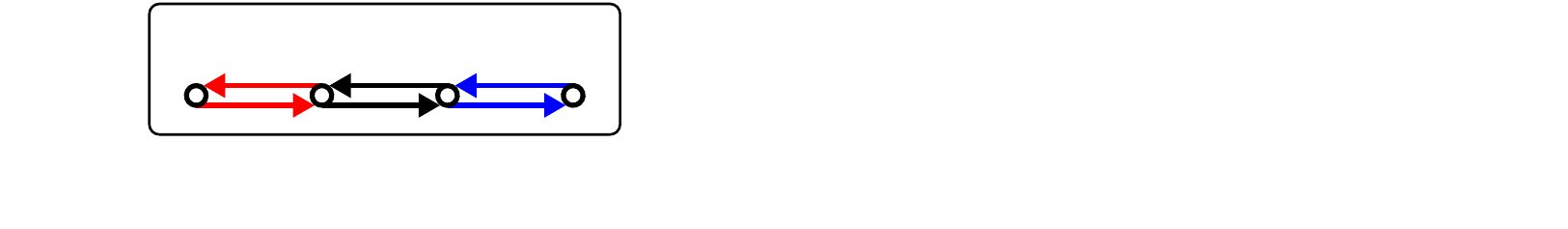}}%
    \put(0.14932674,0.10563498){\color[rgb]{0,0,0}\makebox(0,0)[lb]{\smash{$\mbf{\Theta}_n^{(i_4)}$}}}%
    \put(0.22945492,0.10563498){\color[rgb]{0,0,0}\makebox(0,0)[lb]{\smash{$\mbf{\Theta}_n^{(i_2)}$}}}%
    \put(0.30958154,0.10563498){\color[rgb]{0,0,0}\makebox(0,0)[lb]{\smash{$\mbf{\Theta}_n^{(j_4)}$}}}%
    \put(-0.00199538,0.09541175){\color[rgb]{0,0,0}\makebox(0,0)[lb]{\smash{$\displaystyle \lim_{n \to \infty} \tau^0\big[$}}}%
    \put(0,0){\includegraphics[width=\unitlength,page=2]{fig25_non.pdf}}%
    \put(0.40703505,0.09541217){\color[rgb]{0,0,0}\makebox(0,0)[lb]{\smash{$\big] = \displaystyle \lim_{n \to \infty} \tau^0\big[$}}}%
    \put(0.77993568,0.09474318){\color[rgb]{0,0,0}\makebox(0,0)[lb]{\smash{$\big]$}}}%
    \put(0.60255687,0.10563498){\color[rgb]{0,0,0}\makebox(0,0)[lb]{\smash{$\mbf{\Theta}_n^{(i_4)}$}}}%
    \put(0.68268508,0.10563498){\color[rgb]{0,0,0}\makebox(0,0)[lb]{\smash{$\mbf{\Theta}_n^{(j_4)}$}}}%
    \put(0.40703505,0.0056692){\color[rgb]{0,0,0}\makebox(0,0)[lb]{\smash{$\phantom{\big]} = \displaystyle \lim_{n \to \infty} \tau^0\big[S(\mbf{\Theta}_n^{(i_4)}, \mbf{\Theta}_n^{(i_4)}, \mbf{\Theta}_n^{(j_4)}, \mbf{\Theta}_n^{(j_4)})\big] = p_S(\{c_{i_4}, c_{j_4}\}),$}}}%
  \end{picture}%
\endgroup%

\end{center}

\noindent where \phantomsection{}
\begin{equation*}\label{eq:non}
p_S(\{c_{i_4}, c_{j_4}\}) =
\begin{dcases}
\hfil \frac{-\frac{1}{3}c_{i_4}^3 - c_{i_4}^2c_{j_4} - 2c_{i_4}c_{j_4}^2 + 4c_{i_4}c_{j_4}}{(2c_{i_4} - c_{i_4}^2)(2c_{j_4} - c_{j_4}^2)} & \text{if $c_{i_4} \leq c_{j_4} \leq \frac{1}{2}$,} \\
\frac{\frac{1}{3}c_{j_4}^3 - c_{i_4}c_{j_4}^2 - c_{i_4}^2 - c_{j_4}^2 + 2c_{i_4}c_{j_4} + c_{i_4} + c_{j_4} - \frac{1}{3}}{(2c_{i_4} - c_{i_4}^2)(2c_{j_4} - c_{j_4}^2)} & \text{if $1 - c_{j_4} \leq c_{i_4} \leq \frac{1}{2}$,} \\
\hfil \frac{-\frac{1}{3}c_{i_4}^3 - c_{i_4}^2c_{j_4} - 2c_{i_4}c_{j_4}^2 + 4c_{i_4}c_{j_4}}{(2c_{i_4} - c_{i_4}^2)(2c_{j_4} - c_{j_4}^2)} & \text{if $c_{i_4} \leq 1 - c_{j_4} \leq \frac{1}{2}$,} \\
\frac{\frac{1}{3}c_{j_4}^3 - c_{i_4}c_{j_4}^2 - c_{i_4}^2 - c_{j_4}^2 + 2c_{i_4}c_{j_4} + c_{i_4} + c_{j_4} - \frac{1}{3}}{(2c_{i_4} - c_{i_4}^2)(2c_{j_4} - c_{j_4}^2)} & \text{if $\frac{1}{2} \leq c_{i_4} \leq c_{j_4}$}.
\end{dcases}
\end{equation*}
In particular,

\begin{center}
\begingroup%
  \makeatletter%
  \providecommand\color[2][]{%
    \errmessage{(Inkscape) Color is used for the text in Inkscape, but the package 'color.sty' is not loaded}%
    \renewcommand\color[2][]{}%
  }%
  \providecommand\transparent[1]{%
    \errmessage{(Inkscape) Transparency is used (non-zero) for the text in Inkscape, but the package 'transparent.sty' is not loaded}%
    \renewcommand\transparent[1]{}%
  }%
  \providecommand\rotatebox[2]{#2}%
  \ifx\svgwidth\undefined%
    \setlength{\unitlength}{495bp}%
    \ifx\svgscale\undefined%
      \relax%
    \else%
      \setlength{\unitlength}{\unitlength * \real{\svgscale}}%
    \fi%
  \else%
    \setlength{\unitlength}{\svgwidth}%
  \fi%
  \global\let\svgwidth\undefined%
  \global\let\svgscale\undefined%
  \makeatother%
  \begin{picture}(1,0.24545455)%
    \put(0.08200042,1.28531038){\color[rgb]{0,0,0}\makebox(0,0)[lt]{\begin{minipage}{0.10501016\unitlength}\raggedright \end{minipage}}}%
    \put(-0.00188654,0.2265721){\color[rgb]{0,0,0}\makebox(0,0)[lb]{\smash{$p_S(\{c_{i_4}, c_{j_4}\}) \neq 1$}}}%
    \put(0,0){\includegraphics[width=\unitlength,page=1]{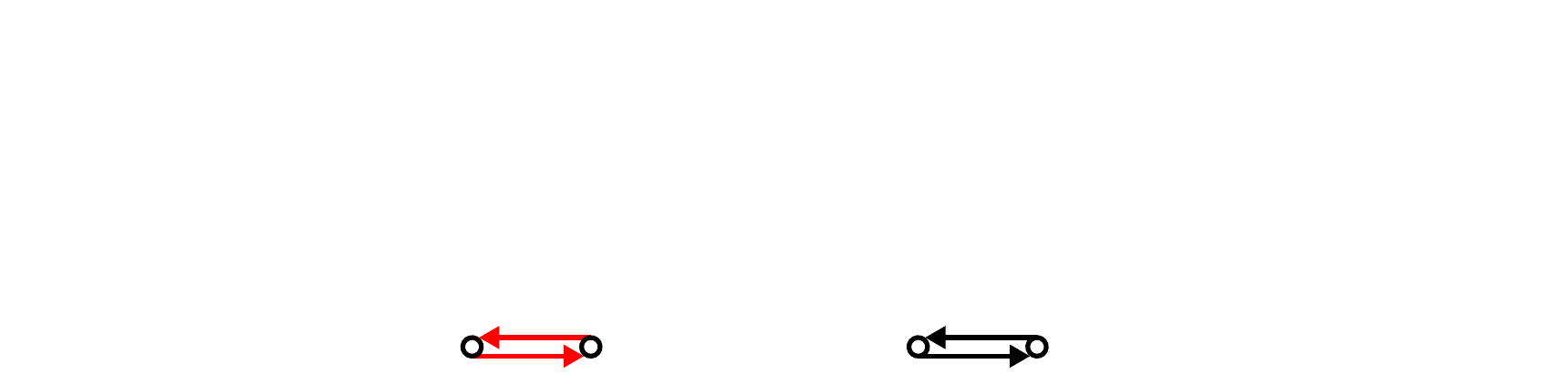}}%
    \put(0.41587865,0.03414646){\color[rgb]{0,0,0}\makebox(0,0)[lb]{\smash{$\displaystyle\big]\bigg) \bigg(\lim_{n \to \infty} \tau^0\big[$}}}%
    \put(0,0){\includegraphics[width=\unitlength,page=2]{fig26_non.pdf}}%
    \put(0.13402256,0.03414764){\color[rgb]{0,0,0}\makebox(0,0)[lb]{\smash{$= \displaystyle\bigg(\lim_{n \to \infty} \tau^0\big[$}}}%
    \put(0,0){\includegraphics[width=\unitlength,page=3]{fig26_non.pdf}}%
    \put(0.70046196,0.03414665){\color[rgb]{0,0,0}\makebox(0,0)[lb]{\smash{$\displaystyle\big]\bigg) \bigg(\lim_{n \to \infty} \tau^0\big[$}}}%
    \put(0.98504537,0.03351376){\color[rgb]{0,0,0}\makebox(0,0)[lb]{\smash{$\big]\bigg)$,}}}%
    \put(0,0){\includegraphics[width=\unitlength,page=4]{fig26_non.pdf}}%
    \put(0.4158787,0.15687394){\color[rgb]{0,0,0}\makebox(0,0)[lb]{\smash{$\displaystyle\big]\bigg) \bigg(\lim_{n \to \infty} \tau^0\big[$}}}%
    \put(0.70046196,0.15624105){\color[rgb]{0,0,0}\makebox(0,0)[lb]{\smash{$\big]\bigg)$}}}%
    \put(0,0){\includegraphics[width=\unitlength,page=5]{fig26_non.pdf}}%
    \put(0.13402256,0.15687512){\color[rgb]{0,0,0}\makebox(0,0)[lb]{\smash{$= \displaystyle\bigg(\lim_{n \to \infty} \tau^0\big[$}}}%
    \put(0.32396497,0.16653421){\color[rgb]{0,0,0}\makebox(0,0)[lb]{\smash{$\mbf{\Theta}_n^{(i_4)}$}}}%
    \put(0.60854833,0.16653421){\color[rgb]{0,0,0}\makebox(0,0)[lb]{\smash{$\mbf{\Theta}_n^{(j_4)}$}}}%
    \put(0.32396497,0.04380693){\color[rgb]{0,0,0}\makebox(0,0)[lb]{\smash{$\mbf{\Theta}_n^{(i_4)}$}}}%
    \put(0.89309381,0.04380693){\color[rgb]{0,0,0}\makebox(0,0)[lb]{\smash{$\mbf{\Theta}_n^{(j_4)}$}}}%
    \put(0.60854833,0.04380693){\color[rgb]{0,0,0}\makebox(0,0)[lb]{\smash{$\mbf{\Theta}_n^{(i_2)}$}}}%
  \end{picture}%
\endgroup%

\end{center}

\noindent which covers both statements.
\end{proof}

\begin{rem}\label{rem4.3.5}
One can also deduce the lack of asymptotic traffic independence for independent proper proportion RBMs $\salg{O}_n^{(4)} = (\mbf{\Theta}_n^{(i)})_{i \in I_4}$ of the same proportion $c_i \equiv c$ from the traffic CLT. Indeed, if the family $\salg{O}_n^{(4)}$ were asymptotically traffic independent, then we could adapt the argument from Section \hyperref[sec3.2]{3.2} to identify the LSD of a single proper proportion RBM $\mbf{\Theta}_n^{(i)}$ as a free convolution $\wsc(0,p^2) \boxplus \gn(0, q^2)$ of the form $p^2 + q^2 = 1$. On the contrary, the actual LSD is known to be non-semicircular and of bounded support \cite{BMP91}, which simultaneously implies that both $q^2 \neq 0$ and $q^2 = 0$ respectively.
\end{rem}

The careful reader will notice that the periodic RBMs $\salg{P}_n^{(1)}$ are conspicuously absent in Theorem \hyperref[thm4.3.3]{4.3.3}. Again, we have the familiar obstruction: without any further assumptions on the band widths $\mbf{b}_n^{(1)} = (b_n^{(i)})_{i \in I_1}$, their fluctuations could preclude the existence of a joint LTD. For example, if a periodic band width $b_n^{(i)}$ has a subsequence of slow growth and another subsequence of proportional growth, then the LTDs along these two subsequences will be different. If we assume that the band widths $\mbf{b}_n^{(1)} = (b_n^{(i)})_{I_1'} \cup (b_n^{(i)})_{i \in I_1''}$ fall into one of these two regimes, slow growth or proportional growth respectively, then we can prove the extension of Theorem \hyperref[thm4.3.3]{4.3.3} to $\salg{P}_n^{(1)} \cup \salg{O}_n^{(2)} \cup \salg{O}_n^{(3)} \cup \salg{O}_n^{(4)}$. In this case, the LTD essentially follows \eqref{eq:4.40} except that we must now also contract the edges with labels in $I_1'$ and remove the edges with labels in $I_1''$ (regardless of the limiting proportions $\lim_{n \to \infty} \frac{b_n^{(i)}}{n}$ for $i \in I_1''$).

The contraction of the edges with labels in $I_1'$ should come as no surprise given Section \hyperref[sec4.2]{4.2}, where we saw that periodizing a slow growth RBM does little to affect the calculations. Just as we contract the labels in $I_2$, we should then also expect to contract the labels in $I_1'$. On the other hand, as we noted before, periodizing a proportional growth RBM changes the situation entirely. Formally, we need to work with the periodic absolute value
\[
|x|_p = \min(x, 1-x), \qquad \forall x \in [0, 1]
\]
in our integral to account for the edges with labels in $I_1''$; however, the analogue of \eqref{eq:4.26} does not depend on where we measure the diameter of our cross section,
\[
g(x_{v_1}) = \int_0^1 \indc{|x_{v_0} - x_{v_1}|_p \leq c_{\gamma([e_0]}}\, dx_{v_0} = 2c_{\gamma([e_0])}, \qquad \forall x_{v_1} \in [0, 1]. 
\]
This balances out perfectly with the normalization of the periodic RBMs,
\[
\mbf{\Lambda}_n^{(\gamma([e_0]))} = \mbf{\Upsilon}_n^{(\gamma([e_0]))} \circ \mbf{\Gamma}_n^{(\gamma([e_0]))} = \frac{1}{\sqrt{2b_n^{(\gamma([e_0]))}}} \mbf{\Gamma}^{(\gamma([e_0]))},
\]
so we can integrate out the vertices that are only adjacent to edges with labels in $I_1''$ without changing the value of the integral. This of course corresponds to simply removing the edges with labels in $I_1''$ when calculating $p_F(\mbf{c}_4)$. In this case, we then know that the periodic RBMs $\salg{P}_n^{(1'')} = (\mbf{\Lambda}_n^{(i)})_{i \in I_1''}$ and the proportional growth RBMs $\salg{O}_n^{(4)}$ are asymptotically traffic independent, whereas the periodic RBMs $\salg{P}_n^{(1')} = (\mbf{\Lambda}_n^{(i)})_{i \in I_1'}$ and the proportional growth RBMs $\salg{O}_n^{(4)}$ are not.

For general $\beta_i \in \C$, we must again settle for convergence in joint distribution.
\begin{thm}\label{thm4.3.6}
Assume that the band widths $(b_n^{(i)})_{i \in I_1}$ of the periodic RBMs fall into one of two categories $I_1 = I_1' \cup I_1''$ as before. For general $\beta_i \in \C$, the families $\salg{P}_n^{(1)} \cup \salg{O}_n^{(2)} \cup \salg{O}_n^{(3)} \cup \salg{O}_n^{(4)}$ converge in joint distribution to a family
\[
\mbf{a} = (a_i)_{i \in I} = (a_i)_{i \in I_1'} \cup (a_i)_{i \in I_1''} \cup (a_i)_{i \in I_2} \cup (a_i)_{i \in I_3} \cup (a_i)_{i \in I_4} = \mbf{a}_{1'} \cup \mbf{a}_{1''} \cup \mbf{a}_2 \cup \mbf{a}_3 \cup \mbf{a}_4.
\]
The family $\mbf{a}_{1'} \cup \mbf{a}_{1''} \cup \mbf{a}_2 \cup \mbf{a}_3$ is a semicircular system; the families $\mbf{a}_{1''}$, $\mbf{a}_3$, and $\mbf{a}_4$ are free; the families $\mbf{a}_2$ and $\mbf{a}_4$ are not free, nor are the families $\mbf{a}_{1'}$ and $\mbf{a}_4$; finally, the family $\mbf{a}_4 = (a_i)_{i \in I_4}$ is not free.
\end{thm}
\begin{proof}
The convergence in joint distribution follows from a modified version of the criteria \eqref{eq:3.18} in Remark \hyperref[rem3.1.3]{3.1.3}. In particular, we do not actually need to know the value of
\[
\lim_{n \to \infty} \tau^0\big[T(\salg{P}_n^{(1)} \cup \salg{O}_n^{(2)} \cup \salg{O}_n^{(3)} \cup \salg{O}_n^{(4)})\big]
\]
for an opposing colored double tree $T$, just that it exists. In this case, we know that the value of this limit is equal to $p_F(\mbf{c}_4)$, which in turn is equal to 1 if there are no edges with labels in $I_4$. This proves the first statement, about $\mbf{a}_{1'} \cup \mbf{a}_{1''} \cup \mbf{a}_2 \cup \mbf{a}_3$.

For the second statement, about $\mbf{a}_{1''} \cup \mbf{a}_3 \cup \mbf{a}_4$, it suffices to prove that $\mbf{a}_3$ and $\mbf{a}_4$ are free. Indeed, this follows from the calculation of $p_F(\mbf{c}_4)$: edges with labels in either $I_{1''}$ or $I_3$ are both treated just the same and simply removed. In particular, this implies that the joint distributions $\mu_{\mbf{a}_{1''} \cup \mbf{a}_3 \cup \mbf{a}_4}$ and $\mu_{\mbf{a}_{3''} \cup \mbf{a}_3 \cup \mbf{a}_4} = \mu_{\mbf{b}_3 \cup \mbf{a}_4}$ are identical, where $\mbf{a}_{3''}$ is the limit of the full proportion RBMs $\salg{O}_n^{(3'')} = (\mbf{\Theta}_n^{(i)})_{i \in I_1''}$ and $\mbf{b}_3 = \mbf{a}_{3''} \cup \mbf{a}_3$ is simply the limit of a larger family of independent full proportion RBMs. Now, since the joint distribution $\mu_{\mbf{a}_3 \cup \mbf{a}_4}$ is universal independent of the parameters $\beta_i$, we can calculate $\mu_{\mbf{a}_3 \cup \mbf{a}_4}$ via a unitarily invariant realization of $\salg{O}_n^{(3)}$. The standard techniques then apply to show that $\mbf{a}_3$ and $\mbf{a}_4$ are free \cite{Voi91}.

Similarly, the joint distributions $\mu_{\mbf{a}_2 \cup \mbf{a}_4}$ and $\mu_{\mbf{a}_{1'} \cup \mbf{a}_4}$ are also identical, so we need only to consider the families $\mbf{a}_2$ and $\mbf{a}_4$. Let $a_{i_2} \in \mbf{a}_2$ and $a_{i_4}\in \mbf{a}_4$. If $a_{i_2}$ and $a_{i_4}$ were free, then
\[
\varphi(a_{i_4}^2a_{i_2}a_{i_4}^2a_{i_2}) = \varphi(a_{i_4}^2)^2\varphi(a_{i_2}^2) = 1;
\]
however, one can easily calculate
\begin{align*}
\lim_{n \to \infty} \E\bigg[\frac{1}{n}\trace\bigg((\mbf{\Theta}_n^{(i_4)})^2\mbf{\Theta}_n^{(i_2)}(\mbf{\Theta}_n^{(i_4)})^2\mbf{\Theta}_n^{(i_2)}\bigg)\bigg] &= p_T(c_{i_4}) \\
&=
\begin{dcases}
\frac{8c_{i_4}^2(\frac{1}{2} - c_{i_4}) + \frac{14}{3}c_{i_4}^3}{(2c_{i_4} - c_{i_4}^2)^2} & \text{if $c_{i_4} \leq \frac{1}{2}$,} \\
\frac{2c_{i_4} - 1 + \frac{2}{3}(1-c_{i_4}^3)}{(2c_{i_4} - c_{i_4}^2)^2} & \text{if $c_{i_4} \geq \frac{1}{2}$}
\end{dcases}\\
&\neq 1
\end{align*}
for $c_{i_4} \in (0, 1)$, where

\begin{center}
\begingroup%
  \makeatletter%
  \providecommand\color[2][]{%
    \errmessage{(Inkscape) Color is used for the text in Inkscape, but the package 'color.sty' is not loaded}%
    \renewcommand\color[2][]{}%
  }%
  \providecommand\transparent[1]{%
    \errmessage{(Inkscape) Transparency is used (non-zero) for the text in Inkscape, but the package 'transparent.sty' is not loaded}%
    \renewcommand\transparent[1]{}%
  }%
  \providecommand\rotatebox[2]{#2}%
  \ifx\svgwidth\undefined%
    \setlength{\unitlength}{468bp}%
    \ifx\svgscale\undefined%
      \relax%
    \else%
      \setlength{\unitlength}{\unitlength * \real{\svgscale}}%
    \fi%
  \else%
    \setlength{\unitlength}{\svgwidth}%
  \fi%
  \global\let\svgwidth\undefined%
  \global\let\svgscale\undefined%
  \makeatother%
  \begin{picture}(1,0.0848277)%
    \put(0.29106592,1.35946302){\color[rgb]{0,0,0}\makebox(0,0)[lt]{\begin{minipage}{0.11106844\unitlength}\raggedright \end{minipage}}}%
    \put(0,0){\includegraphics[width=\unitlength,page=1]{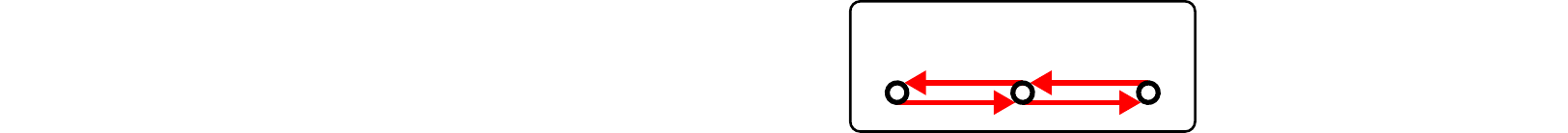}}%
    \put(0.59627457,0.04633917){\color[rgb]{0,0,0}\makebox(0,0)[lb]{\smash{$\mbf{\Theta}_n^{(i_4)}$}}}%
    \put(0.67640279,0.04633917){\color[rgb]{0,0,0}\makebox(0,0)[lb]{\smash{$\mbf{\Theta}_n^{(i_4)}$}}}%
    \put(0.22396617,0.03611678){\color[rgb]{0,0,0}\makebox(0,0)[lb]{\smash{$T(\mbf{\Theta}_n^{(i_4)}, \mbf{\Theta}_n^{(i_4)}, \mbf{\Theta}_n^{(i_4)}, \mbf{\Theta}_n^{(i_4)}) =$ }}}%
    \put(0.77124178,0.03504861){\color[rgb]{0,0,0}\makebox(0,0)[lb]{\smash{$.$}}}%
  \end{picture}%
\endgroup%

\end{center}

Finally, suppose that $a_{i_4} \neq a_{j_4} \in \mbf{a}_4$ with $0 < c_{i_4} \leq c_{j_4} < 1$. If $a_{i_4}$ and $a_{j_4}$ were free, then
\[
\varphi(a_{i_4}^2a_{j_4}^2) = \varphi(a_{i_4}^2)\varphi(a_{j_4}^2) = 1;
\]
however, one can again show that
\begin{align*}
\lim_{n \to \infty} \E\bigg[\frac{1}{n}\trace\bigg((\mbf{\Theta}_n^{(i_4)})^2(\mbf{\Theta}_n^{(j_4)})^2\bigg)\bigg] &= p_S(\{c_{i_4}, c_{j_4}\}) \neq 1,
\end{align*}
where $p_S(\{c_{i_4}, c_{j_4}\})$ is as in the proof of Corollary \hyperref[eq:non]{4.3.4}.
\end{proof}

\begin{rem}\label{rem4.3.7}
We need the assumption on the band widths $(b_n)_{i \in I_1}$ of the periodic RBMs to handle the interaction with the proper proportional growth RBMs $\salg{O}_n^{(4)}$. The families $\salg{P}_n^{(1)} \cup \salg{O}_n^{(2)} \cup \salg{O}_n^{(3)}$ converge in joint distribution to a semicircular system regardless, even without this assumption.
\end{rem}

Finally, the same considerations that allow us to translate Proposition \hyperref[prop3.1.2]{3.1.2} to Theorem \hyperref[thm4.3.3]{4.3.3} also work to prove the RBM version of the concentration inequalities in Theorem \hyperref[thm3.2.2]{3.2.2}. Here, we do not make any assumptions on the band widths $(b_n^{(i)})_{i \in I_1}$ beyond their divergence \eqref{eq:4.1}, nor on the parameters $\beta_i \in \C$.

\begin{thm}\label{thm4.3.8}
Let $\salg{Q}_n = \salg{P}_n^{(1)} \cup \salg{O}_n^{(2)} \cup \salg{O}_n^{(3)} \cup \salg{O}_n^{(4)}$. For any test graph $T$ in $\mbf{x}$,
\[
\E\bigg[\bigg|\frac{1}{n}\emph{tr}\big[T(\salg{Q}_n)\big] - \E\frac{1}{n}\emph{tr}\big[T(\salg{Q}_n)\big]\bigg|^{2m}\bigg] = O_T(n^{-m}).
\]
The bound is tight in the sense that there exist test graphs $T$ such that
\[
\E\bigg[\bigg|\frac{1}{n}\emph{tr}\big[T(\salg{Q}_n)\big] - \E\frac{1}{n}\emph{tr}\big[T(\salg{Q}_n)\big]\bigg|^{2m}\bigg] = \Theta_T(n^{-m}).
\]
\end{thm}

As before, we can use Theorem \hyperref[thm4.3.8]{4.3.8} to upgrade the convergence in Theorems \hyperref[thm4.3.3]{4.3.3} and \hyperref[thm4.3.6]{4.3.6} to the almost sure sense.

\subsection{Fixed band width}\label{sec4.4}

We have much less to say in the fixed band width regime. For starters, we cannot work in the generality of the Wigner matrices of Section \hyperref[sec3]{3}. Instead, we must further assume that the off-diagonal entries (resp., the diagonal entries) of $\mbf{X}_n$ are identically distributed, independent of $n$; otherwise, in general, the LSD of even a single fixed band width RBM $\mbf{\Theta}_n = \mbf{\Upsilon}_n \circ \mbf{\Xi}_n = \mbf{\Upsilon}_n \circ (\mbf{B}_n \circ \mbf{X}_n)$ might not exist, never mind the LTD. We assume hereafter that any fixed band width RBM arises from this restricted setting.

Assuming a symmetric distribution for the entries of $\mbf{X}_n$, Section 6 in \cite{BMP91} proves the existence of a symmetric non-universal LSD $\mu_b$ for a real symmetric RBM $\mbf{\Theta}_n$ of fixed band width $b_n \equiv b$. The authors further prove that the distribution $\mu_b$ converges weakly to the standard semicircle distribution $\mu_{\wsc}$ in the limit $b \to \infty$. We consider the joint LTD of independent fixed band width RBMs (real and complex) without this symmetry assumption and prove the analogous convergence to the semicircular traffic distribution in the large band width limit.

To formalize our result, we consider a class of fixed band widths $\mbf{b} = (b_n^{(i)})_{i \in I} = (b_i)_{i \in I}$. We form the corresponding family of fixed band width RBMs
\[
\mathcal{J}_n = (\mbf{\Xi}_n^{(i)})_{i \in I} = (\mbf{B}_n^{(i)} \circ \mbf{X}_n^{(i)})_{i \in I}, \qquad \mathcal{O}_n = (\mbf{\Theta}_n^{(i)})_{i \in I} = (\mbf{\Upsilon}_n^{(i)} \circ \mbf{\Xi}_n^{(i)})_{i \in I}.
\] 
We write $\mu_i$ (resp., $\nu_i$) for the distribution of the strictly upper triangular entries $\mbf{X}_n^{(i)}(j, k)$ (resp., the diagonal entries $\mbf{X}_n^{(i)}(j, j)$) so that
\[
\mu_i = \salg{L}(\mbf{X}_n^{(i)}(j, k)) \quad \text{and} \quad \nu_i = \salg{L}(\mbf{X}_n^{(i)}(j, j)), \qquad  \forall j < k.
\]
In contrast to the previous sections, our fixed normalizations $\mbf{\Upsilon}_n^{(i)} = (2b_i + 1)^{-1/2}\mbf{J}_n$ force us to also consider non-tree-like test graphs $T$ in the limit $n \to \infty$.

\begin{thm}\label{thm4.4.1}
The family of fixed band width RBMs $\salg{O}_n$ converges in traffic distribution; moreover, for any test graph $T = (V, E, \gamma)$ in $\mbf{x}$, we have the bound
\begin{equation}\label{eq:4.42}
\lim_{n \to \infty} \tau^0\big[T(\salg{O}_n)\big] = O_{T, \bm{\mu}, \bm{\nu}}\bigg(\frac{\prod_{[e] \in \wtilde{N}_0} \min_{e' \in [e]} 2b_{\gamma(e')}}{\prod_{e \in E} \sqrt{2b_{\gamma(e)} + 1}}\bigg),
\end{equation}
where $(V, \wtilde{N}_0)$ is any spanning tree of $(V, \wtilde{N})$ and
\[
\bm{\mu} = (\mu_i)_{i \in I}, \qquad \bm{\nu} = (\nu_i)_{i \in I}.
\] 
\end{thm}
\begin{proof}
We have the familiar expansion
\begin{equation}\label{eq:4.43}
\tau^0\big[T(\salg{O}_n)\big] = \frac{1}{n\prod_{e \in N} \sqrt{2b_{\gamma(e)} + 1}} \sum_{\phi: V \hookrightarrow [n]} \E\bigg[\prod_{e \in E} \mbf{\Xi}_n^{(\gamma(e))}(\phi(e))\bigg],
\end{equation}
where the sum can be written as
\[
\sum_{\phi: V \hookrightarrow [n]} \bigg(\prod_{[\ell] \in \wtilde{L}} \E \bigg[\prod_{\ell' \in [\ell]} \mbf{X}_n^{(\gamma(\ell'))}(\phi(\ell'))\bigg]\bigg) \bigg(\prod_{[e] \in \wtilde{N}} \mathbbm{1}\{|\phi([e])| \leq \min_{e' \in [e]} b_{\gamma(e')}\}\E \bigg[\prod_{e' \in [e]} \mbf{X}_n^{(\gamma(e'))}(\phi(e'))\bigg]\bigg).
\]
Note that an injective map $\phi: V \hookrightarrow [n]$ satisfying the band width condition
\[
|\phi([e])| \leq \min_{e' \in [e]} b_{\gamma(e')}, \qquad \forall [e] \in \wtilde{N}
\]
might not exist (e.g., if $\salg{O}_n$ consists of a single RBM $\mbf{\Theta}_n$ of fixed band width $b$ and $T$ is a star graph $S_k$ with $k > 2b$); however, we can certainly bound the number of such maps by
\[
n \prod_{[e] \in \wtilde{N}_0} \min_{e' \in [e]} 2b_{\gamma(e')},
\]
where $(V, \wtilde{N}_0)$ is any spanning tree of $(V, \wtilde{N})$. Here, we are simply recycling the bound \eqref{eq:4.8}. Our strong moment assumption \eqref{eq:3.1} then already proves \eqref{eq:4.42}.

As before, we see that $\tau^0\big[T(\salg{O}_n)\big]$ vanishes unless
\[
m_{i, [e]} = 0 \text{ or } m_{i, [e]} \geq 2, \qquad \forall (i, [e]) \in I \times \wtilde{N}.
\]
Unfortunately, our fixed normalizations $\sqrt{2b_i + 1}$ allow $\tau^0\big[T(\salg{O}_n)\big]$ to survive in the limit for test graphs $T$ with $m_{i, [e]} > 2$. In this case, the assumption that $\beta_i \in \R$ no longer suffices to spare us the consideration of the ordering $\psi_\phi: [\#(V)] \stackrel{\sim}{\to} V$ on the vertices. Nevertheless, our i.i.d. assumption ensures that if $\phi_1: V \hookrightarrow [n_1]$ and $\phi_2: V \hookrightarrow [n_2]$ satisfy the band width condition and induce the same ordering $\psi_{\phi_1} = \psi_{\phi_2}$, then the corresponding summands of \eqref{eq:4.43} are equal, i.e.,
\[
S_{\phi_1}(T) = \E\bigg[\prod_{e \in E} \mbf{\Xi}_{n_1}^{(\gamma(e))}(\phi_1(e))\bigg] = \E\bigg[\prod_{e \in E} \mbf{\Xi}_{n_2}^{(\gamma(e))}(\phi_2(e))\bigg] = S_{\phi_2}(T).
\]
For an ordering $\psi: [\#(V)] \stackrel{\sim}{\to} V$, we can again write $S_\psi$ for the common value of
\[
\{S_\phi: \psi_\phi = \psi \text{ and } |\phi([e])| \leq \min_{e' \in [e]} b_{\gamma(e')} \text{ for all } [e] \in \wtilde{N}\}.
\]
This allows us to rewrite \eqref{eq:4.43} as 
\[
\tau^0\big[T(\salg{O}_n)\big] = \sum_{\psi: [\#(V)] \stackrel{\sim}{\to} V} \frac{p_n^{(\psi)}}{\prod_{e \in E} \sqrt{2b_{\gamma(e)} + 1}} S_\psi(T) = \sum_{\psi: [\#(V)] \stackrel{\sim}{\to} V} q_n^{(\psi)} S_\psi(T),
\]
where
\[
p_n^{(\psi)} = \frac{\sum_{\phi: V \hookrightarrow [n]} \bigg(\mathbbm{1}\{\psi_\phi = \psi\}\prod_{[e] \in \wtilde{N}} \mathbbm{1}\{|\phi([e])| \leq \min_{e' \in [e]} b_{\gamma(e')}\}\bigg)}{n}.
\]
We note the contrast to the situation in \eqref{eq:3.14}. In particular, we cannot use the same weak convergence argument to give an integral representation of $\lim_{n \to \infty} p_n^{(\psi)}$ as in \eqref{eq:3.15} due to the vanishing scales $\lim_{n \to \infty} \frac{b_i}{n} = 0$. Instead, we must opt for a discrete approach.

Let $(a_n^{(\psi)})$ denote the sequence defined by the numerator of $p_n^{(\psi)}$ so that
\[
a_n^{(\psi)} = \sum_{\phi: V \hookrightarrow [n]} \bigg(\mathbbm{1}\{\psi_\phi = \psi\}\prod_{[e] \in \wtilde{N}} \mathbbm{1}\{|\phi([e])| \leq \min_{e' \in [e]} b_{\gamma(e')}\}\bigg).
\]
By considering a map $\phi_1: V \hookrightarrow [n]$ (resp., $\phi_2: V \hookrightarrow [m]$) as a map $\Phi_1: V \hookrightarrow [n + m]$ (resp., $\Phi_2: V \hookrightarrow [n + m]$), viz.
\[
\Phi_1(v) = \phi_1(v) \qquad (\text{resp., } \Phi_2(v) = \phi_2(v) + n),
\]
we see that the sequence $(a_n^{(\psi)})$ is superadditive:
\[
a_{n+m}^{(\psi)} \geq a_n^{(\psi)} + a_m^{(\psi)}.
\]
Fekete's lemma then implies that
\[
p_\psi = \lim_{n \to \infty} p_n^{(\psi)} = \sup_n \frac{a_n^{(\psi)}}{n} \leq \prod_{[e] \in \wtilde{N}} \min_{e' \in [e]} 2b_{\gamma(e')},
\]
which proves the convergence
\begin{equation}\label{eq:4.44}
\lim_{n \to \infty} \tau^0\big[T(\salg{O}_n)\big] = \sum_{\psi: [\#(V)] \stackrel{\sim}{\to} V} \frac{p_\psi}{\prod_{e \in E} \sqrt{2b_{\gamma(e)} + 1}} S_\psi(T) = \sum_{\psi: [\#(V)] \stackrel{\sim}{\to} V} q_\psi S_\psi(T).
\end{equation}
\end{proof}
Note that our bound \eqref{eq:4.42} implies the convergence
\begin{equation}\label{eq:4.45}
\lim_{\underline{b} \to \infty} \sum_{\psi: [\#(V)] \stackrel{\sim}{\to} V} q_\psi S_\psi(T) =
\begin{cases}
\prod_{i \in I} \beta_i^{c_i(T)}  & \text{if $T$ is a colored double tree,} \\
0  & \text{otherwise,}
\end{cases}
\end{equation}
where
\[
\underline{b} = \min_{e \in E} b_{\gamma(e)}.
\]

Theorem \hyperref[thm4.4.1]{4.4.1} still holds for general $\beta_i \in \C$: in fact, since we already keep track of the orderings $\psi$, the same proof goes through just as well (except with different values for $S_\psi(T)$). In this case, the limit \eqref{eq:4.45} might not exist depending on the relative rates of growth in the band widths $b_i$. If we assume that the band widths grow at the same rate in the limit $\underline{b} \to \infty$, then the proportions $q_n^{(\psi)}$ will tend to $\frac{1}{\#(V)}$ as in \eqref{eq:3.16}, but one can skew these proportions along different subsequences to create an obstruction. One can also periodize the fixed band width RBMs without affecting the calculations (a fixed band width is in some sense the slowest growth possible, and so we can adapt the techniques from Section \hyperref[sec4.2]{4.2}). 

At this point, we can combine everything into a result for the joint (traffic) distribution of periodic RBMs, slow growth RBMs, proportional growth RBMs, and fixed band width RBMs; however, the result is not much more interesting than what is already known from the previous section due to the form of the LTD \eqref{eq:4.44}. In particular, we do not have any interesting asymptotic independences arising between the fixed band width RBMs and those of the previously considered regimes, nor amongst the fixed band width RBMs themselves (except in the trivial case $b_i = 0$ of the diagonal fixed band width RBMs, which are permutation invariant and satisfy the conditions of Theorem \hyperref[thm2.5.5]{2.5.5}).

\appendix
\section*{Appendix}\label{appendix}
\setcounter{section}{1}
\setcounter{equation}{0}

We gather some miscellaneous results in this appendix. In the first section, we consider the analogue of the Markov matrix problem from Section \hyperref[sec3.2]{3.2} for the proportional growth RBMs. In particular, we compute the LSD of the degree matrix $\mbf{D}_n = \text{row}(t_x)(\mbf{\Theta}_n)$ of a proportional growth RBM $\mbf{\Theta}_n$ and consider the joint distribution of $(\mbf{\Theta}_n, \mbf{D}_n)$. Here, we find that the free product decomposition of \cite{AM17} cannot be extended to the proportional growth regime (in contrast to the periodic regime \hyperref[sec4.1]{\S 4.1} and the slow growth regime \hyperref[sec4.2]{\S 4.2}). In the second section, we pursue an orthogonal computation, namely, the limiting traffic distribution of a Haar distributed orthogonal matrix $\mbf{O}_n$. The proof essentially follows the unitary case \cite[Proposition 6.2]{Mal11} except that we must now take care to apply the orthogonal Weingarten calculus \cite{CS06}.

\addtocontents{toc}{\SkipTocEntry}
\subsection{An almost Gaussian degree matrix}\label{almost_gaussian}

Again, for simplicity, we restrict our attention to real Wigner matrices $\mbf{X}_n$ as in Section \hyperref[sec3.2]{3.2}. We form the corresponding proportional growth RBMs, unnormalized $\mbf{\Xi}_n$ and otherwise $\mbf{\Theta}_n$. Let $c \in (0, 1]$ denote the limiting proportion of the band width $b_n$, i.e.,
\[
\lim_{n \to \infty} \frac{b_n}{n} = c.
\]
We form the degree matrix $\mbf{D}_n = \text{row}(t_x)(\mbf{\Theta}_n)$ of $\mbf{\Theta}_n$, where
\begin{align*}
\mbf{D}_n(i, j) &= \indc{i = j} \sum_{k = 1}^n \mbf{\Theta}_n(i, k) \\
&= \indc{i = j} \sum_{k = 1}^n \frac{\mbf{\Xi}_n(i, k)}{\sqrt{n}\sqrt{2c-c^2}} = \indc{i = j} \sum_{k = 1}^n \frac{\indc{|i - k| \leq b_n}\mbf{X}_n(i, k)}{\sqrt{n}\sqrt{2c-c^2}}.
\end{align*}
One can then use the asymptotics of partial sums of falling factorials to compute the limiting moments
\[
\lim_{n \to \infty} \E\bigg[\frac{1}{n}\trace(\mbf{D}_n^m)\bigg], \qquad \forall m \in \N,
\]
for example, by choosing a convenient realization of the random variables $\mbf{X}_n(i, k)$ and then appealing to the universality of \eqref{eq:4.40}; however, one can even avoid such a tedious calculation and obtain the answer from \eqref{eq:4.40} directly. In particular, we can factor the expected moments of the spectral distribution $\mu_{\mbf{D}_n}$ through the traffic distribution of $\mbf{\Theta}_n$ via
\[
\E\bigg[\frac{1}{n}\trace(\mbf{D}_n^m)\bigg] = \tau\big[C_m(\mbf{D}_n, \ldots, \mbf{D}_n)\big] = \tau\big[S_m(\mbf{\Theta}_n, \ldots, \mbf{\Theta}_n)\big],
\]
where $C_m$ is the directed cycle with $m$ edges and $S_m = (V, E)$ is the inward facing directed $m$-star graph, i.e.,

\begin{center}
\begingroup%
  \makeatletter%
  \providecommand\color[2][]{%
    \errmessage{(Inkscape) Color is used for the text in Inkscape, but the package 'color.sty' is not loaded}%
    \renewcommand\color[2][]{}%
  }%
  \providecommand\transparent[1]{%
    \errmessage{(Inkscape) Transparency is used (non-zero) for the text in Inkscape, but the package 'transparent.sty' is not loaded}%
    \renewcommand\transparent[1]{}%
  }%
  \providecommand\rotatebox[2]{#2}%
  \ifx\svgwidth\undefined%
    \setlength{\unitlength}{468bp}%
    \ifx\svgscale\undefined%
      \relax%
    \else%
      \setlength{\unitlength}{\unitlength * \real{\svgscale}}%
    \fi%
  \else%
    \setlength{\unitlength}{\svgwidth}%
  \fi%
  \global\let\svgwidth\undefined%
  \global\let\svgscale\undefined%
  \makeatother%
  \begin{picture}(1,0.25)%
    \put(0.12170718,1.74313232){\color[rgb]{0,0,0}\makebox(0,0)[lt]{\begin{minipage}{0.11106844\unitlength}\raggedright \end{minipage}}}%
    \put(0,0){\includegraphics[width=\unitlength,page=1]{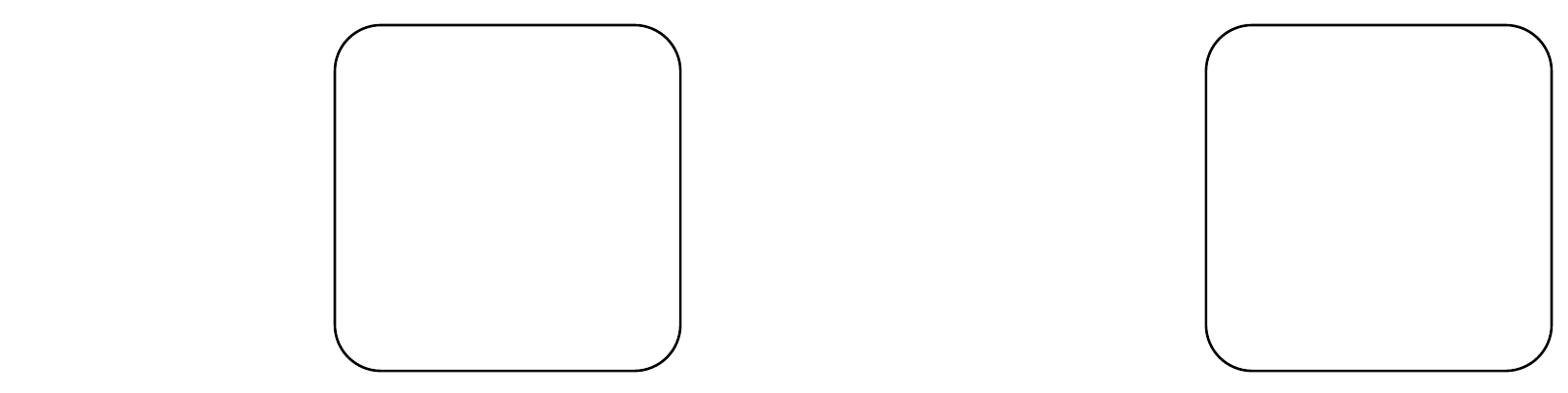}}%
    \put(-0.00199538,0.11752325){\color[rgb]{0,0,0}\makebox(0,0)[lb]{\smash{$C_m(\mbf{D}_n, \ldots, \mbf{D}_n) =$}}}%
    \put(0.55549662,0.11752304){\color[rgb]{0,0,0}\makebox(0,0)[lb]{\smash{$S_m(\mbf{\Theta}_n, \ldots, \mbf{\Theta}_n) =$}}}%
    \put(0,0){\includegraphics[width=\unitlength,page=2]{fig28_degree.pdf}}%
    \put(0.41361217,0.36502126){\color[rgb]{0,0,0}\makebox(0,0)[lb]{\smash{}}}%
    \put(0.86451103,0.06718729){\color[rgb]{0,0,0}\makebox(0,0)[lb]{\smash{$\cdots$}}}%
    \put(0.47531911,0.11719982){\color[rgb]{0,0,0}\makebox(0,0)[lb]{\smash{and}}}%
    \put(0.83374179,0.16634857){\color[rgb]{0,0,0}\makebox(0,0)[lb]{\smash{$\mbf{\Theta}_n$}}}%
    \put(0.80689884,0.11752325){\color[rgb]{0,0,0}\makebox(0,0)[lb]{\smash{$\mbf{\Theta}_n$}}}%
    \put(0.89463924,0.16634857){\color[rgb]{0,0,0}\makebox(0,0)[lb]{\smash{$\mbf{\Theta}_n$}}}%
    \put(0.91787642,0.11752325){\color[rgb]{0,0,0}\makebox(0,0)[lb]{\smash{$\mbf{\Theta}_n$}}}%
    \put(0.83374179,0.0661883){\color[rgb]{0,0,0}\makebox(0,0)[lb]{\smash{$\mbf{\Theta}_n$}}}%
    \put(0.89463924,0.0661883){\color[rgb]{0,0,0}\makebox(0,0)[lb]{\smash{$\mbf{\Theta}_n$}}}%
    \put(0.29928665,0.2028069){\color[rgb]{0,0,0}\makebox(0,0)[lb]{\smash{$\mbf{D}_n$}}}%
    \put(0.38662641,0.16714985){\color[rgb]{0,0,0}\makebox(0,0)[lb]{\smash{$\mbf{D}_n$}}}%
    \put(0.39463923,0.08221384){\color[rgb]{0,0,0}\makebox(0,0)[lb]{\smash{$\mbf{D}_n$}}}%
    \put(0.31651422,0.0281274){\color[rgb]{0,0,0}\makebox(0,0)[lb]{\smash{$\mbf{D}_n$}}}%
    \put(0,0){\includegraphics[width=\unitlength,page=3]{fig28_degree.pdf}}%
    \put(0.31435076,0.05083091){\color[rgb]{0,0,0}\makebox(0,0)[lb]{\smash{$\cdots$}}}%
    \put(0.22236359,0.15432933){\color[rgb]{0,0,0}\makebox(0,0)[lb]{\smash{$\mbf{D}_n$}}}%
    \put(0.23438283,0.06819161){\color[rgb]{0,0,0}\makebox(0,0)[lb]{\smash{$\mbf{D}_n$}}}%
    \put(0.99560079,0.11645498){\color[rgb]{0,0,0}\makebox(0,0)[lb]{\smash{$.$}}}%
  \end{picture}%
\endgroup%

\end{center}

\noindent Here, we have made the substitution $\mbf{D}_n = \op{row}(t_x)(\mbf{\Theta}_n)$. We rewrite this in terms of the injective trace to obtain
\[
\tau\big[S_m(\mbf{\Theta}_n, \ldots, \mbf{\Theta}_n)\big] = \sum_{\pi \in \salg{P}(V)} \tau^0\big[S_m^\pi(\mbf{\Theta}_n, \ldots, \mbf{\Theta}_n)\big].
\]

In the limit, \eqref{eq:4.40} tells us that the only contributions come from double trees $S_m^\pi(\mbf{\Theta}_n, \ldots, \mbf{\Theta}_n)$. For odd $m$, this is not possible since a double tree has an even number of edges, while $S_m$ has $m$ edges. This implies that
\begin{equation}\label{eq:A.1}
\lim_{n \to \infty} \E\bigg[\frac{1}{n}\trace(\mbf{D}_n^m)\bigg] = 0 \quad \text{ if $m$ is odd.}
\end{equation}
Henceforth, we assume that $m = 2\ell$. Let $v_1, \ldots, v_{2\ell}$ denote the leaf vertices of $S_{2\ell}$ with the internal node $v_0$. We see that
\[
S_{2\ell}^\pi \text{ is a double tree} \quad \Longleftrightarrow \quad \pi = \{\{v_0\}\} \cup \rho, 
\]
where $\rho$ is a pair partition of $\{v_1, \ldots, v_{2\ell}\}$. In particular, each such $\pi$ produces the same double tree $T_{\ell}(\mbf{\Theta}_n, \ldots, \mbf{\Theta}_n) = S_{2\ell}^\pi(\mbf{\Theta}_n, \ldots, \mbf{\Theta}_n)$, where $T_\ell$ is the inward facing double $\ell$-star graph. It follows that
\begin{align}
\notag \lim_{n \to \infty} \E\bigg[\frac{1}{n}\trace(\mbf{D}_n^{2\ell})\bigg] &= \lim_{n \to \infty} \sum_{\pi \in \salg{P}(V)} \tau^0\big[S_{2\ell}^\pi(\mbf{\Theta}_n, \ldots, \mbf{\Theta}_n)\big] \\
\notag &= \#(\salg{P}_2(2\ell)) p_{T_\ell}(c) = (2\ell-1)!!\frac{\op{Int_{T_\ell}(c)}}{\op{Norm}_{T_\ell}(c)} \\
\notag &= (2\ell-1)!!\frac{\int_{[0, 1]^{\ell+1}} \prod_{k = 0}^\ell \indc{|x_0 - x_k| \leq c} \, dx_\ell \cdots dx_0}{(2c-c^2)^\ell} \\
\notag &= (2\ell-1)!!\frac{\int_0^1 \bigg(\int_0^1 \indc{|x_0 - x_1| \leq c} \, dx_1\bigg)^\ell dx_0}{(2c-c^2)^\ell} \\
&= (2\ell-1)!!\frac{\frac{2}{\ell+1}((2c \wedge 1)^{\ell+1} - c^{\ell+1}) + |2c - 1|(2c \wedge 1)^\ell}{(2c-c^2)^{\ell}}, \label{eq:A.2}
\end{align}
where we have made use of \eqref{eq:4.26} in the last equality. We recognize the double factorial $(2\ell-1)!!$ as the $2\ell$-th moment of the standard normal distribution. In view of Theorem \hyperref[thm4.3.8]{4.3.8}, the limits \eqref{eq:A.1} and \eqref{eq:A.2} then show that $\mu_{\mbf{D}_n}$ converges weakly almost surely to a symmetric distribution of unit variance with \emph{almost} Gaussian moments (if $c = 1$, then these moments are precisely Gaussian). In particular, we can compute the limits 
\[
\lim_{c \to 0^+} \frac{\frac{2}{\ell+1}((2c \wedge 1)^{\ell+1} - c^{\ell+1}) + |2c - 1|(2c \wedge 1)^\ell}{(2c-c^2)^{\ell}} = 1, \qquad \forall \ell \in \N
\]
and
\[
\lim_{c \to 1^-} \frac{\frac{2}{\ell+1}((2c \wedge 1)^{\ell+1} - c^{\ell+1}) + |2c - 1|(2c \wedge 1)^\ell}{(2c-c^2)^{\ell}} = 1, \qquad \forall \ell \in \N,
\]
both of which are special cases of \eqref{eq:4.32}.

\begin{center}
\begingroup%
  \makeatletter%
  \providecommand\color[2][]{%
    \errmessage{(Inkscape) Color is used for the text in Inkscape, but the package 'color.sty' is not loaded}%
    \renewcommand\color[2][]{}%
  }%
  \providecommand\transparent[1]{%
    \errmessage{(Inkscape) Transparency is used (non-zero) for the text in Inkscape, but the package 'transparent.sty' is not loaded}%
    \renewcommand\transparent[1]{}%
  }%
  \providecommand\rotatebox[2]{#2}%
  \ifx\svgwidth\undefined%
    \setlength{\unitlength}{468bp}%
    \ifx\svgscale\undefined%
      \relax%
    \else%
      \setlength{\unitlength}{\unitlength * \real{\svgscale}}%
    \fi%
  \else%
    \setlength{\unitlength}{\svgwidth}%
  \fi%
  \global\let\svgwidth\undefined%
  \global\let\svgscale\undefined%
  \makeatother%
  \begin{picture}(1,0.51923077)%
    \put(0,0){\includegraphics[width=\unitlength,page=1]{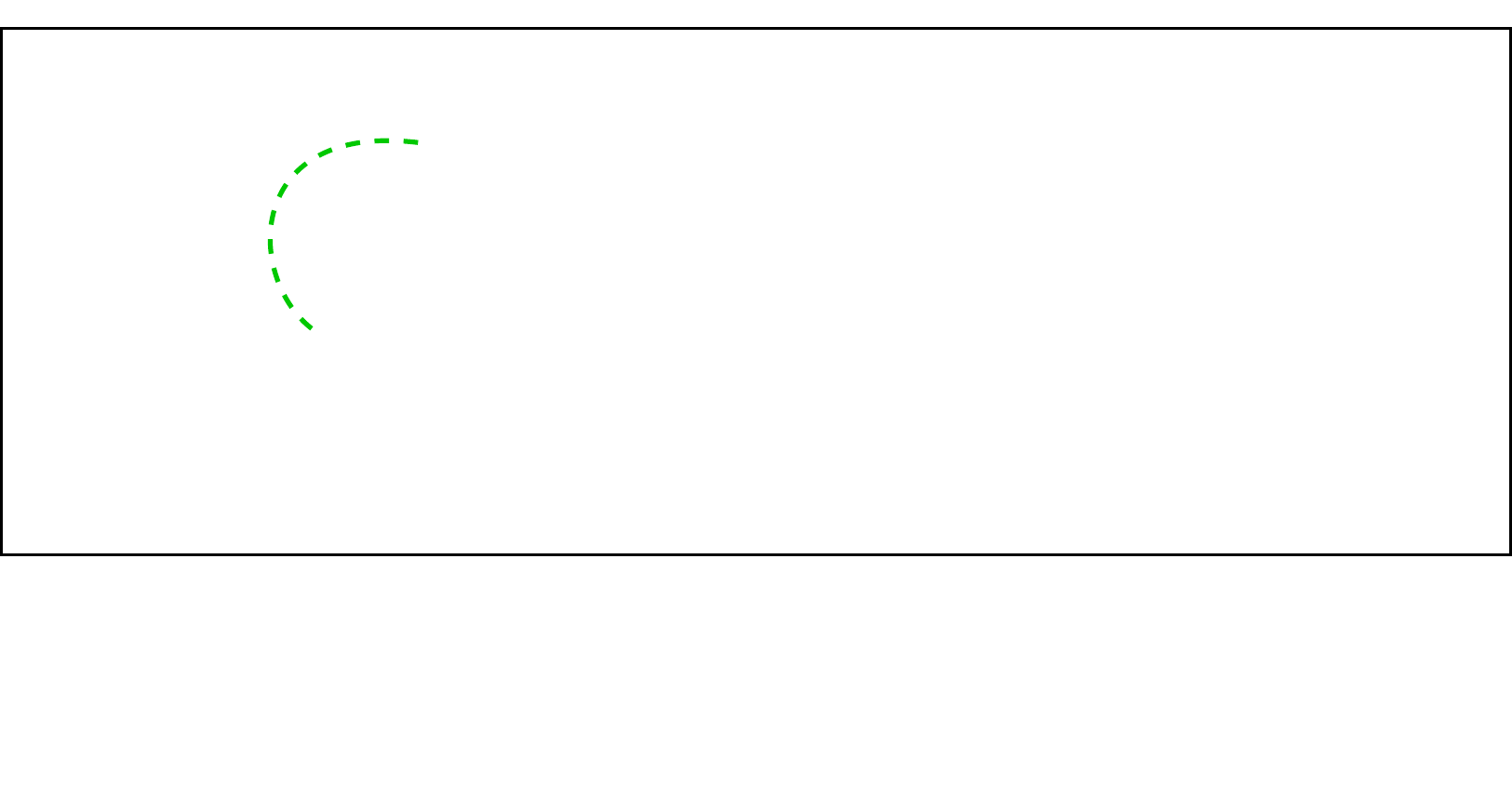}}%
    \put(-0.00141486,0.13610406){\color[rgb]{0,0,0}\makebox(0,0)[lt]{\begin{minipage}{0.99822208\unitlength}\raggedright Figure 25: An example of a pair partition $\rho$ of the leaf vertices of $S_{2\ell}$ giving rise to an inward facing double $\ell$-star graph $T_\ell$ for $\ell = 3$. Here, we use different colors for the different blocks of the pair partition. Note that any pair partition of the leaf vertices gives rise to the same double tree $T_\ell$. \end{minipage}}}%
    \put(0.12170718,1.70402897){\color[rgb]{0,0,0}\makebox(0,0)[lt]{\begin{minipage}{0.11106844\unitlength}\raggedright \end{minipage}}}%
    \put(0,0){\includegraphics[width=\unitlength,page=2]{fig29_double.pdf}}%
    \put(0.48718709,0.33418953){\color[rgb]{0,0,0}\makebox(0,0)[lb]{\smash{$\stackrel{\rho}{\mapsto}$}}}%
    \put(0.19410236,0.17373817){\color[rgb]{0,0,0}\makebox(0,0)[lb]{\smash{$S_{2\ell}(\mbf{\Theta}_n, \ldots, \mbf{\Theta}_n)$}}}%
    \put(0.63542451,0.17373817){\color[rgb]{0,0,0}\makebox(0,0)[lb]{\smash{$T_\ell(\mbf{\Theta}_n, \ldots, \mbf{\Theta}_n)$}}}%
    \put(0.2364261,0.38335035){\color[rgb]{0,0,0}\makebox(0,0)[lb]{\smash{$\mbf{\Theta}_n$}}}%
    \put(0.20958313,0.33452494){\color[rgb]{0,0,0}\makebox(0,0)[lb]{\smash{$\mbf{\Theta}_n$}}}%
    \put(0.29732354,0.38335035){\color[rgb]{0,0,0}\makebox(0,0)[lb]{\smash{$\mbf{\Theta}_n$}}}%
    \put(0.32056072,0.33452494){\color[rgb]{0,0,0}\makebox(0,0)[lb]{\smash{$\mbf{\Theta}_n$}}}%
    \put(0.2364261,0.28319009){\color[rgb]{0,0,0}\makebox(0,0)[lb]{\smash{$\mbf{\Theta}_n$}}}%
    \put(0.29732354,0.28319009){\color[rgb]{0,0,0}\makebox(0,0)[lb]{\smash{$\mbf{\Theta}_n$}}}%
    \put(0.64503985,0.31816918){\color[rgb]{0,0,0}\makebox(0,0)[lb]{\smash{$\mbf{\Theta}_n$}}}%
    \put(0.7391905,0.36003617){\color[rgb]{0,0,0}\makebox(0,0)[lb]{\smash{$\mbf{\Theta}_n$}}}%
    \put(0.72897416,0.26027665){\color[rgb]{0,0,0}\makebox(0,0)[lb]{\smash{$\mbf{\Theta}_n$}}}%
  \end{picture}%
\endgroup%

\end{center}

We note that $\mbf{\Theta}_n$ and $\mbf{D}_n$ are asymptotically free iff $c = 1$. Indeed, this follows from the calculation

\begin{center}
\begingroup%
  \makeatletter%
  \providecommand\color[2][]{%
    \errmessage{(Inkscape) Color is used for the text in Inkscape, but the package 'color.sty' is not loaded}%
    \renewcommand\color[2][]{}%
  }%
  \providecommand\transparent[1]{%
    \errmessage{(Inkscape) Transparency is used (non-zero) for the text in Inkscape, but the package 'transparent.sty' is not loaded}%
    \renewcommand\transparent[1]{}%
  }%
  \providecommand\rotatebox[2]{#2}%
  \ifx\svgwidth\undefined%
    \setlength{\unitlength}{468bp}%
    \ifx\svgscale\undefined%
      \relax%
    \else%
      \setlength{\unitlength}{\unitlength * \real{\svgscale}}%
    \fi%
  \else%
    \setlength{\unitlength}{\svgwidth}%
  \fi%
  \global\let\svgwidth\undefined%
  \global\let\svgscale\undefined%
  \makeatother%
  \begin{picture}(1,0.26923077)%
    \put(0.01012863,1.46563285){\color[rgb]{0,0,0}\makebox(0,0)[lt]{\begin{minipage}{0.11106843\unitlength}\raggedright \end{minipage}}}%
    \put(0,0){\includegraphics[width=\unitlength,page=1]{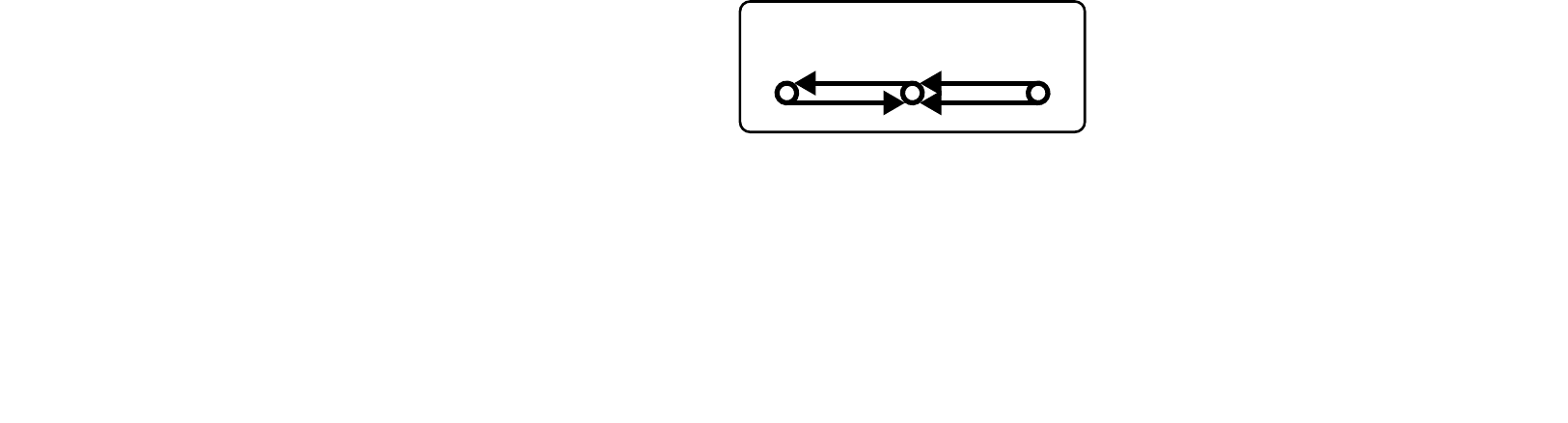}}%
    \put(0.70333311,0.21974224){\color[rgb]{0,0,0}\makebox(0,0)[lb]{\smash{$\big]$}}}%
    \put(0.5308572,0.23063405){\color[rgb]{0,0,0}\makebox(0,0)[lb]{\smash{$\mbf{\Theta}_n$}}}%
    \put(0.61499181,0.23063405){\color[rgb]{0,0,0}\makebox(0,0)[lb]{\smash{$\mbf{\Theta}_n$}}}%
    \put(0.11819693,0.22041154){\color[rgb]{0,0,0}\makebox(0,0)[lb]{\smash{$\displaystyle \lim_{n \to \infty} \E\bigg[\frac{1}{n}\trace(\mbf{\Theta}_n^2\mbf{D}_n^2)\bigg] = \lim_{n \to \infty} \tau^0\big[$}}}%
    \put(0.34808476,0.11344164){\color[rgb]{0,0,0}\makebox(0,0)[lb]{\smash{$\displaystyle = \frac{2((2c \wedge 1)^3 - c^3) + 3|2c-1|(2c \wedge 1)^2}{(2c-c^2)^2}$}}}%
    \put(0.34808476,0.02530061){\color[rgb]{0,0,0}\makebox(0,0)[lb]{\smash{$\displaystyle \neq 1 = \bigg(\lim_{n \to \infty} \E\bigg[\frac{1}{n}\trace(\mbf{\Theta}_n^2)\bigg]\bigg) \bigg(\lim_{n \to \infty} \E\bigg[\frac{1}{n}\trace(\mbf{D}_n^2)\bigg]\bigg)$}}}%
  \end{picture}%
\endgroup%

\end{center}

\noindent unless $c = 1$. In this case, we see that the free product decomposition of \cite{AM17} cannot be extended to the proper proportional growth regime.

\begin{center}
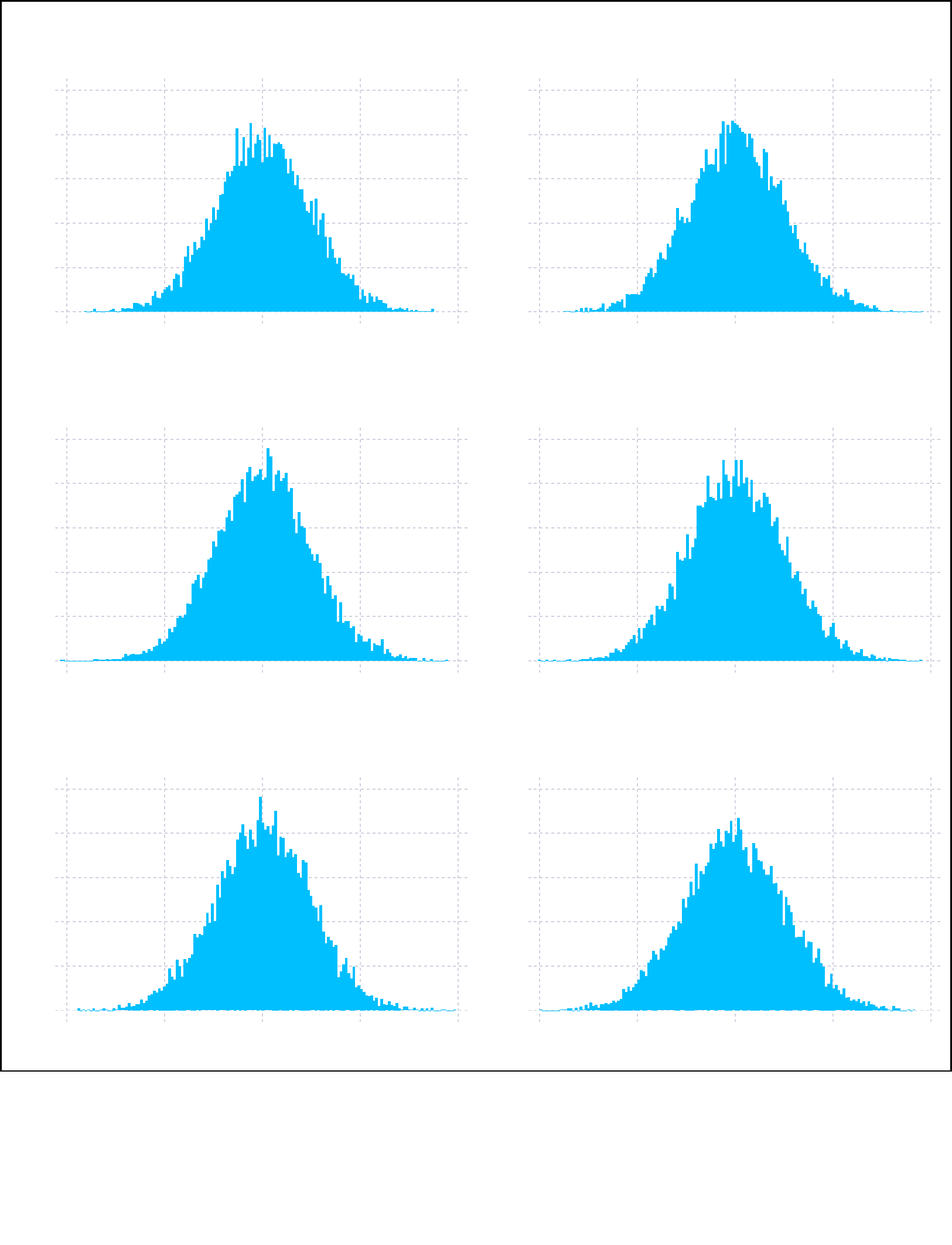
\end{center}

\addtocontents{toc}{\SkipTocEntry}
\subsection{Haar distributed orthogonal matrices}\label{haar_orthogonal}

Let $\mbf{O}_n$ denote an $n \times n$ Haar orthogonal matrix, for which we compute the limiting traffic distribution. Our proof derives from the analogous result for a Haar unitary matrix \cite[Proposition 6.2]{Mal11}. We commit the formal details here to fill out the traffic probability literature. As usual, we restrict our attention to test graphs $T \in \salg{T}\langle x \rangle$. The general case of a $*$-test graph $T = (V, E, \gamma, \varepsilon)$ follows from the relation $\mbf{O}_n^*=\mbf{O}_n^t$, which allows us to freely interchange any edge $e$ with $*$-label $\varepsilon(e) = *$ with an edge $e'$ with $*$-label $\varepsilon(e') = 1$  in the opposite direction, i.e.,
\[
(\source(e), \target(e)) = (\target(e'), \source(e')).
\]
In this case, we suppress the map $\gamma$ since there is only one indeterminate $x$ in consideration.
\begin{defn}[Orthogonal cactus]\label{defnA.2.1}
For a test graph $T = (V, E) \in \salg{T}\langle x \rangle$, we write $\interior{T} = (V, \interior{E})$ for the underlying undirected multigraph. We further write $P: E \to \interior{E}$ for the canonical projection onto the undirected edge set. We say that $T$ is a \emph{cactus} if each edge $\mathring{e}$ of $\interior{T}$ belongs to a unique simple cycle $C_{\mathring{e}}$. We further say that $T$ is an \emph{orthogonal cactus} if $T$ is a cactus such that each cycle $C_{\mathring{e}}$ corresponds to an anti-directed cycle $P^{-1}(C_{\mathring{e}})$ in $T$. By an anti-directed cycle, we mean that $P^{-1}(C_{\mathring{e}}) = (e_1, \ldots, e_k)$ alternates in direction (as opposed to a directed cycle), i.e.,
\begin{equation}\label{eq:A.3}
\exists j \in [k]: \target(e_j) = \target(e_{j+1}), \, \source(e_{j+1}) = \source(e_{j+2}), \, \target(e_{j+2}) = \target(e_{j+3}), \ldots
\end{equation}
where $e_{k+1} = e_1$, $e_{k+2} = e_2$, and so on.
\end{defn}

\phantomsection\label{fig31_cacti}
\begin{center}
\begingroup%
  \makeatletter%
  \providecommand\color[2][]{%
    \errmessage{(Inkscape) Color is used for the text in Inkscape, but the package 'color.sty' is not loaded}%
    \renewcommand\color[2][]{}%
  }%
  \providecommand\transparent[1]{%
    \errmessage{(Inkscape) Transparency is used (non-zero) for the text in Inkscape, but the package 'transparent.sty' is not loaded}%
    \renewcommand\transparent[1]{}%
  }%
  \providecommand\rotatebox[2]{#2}%
  \ifx\svgwidth\undefined%
    \setlength{\unitlength}{468bp}%
    \ifx\svgscale\undefined%
      \relax%
    \else%
      \setlength{\unitlength}{\unitlength * \real{\svgscale}}%
    \fi%
  \else%
    \setlength{\unitlength}{\svgwidth}%
  \fi%
  \global\let\svgwidth\undefined%
  \global\let\svgscale\undefined%
  \makeatother%
  \begin{picture}(1,0.44230769)%
    \put(0,0){\includegraphics[width=\unitlength,page=1]{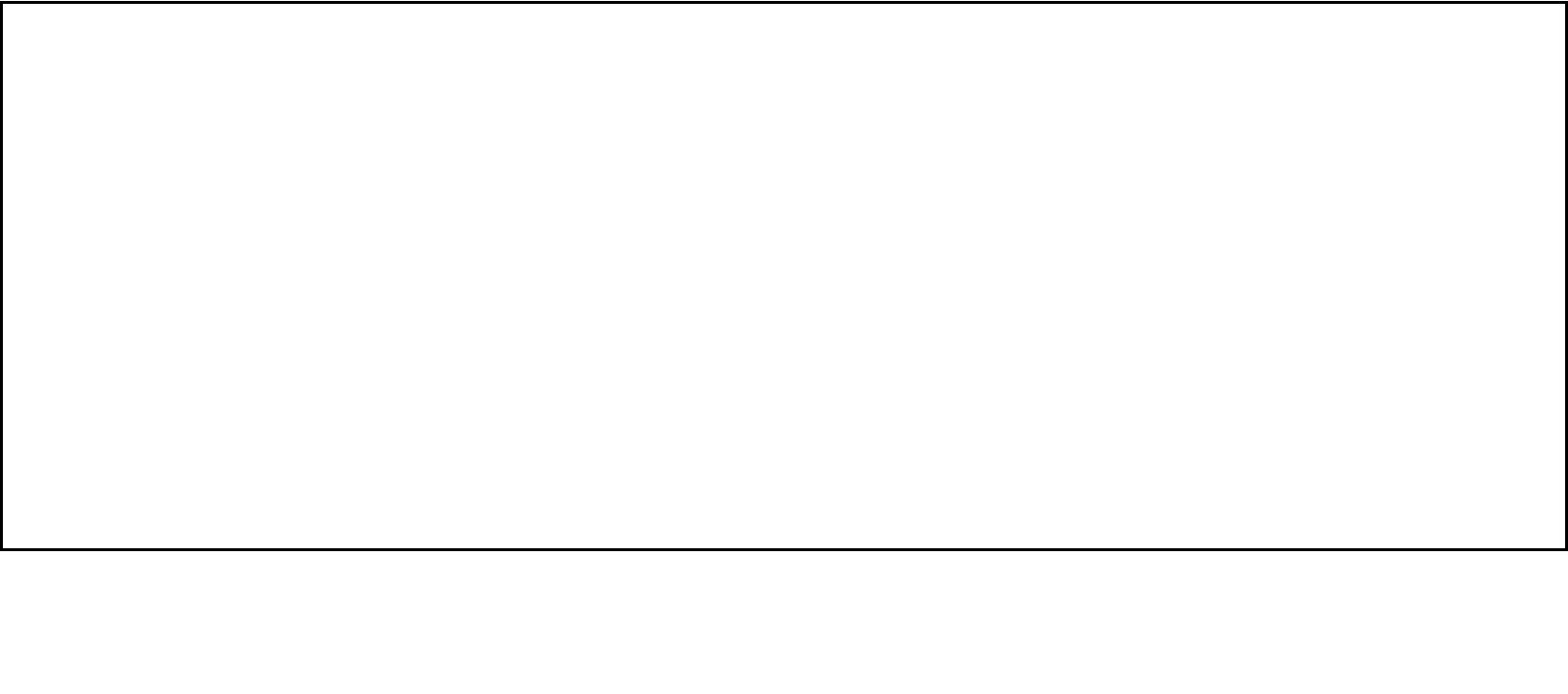}}%
    \put(-0.00141486,0.07609554){\color[rgb]{0,0,0}\makebox(0,0)[lt]{\begin{minipage}{0.99822206\unitlength}\raggedright Figure 27: Examples of a cactus and an orthogonal cactus respectively. For convenience, we omit the sole edge label $x$. \end{minipage}}}%
    \put(0.12170718,1.64402042){\color[rgb]{0,0,0}\makebox(0,0)[lt]{\begin{minipage}{0.11106844\unitlength}\raggedright \end{minipage}}}%
    \put(0,0){\includegraphics[width=\unitlength,page=2]{fig31_cacti.pdf}}%
  \end{picture}%
\endgroup%

\end{center}

For a cactus $T$, we record the length $\#(C)$ of each of its simple (undirected) cycles $C$ in $\interior{T}$. By a slight abuse of notation, we also write $C$ for the corresponding pullback $P^{-1}(C)$ in $T$. For an orthogonal cactus, we know that $\#(C) \in 2\N$ for each such cycle $C$ due to the anti-directedness \eqref{eq:A.3}

We can of course reconstruct a cactus $T$ from its simple cycles (or ``pads'') by starting with an arbitrary simple cycle $C$ of $T$ (level 0), reintroducing the simple cycles that share a common vertex with $C$ (level 1), reintroducing the simples cycles that share a common vertex with the simple cycles from level 1 (level 2), and so on. We imagine this process as ``growing'' the cactus $T$.

\begin{thm}\label{thmA.2.2}
For any test graph $T$ in $x$,
\begin{equation}\label{eq:A.4}
\lim_{n \to \infty} \tau^0\big[T(\mbf{O}_n)\big] =
\begin{dcases}
\prod_{C \in \emph{Pads}(T)} (-1)^{\frac{\#(C)}{2} - 1}c_{\frac{\#(C)}{2}} & \text{if $T$ is an orthogonal cactus,} \\
\hfil 0 & \text{otherwise},
\end{dcases}
\end{equation}
where the product is over the pads $\emph{Pads}(T)$ of $T$ and $c_k = \frac{\binom{2k}{k}}{k+1}$ is the $k$-th Catalan number.
\end{thm}
\begin{proof}
We start with the usual expansion of the injective trace
\begin{align*}
\tau_n^0\big[T(O_n)\big] &= \frac{1}{n} \sum_{\phi: V \hookrightarrow [n]} \E \bigg[\prod_{e \in E} \mbf{O}_n(\phi(e))\bigg] \\
&= \frac{1}{n} \sum_{\phi: V \hookrightarrow [n]} \E \bigg[\prod_{(v, w) \in E} \mbf{O}_n(\phi(w), \phi(v))\bigg],
\end{align*}
where we now consider $E$ as a multiset to do away with the source and target functions. In particular, $\source((v,w)) = v$ and $\target((v, w)) = w$. Note that the distributional invariance of $\mbf{O}_n$ under conjugation by the permutation matrices implies that the value of a summand
\[
S_\phi(T) = \E \bigg[\prod_{(v, w) \in E} \mbf{O}_n(\phi(w), \phi(v))\bigg] = \E \bigg[\prod_{\ell = 1}^{\#(E)} \mbf{O}_n(\phi(w_\ell), \phi(v_\ell))\bigg]
\]
does not depend on the particular choice of labeling $\phi: V \hookrightarrow [n]$ of the vertices. In this case, we can fix a labeling $\phi_0 : V \hookrightarrow [n]$ for all large $n$ (for example, by enumerating the vertices $V = (u_r)_{r=1}^s$ and defining $\phi_0(u_r) = r$) to obtain
\begin{align}
\notag \tau_n^0\big[T(O_n)\big] &= \frac{n^{\underline{\#(V)}}}{n} \E\bigg[\prod_{\ell = 1}^{\#(E)} \mbf{O}_n(\phi_0(w_\ell), \phi_0(v_\ell))\bigg] \\
&\sim n^{\#(V)-1} \E\bigg[\mbf{O}_n(i_1,j_1) \cdots \mbf{O}_n(i_m,j_m)\bigg], \label{eq:A.5}
\end{align}
where $(i_\ell,j_\ell) = (\phi_0(w_\ell), \phi_0(v_\ell))$ and $m = \#(E)$. The string $\mbf{i} = (i_1, \ldots, i_m)$ defines a partition $\ker(\mbf{i})$ of $[m]$ by
\[
\text{ker}(\mbf{i}) = \{\{\ell' : i_\ell = i_{\ell '}\} : \ell \in [m]\},
\]
and similarly for $\mbf{j} = (j_1, \ldots, j_m)$. The orthogonal Weingarten calculus (in the form of \cite[Corollary 3.4]{CS06}) tells us that the expectation in \eqref{eq:A.5} equals $0$ if $m$ is odd; otherwise, $m=2k$ and 
\begin{equation}\label{eq:A.6}
\E\bigg[\mbf{O}_n(i_1,j_1) \cdots \mbf{O}_n(i_{2k},j_{2k})\bigg] = \sum_{p_1,p_2 \in \salg{P}_2(2k)} \delta_{\mbf{i}}(p_1)\delta_{\mbf{j}}(p_2)\inn{p_1}{\op{Wg}_n(p_2)},
\end{equation}
where $\salg{P}_2(2k)$ is the set of pair partitions of $[2k]$, $\op{Wg}_n$ is the $n \times n$ orthogonal Weingarten function, and 
\[
\delta_{\bm{\iota}}(p) = \begin{dcases}
1 & \text{if } p \leq \ker(\bm{\iota}), \\
0 & \text{otherwise.}
\end{dcases}
\]
Here, we use the usual refinement order $\leq$ on the set of partitions $\salg{P}(2k)$.

Of course, the injectivity of the map $\phi_0$ implies that
\begin{alignat*}{3}
i_\ell &= i_{\ell'} \quad &&\Longleftrightarrow \quad w_\ell &&= w_{\ell'}, \\
j_\ell &= j_{\ell'} \quad &&\Longleftrightarrow \quad v_\ell &&= v_{\ell'}. 
\end{alignat*}
We use this correspondence to interpret a pair partition
\[
p_1 = \{\{a_\ell, b_\ell\} : \ell \in [k]\} \in \salg{P}_2(2k) \qquad (\text{resp., } p_2 = \{\{\alpha_\ell, \beta_\ell\} : \ell \in [k]\} \in \salg{P}_2(2k))
\]
such that $\delta_{\mbf{i}}(p_1) = 1$ (resp., $\delta_{\mbf{j}}(p_2) = 1$) as a pair partition
\[
\pi_1 = \{\{(v_{a_\ell}, w_{a_\ell}), (v_{b_\ell}, w_{b_\ell})\} : \ell \in [k]\} \qquad (\text{resp., } \pi_2 = \{\{(\nu_{\alpha_\ell}, \omega_{\alpha_\ell}), (\nu_{\beta_\ell}, \omega_{\beta_\ell})\} : \ell \in [k]\})
\]
of the edges $E$ such that the two edges
\[
(v_{a_\ell}, w_{a_\ell}) \text{ and } (v_{b_\ell}, w_{b_\ell}) \qquad (\text{resp., } (\nu_{\alpha_\ell}, \omega_{\alpha_\ell}) \text{ and } (\nu_{\beta_\ell}, \omega_{\beta_\ell}))
\]
in any block of the partition have a common target $w_{a_\ell} = w_{b_\ell}$ (resp., a common source $\nu_{\alpha_\ell} = \nu_{\beta_\ell}$). We further interpret the pair partition $\pi_1$ as a permutation of the edges $E$ by considering each block $\{(v_{a_\ell}, w_{a_\ell}), (v_{b_\ell}, w_{b_\ell})\}$ as a transposition $((v_{a_\ell}, w_{a_\ell})\text{ }(v_{b_\ell}, w_{b_\ell}))$. In this case, $\pi_1$ corresponds to a product of disjoint transpositions
\[
\pi_1 = \prod_{\ell = 1}^{k} ((v_{a_\ell}, w_{a_\ell})\text{ }(v_{b_\ell}, w_{b_\ell})),
\]
and similarly for
\[
\pi_2 = \prod_{\ell = 1}^{k} ((\nu_{\alpha_\ell}, \omega_{\alpha_\ell})\text{ }(\nu_{\beta_\ell}, \omega_{\beta_\ell})).
\]
A pair $(p_1, p_2)$ such that $\delta_{\mbf{i}}(p_1) = \delta_{\mbf{j}}(p_2) = 1$ then partitions the edges of $T$ into anti-directed cycles
\begin{equation}\label{eq:A.7}
\mfk{C}(\pi_1, \pi_2) = \{(e, \pi_2(e), \pi_1\pi_2(e), \pi_2\pi_1\pi_2(e), \ldots) : e \in E\},
\end{equation}
where we of course assume that cycles are only defined up to a cyclic ordering of the edges. We note that a cycle $C \in \mfk{C}(\pi_1, \pi_2)$ need not be simple.

As a sanity check, one can verify the following equivalent construction of $\mfk{C}(\pi_1, \pi_2)$. We consider a partition $p \in \salg{P}(2k)$ as an element of the symmetric group $\mfk{S}_{2k}$ by associating a block $b = \{\ell_1, \ldots, \ell_{q(b)}\}$ with the cycle $(\ell_1\text{ }\cdots\text{ }\ell_{q(b)})$. A pair $(p_1, p_2)$ as before then partitions the edges of $T$ into anti-directed cycles
\begin{equation}\label{eq:A.8}
\begin{aligned}
\mfk{C}(\pi_1, \pi_2) = \{((v_\ell, w_\ell), &(v_{p_2(\ell)}, w_{p_2(\ell)}), \\
&(v_{p_1 p_2(\ell)}, w_{p_1 p_2(\ell)}), (v_{p_2 p_1 p_2(\ell)}, w_{p_2 p_1 p_2(\ell)}), \ldots) : \ell \in [2k]\}.
\end{aligned}
\end{equation}
Note that the cycle decomposition of the permutation $p_1p_2 \in \mfk{S}_{2k}$ further splits each cycle $C$ in \eqref{eq:A.8} into a pair
\begin{equation*}
(w_\ell, w_{p_1p_2(\ell)}, w_{(p_1p_2)^2(\ell)}, \ldots) \quad \text{and} \quad (v_{p_2(\ell)}, v_{p_2p_1(p_2(\ell))}, v_{(p_2p_1)^2(p_2(\ell))}, \ldots).
\end{equation*}
In terms of \eqref{eq:A.7}, this corresponds to the cycle decomposition of the permutation $\pi_1\pi_2$ of the edges, namely, 
\[
(e, \pi_1\pi_2(e), (\pi_1\pi_2)^2(e), \ldots) \quad \text{and} \quad (\pi_2(e), (\pi_2\pi_1)\pi_2(e), (\pi_2\pi_1)^2\pi_2(e), \ldots).
\]
This implies that
\begin{equation}\label{eq:A.9}
\frac{\#(p_1p_2)}{2} = \frac{\#(\pi_1\pi_2)}{2} = \#(\mfk{C}(\pi_1, \pi_2)),
\end{equation}
where $\#(p_1p_2)$ denotes the number of cycles of $p_1p_2$. We assume hereafter that the partitions $p_1$ and $p_2$ satisfy $\delta_{\mbf{i}}(p_1) = \delta_{\mbf{j}}(p_2) = 1$.

\phantomsection\label{fig32_cycle}
\begin{center}
\begingroup%
  \makeatletter%
  \providecommand\color[2][]{%
    \errmessage{(Inkscape) Color is used for the text in Inkscape, but the package 'color.sty' is not loaded}%
    \renewcommand\color[2][]{}%
  }%
  \providecommand\transparent[1]{%
    \errmessage{(Inkscape) Transparency is used (non-zero) for the text in Inkscape, but the package 'transparent.sty' is not loaded}%
    \renewcommand\transparent[1]{}%
  }%
  \providecommand\rotatebox[2]{#2}%
  \ifx\svgwidth\undefined%
    \setlength{\unitlength}{468bp}%
    \ifx\svgscale\undefined%
      \relax%
    \else%
      \setlength{\unitlength}{\unitlength * \real{\svgscale}}%
    \fi%
  \else%
    \setlength{\unitlength}{\svgwidth}%
  \fi%
  \global\let\svgwidth\undefined%
  \global\let\svgscale\undefined%
  \makeatother%
  \begin{picture}(1,0.57692308)%
    \put(0,0){\includegraphics[width=\unitlength,page=1]{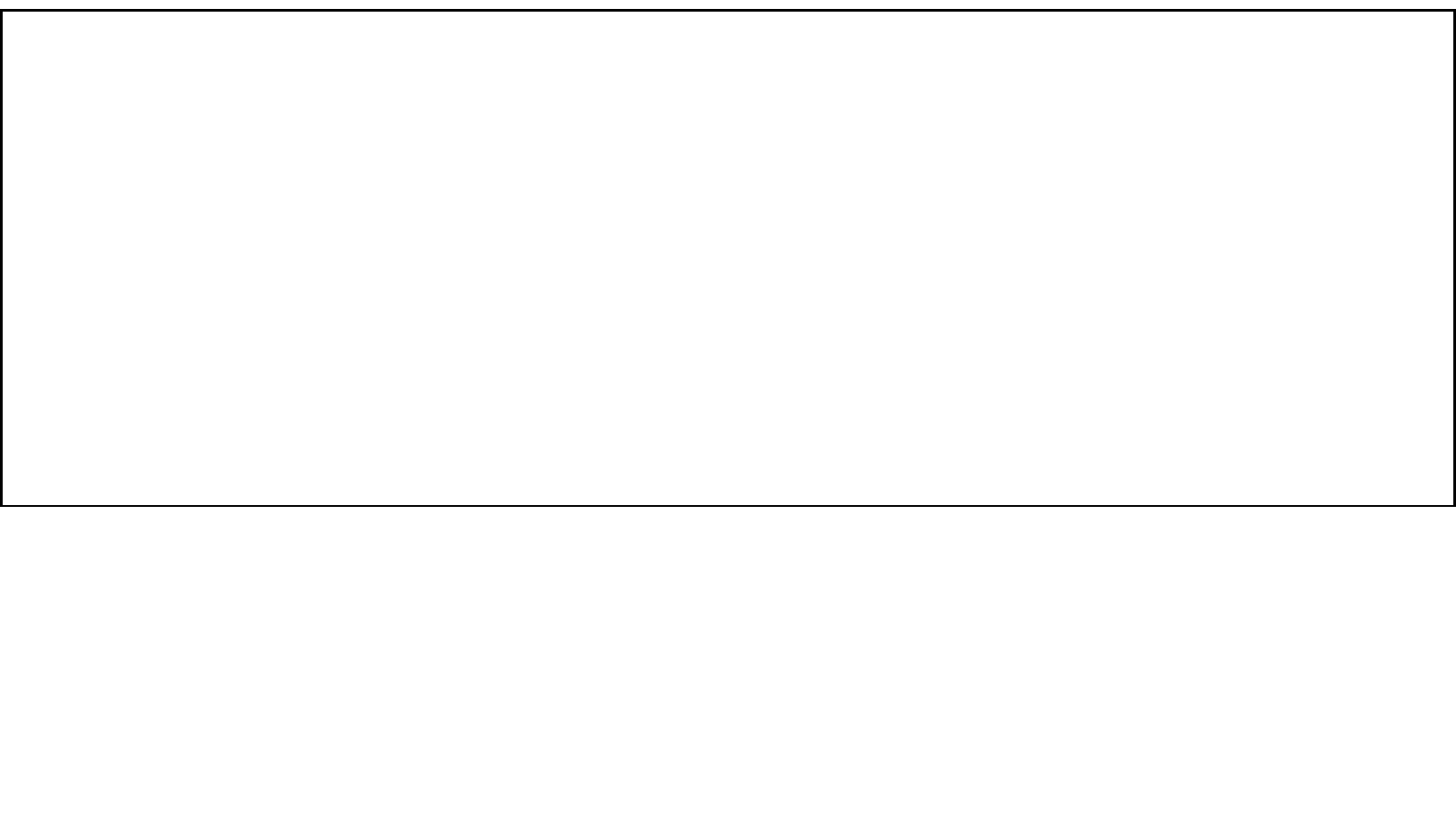}}%
    \put(-0.00141487,0.2147549){\color[rgb]{0,0,0}\makebox(0,0)[lt]{\begin{minipage}{0.99822199\unitlength}\raggedright Figure 28: An example of the construction of $\mfk{C}(\pi_1, \pi_2)$. Here, we start with a test graph $T$ that itself is already a (non-simple) anti-directed cycle $C = (e_1, \ldots, e_6)$, where $e_\ell = (v_\ell, w_\ell)$. Any injective labeling $(i_\ell, j_\ell) = (\phi(w_\ell), \phi(v_\ell))$ of the vertices then generates the partitions $\ker(\mbf{i}) = \{\{1, 2\}, \{3, 4\}, \{5, 6\}\}$ and $\ker(\mbf{j}) = \{\{1, 6\}, \{2, 3\}, \{4, 5\}\}$. In this case, there is a unique pair partition $p_1 \leq \ker(\mbf{i})$, namely $p_1 = \ker(\mbf{i})$, and similarly for $p_2 \leq \ker(\mbf{j})$. One can then easily verify the corresponding permutation of the edges $\pi_1 = (e_1\text{ }e_2)(e_3\text{ }e_4)(e_5\text{ }e_6)$ (resp., $\pi_2 = (e_6\text{ }e_1)(e_2\text{ }e_3)(e_4\text{ }e_5)$), from which it follows that $\mfk{C}(\pi_1, \pi_2) = \{C\}$.\end{minipage}}}%
    \put(0.12170717,1.59361327){\color[rgb]{0,0,0}\makebox(0,0)[lt]{\begin{minipage}{0.11106843\unitlength}\raggedright \end{minipage}}}%
    \put(0,0){\includegraphics[width=\unitlength,page=2]{fig32_cycle.pdf}}%
    \put(0.18400942,0.43936308){\color[rgb]{0,0,0}\makebox(0,0)[lb]{\smash{$e_1$}}}%
    \put(0.22936201,0.48824128){\color[rgb]{0,0,0}\makebox(0,0)[lb]{\smash{$e_2$}}}%
    \put(0.27215045,0.43936308){\color[rgb]{0,0,0}\makebox(0,0)[lb]{\smash{$e_3$}}}%
    \put(0.27215045,0.35122205){\color[rgb]{0,0,0}\makebox(0,0)[lb]{\smash{$e_4$}}}%
    \put(0.22936201,0.3019432){\color[rgb]{0,0,0}\makebox(0,0)[lb]{\smash{$e_5$}}}%
    \put(0.18400942,0.35122205){\color[rgb]{0,0,0}\makebox(0,0)[lb]{\smash{$e_6$}}}%
    \put(0,0){\includegraphics[width=\unitlength,page=3]{fig32_cycle.pdf}}%
    \put(0.4939868,0.39314193){\color[rgb]{0,0,0}\makebox(0,0)[lb]{\smash{$v_6 = v_1$}}}%
    \put(0.68509258,0.45546697){\color[rgb]{0,0,0}\makebox(0,0)[lb]{\smash{$w_1 = w_2 = v_2 = v_3$}}}%
    \put(0.73961341,0.39315602){\color[rgb]{0,0,0}\makebox(0,0)[lb]{\smash{$w_3 = w_4$}}}%
    \put(0.68509258,0.33083088){\color[rgb]{0,0,0}\makebox(0,0)[lb]{\smash{$v_4  = v_5 = w_5 = w_6$}}}%
  \end{picture}%
\endgroup%

\end{center}

Strictly speaking, we should consider a pair partition $p \in \salg{P}_2(2k)$ as a basis element of the Brauer algebra (see, e.g., \cite{HR05}); however, we will only need the very basics of this structure. In particular, we consider a partition $p$ as a graph on $2k$ vertices. We arrange the vertices into two evenly distributed rows, the first of which we consider as given by $1, 2, \ldots, k$; the second by $k+1, k+2, \ldots, 2k$. We then connect the vertices in a given block of $p$ with a line. In this way, we obtain a graph with $k$ connected components, each of size two. For two pair partitions $p_1, p_2 \in \salg{P}_2(2k)$, we define $p_1 \circ p_2$ as the graph obtained by overlaying the two graphs corresponding to $p_1$ and $p_2$ respectively, which we can again interpret as a partition $p_1 \circ p_2 \in \salg{P}(2k)$. The correspondence \eqref{eq:A.7} and \eqref{eq:A.8} between the pairs $(p_1, p_2)$ and $(\pi_1, \pi_2)$ pushes forward to a correspondence between the blocks of $p_1 \circ p_2$ and the anti-directed cycles $\mfk{C}(\pi_1, \pi_2)$. In particular, we have a cardinality-preserving bijection
\begin{equation}\label{eq:A.10}
\op{blocks}(p_1 \circ p_2) \cong \mfk{C}(\pi_1, \pi_2), \qquad b \mapsto C_b,
\end{equation}
where $\#(b) = \#(C_b)$. Indeed, we construct this bijection as follows. For the partition $p_1$ (resp., $p_2$), we imagine the vertices $\ell \in [2k]$ in its graph as the vertices $w_\ell \in V$ (resp., $v_\ell \in V$). In this way, a block $b$ of $p_1 \circ p_2$ then naturally corresponds to a cycle $C \in \mfk{C}(\pi_1, \pi_2)$ in the form of \eqref{eq:A.8}.


Finally, we need to understand the asymptotics of the Weingarten term $\inn{p_1}{\op{Wg}_n(p_2)}$ in \eqref{eq:A.6}. Theorem 3.13 in \cite{CS06} shows that
\[
\inn{p_1}{\op{Wg}_n(p_2)} = n^{-2k + \frac{\#(p_1p_2)}{2}}\prod_{b \in \op{blocks}(p_1 \circ p_2)} (-1)^{\frac{\#(b)}{2}-1}c_{\frac{\#(b)}{2}} + O(n^{-2k + \frac{\#(p_1p_2)}{2} - 1}).
\]
We can rewrite this in terms of $\mfk{C}(\pi_1, \pi_2)$ grace of \eqref{eq:A.9} and \eqref{eq:A.10} to obtain the equivalent asymptotic
\[
\inn{p_1}{\op{Wg}_n(p_2)} = n^{-2k + \#(\mfk{C}(\pi_1, \pi_2))}\prod_{C \in \mfk{C}(\pi_1, \pi_2)} (-1)^{\frac{\#(C)}{2}-1}c_{\frac{\#(C)}{2}} + O(n^{-2k + \#(\mfk{C}(\pi_1, \pi_2)) - 1}).
\]
At this point, we reintroduce this asymptotic for our matrix integral \eqref{eq:A.6} back into the injective trace \eqref{eq:A.5}. This reduces the problem to computing
\begin{equation}\label{eq:A.11}
S_{(\pi_1, \pi_2)} = \lim_{n \to \infty} n^{\#(V) - 1 - 2k + \#(\mfk{C}(\pi_1, \pi_2))}\bigg(\prod_{C \in \mfk{C}(\pi_1, \pi_2)} (-1)^{\frac{\#(C)}{2}-1}c_{\frac{\#(C)}{2}} + O(n^{-1})\bigg)
\end{equation}
for a given pair $(\pi_1, \pi_2)$ as before. To this end, we introduce the bipartite multigraph $\mfk{G} = (\mfk{V}, \mfk{E})$, where $\mfk{V} = V \cup \mfk{C}(\pi_1, \pi_2)$ is the union of the vertices of our original graph $T$ and the anti-directed cycle partition $\mfk{C}(\pi_1, \pi_2)$ of the edges $E$ of $T$. We draw an edge between a vertex $v \in V$ and a cycle $C \in \mfk{C}(\pi_1, \pi_2)$ if $v$ is a vertex in the cycle $C$, in which case the edge comes with multiplicity equal to the number of occurrences of $v$ in $C$ as an undirected cycle. For example, if $C$ is a simple cycle, then we only draw one edge between $v$ and $C$.

\begin{center}
\begingroup%
  \makeatletter%
  \providecommand\color[2][]{%
    \errmessage{(Inkscape) Color is used for the text in Inkscape, but the package 'color.sty' is not loaded}%
    \renewcommand\color[2][]{}%
  }%
  \providecommand\transparent[1]{%
    \errmessage{(Inkscape) Transparency is used (non-zero) for the text in Inkscape, but the package 'transparent.sty' is not loaded}%
    \renewcommand\transparent[1]{}%
  }%
  \providecommand\rotatebox[2]{#2}%
  \ifx\svgwidth\undefined%
    \setlength{\unitlength}{468bp}%
    \ifx\svgscale\undefined%
      \relax%
    \else%
      \setlength{\unitlength}{\unitlength * \real{\svgscale}}%
    \fi%
  \else%
    \setlength{\unitlength}{\svgwidth}%
  \fi%
  \global\let\svgwidth\undefined%
  \global\let\svgscale\undefined%
  \makeatother%
  \begin{picture}(1,0.42307692)%
    \put(0,0){\includegraphics[width=\unitlength,page=1]{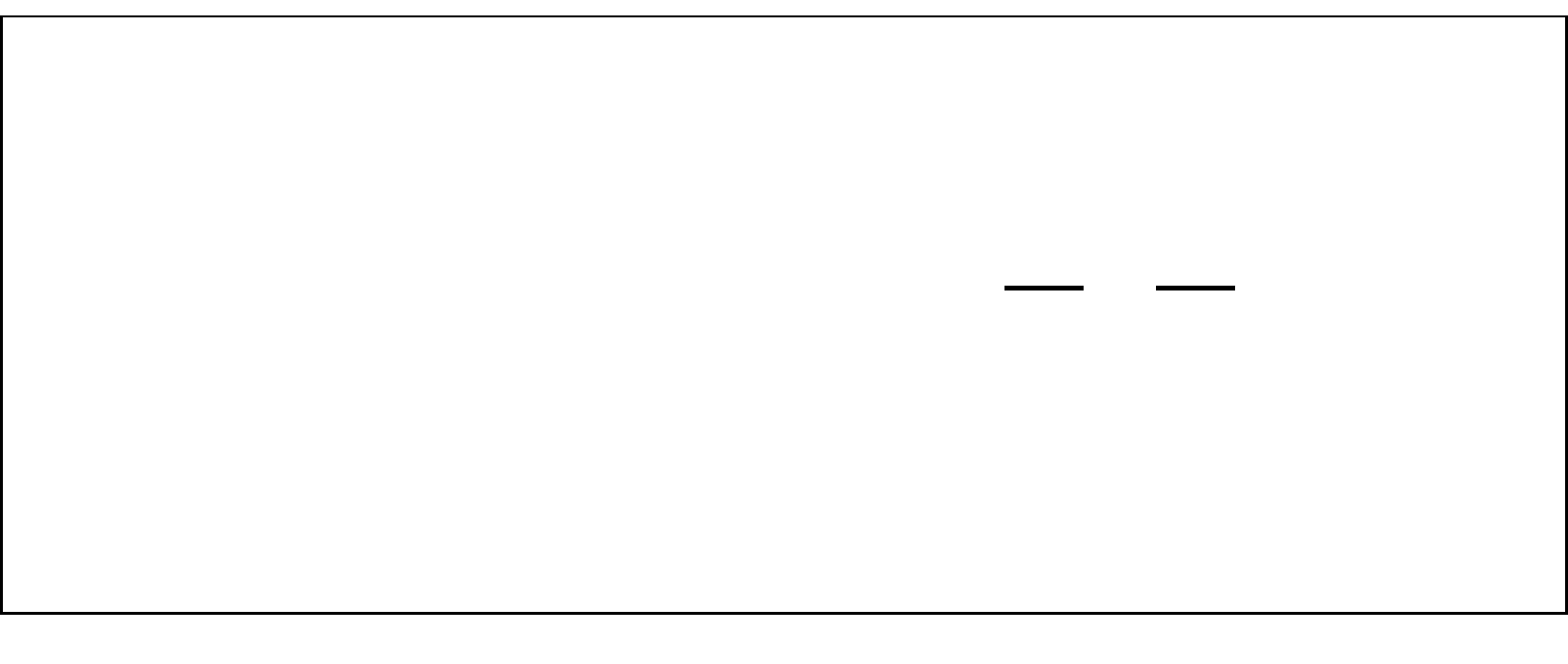}}%
    \put(-0.00141487,0.01905427){\color[rgb]{0,0,0}\makebox(0,0)[lt]{\begin{minipage}{0.99822199\unitlength}\raggedright Figure 29: An example of the construction of the graph $\mfk{G} = (\mfk{V}, \mfk{E})$. Here, we start with\end{minipage}}}%
    \put(0.12170717,1.4396419){\color[rgb]{0,0,0}\makebox(0,0)[lt]{\begin{minipage}{0.11106843\unitlength}\raggedright \end{minipage}}}%
    \put(0,0){\includegraphics[width=\unitlength,page=2]{fig33_bipartite.pdf}}%
    \put(0.24499982,0.0577527){\color[rgb]{0,0,0}\makebox(0,0)[lb]{\smash{$\mfk{C}(\pi, \pi_2)$}}}%
    \put(0.70429471,0.05600365){\color[rgb]{0,0,0}\makebox(0,0)[lb]{\smash{$\mfk{G}$}}}%
    \put(0,0){\includegraphics[width=\unitlength,page=3]{fig33_bipartite.pdf}}%
    \put(0.48878989,0.23262834){\color[rgb]{0,0,0}\makebox(0,0)[lb]{\smash{$\Rightarrow$}}}%
  \end{picture}%
\endgroup%

\end{center}

\noindent the anti-directed cycle $\mfk{C}(\pi_1, \pi_2) = \{C\}$ from Figure \hyperref[fig32_cycle]{28}. We color the vertices to clarify the construction.\vspace{.5cm}

By construction,
\[
\#(\mfk{V}) = \#(V) + \#(\mfk{C}(\pi_1, \pi_2)) \quad \text{and} \quad \#(\mfk{E}) = \#(E) = 2k.
\]
Moreover, the graph $\mfk{G}$ is clearly connected (by virtue of the connectedness of $T$), whence
\[
\#(\mfk{V}) \leq \#(\mfk{E}) + 1.
\]
This allows us to recast \eqref{eq:A.11} as
\begin{align*}
S_{(\pi_1, \pi_2)} &= \lim_{n \to \infty} n^{\#(\mfk{V}) - (\#(\mfk{E}) + 1)}\bigg(\prod_{C \in \mfk{C}(\pi_1, \pi_2)} (-1)^{\frac{\#(C)}{2}-1}c_{\frac{\#(C)}{2}} + O(n^{-1})\bigg) \\
&= \indc{\mfk{G} \text{ is a tree}}\prod_{C \in \mfk{C}(\pi_1, \pi_2)} (-1)^{\frac{\#(C)}{2}-1}c_{\frac{\#(C)}{2}}.
\end{align*}

Assume that $\mfk{G}$ is a tree. Of course, in this case, $\mfk{G}$ cannot have any multi-edges, which implies that each cycle $C \in \mfk{C}(\pi_1, \pi_2)$ is simple. In fact, the treeness of $\mfk{G}$ implies that $T$ is an orthogonal cactus. Indeed, the tree $\mfk{G}$ contains all of the information for how to properly grow the cactus $T$ from the simple anti-directed cycles $\mfk{C}(\pi_1, \pi_2)$. We describe this algorithm, as suggested at the beginning of the section. Start with an arbitrary pad $C_0 \in \mfk{C}(\pi_1, \pi_2)$ (level 0) and grow (i.e., attach) the pads $C_1 \in \mfk{C}(\pi_1, \pi_2)$ at distance two away from $C_0$ in $\mfk{G}$. Note that the pads introduced at level 1 cannot intersect outside of $C_0$ (this would contradict the treeness of $\mfk{G}$). We then introduce the pads $C_2 \in \mfk{C}(\pi_1, \pi_2)$ at distance four away from $C_0$ in $\mfk{G}$ (level 2). Each pad at level 2 is only  attached to a single pad at level 1 and can only intersect another pad at level 2 in a vertex of a pad $C_1$. We continue this process until we run out of pads. If we imagine rooting the graph $\mfk{G}$ at the vertex $C_0$ and orienting the rest of the graph upwards, then this process simply amounts to contracting the edges of $\mfk{V}$ as we move up.  

On the other hand, if $T$ is an orthogonal cactus, then there is a unique pair of pair partitions $(p_1, p_2)$ such that $\delta_{\mbf{i}}(p_1) = \delta_{\mbf{j}}(p_2) = 1$ in \eqref{eq:A.6}. The associated pair of partitions $(\pi_1, \pi_2)$ will then correspond precisely to the cycles of this cactus. In this way, we finally arrive at the prescribed limit \eqref{eq:A.4}.
\end{proof}

Naturally, one can of course ask the same question for a family of independent $n \times n$ Haar orthogonal matrices $(\mbf{O}_n^{(i)})_{i \in I}$. We can use the same approach to prove the existence of a joint LTD, now supported on \emph{colored orthogonal cacti} (i.e., cacti with anti-directed pads such that each pad is of a uniform color). We leave the details to the interested reader. Instead, we note that the same result can be obtained via Theorem \hyperref[thm2.5.5]{2.5.5}. One need only to prove the factorization property \eqref{eq:2.15} for $\mbf{O}_n$, which now follows as in the unitary case \cite[Proposition 6.2]{Mal11}. In particular, we note that the family $(\mbf{O}_n^{(i)})_{i \in I}$ is asymptotically traffic independent.

\begin{center}
\begingroup%
  \makeatletter%
  \providecommand\color[2][]{%
    \errmessage{(Inkscape) Color is used for the text in Inkscape, but the package 'color.sty' is not loaded}%
    \renewcommand\color[2][]{}%
  }%
  \providecommand\transparent[1]{%
    \errmessage{(Inkscape) Transparency is used (non-zero) for the text in Inkscape, but the package 'transparent.sty' is not loaded}%
    \renewcommand\transparent[1]{}%
  }%
  \providecommand\rotatebox[2]{#2}%
  \ifx\svgwidth\undefined%
    \setlength{\unitlength}{468bp}%
    \ifx\svgscale\undefined%
      \relax%
    \else%
      \setlength{\unitlength}{\unitlength * \real{\svgscale}}%
    \fi%
  \else%
    \setlength{\unitlength}{\svgwidth}%
  \fi%
  \global\let\svgwidth\undefined%
  \global\let\svgscale\undefined%
  \makeatother%
  \begin{picture}(1,0.52403846)%
    \put(0,0){\includegraphics[width=\unitlength,page=1]{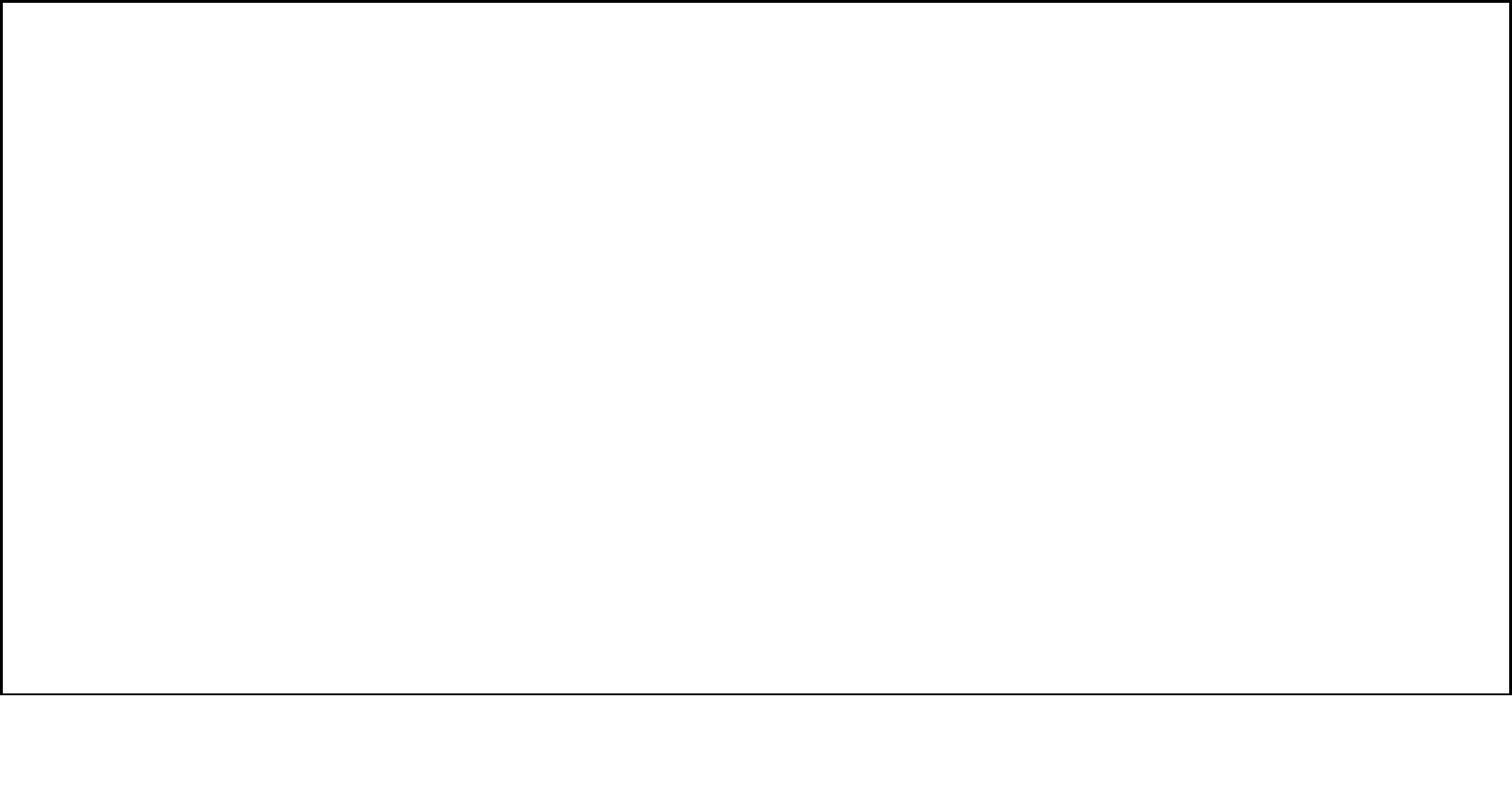}}%
    \put(-0.00141487,0.04885193){\color[rgb]{0,0,0}\makebox(0,0)[lt]{\begin{minipage}{0.99822199\unitlength}\raggedright Figure 30: An example of the construction of the graph $\mfk{G}$ for a colored version $T$ of the orthogonal cactus in Figure \hyperref[fig31_cacti]{27}. We color the edges of the tree $\mfk{G}$ to clarify the construction. \end{minipage}}}%
    \put(0.12170717,1.42761725){\color[rgb]{0,0,0}\makebox(0,0)[lt]{\begin{minipage}{0.11106843\unitlength}\raggedright \end{minipage}}}%
    \put(0.26643412,0.09427655){\color[rgb]{0,0,0}\makebox(0,0)[lb]{\smash{$T$}}}%
    \put(0.72316491,0.09466426){\color[rgb]{0,0,0}\makebox(0,0)[lb]{\smash{$\mfk{G}$}}}%
    \put(0,0){\includegraphics[width=\unitlength,page=2]{fig34_bipartite.pdf}}%
  \end{picture}%
\endgroup%

\end{center}

\bibliographystyle{amsalpha}
\bibliography{traffics}

\end{document}